\documentclass[12pt,a4paper]{amsart}
\makeatletter
\renewcommand\normalsize{%
	\@setfontsize\normalsize{11.7}{14pt plus .3pt minus .3pt}%
	\abovedisplayskip 10\p@ \@plus4\p@ \@minus4\p@
	\abovedisplayshortskip 6\p@ \@plus2\p@
	\belowdisplayshortskip 6\p@ \@plus2\p@
	\belowdisplayskip \abovedisplayskip}
\renewcommand\small{%
	\@setfontsize\small{9.5}{12\p@ plus .2\p@ minus .2\p@}%
	\abovedisplayskip 8.5\p@ \@plus4\p@ \@minus1\p@
	\belowdisplayskip \abovedisplayskip
	\abovedisplayshortskip \abovedisplayskip
	\belowdisplayshortskip \abovedisplayskip}
\renewcommand\footnotesize{%
	\@setfontsize\footnotesize{8.5}{9.25\p@ plus .1pt minus .1pt}
	\abovedisplayskip 6\p@ \@plus4\p@ \@minus1\p@
	\belowdisplayskip \abovedisplayskip
	\abovedisplayshortskip \abovedisplayskip
	\belowdisplayshortskip \abovedisplayskip}
\ifdefined\pdfpagewidth
\else
\fi
\calclayout
\makeatother

\usepackage{times}
\usepackage{todonotes}

\usepackage{amsmath,amsfonts,amsthm,amssymb,amscd}
\usepackage{graphicx, epsfig, psfrag}
\usepackage{pb-diagram}
\usepackage{hyperref}
\usepackage{epstopdf,yfonts}
\usepackage{fancyhdr}
\usepackage{tikz}
\usepackage[all]{xy}
\usepackage{color}
\usepackage{xargs}
\usepackage{comment}   
\usepackage[sc]{mathpazo}
\usepackage{eulervm}
\usepackage{float}
\usepackage[refpage]{nomencl}

\theoremstyle{plain}
\newtheorem{cor}{Corollary}[section]
\newtheorem{thm}[cor]{Theorem}
\newtheorem{prop}[cor]{Proposition}

\newtheorem{lemma}[cor]{Lemma}
\newtheorem{lem}[cor]{Lemma}
\theoremstyle{definition}
\newtheorem{defi}[cor]{Definition}
\newtheorem{defn}[cor]{Definition}
\theoremstyle{remark}
\newtheorem{remark}[cor]{Remark}

\newtheorem{ex}[cor]{Example}
\newtheorem{exs}[cor]{Examples}
\newtheorem{example}[cor]{Example}

\newtheorem{cv}[cor]{Convention}
\newtheorem{ppts}[cor]{Properties}

\newcommand\bqn{\begin{equation*}}
\newcommand\eqn{\end{equation*}}
\newcommand\bq{\begin{equation}}
\newcommand\eq{\end{equation}}
\newcommand\be{\begin{enumerate}}
\newcommand\ee{\end{enumerate}}
\newcommand\bei{\begin{itemize}}
\newcommand\eei{\end{itemize}}
\newcommand\ba{\begin{aligned}}
\newcommand\ea{\end{aligned}}
\newcommand\ban{\begin{aligned*}}
\newcommand\ean{\end{aligned*}}
\newcommand{\bsm}{\left(\begin{smallmatrix}}
\newcommand{\esm}{\end{smallmatrix}\right)}                   
\newcommand{\bpm}{\begin{pmatrix}}
\newcommand{\epm}{\end{pmatrix}}

\newcommand{\thismonth}{\ifcase\month 
	\or January\or February\or March\or April\or May\or June%
	\or July\or August\or September\or October\or November%
	\or December\fi}

\newcommand{\calA}{\mathcal A}
\newcommand{\calB}{\mathcal B}

\newcommand{\calC}{\mathcal C}
\newcommand{\calD}{\mathcal D}
\newcommand{\calE}{\mathcal E}
\newcommand{\calF}{\mathcal F}
\newcommand{\calG}{\mathcal G}
\newcommand{\calH}{\mathcal H}
\newcommand{\calI}{\mathcal I}
\newcommand{\calL}{\mathcal L}
\newcommand{\calM}{\mathcal M}
\newcommand{\calN}{\mathcal N}
\newcommand{\calO}{\mathcal O}
\newcommand{\calP}{\mathcal P}

\newcommand{\calR}{\mathcal R}
\newcommand{\calS}{\mathcal S}
\newcommand{\calT}{\mathcal T}
\newcommand{\calU}{\mathcal U}
\newcommand{\calV}{\mathcal V}

\newcommand{\calX}{\mathcal X}

\newcommand{\calZ}{\mathcal Z}

\newcommand{\fa}{\mathfrak a}

\newcommand{\fg}{\mathfrak g}

\newcommand{\fk}{\mathfrak k}

\newcommand{\fp}{\mathfrak p}

\newcommand{\CC}{\mathbb C}

\newcommand{\EE}{\mathbb E}
\newcommand{\FF}{\mathbb F}

\newcommand{\HH}{\mathbb H}

\newcommand{\KK}{\mathbb K}
\newcommand{\LL}{\mathbb L}
\newcommand{\MM}{\mathbb M}
\newcommand{\NN}{\mathbb N}

\newcommand{\PP}{\mathbb P}
\newcommand{\QQ}{\mathbb Q}
\newcommand{\RR}{\mathbb R}

\newcommand{\ZZ}{\mathbb Z}

\newcommand{\bB}{\mathbf B}

\newcommand{\bG}{\mathbf G}
\newcommand{\bK}{\mathbf K}
\newcommand{\bH}{\mathbf H}

\newcommand{\bL}{\mathbf L}

\newcommand{\bN}{\mathbf N}

\newcommand{\bP}{\mathbf P}
\newcommand{\bQ}{\mathbf Q}

\newcommand{\bS}{\mathbf S}
\newcommand{\bT}{\mathbf T}
\newcommand{\bU}{\mathbf U}

\newcommand{\bGmPI}{\quotient{\bG}{{}_\KK\bP_I}}
\newcommand{\bGmQ}{\quotient{\bG}{\bQ}}

\newcommand{\Gr}{{\rm Gr}}

\newcommand{\SL}{\operatorname{SL}}
\newcommand{\Sp}{\operatorname{Sp}}
\newcommand{\PSp}{\operatorname{PSp}}
\newcommand{\PSL}{\operatorname{PSL}}

\newcommand{\SO}{\operatorname{SO}}
\newcommand{\GL}{\operatorname{GL}}

\newcommand{\Aut}{\operatorname{Aut}}
\newcommand{\Out}{\operatorname{Out}}

\newcommand{\Iso}{\operatorname{Iso}}

\newcommand{\Hom}{\mathrm{Hom}}

\newcommand{\Teich}{\mathrm{Teich}}

\newcommand{\<}{\langle}
\renewcommand{\>}{\rangle}
\newcommand{\Ad}{\mathrm{Ad}}
\newcommand{\cons}{\mathrm{c}}
\newcommand{\cl}{\mathrm{cl}}
\newcommand{\Cone}{\mathrm{Cone}}
\newcommand\diag{\mathrm{diag}}

\newcommand{\ev}{\mathrm{ev}}
\newcommand{\Graph}{\mathrm{Graph}}
\newcommand\Id{\mathrm{Id}}
\newcommand{\lcm}{\mathrm{lcm}}
\newcommand{\Lie}{\mathrm{Lie}}
\newcommand{\Min}{\mathrm{Min}}
\newcommand\ov{\overline}
\newcommand\per{\mathrm{per}}

\newcommand{\Ret}{\mathrm{Ret}}
\newcommand\Stab{\mathrm{Stab}}
\newcommand\tr{\mathrm{tr}}
\newcommand\wt{\widetilde}
\newcommand\Prox{\mathrm{Prox}}

\newcommand\Rspec{{\operatorname{RSp}}}
\newcommand\Rspecc{{\operatorname{RSp}_\mathrm{cl}}}

\newcommand{\qbarr}{\overline{\mathbb Q}^\mathrm{r}}

\newcommand\rol{{\mathbb R^\omega_{\pmb{\mu}}}}
\newcommand\grom{\bG({\mathbb R^\omega_{\pmb{\mu}}})}
\newcommand\rom{{\mathbb R^\omega_{\pmb{\mu}}}}
\newcommand\rhol{{\rho^\omega_{\pmb{\mu}}}}
\newcommand\rhomp{{\rho^{\omega'}_{\pmb{\mu}'}}}
\newcommand\rhom{{\rho^\omega_{\pmb{\mu}}}}

\newcommand\oK{{\KK}^\omega}

\newcommand{\oKl}{\KK^\omega_{\pmb{\mu}}}

\newcommand{\ox}{x^\omega}
\newcommand{\oxl}{x^\omega_{\pmb{\mu}}}
\newcommand{\plambda}{{\pmb{\mu}}}
\newcommand{\prlambda}{{\pmb{\lambda}}}
\newcommand{\pmu}{{\pmb{\mu}}}

\newcommand\fabp{\overline{\mathfrak a}^+}
\newcommand{\fap}{\overline{\mathfrak a}^+}

\newcommand\bgof{\mathcal B_{\bG(\mathbb F)}}

\newcommand\bgf{\bgof}

\newcommand\bgrol{\mathcal B_{\bG(\mathbb R^\omega_{\pmb{\mu}})}}
\newcommand\bgrom{\mathcal B_{\bG(\mathbb R^\omega_{\pmb{\mu}})}}
\newcommand\bgromp{\mathcal B_{\bG({\mathbb R^{\omega'}_{\pmb{\mu}'}})}}

\newcommand\obgof{\overline{\mathcal B_{\bG(\mathbb F)}}}
\newcommand\obgf{\obgof}

\def\rsp#1{#1^\mathrm{RSp}}
\def\rspcl#1{#1^\mathrm{RSp}_\mathrm{cl}}

\def\thp#1{#1^\mathrm{WL}}
\def\thpN#1{#1^\mathrm{NL}}

\def\nl#1{#1^\mathrm{NL}}

\newcommandx{\ioo}[2]{I_{(#1,#2)}}
\newcommandx{\ioc}[2]{I_{(#1,#2]}}
\newcommandx{\ico}[2]{I_{[#1,#2)}}
\newcommandx{\icc}[2]{I_{[#1,#2]}}

\newcommand\bGFF{\mathbf G(\mathbb F)}
\newcommand\bGRR{\mathbf G(\mathbb R)}
\newcommand\GRR{ G_\mathbb R}
\newcommand\GFF{ G_\mathbb F}

\newcommand\G{\Gamma}
\newcommand\g{\gamma}

\newcommand{\frakT}{\mathfrak T}
\newcommand{\frakM}{\mathfrak M}
\newcommand{\frakK}{\mathfrak K}
\newcommand{\Ln}{\mathrm{Ln}}
\newcommand{\Log}{\mathrm{Log}}
\newcommand{\CR}{\mathrm{cr}}
\newcommand{\mindisp}{\lambda}
\newcommand{\norm}{g}
\newcommand{\valu}{\Lambda^0}

\newcommand{\Gal}{\mathrm{Gal}}

\newcommand{\Jord}{J}
\newcommand{\jord}{j}
\newcommand{\FP}{{\rm FP}}
\newcommand{\faL}{\mathfrak a_\Lambda}
\newcommand{\faZ}{\mathfrak a_\ZZ}

\newcommand\quotient[2]{
	\mathchoice
	{
		\text{\raise.5ex\hbox{$#1$}\big/\lower.5ex\hbox{$#2$}}%
	}
	{
		\text{\raise.25ex\hbox{$#1$}\big/\lower.25ex\hbox{$#2$}}%
	}
	{
		#1\,/\,#2
	}
	{
		#1\,/\,#2
	}
}
\newcommand\fracmod[2]{
	\mathchoice
	{
		\text{\lower.5ex\hbox{$#1$}\big\backslash\raise.5ex\hbox{$#2$}}%
	}
	{
		\text{\lower.25ex\hbox{$#1$}\big\backslash\raise.25ex\hbox{$#2$}}%
	}
	{
		#1\,\backslash\,#2
	}
	{
		#1\,\backslash\,#2
	}
}
\DeclareMathOperator{\Frac}{Frac} 

\makeindex

\begin{document}

\title[The real spectrum compactification]{The real spectrum compactification \\ of character varieties}
\begin{abstract}
We study the real spectrum compactification of character varieties of finitely generated groups in semisimple Lie groups. 
This provides a compactification with good topological properties, 
and we interpret the boundary points in terms of actions on building-like spaces.  
Among the applications we give a general framework guaranteeing the existence of equivariant harmonic maps in building-like spaces.
\end{abstract}

\author[M.~Burger]{Marc Burger}
\address{Departement Mathematik, ETHZ, R\"amistrasse 101, CH-8092 Z\"urich, Switzerland}
\email{burger@math.ethz.ch}

\author[A.~Iozzi]{Alessandra Iozzi}
\address{Departement Mathematik, ETHZ, R\"amistrasse 101, CH-8092 Z\"urich, Switzerland}
\email{iozzi@math.ethz.ch}

\author[A.~Parreau]{Anne Parreau}
\address{Univ. Grenoble Alpes, CNRS, Institut Fourier, F-38000  Grenoble, France}
\email{Anne.Parreau@univ-grenoble-alpes.fr}

\author[M.B.~Pozzetti]{Maria Beatrice Pozzetti}
\address{Dipartimento di Matematica, Universita' di Bologna, Piazza di Porta San Donato 5, 40126 Bologna, Italy}
\email{beatrice.pozzetti@unibo.it}

\thanks{We thank Xenia Flamm, Konstantin Andritsch,
	Raphael Appenzeller, Segev Gonen Cohen, Victor Jaeck and the anonimous referee for valuable comments. The first author thanks Peter Sarnak and Gopal Prasad for a very useful email exchange. Beatrice Pozzetti acknowledges support through Germany’s Excellence Strategy EXC-2181/1-390900948, the DFG project 338644254 and the DFG Emmy Noether grant 427903332.
Marc Burger and Alessandra Iozzi were partially supported by the SNF grant 2-77196-16.
Alessandra Iozzi acknowledges moreover support from U.S. National Science Foundation grants DMS 1107452, 1107263, 1107367 
"RNMS: Geometric Structures and Representation Varieties" (the GEAR Network). 
Marc Burger, Alessandra Iozzi and Beatrice Pozzetti would like to thank 
the Mathematisches Forschungsinstitut Oberwolfach for its hospitality
and the National Science Foundation under Grant No. 1440140 
that supported their residence at the Mathematical Sciences Research
Institute in Berkeley, California, during the semester of Fall 2019
where work on this paper was undertaken.
Anne Parreau is partially supported by ANR GoFR - ANR-22-CE40-0004,
and thanks the FIM for its hospitality.}

\date{\today}

\maketitle

\setcounter{tocdepth}{2} 
\tableofcontents

\section{Introduction}
\addtocontents{toc}{\protect\setcounter{tocdepth}{1}}
Let $\mathbf G<\SL(n,\mathbb C)$ be a connected semisimple algebraic group defined over the field of real algebraic numbers  $\ov{\mathbb Q}^r$
and let $\mathbf G(\mathbb R):=\mathbf G\cap\SL(n,\mathbb R)$.
The {\em character variety} $\Xi(\Gamma,\bGRR)$ of a finitely generated group $\Gamma$ in the real semisimple group $\bGRR$ 
is the quotient by $\bGRR$-conjugation of the topological space
$\Hom_\mathrm{red}(\Gamma, \bGRR)$ of reductive representations.
Our aim is to construct a compactification of $\Xi(\Gamma,\bGRR)$ with good topological properties and 
give an interpretation of its boundary points in terms of actions on building-like spaces.
Among the applications we obtain existence results concerning harmonic maps and structural properties of the Weyl chamber length  compactification of  
$\Xi(\Gamma,\bGRR)$.

\medskip
The character variety $\Xi(\Gamma,\bGRR)$ is homeomorphic, in an explicit way, 
to a real semialgebraic subset of $\RR^l$ defined over $\qbarr$, for an appropriate $l\in\NN$
(see \cite{RS} and \S~\ref{s.RS}).
As such it admits a {\em real spectrum compactification} $\Xi^\mathrm{RSp}(\Gamma,\bGRR)$, 
whose subset $\Xi^\mathrm{RSp}_\mathrm{cl}(\Gamma,\bGRR)$ of closed points is our main object of study. 
The realization of $\Xi(\Gamma,\bGRR)$ as a real semialgebraic set depends on a number of choices (see \S~\ref{s.char}), 
but any two such choices lead to semialgebraic sets that are related via a canonical semialgebraic homeomorphism. 
This allows to give an intrinsic model of the  real spectrum compactification  
$\Xi^\mathrm{RSp}(\Gamma,\bGRR)$ (see \S~\ref{s.canonicity}).


\subsection{Properties and characterizations of points in the real spectrum compactification}

Here are some key properties of the real spectrum compactification $\rsp\Xi(\Gamma,\bGRR)$ and of its closed points $\rspcl\Xi(\Gamma,\bGRR)$:
\begin{enumerate}
\item[(P1)]\label{item:2} there are injections $i$ and $i_\mathrm{cl}$ with homeomorphic dense image
	\bqn
	\xymatrix{
	&\rspcl{\Xi}(\Gamma,G)\ar@{^{(}->}[dd]\\
	\Xi(\Gamma,G)\quad\ar@{^{(}->}[ur]^{i_\mathrm{cl}}\ar@{^{(}->}[dr]_i&\\
	&\rsp{\Xi}(\Gamma,G)\,,
	}
	\eqn
	 the image of $i_\mathrm{cl}$ is open;
	\item[(P2)]\label{item:1} $\Xi^\mathrm{RSp}_\mathrm{cl}(\Gamma,\bGRR)$ is a compact metrizable space on which $\Out(\Gamma)$ acts by homeomorphisms;
		\item[(P3)]\label{item:3} the injection $i_\mathrm{cl}$ induces a bijection at the level of connected components;
	\item[(P4)]\label{item:4} the closure of every point contains a unique closed point, 
	and thus there is a continuous retraction $\rsp\Xi(\Gamma,\bGRR)\to \rspcl\Xi(\Gamma,\bGRR)$.
	
	\end{enumerate}
For the next property we observe that the group $\Out(\Gamma)$ acts by semialgebraic homeomorphisms on $\Xi(\Gamma,\bGRR)$, and hence by homeomorphisms on $\rsp\Xi(\Gamma,\bGRR)$, preserving the subset $\rspcl\Xi(\Gamma,\bGRR)$:
\begin{enumerate}

	\item[(P5)]\label{item:5} if $\calC\subset\Xi(\Gamma,\bGRR)$ is a connected component (hence a semialgebraic set) 
	that is contractible and invariant under  an element $\phi\in\Out(\Gamma)$, 
	then $\phi$ has a fixed point in $\calC^\mathrm{RSp}_\mathrm{cl}$. 
\end{enumerate}
The first four properties are standard for real semialgebraic sets and can be found in \cite[Chapter 7]{BCR}, while (P5) is due to Brumfiel, \cite{Brum88B, Brum92}.

\medskip
Our first result gives a geometric interpretation of the points in 
  \bq\label{eq:no_bd}
\rsp\partial\Xi(\Gamma,\bGRR)= \Xi^\mathrm{RSp}(\Gamma,\bGRR)\smallsetminus \Xi(\Gamma,\bGRR)\,,
\eq
where we are using this notation despite the fact that in general $\Xi(\Gamma,\bGRR)$ is not open in   $\Xi^\mathrm{RSp}(\Gamma,\bGRR)$.
Our considerations bring us to consider fields more general than $\RR$, so if $\FF$ is a real closed field
(see \S~\ref{subsec:intro1}) we denote by $\bGFF$ the ``extension'' of $\mathbf G(\qbarr)$ to $\FF$,
that is the algebraic group defined by the same polynomial equations as $\mathbf G(\qbarr)$
but with matrix entries in $\FF$ (see \S~\ref{s:Intro2}).
Given a representation $\rho:\Gamma\to {\bGFF}$,  
we say that the real closed field $\FF$  is {\em $\rho$-minimal} (or just {\em minimal})
if $\rho$ cannot be ${\bGFF}$-conjugated into a representation $\rho:\Gamma\to G(\LL)$, where $\LL\subset\FF$ is a proper real closed subfield.
We prove that if $\rho$ is reductive and if $\FF$ is real closed, 
a minimal real closed field $\FF_\rho\subset\FF$ always exists and is unique (Corollary~\ref{c.minimalfield}).
If $\FF_1$ and $\FF_2$ are two real closed fields, 
we say that  two representations $(\rho_1,\FF_1)$ and $(\rho_2,\FF_2)$ are {\em equivalent} 
if there exists an isomorphism 
$i:\FF_1\to\FF_2$ such that $i\circ\rho_1$ and $\rho_2$ are $G(\FF_2)$-conjugated.
%
 \begin{thm}[See Proposition~\ref{prop:Rspecrep} and Corollary~\ref{c.minimalfield}]\label{thm:1}
 	Points in $ \rsp\partial\Xi(\Gamma,\bGRR)$ are in bijective correspondence with equivalence classes of pairs 
	$(\rho,\FF_\rho)$ where $\rho:\Gamma\to \bG(\FF_\rho)$ is reductive and $\FF_\rho$ is real closed, non-Archimedean and minimal. 
	{Moreover $\FF_\rho$ is of finite transcendence degree over $\qbarr$.}
	\end{thm}

In particular, the pair $(\rho,\FF_\rho)$ represents
%
a fixed point of an automorphism $\phi\in\Out(\Gamma)$, 
if and only if there exists an automorphism $\alpha\in\Aut(\FF_\rho)$  
such that $\rho\circ\phi$ and $\alpha\circ\rho$ are conjugate in $\mathbf G(\mathbb F_\rho)$. 
We refer the reader to \S~\ref{sec.fixp} for applications to Gothen components and length functions. 

\medskip
If $\FF$ is a real closed fields of finite transcendency degree over $\qbarr$,
it admits an order compatible valuation, which is unique up to multiplication by a positive scalar (see \S~\ref{s.valuation} for details). 
However many results hold for more general real closed fields; we refer to  \S~\ref{subsec:intro1} for basic concepts pertaining 
to ordered fields and 
 details about real closed fields.
 
\subsection{Actions on CAT(0)-spaces and closed points}
We turn now to the characterization of 
\bqn
\rsp\partial_\cl \Xi(\Gamma,\bGRR):=\rspcl\Xi(\Gamma,\bGRR)\smallsetminus\Xi(\Gamma,\bGRR)\,,
\eqn
where now, contrary to \eqref{eq:no_bd}, this is a boundary in the topological sense thanks to (P1), 
and we will do so in terms of actions on certain metric spaces arising naturally in this context. 

We assume as we may that $\bG$ is invariant under transposition. 
Then the symmetric space $\calX_\RR$ associated to $\bGRR<\SL(n,\RR)$ has a natural semialgebraic model 
as the $\bGRR$-orbit of the identity in $\calP^1(n,\RR)$, the space of symmetric positive definite matrices with real coefficients of determinant one. 
For every real closed field $\FF$, $\calX_\FF\subset   \calP^1(n,\FF)$ denotes the $\FF$-extension of $\calX$ (see Definition \ref{d.extension}),
that is the subset of $\calP^1(n,\FF)$ defined by the same polynomial equalities and inequalities 
with coefficients in $\qbarr$ defining $\calX_\RR$.
The group $\bG(\FF)$ acts on $\calX_\FF$ and when $\FF$ is non-Archimedean 
we refer to $\calX_\FF$ as the \emph{non-standard symmetric space} associated to $\bG(\FF)$.

When $\FF$ admits a non-trivial order compatible $\RR$-valued valuation
we construct an invariant pseudodistance on $\calX_\FF$
whose metric quotient $\bgof$ is in general not complete but whose completion $\ov{\bgof}$  is CAT(0).  For more on $\bgof$ see \S~\ref{subsec:metric shadows}. 
\begin{thm}[See Theorem ~\ref{t.1.2text}]\label{thm:1.2}	
  Let $\bG<\SL(n,\CC)$ be a connected semisimple algebraic group defined over $\qbarr$,
$F=F^{-1}$ a finite generating set of $\Gamma$
and  $E:=F^{2^n-1}\subset\Gamma$. Let $(\rho,\FF)$ represent a point in $\rsp\partial\Xi(\Gamma,\bGRR)$, 
that is $\rho:\G\to \bG(\FF)$ is reductive, $\FF$ is real closed non-Archimedean and minimal. 
The following assertions are equivalent:
\be
\item\label{it:bdry_cl1} $(\rho,\FF)$ represents a closed point.
\item\label{it:bdry_cl2} The $\Gamma$-action on $\bgf$ does not have a global fixed point.
\item\label{it:bdry_cl3} There exists $\eta\in E$ such that $\rho(\eta)$ has positive translation length on $\ov\bgof$.
\ee
\end{thm}

The equivalence between the second and the third statement gives an effective way to determine 
when there is a global fixed point for a $\Gamma$-action on $\bgf$. 
A similar characterization for points in the real spectrum compactification of a {\em representation variety} $\Hom(\Gamma,\bGRR)$, 
as opposed to a character variety $\Xi(\Gamma,\bGRR)$, is much easier (see Proposition ~\ref{prop:Rspecrepcl} for a precise statement).


\subsection{Harmonic maps into CAT(0) spaces}
In many interesting cases the actions arising in our construction have strong properness properties.
Recall  that a $\Gamma$-action by isometries on a
CAT(0)-space $(X,d)$ is \emph{non-evanescent} (see \cite[\S 1.2]{MonJams2006} and \cite{KSch})
if for one, and hence any, finite generating set $F\subset \Gamma$, and for any $R\geq 0$ the set
\bqn
\left\{x\in X\,\left| \,\max_{g\in F} d(gx,x)\leq R\right.\right\}
\eqn
is bounded.  It is said {\em non-parabolic} if $\Gamma$ does not fix a point in the geodesic ray boundary of $X$.

%
Notice that the property of being non-evanescent is stronger than being non-parabolic, that is
if the $\Gamma$-action is non-evanescent, then $\Gamma$ does not fix a point in the geodesic ray boundary of $X$,
while the two properties are equivalent if $X$ is proper, that is if all closed balls of finite radius are compact.
Notice that the CAT(0)-spaces $\ov\bgof$ occurring in Theorem~\ref{thm:1.2} are never proper spaces.
 
\begin{thm}[See Theorem ~\ref{p.properact}]\label{thm:2}
Let $\frakM\subset \Xi(\Gamma,\bGRR)$ be a union of connected components consisting of non-parabolic representations. 
Then for every $(\rho,\FF)$ representing a point in $\rsp\partial\frakM$, the $\Gamma$-action on $\obgf$ is non-evanescent.
\end{thm}
%
 %
It then immediately follows from \cite{KSch} that

\begin{cor}
Let $\Gamma=\pi_1(M)$ where $M$ is a compact Riemannian manifold with universal covering $\wt M$
and let $\frakM\subset \Xi(\Gamma,\bGRR)$ be a union of connected components consisting of non-parabolic representations. 
Then for every $(\rho,\FF_\rho)$ representing a point in $\partial\rsp{\frakM}$ there exists a $\Gamma$-equivariant harmonic map 
$\wt M\to\overline{\mathcal B_{\bG(\mathbb F_\rho)}}$. 
\end{cor}
 For example if $\Gamma=\pi_1(\Sigma)$ where $\Sigma$ is a compact oriented surface without boundary, 
 and $\bGRR$ is a connected Lie group of Hermitian type, one can take $\frakM$ to be the set of maximal representations;
if $\bGRR$ is real split, one can take $\frakM$ to be the Hitchin component.
 
 \medskip

 %

\subsection{Relation with the Weyl chamber length spectrum
  compactification}
\label{subsec:RelWithWLcomp}
 We now turn to the relation between the real spectrum compactification 
and the Weyl chamber length compactification;
the latter has been studied by the third author in \cite{Parreau12} using Weyl chamber valued length functions. 

Let $\frakM$ be a union of connected components of $\Xi(\Gamma,\bGRR)$
that does not contain {\em bounded} representations (that is representations having a global fixed point in the symmetric space);
again, to keep an example in mind, one can think of Hitchin or maximal
representations. For any representation $\rho\in\frakM$, one can
consider the projectivized Weyl chamber length function
$\PP(\jord\circ\rho)$ on $\G$ obtained by projectivizing the
composition of $\rho$ with the Jordan projection $\jord:\bGRR\to \fap$
with values in a closed Weyl chamber $\fap$.
This generalizes the length function of the action on the symmetric space
and leads to the Weyl chamber length compactification
$\thp\frakM$ of $\frakM$ as follows (see \S~\ref{s.WL}).  
Let $\widehat{\frakM}=\frakM\cup\{*\}$ be the one point compactification of $\frakM$ and consider the map
\bq
\label{eq:WLcomp}
\begin{array}{ccccc}
\widehat{\mathbb P L}	:&	\frakM&\to&\widehat{\frakM}\times\mathbb P\left((\fap)^\Gamma\right)\\
	&\rho&\mapsto&(\rho,\mathbb P(\jord\circ\rho)).
\end{array}\eq
Then $\thp\frakM$ is the closure of $\widehat{\mathbb P L}(\frakM)$ and, by construction, is compact and contains $\widehat{\mathbb P L}(\frakM)$ as an open set.
When $\bG$ has real rank 1, this is the Morgan-Shalen compactification \cite{Mor-Sha}; it has been extended to higher rank in \cite{Parreau12}.

We show that the inclusion
$\frakM\hookrightarrow \thp{\frakM}$
extends to an explicit map 
\bq\label{eq:bdry}
\rspcl{\frakM}\to \thp{\frakM},
\eq
where the image of $[(\rho,\FF_\rho)]\in \partial\rspcl{\frakM}$ encodes a
 Weyl chamber valued length function of the $\rho(\Gamma)$-action
on  $\ov{\calB_{\bG(\FF_\rho)}}$ refining the length function discussed in Theorem~\ref{thm:1.2}.
%
Notice that  if $[(\rho,\FF_\rho)]\in\partial \rsp{\frakM}$ is not closed,
the corresponding length function vanishes because of Theorem~\ref{thm:1.2}.

 \begin{thm}[See Theorem ~\ref{t:Rspec-ThP}]\label{t.ThPINTRO}
 Let $\frakM$ be a union of connected components of $\Xi(\Gamma,\bGRR)$ that does not contain bounded representations.
 Let $\Out(\Gamma,\frakM)$ be the subgroup of the outer automorphism group of $\G$ preserving the set $\frakM$.
The map in  \eqref{eq:bdry} 
  \bqn
 \rspcl{\frakM}\twoheadrightarrow \thp{\frakM}
 \eqn
 is continuous, surjective and $\Out(\Gamma,\frakM)$-equivariant.
 \end{thm} 

There is no reason, in general, to expect that the boundary $\thp\partial\frakM:=\thp{\frakM}\setminus\frakM$ 
reflects topological properties of $\frakM$  as  the next example shows.  We see here also
that the map in  Theorem ~\ref{t.ThPINTRO} is not injective: 
\begin{ex}\label{ex:wolff}
Let $\Sigma$ be a closed oriented surface with genus greater than or equal to 2, 
$\Gamma=\pi_1(\Sigma)$ and $\bGRR=\PSL(2,\RR)$. The character variety $\Xi(\Gamma,\bGRR)$ has $4g-3$ connected components, 
distinguished by the $4g-3$ possible values of the Toledo invariant in $[2-2g, 2g-2]\cap \ZZ.$ 
Each of the components corresponding to the maximal value of the (absolute value of the) Toledo invariant
gives a copy of the Teichm\"uller space $\calT_g$.
The Weyl chamber length compactification $\calT_g^\mathrm{WL}$ coincides in this case with the Thurston compactification, 
and its boundary $\thp\partial\calT_g$ is identified with the set of projective measured laminations. 
Wolff  has shown in \cite[Theorem 1.1]{Wolff} that 
the boundary $\thp\partial\calC$ of any connected component $\calC\subset\Xi(\Gamma,\bGRR)$ 
different from the Teichm\"uller components and 
with non-vanishing Toledo invariant contains the boundary of $\thp\partial \calT_g$ as a subset with empty interior.
This is in contrast with the fact that $\Xi^\mathrm{RSp}(\Gamma,\bGRR)$ still has $4g-3$ connected components (see (P3)).
\end{ex}

One of the interests in the Weyl chamber length compactification is that 
the boundary length functions arise from $\Gamma$-actions on buildings of the form $\bgrol$
for appropriate $\plambda$ and $\omega$, \cite{Parreau12}. However neither the action giving rise to a point  in $\thp{\frakM}$, 
nor the associated Robinson field 
is by any means canonical;
see Example~\ref{e.twistbend} for a source of non-injectivity of the surjection in Theorem ~\ref{t.ThPINTRO}. 
A key point in the real spectrum compactification is that it arranges all such possible actions, up to suitable equivalence relation, 
in a compact metrizable space.

\subsection{Density of integral points, properness and real-valued
  length functions}
The real spectrum compactification can be used to deduce properties of
the Weyl chamber length compactification: for instance we  use it
to show that integral length functions are dense in $\thp\partial{\frakM}$,
generalizing the density of integral points in Thurston's boundary for
Teichm\"uller space, following the methods of \cite{Brum2}.
To be more precise, choose a Cartan subspace $\fa$ and let $\Phi\subset \fa^*$ denote the set of roots of $\bGRR$.
Given a subgroup $\Lambda<\RR$, let 
$$\faL:=\{v\in\fa\,|\; \alpha(v)\in\Lambda,\; \forall \alpha\in\Phi\}.$$
We then have
\begin{thm}[See Proposition ~\ref{p.finlen} and Corollary ~\ref{c.integrallengths}]\label{thmI:fl} Let $\frakM$ be a union of connected components of $\Xi(\Gamma,\bGRR)$ that does not contain bounded representations.
	\begin{enumerate}
		\item For every length function $L$ with $[(*,L)]\in\thp\partial\frakM$,
		there exists a finite dimensional $\QQ$-vector space $\Lambda<\RR$ such that $L(\gamma)\in \faL$ for all $\g\in\G$.
		\item The set of length functions $[(*,L)]$ with $L$ taking values in $\faZ$ is dense in $\thp\partial\frakM$.
	\end{enumerate}
	\end{thm}
On the other hand we deduce from results of Prasad--Rapinchuk \cite{PraRap} the following:
\begin{prop}[See Proposition ~\ref{p.infdim}]\label{pIn.inf}
Let $\rho:\G\to\bG(\RR)$ be non-elementary and $\G$ finitely generated. Then $\langle j(\rho(\g))|\;\g\in\G\rangle_\QQ$ is an infinite dimensional $\QQ$-vector subspace of $\fa$, where $j$ is the Jordan projection.
\end{prop}	
Here we say that a reductive representation is \emph{elementary} if its image is virtually contained in a compact extension of an Abelian group
(see Definition~\ref{def:elementary_rep}).

Combining Theorem ~\ref{t.ThPINTRO}, Theorem ~\ref{thmI:fl} (2) and Proposition ~\ref{pIn.inf} we show:

\begin{thm}[See Theorem ~\ref{thm:realnotfinite}]
	Let $\frakM$ be a connected component consisting of non-elementary representations, and consider the map 
	$$\begin{array}{cccc}
	{\mathbb P L}:&\frakM&\to&\mathbb P\left((\fap)^\G\right)\\
	&\rho&\mapsto&\mathbb P(j\circ\rho)\,.
	\end{array}$$ 
	Then ${\rm Im}({\mathbb P L})$ is open in its closure and ${\mathbb P L}:\frakM\to{\rm Im}({\mathbb P L})$ is proper.
\end{thm}

{Let us close this section with a discussion on real valued length
functions. 
In \S~\ref{s.realvaluedlength} we associate to a dominant weight 
of an almost faithful representation of $\bG(\CC)$ 
defined over $\RR$, and to every $[\rho]\in \Xi(\Gamma,\bGRR)$,
 a length function $\ell_N([\rho]) \in \RR^{\Gamma}_{\geq0}$; 
here the subscript $N$ refers to the semialgebraic norm associated to
said dominant weight, see \S~\ref{s.seminorm} and Example
\ref{e.fundExSemialgebraicNorm}.
Given $\frakM \subset \Xi(\Gamma,\bGRR)$ a union of connected components not
containing any bounded representation, we define as in 
\S~\ref{subsec:RelWithWLcomp} Equation \eqref{eq:WLcomp}
a corresponding length function compactification 
$\thpN{\frakM} \subset \mathbb P(\RR^{\G}_{\geq0})$.
One advantage of working with such length functions is that $\ell_N([\rho])$
can be explicitely defined for any representation 
$\rho:\G\to \bGFF$ where $\FF$ is real closed of finite transcendence
degree over $\qbarr$. We proceed then to construct a continuous
surjective map 
$\widehat{\mathbb P \ell_N}:\rspcl{\frakM}\to\thpN{\frakM}$
(Theorem \ref{t.lengthA}) in analogy with Theorem \ref{t.ThPINTRO}.
We obtain furthermore that the set of length functions in
$\thpN{\partial}\frakM$ taking values in $\ZZ$ is dense 
(Corollary \ref{c.denseZNlenght}).
Finally if $\frakM$ consists of non-elementary representations and
$$\mathbb P \ell_N:\frakM\to\mathbb P(\RR^{\G}_{\geq0})$$
is the map obtained by projectivizing length functions, we show that
$\mathbb P \ell_N(\frakM)$ is open in its closure 
$\overline{\mathbb P \ell_N(\frakM)}$
which is compact, and that the map
$\mathbb P \ell_N:\frakM\to\mathbb P \ell_N(\frakM)$
taking values in the locally compact space 
$\mathbb P \ell_N(\frakM)$ is proper (Theorem \ref{thm:PPLNproper}).
}

\subsection{Metric shadows and $\Lambda$-buildings}
\label{subsec:metric shadows}
The real valued length functions also arise as length functions for the action on the metric space $\ov\bgof$ endowed with suitable distances.
Let $\fa \subset \Lie(\bGRR)$ be a Cartan subspace; 
then (see \cite{Planche}) to every Weyl group invariant norm $\|\,\cdot\,\|$ on $\fa$ corresponds 
a $\bGRR$-invariant Finsler metric $d_{\|\,\cdot\,\|}$ on $\calX_\RR$ and every invariant Finsler metric is obtained this way. 
We show that if $\FF$ is real closed and admits an order compatible (non-Archimedean) valuation $\FF\to \RR\cup\{\infty\}$, 
every Weyl group invariant norm $\|\,\cdot\,\|$ on $\fa$ leads to a 
$\bG(\FF)$-invariant pseudodistance $d_{\|\,\cdot\,\|}^\FF$ on $\calX_\FF$;
these pseudodistances are all equivalent and thus the quotient $\bgof$ 
of $\calX_\FF$ by the $d_{\|\,\cdot\,\|}^\FF=0$ relation is independent of $\|\,\cdot\,\|$:
we call it the {\em metric shadow} of $\calX_\FF$. 
This object is canonical since any two order compatible valutations (when they exist) are real multiple of each other. 
With an abuse of notation we will still denote by 
$d_{\|\,\cdot\,\|}^\FF$ the distance induced on $\bgof$.
Then we have
\begin{thm}[See Corollary~\ref{cor:4.17}]\
	\begin{enumerate}
		\item The distance $d_{\|\,\cdot\,\|}^\FF$ has the midpoint property.
		\item If the norm $\|\,\cdot\,\|$ comes from a scalar product, 
		then $d_{\|\,\cdot\,\|}^\FF$ verifies the median inequality.
		In particular the completion $\obgof$ is CAT(0).
	\end{enumerate}
\end{thm}

A particularly interesting case is when $\FF=\rol$
is a Robinson field associated to a non-principal ultrafilter $\omega$ on $\NN$  and $\pmb\mu:=(\mu_n)_{n\in\NN}$  a sequence of positive reals with $\lim_\omega \mu_n=+\infty$. Indeed

\begin{thm}[See~Theorem \ref{thm:4.15} and Corollary~\ref{cor:4.17}]\
	\begin{enumerate}
		\item 
		Let $(\lambda_n)_{n\in\NN}$ be  a sequence of  reals with $\lim_\omega \lambda_n=+\infty$, and $\mu_n=e^{\lambda_{n}}$. Then 
		$(\bgrom,d_{\|\,\cdot\,\|}^\rom)$ is isometric 
		to the asymptotic cone of the sequence 
		$(\calX_\RR,\prlambda,\frac{d_{\|\,\cdot\,\|}}{\lambda_n})$.
		\item A valuation compatible
		field embedding  $\FF\hookrightarrow\rom$ induces an isometric embedding
		$$\ov\bgof\hookrightarrow\bgrom \;.$$
	\end{enumerate}
\end{thm}

Crucial in the proofs of the above results is 
the existence of a $\bGFF$-invariant semi-algebraic $\FF_{\geq 1}$-valued multiplicative distance function on $\calX_\FF$; 
combined with an order compatible valuation $v:\FF\to\RR\cup\{\infty\}$, 
this gives a pseudodistance on $\calX_\FF$ whose metric quotient is again $\bgof$ endowed now with a $v(\FF^*)$-valued distance. 
In fact, more generally one can perform this construction with any order compatible valuation $v:\FF\to\Lambda\cup\{\infty\}$ 
where the value group is now a, not necessarily Archimedean, ordered Abelian group $\Lambda$
and denote the quotient by $\mathcal B_{\bG(\mathbb F),\Lambda}$. One has then: 

\begin{thm}\cite{Appenzeller}
	Assume the the root system of $\bGRR$ is reduced. The metric quotient 
	$\mathcal B_{\bG(\mathbb F),\Lambda}$ has the structure of a $\Lambda$-building.
\end{thm}

In view of the applications to higher rank Teichmüller spaces we mention here that if $\bGRR$ admits a $\Theta$-positive structure in the sense of Guichard-Wienhard, then its root system is reduced.

{
\subsection{Perspectives}
While the theory developed in this article applies to any finitely
generated group, one motivation for developing it are its potential
applications to the study of compactifications of spaces of geometric
structures and to higher rank Teichmüller spaces.
One instance of geometric structures are projective structures on
manifolds. More precisely, let $M$ be a compact closed connected
manifold of dimension $d\geq 2$ with hyperbolic fundamental group
$\Gamma$. The space $\calP(M)$ of marked properly convex real projective
structures on $M$ identifies via the holonomy representation with a
union of connected components of $\Xi(\Gamma,\SL_{d+1}(\RR))$ \cite{BenoistIII}.
Thus $\calP(M)$ inherits the structure of a semialgebraic set and
thus $\rspcl{\calP(M)}$ provides a (Hausdorff) compactification of 
$\calP(M)$ (see Proposition \ref{prop:closed}) 
and 
\S~\ref{subsec:compactification_semialgebraic} Equation \eqref{e.rspS}). 
A major open problem is then the characterization of points  
in $\rspcl{\calP(M)}$ in terms of some geometric objects associated to
$M$. A step in this direction is \cite{FlPa} on the Hilbert geometry
over ordered valued fields: the authors associate to every
point in $\rspcl{\calP(M)}$ a real closed field $\FF$ and an action of
$\G$ on a properly convex subset $\Omega$ of $\PP^d(\FF)$ so that the induced action on the corresponding Hilbert metric space $X_\Omega$ does not have a global fixed
point  (\cite[Theorem B]{FlPa}),
thus echoing our Theorem \ref{thm:1.2}.

Turning to higher rank Teichmüller theory, let  $\Sigma$ be a  compact closed
connected surface of genus at least two and let $\G$ be its fundamental
group. A higher rank Teichmüller space is a connected component of the
character variety $\Xi(\Gamma,\bGRR)$ consisting entirely of faithful
representations with discrete image (see \cite[Definition 2]{Wienhard_ICM}). Examples
are given, when $\bGRR$ is of Hermitian type, by the subset of maximal
representations, which form a union of connected components of
$\Xi(\Gamma,\bGRR)$ (see \cite{BIW}), and when $\bG$ is $\RR$-split, by the
Hitchin components (see \cite{Labourie2006}). These examples fall into a more
general framework, namely when $\bGRR$ admits a $\Theta$-positive structure,
in which case one speaks of $\Theta$-positive representations (see
\cite{GWpos, GWalg}). It is known that $\Theta$-positive representations
are faithful with discrete image and it was shown recently that the
set $\Xi^\Theta(\Gamma,\bGRR)$ of $\Theta$-positive representations is a union of
connected components \cite{BGLPW}.
A characterization of points in $\Xi^\Theta(\Gamma,\bGRR)^\mathrm{RSp}$ is still
an open problem; it has however been obtained in the case of Hitchin
components \cite{Flamm} and in the case of maximal representations;
see \cite{BIPP-ann} and the forthcoming article 
\cite{BIPP_RspecMax}.
In \cite{BIPP_RspecMax} we apply the theory developed here to
establish the following property of the action of the mapping class
group $M(\Sigma)$ on $\Xi^\Theta(\Gamma,\bGRR)^\mathrm{RSp}$:

\begin{thm}
  Stabilizers in $M(\Sigma)$ of points in $\Xi^\Theta(\Gamma,\bGRR)^\mathrm{RSp}$ are
  virtually abelian.
\end{thm}

An important role here is played by Theorem~\ref{thm:1}
as well as the relation with geodesic currents. We explore
this relation in \S~\ref{s.crossratio} where we introduce the concept
of $k$-positively ratioed representation
$\rho:\G\to \SL_d(\FF)$ where $\FF$ is any real closed field and $1\leq 2k \leq
d$.
We then show that for every $k$-positively ratioed representation $\rho$, 
there is a geodesic current $\mu_\rho$ such that, if $i(\cdot,\cdot)$ denotes the
Bonahon intersection form, for every $\gamma\in\G\setminus\{\Id\}$
$$i(\mu_\rho,\gamma)=\ell_{N_k}(\rho(\gamma)) \ \ \text{(Corollary \ref{c.posrat})}$$
Here $N_k$ is the semialgebraic norm associated to the $k$-th
fundamental  weight of $\SL_d$. 
Given a component $\frakM \subset \Xi(\Gamma,\bGRR)$ where $\bG <\SL_d$, consisting
of $k$-positively ratioed representations, we use this to define a map 
$\PP \calC : \rspcl{\frakM}\to \PP(\calC(\Sigma))$ into the projectivization
of the space $\calC(\Sigma)$ of geodesic currents on $\Sigma$ and show that it
is continuous. We then conclude:

\begin{thm}[See Theorem \ref{t.crmain}]\label{thm:comm_diagr}
The diagram
	$$\xymatrix{\rspcl{\frakM}\ar[r]^{\mathbb P\calC}\ar[dr]_{\mathbb P\ell_{N_{k}}}&\mathbb P(\calC(\Sigma))\ar[d]^{i(\cdot,\cdot)}\\&\mathbb P(\RR^\G_{\geq 0})}$$
is a commutative diagram of continuous maps.
\end{thm}
}

\subsection{Tools and techniques I: Artin--Lang principle and framings}\label{subsec:T&TI}
We conclude the introduction by mentioning some of the tools that play an important role in the paper. 
The first is the Artin--Lang principle. 
A semialgebraic set $S$ is defined by finitely many  polynomial equalities and inequalities satisfied by elements of $\RR^n$.
Then its real spectrum compactification $S^{\mathrm{RSp}}$ organizes into a topological space the points
in all $\FF$-extensions of $S$ over all real closed extension $\FF\supset \qbarr$,
by means of a suitable equivalence relation.
One of the main tools used in the paper, that is essentially the \emph{Artin--Lang finiteness principle}, 
is that any polynomial inequality satisfied by all points in $S$ is satisfied by any point in its real spectrum. 

This has striking consequences when applied to character varieties. 
For example if all representations in a union $\frakM$ of connected components of a character variety
admit an equivariant boundary map, it allows to prove the existence of well behaved equivariant boundary maps 
for all representations in $\rsp\partial\frakM$. 
To give a more specific example, recall that a feature of all higher rank Teichm\"uller spaces, 
encompassed in the unifying framework of  $\Theta$-positive representations \cite{GWpos}, 
is that they consist of representations
admitting canonical boundary maps in appropriate flag manifolds
satisfying certain transversality properties.
%
One of the strengths of real algebraic methods in general, and of the Artin--Lang principle in particular, 
is that algebraic properties satisfied by representations in such components
are satisfied also by representations in their real spectrum compactification.
%
\begin{thm}[See Proposition ~\ref{p.dpCframing}]\label{thmINTRO:framing}
Let $\G$ be the fundamental group of a closed orientable surface of genus at least 2, 
and  let $\frakM\subset\Xi(\G,\bGRR)$ be a component consisting only of $\Theta$-positive representations. 
Then every representation in $\rsp{\frakM}$ is $\Theta$-positive.
\end{thm}	

See \S~\ref{s.framing} for other situations in which we can guarantee the existence of framings for representations in $\frakM^\Rspec$.


\subsection{Tools and techniques II: a concrete model of the character variety}
Given a finite generating set $F$ of $\Gamma$ with $|F|=f$, the representation variety $\Hom(\Gamma,\bGRR)$ is homeomorphic, 
via the evaluation map, to a real algebraic subset $\calR_F(\Gamma,\bGRR)$ of the vector space $M_{n,n}(\RR)^f$, 
and closed $\bGRR$-orbits correspond to reductive representations.
We will use and adapt the approach to geometric invariant theory due to Richardson--Slodowy \cite{RS}: 
using the natural scalar product on $M_{n,n}(\RR)^f$
they show that reductive representations, 
hence closed $\bGRR$-orbits, correspond exactly to the orbits containing a vector of minimal length.
The set  $\calM_F(\Gamma,\bGRR)$ of \emph{minimal vectors}  
turns out to be a coarse fundamental domain for the $\bGRR$-action on the set $\calR^\mathrm{red}_F(\Gamma,\bGRR)$ 
of reductive representations, in the sense that every $\bGRR$-orbit meets $\calM_F(\Gamma,\bGRR)$ 
and the intersection of a $\bGRR$-orbit with $\calM_F(\Gamma,\bGRR)$ is a $\bK(\RR)$-orbit, 
where $\bK(\RR):=SO(n)\cap\bGRR$ is a maximal compact subgroup. 
A further important property of $\calM_F(\Gamma,\bGRR)$ is that it is a real algebraic subset 
that is closed in $M_{n,n}(\RR)^f$ (and thus also in $\calR_F(\Gamma,\bGRR)$). 
The semialgebraic model for the character variety is obtained by taking the image of $\calM_F(\Gamma,\bGRR)$ 
in an appropriate affine space using a finite set of generators of the algebra of $\bK$-invariant polynomials.

Combining a properness statement implicit in \cite{RS}, results of Procesi \cite{Procesi} 
and a careful study of the geometry of symmetric spaces,
we obtain the following bound on the norm of a minimal vector $\rho$
in terms of the traces of the elements $h\in\Gamma$
that can be written as words of length at most $2^n-1$ in the generating set $F$. 
For such an element $h$ we denote by $|h|_F$ its word length with respect to $F$.
This result plays a crucial role in the characterization of closed points in Theorem ~\ref{thm:1.2}.
\begin{thm}[See Theorem ~\ref{t.traceanddisp}]\label{thmINTRO:Procesi}
  Let $\bG<\SL(n,\CC)$ be a connected semisimple algebraic group defined over $\qbarr$.
  Let $F$ be a finite generating set of $\Gamma$ with $|F|=f$, 
 let $E:=F^{2^n-1}\subset\Gamma$, and $m=\lcm\{1,\ldots, 2^n-1\}$. 
 Then there exists $c_1=c_1(f,n)$ such that for any representation $\rho:\G\to \bGRR$ with $(\rho(\g))_{\g\in F}$ minimal,
 the inequality
	\bqn
	\left(\sum_{\gamma\in F}\tr(\rho(\gamma)\rho(\gamma)^t)\right)^{2m}\leq n^m\left(c_1\sum_{h\in E}\tr(\rho(h))^{\frac {2m}{|h|_F}}\right)^{2(n-1)}\,.
	\eqn
	holds.
\end{thm}
For our purposes it is necessary to adapt the theory of Richardson--Slodowy to general real closed fields $\FF$, 
and for this reason it is important to ensure that all inequalities involved, 
as for example the statement of Theorem ~\ref{thmINTRO:Procesi}, are given by polynomials. 


\subsection{Tools and techniques III: Robinson field realization}
Another important theme in the paper is a bridge between geometric group theory and real algebraic geometry, 
given by the {Robinson field realization}. 
Given a non-principal ultrafilter $\omega$, 
we say that a sequence of scales $\plambda=(\mu_k)$ is \emph{well adapted} to a sequence of representations $\rho_k:\G\to\bGRR$ 
if there exist $c_1,c_2\in\RR$ with $c_i>0$, such that in the ultraproduct $\RR^\omega$ we have the inequalities 
$$c_1\plambda\leq \left(\sum_{\gamma\in F}\tr(\rho_k(\gamma)\rho_k(\gamma)^t)\right)_{k\in\NN}\leq c_2\plambda,$$
and it is \emph{adapted}   if just the second  inequality holds.
If one starts with a sequence of representations, a non-principal ultrafilter and an adapted sequence of scales, 
one obtains a corresponding action on the asymptotic cone that admits an algebraic description as $\calB_{\bG(\rol)}$. 
In this algebraic description the isometric action on the building comes from a representation $\rho:\Gamma\to \bG(\rol)$. 
For reductive representations,
%
  this leads to a point in the real spectrum compactification of the character variety. 
A remarkable fact linking real algebraic geometry and geometric group theory is that any point in the real spectrum compactification arises this way:

\begin{thm}[See Theorem~\ref{thm:7.16}]\label{thm:ggt}  Let $\omega$ a non-principal ultrafilter on $\NN$, 
let $((\rho_k,\RR))_{k\geq1}$ be a sequence of representations with 
$(\rho_k(\g))_{\g\in F}$ minimal for all $k\geq 1$, 
and let $(\rhol,\rol)$ be its $(\omega,\plambda)$-limit for an adapted sequence of scales $\plambda$.
Then:
\begin{itemize}
\item $\rhol$ is reductive.
\item Moreover if $\plambda$ is well adapted and an infinite element, and $\FF_{\rhol}$ is the $\rhol$-minimal field, 
then $\FF_{\rhol}\subset\rol$ and $(\rhol,\rol)$ is $\bK(\rol)$-conjugate to a representation $(\pi,\FF_{\rhol})$ 
that represents a closed point in $\rsp\partial\Xi(\Gamma,\bG(\RR))$.
\end{itemize}
Conversely, any $(\rho,\FF)$ representing a closed point in $\rsp\partial\Xi(\Gamma,\bG(\RR))$ arises in this way.
More precisely for any non-principal ultrafilter $\omega$  and any sequence of scales $\plambda$ giving an infinite element,  
there exist an order preserving field injection
$i\colon \FF\hookrightarrow\rol$ and
 a sequence of  representations $((\rho_k,\RR))_{k\geq1}$  with $(\rho_k(\g))_{\g\in F}$ minimal for all $k\geq 1$ for which $\plambda$ is well adapted,
and such that $i\circ\rho$ and $\rhol$ are $\bG(\rol)$-conjugate.
\end{thm}

This interplay is useful in both directions: on the one hand we use real algebraic methods to deduce properties of actions on asymptotic cones, 
such as the non-evanescence  as in  Theorem~\ref{thm:2};
 on the other hand, we use asymptotic cones and ultrafilters to establish properties of points in the real spectrum compactification: 
this is the key in proving the characterization of points in the real spectrum compactification of maximal character varieties achieved in \cite{BIPP_RspecMax}.

\subsection{Tools and techniques IV: accessibility}
The classical curve selection lemma in real algebraic geometry states 
that any point $x$ in the closure of a semialgebraic set $S\subset \RR^l$ is the endpoint of a continuous semialgebraic segment $s\colon[0,1)\to S$. 
We prove an analogous result for closed points in the real spectrum compactification of a semialgebraic set $S$; 
in this case, however, a point $x\in\rspcl S$ is accessible by a continuous ray $s\colon[0,1)\to S$ 
that we can only guarantee being locally semialgebraic (Proposition ~\ref{prop:2}).

This has as consequence the following property of the Weyl chamber length compactification, which is of independent interest: 
while $\thp{\frakM}$ might, in general, be topologically badly behaved (see Example ~\ref{ex:wolff}), 
its connected components are always path connected:

\begin{thm}[See Corollary ~\ref{c.length}]
	For every $(*,[L])\in\thp\partial\frakM$ there exists a continuous, locally semialgebraic path $\rho_t:[0,\infty)\to \Hom_{\rm red}(\Gamma,\bGRR)$ such that, for every $\gamma\in\Gamma$, we have
	\bqn
	\lim_{t\to\infty}\frac{j(\rho_t(\gamma))}{t}=L(\gamma)\,,
	\eqn
	where $j$ is the Jordan projection.
\end{thm}

\subsection*{Structure of the paper}
In \S~\ref{sec.prelim} we offer a self-contained introduction to real algebraic geometry and the real spectrum compactification. 
The setting there, and in the rest of the paper, is considerably more general than what is outlined in the introduction: 
for instance we do not restrict to $\qbarr$ as our field of definition, 
but rather consider general real closed fields $\qbarr \subset\KK\subset\RR$.
While character varieties of semisimple Lie groups could always be defined over $\qbarr$, 
different choices of $\KK$ give rise to different real spectrum compactifications, the one for $\qbarr$ being the smallest. 
If $\KK$ is countable, the set of closed points in the real spectrum compactification has a countable basis for its topology, and is thus metrizable.

In \S~\ref{s.RspGenProp} we discuss several general properties of the real spectrum compactification of semialgebraic sets that are not readily available in the literature, 
and are useful in our study of compactifications of character varieties. 
More specifically we prove the Robinson field realization and the accessibility described above
(respectively Theorem~\ref{thm:ggt} and Corollary~\ref{c.length}), 
as well as a continuity statement for real valued functions 
that is needed to relate the real spectrum compactification and the Weyl chamber length compactification.

In \S~\ref{sec:3} we adapt the classical theory of semisimple algebraic groups to hold over real closed fields $\FF$ 
and discuss extensions of the classical Cartan and Jordan decompositions to this setting, 
as well as properties of proximal elements and their fixed points in  suitable flag manifolds. 
We work there with semialgebraic semisimple Lie groups that are not necessarily algebraic.

In \S~\ref{sec:5} we carry out the program outlined above to construct
the CAT(0) space $\obgof$ and discuss the relation of its visual boundary with the set of parabolic subgroups introduced in \S~\ref{sec:3}.

In \S~\ref{s.symm} we discuss representation varieties.
We prove Theorem ~\ref{thm:2} and Theorem ~\ref{thmINTRO:framing} mentioned above, which fit more naturally in this context. 

In \S~\ref{s.char} we give the details of the construction of the character variety, introduced as second main tool, 
leading to the proof of Theorems ~\ref{thm:1} and ~\ref{thm:1.2}. 

In \S~\ref{s.WL} we discuss the relation of the Weyl chamber length compactification and the real spectrum compactification. 
The results there are more general than what is stated in the introduction, in particular because we do not restrict ourselves to connected components
of the character variety but we allow general closed semialgebraic subsets thereof. 

In \S~\ref{s.crossratio} we turn to the case in which $\Gamma=\pi_1(\Sigma)$ 
is the fundamental group of a compact closed connected surface $\Sigma$ and we prove Theorem~\ref{thm:comm_diagr}.
\addtocontents{toc}{\protect\setcounter{tocdepth}{2}}

\section{Preliminaries}\label{sec.prelim}

\subsection{Real closed fields}\label{subsec:intro1}

A general reference of this section is \cite[Ch. 1]{BCR}.
An {\em ordered field} is a field $\FF$ with an {\em ordering},
that is a total order relation $\leq$ satisfying the following properties:
\be
\item If $x,y\in \FF$ with $x\leq y$, then $x+z\leq y+z$ for all $z\in\FF$.
\item If $x,y\in\FF$ with $0\leq x$ and $0\leq y$, then $0\leq xy$.
\ee

In ordered fields  the absolute value $|x|\in\FF$ of
$x\in\FF$ is defined as $|x|=\max\{x, -x\}$.
An order is {\em Archimedean} if every $x\in\FF$ is bounded above by some $r\in\QQ$.
A {\em real closed field} $\FF$ is an ordered field that has no non-trivial algebraic extensions
to which the order extends. 
An ordered field $\FF$ is real closed if and only if it has the following two properties:
\be
\item Every polynomial of odd degree has a root in $\FF$;
\item Every positive element is a square;
\ee
in particular the order is then unique.
Given an ordered field $\LL$, an ordered algebraic extension $\FF$ is a {\em real closure} of $\LL$ 
if $\FF$ is real closed and the ordering on $\LL$ extends to the ordering of $\FF$.
The real closure denoted $\ov\LL^r$ \label{n.rc}
exists for any ordered field $\LL$ and is unique up to a unique order preserving isomorphism over $\LL$.
In other words if $\FF_1$ and $\FF_2$ are real closures of $\LL$, there exists a unique order preserving isomorphism
$\FF_1\to\FF_2$ that is the identity on $\LL$.

\begin{example}\label{ex:orders}\
	\be
	\item The field $\qbarr$  of real algebraic numbers and the field $\RR$ of real numbers are real closed fields.
	
	\item The field $\RR(X)$ of rational functions over $\RR$ admits many orders extending the order of $\RR$.
	For example we consider the order 
	\bqn
	f>0\text{ if and only if there exists }T\in\RR\text{ such that }f|_{(T,\infty)}>0\,.
	\eqn
	However $\RR(X)$ is not a real closed field as, for example, $\sqrt{X}\notin\RR(X)$.  
	Its real closure {with respect to the order introduced above }is contained in the field of formal Puiseux series with coefficients in $\RR$ 
(centered at infinity)
	\bqn
	\left\{\sum_{q\leq k}a_qX^{q/n}:\, n\in\NN^+, k\in\ZZ\right\}\,,
	\eqn
	where $q\in\ZZ$ and $a_q\in\RR$, and the order is given by $\sum_{q\leq k}a_qX^{q/n}>0$ if $a_k>0$.
	The field of formal Puiseux series is a real closed field,
	that is larger than the real closure $\ov{\RR(X)}^r$ of $\RR(X)$;
	the latter consists of the Puiseux series that are algebraic over $\RR(X)$
	and hence converge at infinity (\cite[Example 1.3.6]{BCR}). The real closure of $\RR(X)$, however, depends on the choice of the order.
	\ee
\end{example}

We remark that $\qbarr$ is a subfield of any real closed field, while
a real closed field $\FF$ is isomorphic to a subfield of $\RR$ if and
only if $\FF$ is Archimedean, \cite{Priess-Crampe}.

\medskip

An useful operation that produces new real closed fields from a given one is the reduction modulo an {\em order convex subring} \label{n.2}\nomenclature{$\calO$}{order convex subring of an ordered field $\FF$ } 
$\calO\subset\FF$; 
Recall that a subring $\calO$ in an ordered field $\FF$ 
is {\em order convex}
if whenever $0 \leq x \leq y$ and $y\in\calO$, also $x\in\calO$.
Then $\calO$ is a valuation ring
with maximal ideal
\bqn
\calI=\{x\in\calO:\, x\neq 0, \;x^{-1}\notin\calO\}\cup\{0\} \;,
\eqn
which is is also order convex.
This implies that the residue field \label{n.3}
$\FF_\calO:=\calO/\calI$ inherits a unique ordering compatible
with the projection, and is a real closed field if $\FF$ is.
For specific details about this construction see \cite[Chapter 3]{Lightstone_Robinson}.

Let now $\KK\subset \RR$ be a field, and $\omega$ an ultrafilter over $\NN$.
We consider the {\em hyper-$\KK$ field}, \label{n.4} 
defined as the ultraproduct 
\bqn
\oK:=\KK^\NN/\sim\,,
\eqn
which is by definition the quotient of the product ring $\KK^\NN$ by the equivalence relation
where $(x_k)_{k\in\NN}\sim(y_k)_{k\in\NN}$ if and only if the two
sequences coincide $\omega$-almost everywhere. The corresponding class
in $\oK$ will be denoted by $[(x_k)]_\omega$.
Notice that there is an embedding $\KK\hookrightarrow\oK$ given by $x\mapsto[(x,x,\dots,)]_\omega$
and that if $\omega$ is principal,  this embedding is an isomorphism.
If $\KK$ is an ordered field, then $\oK$ is also an ordered field
with order given by $[(x_k)]_\omega > 0$ if and only if $x_k>0$ 
for $\omega$-almost everywhere, 
and in fact it is real closed if $\KK$ is.
%

There is an important construction of order convex subrings
that lead to specializations of interest to us.
Observe that if $\omega$ is non-principal
there are positive {\em infinite} elements in $\oK$,
that is elements  larger than any integer.
For such an element $\plambda$
then \label{n.5}
\bqn
\calO_{\plambda}:=\{x\in\oK:\, |x|<{\plambda}^m \text{ for some } m\in\ZZ \}
\eqn
is an order convex subring of $\oK$ with maximal ideal
\bqn
\calI_\plambda:=\{x\in\oK:\, |x|<\plambda^m \text{ for all }m\in\ZZ \}\,.
\eqn
The quotient $\oKl:=\calO_\plambda/\calI_\plambda$ is the {\em Robinson field} \label{n.6} 
associated to the non-principal ultrafilter $\omega$ and the infinite element $\plambda$. 
This is necessarily a non-Archimedean field, as, denoting by $\plambda$ also the corresponding element in $\oKl$, all rational numbers are smaller than $\plambda$.  For background on Robinson fields see \cite[Chapter 3]{Lightstone_Robinson}.

\subsection{Valuations}\label{s.valuation}
Often real closed fields admit order compatible valuations:
\begin{defn}\label{defn:valuation} A {\em valuation} \label{n.7}
	on a field $\FF$
	is a function $v\colon \FF\to\RR\cup\{\infty\}$ such that the
	function $\|x\|:=e^{-v(x)}$ satisfies the following
	properties:
	\be
	\item $\|x\|\geq 0$ and $\|x\|=0$ if and only if $x=0$;
	\item $\|xy\|=\|x\|\,\|y\|$;
	\item $\|x+y\|\leq\|x\|+\|y\|$.
	\ee
	A valuation on an ordered field $\FF$ is {\em order compatible} if 
	whenever $0\leq x\leq y$, then $\|x\|\leq\|y\|$. It is {\em trivial} if
	$\|x\|=1$  for all $x\neq 0$.
\end{defn}
\begin{remark}
	This definition, used in \cite{EP}, is more general than the usual one as it is designed to include $-\ln|\cdot|:\RR\to\RR\cup\{\infty\}$ as an example of valuation; here $\ln$ is the natural logarithm.
\end{remark}	
Notice that if (the order on) $\FF$ is non-Archimedean, then $\|n\|=1$
for every
positive integer $n$,
and  the condition (3) in Definition~\ref{defn:valuation} becomes 
the stronger {\em ultrametric} inequality
\bqn
\|x+y\|\leq\max\{\|x\|,\|y\|\}\,.
\eqn
and hence (1), (2) and (3) can be replaced by the following, which
define a {\em non-Archimedean valuation}: 

\be
\item[(1')] $v(x)=\infty$ if and only if $x=0$;
\item[(2')] $v(xy)=v(x)+v(y)$;
\item[(3')] $v(x+y)\geq\min\{v(x),v(y)\}$.
\ee

\begin{remark}
	A valuation satisfying (1'), (2'), (3') is often called a \emph{rank one valuation}.
\end{remark}

\begin{example}\
	\be
	\item Ostrowski classified all valuations on $\QQ$:  if there exists a
	function $\|\,\cdot\,\|$ on $\QQ$
	that satisfies the three properties in Definition~\ref{defn:valuation} and is order compatible, 
	then there exists $C\in\RR$, $0<C\leq1$ such that $\|x\|=|x|^C$.  Thus all order compatible valuations on $\qbarr$ are Archimedean.
	\item A valuation on $\RR$ is $-\ln|\cdot|:\RR\to\RR\cup\{\infty\}$, where $\ln$ denotes the natural logarithm.
	\item On $\RR(X)$ one can define a valuation compatible with the order defined in Example~\ref{ex:orders}(2) by setting
	\bqn
	v\left(\frac{p(X)}{q(X)}\right):=\deg q-\deg p\,,
	\eqn
	which can then be extended in an obvious way to Puiseux series and hence to $\ov{\RR(X)}^r$.
	\ee
\end{example}
One can characterize ordered fields admitting an order compatible non-trivial valuation purely in terms of their ordering.This leads to an explicit description of said valuations which we will need later on; for this we follow closely Brumfiel \cite{Brum2}.

A positive element $b\in\FF$ in an ordered field is a {\em big element} \label{n.8}
if for any $a\in\FF$, there exists $m\in\NN$, $m\geq1$,
such that $a<b^m$.  For $a\in\FF$, with $a>0$, the two subsets of $\QQ$ 
\bq\label{eq:Aa}
A_a:=\left\{\frac{m}{n}:\, b^m\leq a^n, \, n\in\NN^+,\, m\in\ZZ\right\}
\eq
and 
\bq\label{eq:Ba}
B_a:=\left\{\frac{m'}{n'}:\, b^{m'}\geq a^{n'}, \, n'\in\NN^+,\, m'\in\ZZ\right\}
\eq
define a Dedekind cut of $\QQ$ and hence a real number denoted $\log_b(a)$.  Then 
\bqn
v_b(a):=-\log_b|a|\,
\eqn
is a non-trivial order compatible valuation on $\FF$.
Conversely if $v$ is a non-trivial order compatible valuation and $b\in\FF$, $b>0$ satisfies $v(b)<0$ then $b$ is a big element and $v(a)=v(b)\log_b(|a|)$ for all $a\in\FF$.

If on the other hand $\FF=\RR$, then $b\in \RR$ is a big element if and only if $b>1$ and
$v_b(a)=\log_b|a|$ is the usual logarithm in base $b$.  

\begin{example}\label{e.RobinsonValuation} If $\omega$ is a non-principal ultrafilter, then $\KK^\omega$ has no big element. The construction of the Robinson field $\oKl:=\calO_\plambda/\calI_\plambda$  discussed at the end of Section~\ref{subsec:intro1}  is designed in such a way that the positive element $\plambda\in\oKl$ is a big element,
	and hence  the Robinson field $\oKl$ has an order compatible valuation.
	One can express the valuation on the field $\oKl$ in terms of limit of logarithms:
	\bqn
	v(\pmb x):=-\lim_\omega\frac{\ln(|x_k|)}{\ln(\lambda_k)}
	\eqn
	where $\plambda=(\lambda_k)_{k\in\NN}$ and $\pmb x=(x_k)_{k\in\NN}$.
	One verifies easily from the density of $\KK$ in $\RR$ that $\oKl=\rol$.
\end{example}

It is important for our needs to isolate a class of fields which have
big elements.
\begin{lemma}[{\cite{Brum2}}]\label{l.bigelt}
	If $\LL$ is an ordered field extension of finite transcendence
	degree over $\KK$ and $\KK$ contains a big element, then $\LL$
	contains a big element as well.
	In particular   any ordered field that has finite transcendence degree
	over $\KK\subset\RR$ has big elements and hence an order compatible
	valuation.
\end{lemma}
%

The following leads to interesting invariants:
\begin{prop}\label{p.valuautom}
	Let $\KK\subset \RR$ be real closed and $\FF$ be real closed, non-Archimedean of transcendence degree $\leq d$ over $\KK$, $\phi\in\Aut(\FF)$ and $b\in\FF_{>1}$ a big element.
	\be
	\item $\log_b(\FF^*)$ is a $\QQ$-vector space of $\RR$ of dimension at most $d$.
	\item $\log_b(\phi(b))$ is a real algebraic number of degree at most $d$ independent of $b$.
	\ee
\end{prop}	
\begin{proof}
	(1) Since every positive element $x\in\FF$ has $m$-th roots in $\FF$ for all $m\geq 1$, $\log_b(\FF^*)$ is divisible and hence a $\QQ$-vector space. By \cite[Corollary 1, \S10.3, VI]{BourbakiCA} its rational rank is bounded by $d$, which shows (1).
	
	(2) Since $\FF$ is real closed, $\phi$ is order preserving and hence $\phi(b)$ is a big element as well, and $$\log_{\phi(b)}(\phi(x))=\log_b(x)\quad \forall x\in\FF^*.$$
	This implies, using that $\log_b(\phi(b))\log_{\phi(b)}(\phi(x))=\log_b(\phi(x))$, that $$\log_b(\phi(b))\log_b(x)=\log_b(\phi(x)),$$
	hence the multiplication map 
	$$\begin{array}{ccc}
		\RR&\to&\RR\\ \xi&\mapsto&\log_b(\phi(b))\xi
	\end{array}$$
	preserves the  $\QQ$-vector subspace $\log_b(\FF^*)$, whose dimension is at most $d$, implying (2).
\end{proof}	

We conclude the section with an example of an ordering on $\KK(T,S)$ together with an order compatible valuation $v$ 
such that $v(\KK(T,S))=\QQ$. 
This will not be needed in the rest of the paper, but shows the surprising fact that, 
despite the field $\KK(T,S)$ having transcendency degree two over $\KK$, the value group is not a finitely generated Abelian group.
\begin{example}\label{ex:weirdval}
	Let $\KK(T^\QQ)$ be the field of formal power series with coefficients in $\KK$
	$$f=\sum_{n=1}^\infty a_nT^{r_n}$$
	whose exponents $r_n$ belong to $\QQ$ and form a monotone increasing sequence with $\lim_{n\to\infty}r_n=+\infty$, endowed with the
	ordering given by $f>0$ if $a_1>0$ and valuation $v(f):=r_1$.
	
	Let $0<r_1<r_2<\cdots$ be a strictly monotone increasing sequence in $\QQ$ such that 
	\bq\label{e.weirdval1}\lim_{n\to\infty}(r_n-qr_{n-1})=+\infty\quad \forall q\in\NN\eq
	Then, using the arguments in \cite[Exercise VI, \S 3, 1a) ]{BourbakiCA} one shows that  $$s:=1+T^{r_1}+T^{r_2}+\ldots\in\KK(T^\QQ)$$ is transcendental over $\KK(T)$.
	In this case we can define a field injection
	$$\KK(T,S)\to\KK(T^\QQ)$$
	by sending $T$ to $T$ and $S$ to $s$. We endow $\KK(T,S)$ with the ordering and the valuation, still denoted $v$, induced by $\KK(T^\QQ)$.
	We are going to state an additional condition \eqref{e.weirdval2} on the sequence $(r_n)_{n\geq 1}$ which, together with \eqref{e.weirdval1},  will imply that $v(\KK(T,S))\supset \ZZ+\ZZ r_1+\ZZ r_2+\ldots \;.$
	Namely for every $n\geq 1$, let $s_n:=1+T^{r_1}+\ldots+ T^{r_{n-1}}$.
	Let furthermore $P_n\in\KK(T)[X]$ be a minimal polynomial of $s_n$ over $\KK(T)$, chosen so that all its coefficients are in $\KK[T]$; if $P_n(X)=\sum_{j=0}^{d_n}a_j^{(n)}(T)X^j$ where  $d_n$ is the degree of $P_n$,  let 
	$$b_n:=\max\{\deg(a_j^{(n)}(T)):\, 0\leq j\leq d_n\}.$$
	Then if 
	\bq\label{e.weirdval2}r_n>b_n+(d_n-1)r_{n-1}\eq
	we have $v(P_n(s))\in r_n+\ZZ r_{n-1}+\ZZ r_{n-2}+\ldots+\ZZ r_1+\ZZ.$
	Indeed write $s=s_n+R_n$ and consider the Taylor expansion of $P_n$ around $s_n$ to obtain, since $P_n(s_n)=0$:
	$$P_n(s)=R_n P'_n(s_n)+\frac{1}{2}R_n^2P_n''(s_n)+\ldots+\frac{1}{d_n!}R_n^{d_n}P_n^{(d_n)}(s_n).$$
	This implies by \eqref{e.weirdval2} for $j\geq 2$:
	$$v(R_nP_n'(s_n))\leq r_n+b_n+(d_n-1)r_{n-1}<jr_n\leq v(R_n^jP_n^{(j)}(s_n))$$
	since $v(P_n^{(j)}(s_n))\geq0$ for every positive $j$. Thus $v(P_n(s))=r_n+ v(P_n'(s_n))$ and we conclude by observing that $v(P_n'(s_n))\in \ZZ+\ZZ r_1+\ldots+\ZZ r_{n-1}$.
	This readily implies that if  \eqref{e.weirdval1} and \eqref{e.weirdval2} hold for every $n\geq 2$,
	$v(\KK(T,S))\supset \ZZ+\ZZ r_1+\ZZ r_2+\ldots \;.$
	To obtain a specific example
	set  $r_1=1$, and define $r_n$ inductively by
	$r_n=\max\{(b_n+(d_n-1)r_{n-1}, n^n\}+\frac1n$ for all $n\geq 2$,
	to obtain
	$$v(\KK(T,S))=\QQ.$$
\end{example}

\subsection{Semialgebraic varieties and $\FF$-extension}\label{s:Intro2}
Let $\KK$ be a real closed field.  

\begin{defn} A set $S\subset\KK^n$ is {\em semialgebraic} \label{n.9}
	if 
	\bq\label{eq:1/2alg}
	S=\bigcup_{\mathrm{finite}}\bigcap_{\mathrm{finite}}\{x\in\KK^n:\, P_j(X)=0\}\cap\{x\in\KK^n:\, Q_i(X)>0\}\,,
	\eq
	where $P_i(X),Q_i(X)\in\KK[X]$.
\end{defn}
For example one can show that a subset $S\subset\KK$ is semialgebraic if and only if it is the finite union of points and open intervals, bounded or unbounded.

\begin{defn} A map $f\colon S_1\to S_2$ between semialgebraic sets if {\em semialgebraic} if its graph
	is a semialgebraic subset of $\KK^n\times\KK^m$, where $S_1\subset\KK^n$ and $S_2\subset\KK^m$.
\end{defn}

The following is a fundamental result:

\begin{prop}[{\cite[Proposition~2.2.7]{BCR}}] The image of a semialgebraic set via a semialgebraic map is semialgebraic.
\end{prop}

For any real closed field $\KK$ we define the norm  $N\colon\KK^n\to\KK_{\geq0}$ by
\bqn
N(x)=:\sqrt{\sum_{j=1}^n x_j^2}\,.
\eqn
Given $x\in\KK^n$ and $r\in\KK_{\geq0}$, we define the {\em open ball}
\bqn
B(x,r):=\{y\in\KK^n:\, N(x-y)<r\} \;.
\eqn
Open balls form a basis of open sets for the product topology on $\KK^n$,
called the {\em Euclidean topology}.
%
If $\KK=\RR$ this is the familiar topology, but if $\KK$ is any other  real closed field and $n\geq1$,
$\KK^n$ is totally disconnected in the Euclidean topology.
We will say that a semialgebraic map is \emph{continuous}, if it is with respect to the Euclidean topology.

\begin{defn} A semialgebraic set $S\subset\KK^m$ is {\em semialgebraically connected}  \label{n.10}
	if it cannot be written as the disjoint union of two non-empty closed semialgebraic subsets.
\end{defn}

The following is a classical theorem of Whitney in the case $\KK=\RR$:

\begin{thm}[{\cite[Theorem 2.4.4]{BCR}}] Any semialgebraic subset $S\subset\KK^n$ is the disjoint union of a finite number 
	of connected semialgebraic sets $C_1,\dots, C_p$ that are both open and closed in $S$.
	These sets are called the {\em semialgebraic connected components} of $S$.
\end{thm}

If the groundfield $\KK$ is the field of real numbers, the notion of connectedness and semialgebraic connectedness coincide.
In general connectedness is of course stronger than semialgebraic connectedness.

\begin{defi}\label{d.extension}
	Given a semialgebraic set $S\subset\KK^n$ and a real closed field $\FF\supset\KK$, 
	we denote by $S_\FF$ the \emph{$\FF$-extension} of $S$, \label{n.11}
	that is the set
	\bqn
	S_\FF:=\bigcup_{\mathrm{finite}}\bigcap_{\mathrm{finite}}\{x\in\FF^n:\, P_j(X)=0\}\cap\{x\in\FF^n:\, Q_i(X)>0\}\,,
	\eqn
	where $P_i(X),Q_i(X)\in\KK[X]$ are the same polynomials that define $S$ in \eqref{eq:1/2alg}.  
\end{defi}
It can be shown that the set $S_\FF$ is independent of the choice of the finite set of polynomials defining $S$, 
\cite[Proposition~5.1.1]{BCR}.

We can also consider extensions of semialgebraic functions $f\colon S_1\to S_2$ between semialgebraic set $S_1\subset\KK^m$ and $S_2\subset \KK^n$.
In fact the set $(\mathrm{graph}\,f)_\FF$ is the graph of a
semialgebraic map $(S_1)_\FF\to (S_2)_\FF$, which will be denoted by $f_\FF$, \cite[Proposition 5.3.1]{BCR}.

\begin{ppts}
	
	\label{ppts.2.10}
	Let $\KK,\FF$ be real closed fields with $\KK\subset\FF$.
	\be
	\item \cite[Proposition 5.3.5(i)]{BCR}
	The semialgebraic set $S\subset\KK^n$ is open (resp. closed, non-empty) if and only if $S_\FF\subset\FF^n$ is open (resp. closed,  non-empty).
	\item \cite[Propositions 5.3.3 and 5.3.5(ii)]{BCR}
	The semialgebraic map $f\colon S_1\to S_2$ is injective (resp. surjective, bijective, continuous) if and only if
	its extension $f_\FF\colon (S_1)_\FF\to(S_2)_\FF$ is injective (resp. surjective, bijective, continuous).
	\item  \cite[Proposition 5.3.6(ii)]{BCR}
	The semialgebraic connected components of a semialgebraic set $S$ correspond to the semialgebraic connected components
	of its extension $S_\FF$.
	\ee
\end{ppts}

The above properties build on the Tarski--Seidenberg Principle:
namely, if $\KK\subset\FF$ and $S\subset\KK^n$ is semialgebraic with $S_\FF\neq\varnothing$, 
then $S\neq\varnothing$, \cite{BCR}.

\subsection{Reduction modulo an order convex subring}
The  $\FF$-extension is compatible with the reduction modulo an order
convex subring (see Section \ref{subsec:intro1}) as we now describe. This will play an important role in the  proof of Theorem  \ref{thm:4.15}, relating non-standard symmetric spaces to asymptotic cones.

Let  $\KK\subset \calO\subset\FF$ where $\KK, \FF$ are real closed
and $\calO$ is an order convex subring, and hence a valuation ring,
whose maximal ideal is denoted $\calI\subset \calO$. The residue field
$\FF_\calO:=\calO/\calI$ is real closed  and $\FF_\calO\supset\KK$.
Given $V\subset\KK^n$ semialgebraic  we are interested in the relation between the extensions $V_\FF$ and $V_{\FF_\calO}$. Let
\bqn
\pi\colon \calO^n\to \FF_\calO^n
\eqn
be the canonical projection map, which is a morphism of $\KK$-vector spaces. We define:
\begin{defi}\label{d.2.15}
	$$V_\FF(\calO):=V_\FF\cap \calO^n.$$ \label{n.12}
\end{defi}	

We have then:

\begin{prop}\label{prop:4.18} Assume that $V\subset\KK^n$ is closed semialgebraic.   
	Then $\pi(V_\FF(\calO))=V_{\FF_\calO}$.
\end{prop}

\begin{remark}
	Proposition \ref{prop:4.18} is due to Thornton \cite[Proposition 2.38]{Thornton} in the case where $V$ is real algebraic, $\FF$ the field of hyperreals and $\FF_\calO$ a Robinson field.
\end{remark}	
In the proof of Proposition \ref{prop:4.18} we will use the distance function to $V$ from a point $x$ in $\KK^n$:
$$d_V(x):=\inf\{N(x-y):\, y\in V\}\in\KK_{\geq 0}$$
and recall its properties in the following lemma:
\begin{lemma}\label{l.2.14}\
	\begin{enumerate}
		\item $d_V:\KK^n\to\KK_{\geq 0}$ is well defined.
		\item $d_V$ is semialgebraic, continuous and when $V\subset \KK^n$ is closed, its vanishing set is precisely $V$.
		\item The $\FF$-extension of $d_V$ is $d_{V_\FF}$.
	\end{enumerate} 	
\end{lemma}	
\begin{proof}
	Assertions (1) and (2) are the content of \cite[Proposition 2.2.8]{BCR}. The third follows from the formula given in \cite[Proposition 2.2.8 (ii)]{BCR} for the graph of $d_V$ and \cite[Proposition 5.1.1]{BCR}.
\end{proof}	
\begin{proof}[Proof of Proposition \ref{prop:4.18}]
	(1) First we treat the case where $V$ is algebraic, in which case we may assume 
	$$V=\{x\in \KK^n:\, f(x)=0\}$$
	for an appropriate $f\in\KK[x_1,\ldots, x_n]$.
	Observe first that since $f$ has coefficients in $\KK$ and $\pi:\calO^n\to \FF_\calO^n$ is an $\calO$-module morphism we have $\pi(V_\FF(\calO))\subset V_{\FF_\calO}$.
	
	Since the null sets of $f$ and $d_V$ coincide there exists \cite[Theorem 2.6.6]{BCR} a continuous semialgebraic function $h:\KK^n\to\KK$ and an integer $L\geq 1$ with 
	\begin{equation}\label{e.4.18.1}
		d_V^L=hf.
	\end{equation}
	By \cite[Proposition 2.6.2]{BCR} there exist $c\in \KK_{\geq 0}$ and $p\in\NN$ with 
	\begin{equation}\label{e.4.18.2}
		|h(x)|\leq c(1+N(x)^2)^p.
	\end{equation}
	The inequalities \eqref{e.4.18.1} and \eqref{e.4.18.2} hold for the $\FF$-extensions of the corresponding functions; taking into account that the $\FF$-extension of $d_V$ is $d_{V_{\FF}}$ (Lemma \ref{l.2.14}) we get
	\begin{equation}\label{e.4.18.3}
		d_{V_\FF}(x)^L\leq c(1+N(x)^2)^p|f(x)| \quad \forall x\in\FF^n.
	\end{equation}
	Let $x\in V_{\FF_\calO}$ and $y\in \calO^n$ with $\pi(y)=x$. Since $f(x)=0$ we have $f(y)\in \calI$ and hence $|f(y)|\in \calI$. Since $c\in\KK\subset \calO$ and $(1+N(y)^2)^p\in \calO$ we conclude that $c(1+N(y)^2)^p|f(y)|\in \calI$ which by \eqref{e.4.18.3} and order convexity of $\calI$ implies
	$$d_{V_\FF}(y)=\inf\{N(y-z):\, z\in V_\FF\}\in \calI.$$
	Let $\epsilon>0$ with $\epsilon\in \calI$ and pick $z\in V_\FF
	$ with $N(y-z)< d_{V_\FF}(y)+\epsilon\in \calI$. This implies 
	$$\sum_{i=1}^n(y_i-z_i)^2\in \calI,$$ 
	thus $y_i-z_i\in \calI$, $1\leq i\leq n$, and hence $x=\pi(y)=\pi(z)$ with $z\in V_\FF\cap \calO^n$.
	
	\medskip
	(2) Now we assume that $V\subset \KK^n$ is semialgebraic closed. Then $V$ is a finite union of basic closed sets \cite[Theorem 2.7.2]{BCR}, say $V=\bigcup_{i=1}^r S_i$ with 
	$$S_i=\{x\in\KK^n:\, f_1^{(i)}(x)\geq 0,\ldots, f_r^{(i)}(x)\geq 0\}.$$
	Then $V_\FF=\bigcup_{i=1}^r S_{i,\FF}$ and $V_{\FF_\calO}=\bigcup_{i=1}^r S_{i,\FF_\calO}$.
	
	Now the projection $\pi:\calO\to \FF_\calO$ is order preserving in the sense that if $x\geq 0$ in $\calO$ then $\pi(x)\geq 0$ in $\FF_\calO$. Thus if $x\in \calO^n$ with $f_l^{(i)}(x)\geq 0$ then $f_l^{(i)}(\pi(x))\geq 0$. This implies $\pi(S_{i,\FF}(\calO))\subset S_{i,\FF_\calO}$ and hence $\pi(V_\FF(\calO))\subset V_{\FF_\calO}$.
	
	Let now $W\subset \KK^l\times \KK^n$ be an algebraic set such that $p_2(W)=V$ where $p_2:\KK^l\times \KK^n\to \KK^n$ denotes the canonical projection. By the facts previously established we have a well defined commutative square
	\bqn
	\xymatrix{
		W_\FF\cap \calO^{l+n}\ar[r]^{\pi}\ar[d]_{p_2} &W_{\FF_\calO}\ar[d]^{p_2}\subset \FF_\calO^{l+n}\\
		V_\FF\cap \calO^n\ar[r]^{\pi}&V_{\FF_\calO}\subset \FF_\calO^n.
	}
	\eqn
	By (1) we have $\pi(W_\FF\cap \calO^{l+n})=W_{\FF_\calO}$ and by Properties \ref{ppts.2.10} (2) we have $p_2(W_{\FF_\calO})=V_{\FF_\calO}$, which implies $\pi(V_\FF\cap \calO^n)=V_{\FF_\calO}$ which concludes the proof.
\end{proof}	
\subsection{The real spectrum $\Rspec(\calA)$ of a ring $\calA$}\label{subsec:real_spectrum}
We proceed recalling basic facts about the real spectrum of a ring, leading to the real spectrum compactification of real semialgebraic sets. 
We adopt the notation in \cite{BCR}. Another useful reference is \cite{Brum2}.

\begin{cv}  In this paper a ring will always denote a commutative associative ring with unit.
\end{cv}

\begin{defn}\label{d.rsp}
	The \emph{real spectrum} $\Rspec(\calA)$ of a ring $\calA$ \label{n.13}
	is the set of prime cones, 
	namely the subsets  $\alpha\subset \calA$ with the properties that 
	\begin{enumerate}
		\item  $-1\notin\alpha$;
		\item $\alpha+\alpha\subset \alpha,\; \alpha\cdot\alpha\subset \alpha$;
		\item$ \alpha\cup -\alpha=\calA$;
		\item $\alpha\cap-\alpha$ is a prime ideal in $\calA$.
	\end{enumerate} 
\end{defn}

It will be useful to have different ways to think about points in the real spectrum of a ring $\calA$. We quote from \cite[Proposition 7.1.2]{BCR}:
\begin{prop}\label{p.charrsp}
	The following data are equivalent:
	\begin{enumerate}
		\item a prime cone $\alpha$ in $\calA$;
		\item a pair $(\fp, \leq)$ consisting of a prime ideal $\fp$ and an
		ordering  $\leq$ on the field of fractions $\Frac\left(\calA/{\fp}\right)$ of $\calA/{\fp}$; \label{n.14}
		\item an equivalence class of pairs $(\phi,\FF)$ where
		$\phi:\calA\to\FF$ is a homomorphism to a real closed field $\FF$ that is the real closure of the field of fractions of $\phi(\calA)$ and $(\phi_1,\FF_1)$, $(\phi_2,\FF_2)$ are equivalent if there exists an isomorphism $\psi:\FF_1\to\FF_2$ such that $\phi_2=\psi\phi_1$;
		\item an equivalence class of pairs $(\phi,\FF)$ where $\phi:\calA\to\FF$ is a homomorphism to a real closed field $\FF$, for the smallest equivalence relation such that  $(\phi_1,\FF_1)$ and $(\phi_2,\FF_2)$ are equivalent if there is a homomorphism $\psi:\FF_1\to\FF_2$ such that $\phi_2=\psi\phi_1$. \label{n.15}
	\end{enumerate}	
\end{prop}
In fact one goes from (1) to (2) by taking $\fp_\alpha:=\alpha\cap(-\alpha)$ for $\fp$ and  the unique order $\leq_{\alpha}$ on $\Frac\left(\calA/{\fp_\alpha}\right)$ whose set of positive elements is given by 
$$\left\{\frac{\ov x}{\ov y}:\; xy\in\alpha, \, y\notin\fp_\alpha\right\}$$
where $\ov x\in\calA/\fp_{\alpha}$ denotes the reduction modulo $\fp_{\alpha}$. One goes from (2) to (3) and (4) by composing the reduction $\calA\to\calA/\fp$ with the inclusion of $\Frac\left(\calA/{\fp}\right)$ into a real closure $\FF$ with respect to the ordering $\leq$. One goes from (4) to (1) by taking 
$$\alpha=\{a\in\calA|\; \phi(a)\geq 0\}.$$

\begin{exs} \label{ex.rspecbase} \
	\be
	\item If $\calA=\CC$, then $\Rspec(\calA)=\emptyset$.
	
	\item If $\KK$ is orderable and $V\subset \KK^n$ is a non-empty real
	algebraic subset with coordinate ring $\calA=\KK[V]$, then its real
	spectrum $\rsp V:=\Rspec(\KK[V])$ is not empty.  In fact, every
	point $x\in V$ gives rise to a prime cone
	$\alpha_x:=\{p\in\KK[V]:\,p(x)\geq0\}$, which gives an injection
	$V\hookrightarrow \rsp V$.
	%
	\item 
	(see \cite[Example 7.1.4]{BCR} in the case where $\KK=\RR$) If
	$\calA=\KK[X]$ with $\KK$ a real closed subfield of $\RR$, the real
	spectrum has the following simple description.  For every point
	$x\in \RR$ we set
	$$\begin{array}{c}\alpha_x:=\{p\in\KK[X]:\, p(x)\geq 0\}\\
		\alpha_{x^+}:=\{p\in\KK[X]:\, \exists\epsilon>0 ,\,
		\forall y\in]x,x+\epsilon[ ,\, p(y)\geq0\}\\
		\alpha_{x^-}:=\{p\in\KK[X]:\, \exists\epsilon>0,\,
		\forall y\in]x-\epsilon,x[ ,\, p(y)\geq0\}.\end{array}$$
	They are prime cones and
	$\alpha_{x^\pm}\subset \alpha_x$, with equality if and  only if $x\notin\KK$.
	These, together with the  prime cones
	$$\begin{array}{c}\alpha_{+\infty}:=\{p\in\KK[X]:\, \exists m\in\KK ,\,
		\forall y\in]m,+\infty[ ,\, p(y)\geq0\}\\
		\alpha_{-\infty}:=\{p\in\KK[X]:\, \exists m\in\KK ,\,
		\forall y\in]-\infty,m[ ,\, p(y)\geq0\}\end{array}$$
	classify all points of $\Rspec(\calA)$.
	%
	\ee
\end{exs}	

In the sequel, given a prime cone $\alpha$ of $\calA$ we denote by $\phi_\alpha:\calA\to\calA/\frak p_\alpha$ the associated reduction, and by $\FF_\alpha$ the real closure of $\Frac(\calA/\frak p_\alpha)$ with respect to the ordering $\leq_{\alpha}$;  for $f\in\calA$ we denote by $f(\alpha)$ the image $\phi_\alpha(f)$  of $f$ in $\FF_\alpha$. With this notation we have
$$\alpha=\{f\in\calA|\; f(\alpha)\geq 0\}.$$
%
%


\subsection{Topologies on $\Rspec(\calA)$, closed points and specializations}\label{s.ret}
Given a ring $\calA$, we will endow $\Rspec(\calA)$  with two topologies, the constructible topology 
and  the spectral topology.  While it is easy to show that $\Rspec(\calA)$ is compact in the constructible topology, it is also totally disconnected, and, for  $\calA=\KK[V]$, the constructible topology  induces the discrete topology on $V\hookrightarrow \rsp V$ rather than the Euclidean  (recall Example \ref{ex.rspecbase}).
The coarser topology given by the spectral topology has better topological properties, but might contain non-closed points. We will see that we can remedy to this by considering the set of closed points in $\Rspec(\calA)$ with respect to the spectral topology. 

\medskip

The {\em constructible sets} \label{n.16}
are essential in both the constructible and the spectral topology; they are defined as
\bqn
\bigcup_{\mathrm{finite}}\bigcap_{\mathrm{finite}}\{\alpha\in\Rspec(\calA):\,g(\alpha)=0\}\cap\{\alpha\in\Rspec(\calA):\,f(\alpha)>0\}\,,
\eqn
where $f,g$ run through finite subsets of $\calA$.
The {\em constructible topology} has all constructible sets as a basis of open sets. 
If $\Rspec(\calA)$ is endowed with the constructible topology, then the  map $\Rspec(\calA)\to\{0,1\}^\calA$
that to a prime cone in $\Rspec(\calA)$ associates its characteristic function, is continuous. 
This identifies $\Rspec(\calA)$ with a closed subset of $\{0,1\}^\calA$, hence compact and totally disconnected.
The \emph{spectral topology}, \label{n.17}
instead, is the coarser topology whose basis of open sets is given by  
\bqn
U(f_1,\ldots,f_r)=\{\alpha\in \Rspec(\calA)|\; f_i(\alpha)>0,\,i=1,\dots,r\} \,,
\eqn
where $f_i\in\calA$. Since the spectral topology has fewer open sets than the constructible topology,
we deduce that $\Rspec(\calA)$ is compact also with respect to this topology.

However $\Rspec(\calA)$ is not Hausdorff in the spectral topology:
non-closed points arise from the so-called specialization.
If $\alpha,\beta\in\Rspec(\calA)$,
we say that $\beta$ is a {\em specialization } \label{n.18}
of $\alpha$ if $\beta\supset\alpha$.
Chasing the definitions, it is easy to see that
\bqn
\alpha\subset\beta\text{ if and only if }\beta\in\ov{\{\alpha\}}\,,
\eqn
where the closure of $\{\alpha\}$ is in the spectral topology.
Moreover one can prove that the set of prime cones that specialize $\alpha$ is totally ordered by inclusion,
\cite[Proposition 7.1.23]{BCR}, and hence contain a unique maximal element that is closed, \cite[Proposition 7.1.24]{BCR}.
We can thus define a
retraction \label{n.19}
\bq\label{eq:retraction}
\Ret:\Rspec(\calA)\to \Rspecc(\calA)
\eq
onto the set  $\Rspecc(\calA)$  of closed points, which is continuous
\cite[Proposition 7.1.25]{BCR}.
To give one more illustration of how the constructible
topology is used as a tool, we give the proof of the following:

\begin{prop}[{\cite[Proposition 7.1.24]{BCR}}]
	The subspace $\Rspecc(\calA)$ of closed points is compact and Hausdorff.
\end{prop}

\begin{proof}  First observe that if $\alpha\in\Rspec(\calA)$ and $U$ is an open set with $U\ni\Ret(\alpha)$,
	then $U\ni\alpha$;  as a result, any open cover of $\Rspecc(\calA)$ also covers $\Rspec(\calA)$,
	which implies the first assertion.
	
	For the second, let $\alpha,\beta$ be closed points
	and assume that for every $U_1:=U(f_1,\ldots,f_r)$ containing $\alpha$
	and $U_2:=U(g_1,\ldots,g_s)$ containing $\beta$, we have $U_1\cap U_2\neq\varnothing$.
	Since $U_1$ and $U_2$ are closed in the constructible topology and $\Rspec(\calA)$ is compact, 
	we deduce that 
	\bqn
	\bigcap_{\substack{U_1\ni \alpha\\ U_2\ni \beta}}(U_1\cap U_2)\neq\varnothing.
	\eqn
	If $\gamma$ is a point in this intersection, then it follows that $\alpha\in\ov{\{\gamma\}}$, and $\beta\in\ov{\{\gamma\}}$.
	Since $\alpha$ and $\beta$ are closed, then $\alpha=\Ret(\gamma)=\beta$.
\end{proof}

\begin{defn} Let $\FF$ be an ordered field, $R_1, R_2\subseteq\FF$ subrings with $R_1\subseteq R_2$.
	We say that {\em $R_2$ is Archimedean over $R_1$} \label{n.20}
	if every element of $R_2$ is bounded above
	by some element of $R_1$.
\end{defn}

\begin{remark}\label{rem:closure_archimedean}  It is often used that a real closure of an ordered field $\LL$ is Archimedean 
over the field itself.  
Indeed the roots of a polynomial in $\LL[X]$ admit an upper bound in terms of the coefficients of this polynomial.
\end{remark}

The following characterization of closed points is essential for our applications to character varieties, 
we thus give a proof as it is not to be found in \cite{BCR}.
\begin{prop}\label{prop:closed_points} Let $\alpha\in\Rspec(\calA)$ be a prime cone 
	and $\phi_\alpha\colon \calA\to \FF_\alpha$ the corresponding homomorphism defined in Section~\ref{subsec:real_spectrum}.
	The following are equivalent:
	\be
	\item $\alpha$ is closed;
	\item $\FF_\alpha$ is Archimedean over $\phi_\alpha(\calA)$.
	\ee
\end{prop}

Clearly if $\FF_\alpha$ is Archimedean then it is Archimedean over any of its subrings  and,
in particular, $\alpha$ is closed by the proposition.

We start the proof with the following elementary lemma:

\begin{lemma}\label{lem:closed_points} Let $\alpha,\beta\in\Rspec(\calA)$ with $\alpha\subseteq\beta$.  The following are equivalent:
	\be
	\item $\alpha=\beta$;
	\item $\fp_\alpha=\fp_\beta$.
	\ee
\end{lemma}
%

In the proof of Proposition~\ref{prop:closed_points} we will use the concept of proper cone, more general  than the one of prime cone introduced in Definition \ref{d.rsp}.

\begin{defn}\label{def:cone} Let $\calA$ be a ring.  A {\em cone} $P$ of $\calA$  is a subset $P\subset\calA$
	satisfying the following properties:
	\be
	\item If $a,b\in P$, then $a+b\in P$ and $ab\in P$;
	\item If $a\in \calA$, then $a^2\in P$.
	\ee
	The cone is {\em proper} \label{n.21}
	if in addition:
	\be
	\item[(3)] $-1\notin P$.
	\ee
\end{defn}

Notice that a prime cone is in particular a proper cone and every proper cone is contained in a prime cone,
\cite[Theorem~4.3.7, proof of (i)$\,\Rightarrow\,$(ii)]{BCR}.

\begin{proof}[Proof of Proposition~\ref{prop:closed_points}]
	(2)$\,\Rightarrow\,$(1) We prove the contrapositive of the statement.
	Let us assume that $\alpha$ is not closed, that is that there exists a prime cone $\beta\subset\calA$
	such that $\alpha\subsetneq\beta$.  Then, by the equivalence in Lemma~\ref{lem:closed_points},
	$\fp_\alpha\subsetneq \fp_\beta$.
	Consider then the commutative diagramm
	\bqn
	\xymatrix{
		\calA\ar[r]^{\phi_\beta}\ar[d]_{\phi_\alpha} &\calA/{\fp_\beta}\\
		\calA/{\fp_\alpha}\ar[ur]_{\ov\phi}
	}
	\eqn
	Let $t\in\fp_\beta\smallsetminus\fp_\alpha$.  We may assume that $\phi_\alpha(t)>0$.
	If every element of $\FF_\alpha$ were bounded above by some element of $\phi_\alpha(\calA)$,
	then in particular there would be $a\in\calA$ with 
	\bqn
	\frac{1}{\phi_\alpha(t)}<\phi_\alpha(a)\,,
	\eqn
	that is $1<\phi_a(a)\phi_\alpha(t)$.  By applying $\ov{\phi}$ we get that
	\bqn
	1<\phi_\beta(a)\phi_\beta(t)=0\,,
	\eqn
	since $t\in \fp_\beta$.

	\medskip
	\noindent
	(1)$\,\Rightarrow\,$(2) Let us assume that $\FF_\alpha$ is not Archimedean over $\phi_\alpha(\calA)$.
	We will construct a cone $\beta$ such that $\alpha\subsetneq\beta$. 
	Since $\FF_\alpha$ is not Archimedean over it, $\Frac(\calA/\fp_\alpha)$ cannot be Archimedean over $\phi_\alpha(\calA)$
	(see Remark~\ref{rem:closure_archimedean}).
	Hence there are $f_0, f_1\in \alpha\smallsetminus(-\alpha)$ with 
	\bqn
	\phi_\alpha(a)<\frac{\phi_\alpha(f_1)}{\phi_\alpha(f_0)}
	\eqn
	for all $a\in\calA$.  Since $\phi_\alpha(f_0)>0$, we get
	\bqn
	\phi_\alpha(f_0)\phi_\alpha(a)<\phi_\alpha(f_1) 
	\eqn
	for every $a\in \calA$.  Replacing $a$ by $af_1$ and using that $\phi_\alpha$ is a homomorphism give
	\bqn
	\phi_\alpha(f_0a)\phi_\alpha(f_1)<\phi_\alpha(f_1) \,.
	\eqn
	Since also $\phi_\alpha(f_1)>0$, this implies that
	\bqn
	\phi_\alpha(f_0a)<1
	\eqn
	for all $a\in\calA$, that is
	$1-f_0a\in\alpha\smallsetminus(-\alpha)$ for all $a\in\calA$.  
	
	We claim that 
	\bqn
	P:=\{a-f_0b:\,a,b\in\alpha\}
	\eqn
	is a proper cone in $\calA$.  In fact (1) in Definition~\ref{def:cone} is straightforward; 
	property (2) is also immediate since $\alpha$ is a cone and $\alpha\subset P$.  
	To see (3), notice that if $-1=a-f_0b$ for $a,b\in \alpha$, then $1-f_0b=-a\in-\alpha$,
	which is a contradiction.  Then $P\subseteq\beta$ for some prime cone $\beta$ that strictly contains $\alpha$.
	In fact, since $f_0\in\alpha\smallsetminus(-\alpha)$, so $-f_0\in(-\alpha)\smallsetminus\alpha$.
	Since $-f_0\in P$, this implies that $\alpha\subsetneq P$ and hence $\alpha\subsetneq\beta$.  
\end{proof}

\subsection{Compactifications of (semi)algebraic sets}\label{subsec:compactification_semialgebraic}
Let $\KK$ be a real closed field and $V\subset\KK^n$ a real algebraic set.
We study the real spectrum of the coordinate ring $\calA=\KK[V]$ of $V$ and we set, as in Example \ref{ex.rspecbase}, \label{n.22}
\bqn
\rsp{V}:=\Rspec(\KK[V])\quad\text{ and } \quad\rspcl{V}:=\Rspecc(\KK[V])\,.
\eqn
We mentioned already that every point in $V$ gives rise to a prime
cone and this gives an injection $V\hookrightarrow\rsp{V}$
defined by $x\mapsto \alpha_x:=\{f\in\KK[V]:\,f(x)\geq0\}$.

We will need the following interplay between semialgebraic subsets of $V$ and constructible subsets of $\rsp V$:
\begin{ppts}\label{prop:constr}  \
	\be
	\item For every semialgebraic subset $S\subset V$ there exists a
	unique constructible set $\cons(S)\subset\rsp{V}$, such that
	$\cons(S)\cap V=S,$ \cite[Proposition 7.2.2 (i)]{BCR}.  \label{n.23}
	
	\item The set of polynomial equalities and inequalities defining
	$\cons(S)$ is the same set of polynomial equalities and inequalities
	defining $S$, \cite[Proposition~7.2.2, (ii)]{BCR}.
	\ee
\end{ppts}

\begin{thm}\label{thm:constr} Let $V\subset\KK^n$ be a real algebraic set.
	\be
	\item The map $S\mapsto \cons(S)$ is an isomorphism between the Boolean algebra of semialgebraic subsets of $V$
	and the Boolean algebra of constructible subsets of $\rsp{V}$.
	\item The semialgebraic set $S\subset V$ is open (resp. closed) if and only if the corresponding constructible set $\cons(S)\subset\rsp{V}$
	is open (resp. closed).
	\ee
\end{thm}

The above two properties \ref{prop:constr}, together with Theorem~\ref{thm:constr}~(1) follow from the Artin--Lang Homomorphism Theorem, \cite[Theorem~4.1.2]{BCR},
which  can be readily obtained from the Tarski--Seidenberg Principle.  Then Theorem~\ref{thm:constr} (2) follows
from the Finiteness Theorem, \cite[Theorem~2.7.2]{BCR}.

\begin{cor}\label{cor:dense}  Let $\KK$ be a real closed field and $S\subset V\subset\KK^n$ a semialgebraic set.
	Then $S$ is dense in $\cons(S)$.  
\end{cor}
\begin{proof}  Let $\alpha\in \cons(S)$ and let $U:=U(f_1,\dots,f_r)$ be an open set containing $\alpha$, where $f_1,\dots,f_r\in\KK[V]$.
	We want to show that $U\cap S\neq\varnothing$.
	By Property~\ref{prop:constr} (2) $U=c(U\cap V)$, where
	$U\cap V=\{x\in V:\,f_i(x)>0, \text{ for } i=1,\dots,r\}$.
	By construction and by Theorem~\ref{thm:constr} (1), $\alpha\in \cons(S)\cap c(U\cap V)=c(S\cap (U\cap V))$,
	which implies that $S\cap (U\cap V)\neq\varnothing$.  Thus $U\cap S=S\cap(U\cap V)\neq\varnothing$.
\end{proof}
It follows from Properties \ref{ppts.2.10} (2) and Corollary \ref{cor:dense} that, for a semialgebraic subset $S\subset V$, the constructible set $\cons(S)$ is the closure of $S$ in $\rsp V$ if and only if $S$ is closed. 

The next proposition holds only for an Archimedean real closed field $\KK\subset\RR$; it provides a compactification of $V_\RR$ for $V\subset \KK^n$ an algebraic set, and more generally for $S_\RR$ with $S\subset V$ semialgebraic.
\begin{prop}\label{prop:closed} \
	Let $\KK$ be an Archimedean real closed field and let
	$V\subset\KK^n$ be an algebraic set.
	\begin{enumerate} 
		
		\item  The map $x \mapsto \alpha_x$ is a topological embedding
		of $V_\RR$ in $\rspcl{V}$ with open and dense image.
		The space $\rspcl{V}$ is compact, Hausdorff and, if $\KK$ is
		countable, metrizable.
		
		\item A point $\alpha$ is in the boundary $\rsp{\partial}V=\rsp{V}-V_\RR$ if and only if
		$\FF_\alpha$ is non-Archimedean.  
		
		\item \label{it:compS}
If $S\subset V$ is semialgebraic, $S_\RR$ is open and dense in 
		$$\cons(S)_{cl}:=\{\alpha\in \cons(S):\, \alpha \text{ is closed in $\cons(S)$}\}$$
		and the latter is compact Hausdorff.
	\end{enumerate}
\end{prop}
\begin{proof}  (1) 
	Let $x,y\in V_\RR$ with $x\neq y$.
	Since $\KK$ is dense in $\RR$, one can find $f\in\KK[V]$ with $f(x)>0$ and $f(y)<0$,
	showing that the corresponding prime cones $\alpha_x,\alpha_y\in\rsp{V}$ are different.
	In addition, by Proposition~\ref{prop:closed_points}, since $\RR$ is
	Archimedean, $\alpha_x$ is a closed point.
	
	Since $V$ is dense in $\rsp{V}$ by Corollary~\ref{cor:dense}, then $V_\RR$ is of course dense in $\rspcl{V}$.
	
	It remains to show that the image of $V_\RR$ in $\rspcl{V}$ is open. To this end we set $Q(x):=\sum_{i=1}^nx_i^2$
	and we claim that 
	\bqn
	V_\RR=\{\alpha\in\rspcl{V}:\, \text{there exists }T\in\KK\text{ with }Q(\alpha)<T\}\,.
	\eqn
	The inclusion $\subseteq$ is clear.  For the reverse one, let $\alpha\in \rspcl{V}$ be such that $Q(\alpha)<T$ for some $T\in \KK$
	and consider the homomorphism $\phi_\alpha\colon\KK[V]\to\FF_\alpha$ defined in Section~\ref{subsec:real_spectrum}.
	Since $\alpha$ is closed, according to Proposition~\ref{prop:closed_points}, $\FF_\alpha$ is Archimedean over $\phi_\alpha(\KK[V])$.
	
	The ring $\phi_\alpha(\KK[V])$ is generated by the $\phi_\alpha(X_i)$,
	which are bounded by $\sqrt{T}\in\KK$, thus it is Archimedean over $\KK$. Hence $\FF_\alpha$ is Archimedean,
	thus isomorphic to a subfield of $\RR$. 
	Thus $\alpha$ is also represented by a homomorphism $\KK[V]\to\RR$ and hence belongs to $V_\RR$. 
	
	(2) This follows from the above description of $V_\RR$.
	
	(3) If $S\subset V$ is closed, $S_\RR=\cons(S)\cap V_\RR$ is open in $\cons(S)\cap \rspcl V=\cons(S)_{cl}$ 
	and dense since $S$ is dense in $\cons(S)$; if $S\subset V$ is not closed, there is $W\subset \KK^m$ semialgebraic closed, 
	and a semialgebraic homeomorphism $f:W\to S$; 
	by \cite[Proposition 7.2.8]{BCR} this induces a homeomorphism $F:c(W)\to \cons(S)$ 
	which hence sends $c(W)_{cl}$ to $\cons(S)_{cl}$; 
	in addition $F|_{W_\RR}=f_\RR$ and hence $F(W_\RR)=S_\RR$. 
	This shows that in general $S_\RR$ is open and dense in $\cons(S)_{cl}$. 
	In addition $\cons(S)_{cl}$ is always Hausdorff and compact \cite[Proposition 7.1.25 (ii)]{BCR}.
\end{proof}

\begin{ex}
	In the case of $V=\KK$, where $\KK$ is a real closed subfield of $\RR$,
	recall that $\rsp{\KK[V]}$ is given by
	the set of prime cones defined in Example \ref{ex.rspecbase}(3):
	$$\{\alpha_x|\; x\in\RR\cup\{\pm\infty\}\}\cup\{\alpha_{x^+},\alpha_{x^-}|\;x\in\KK\}.$$
	The subset of closed points is given by the first set, in which $\RR$ injects as open and dense subset. 
	Furthermore for all $x\in \KK$, $\ov{\{\alpha_{x^+}\}}=\{\alpha_x, \alpha_{x^+}\}$ and $\ov{\{\alpha_{x^-}\}}=\{\alpha_x, \alpha_{x^-}\}$.
	
	Semialgebraic subsets of $\KK$ are finite unions of intervals and half lines with endpoints in $\KK$. 
	For the semialgebraic set $S=(a,b]\cap \KK$ with $a,b\in\KK$, we have
	\bqn
	\cons(S)=(a,b]\cup\{\alpha_{x^\pm}|\;x\in(a,b)\cap\KK\}\cup\{\alpha_{a^+},\alpha_{b^-}\}
	\eqn
	and
	\bqn
	\cons(S)_{cl}=(a,b]\cup\{\alpha_{a^+}\},
	\eqn
	which is homeomorphic to a  closed segment in $\RR$. Note that $\alpha_{a^+}$ is closed
	in $\cons(S)$ but not in $\rsp{\KK}$.
\end{ex}

Let $V\subset\KK^n$ be a real algebraic subset with $\KK\subset\RR$.
The quotient map $\KK[X_1,\ldots, X_n]\to \KK[V]$ to the coordinate ring of $V$ induces canonically an injection 
\bqn
\rsp V\hookrightarrow \rsp{(\KK^n)}
\eqn
which is a homeomorphism onto its image. 
If then $S\subset V$ is a semialgebraic subset, one can associate by Properties~\ref{ppts.2.10}(1) two constructible subsets, 
namely $c_V(S)\subset \rsp V$ and $c_{\KK^n}(S)\subset \rsp{(\KK^n)}$ which coincide via the above inclusion. 
We will henceforth use the notations  \label{n.23b}
\bq\label{e.rspS}
\rsp{S}:= \cons(S) \quad\text{ and }\quad \rspcl{S}:= \cons(S)_{cl}.
\eq
Of course if $S\subset V$ is closed then $\rspcl S= \rspcl V\cap c(S)$, 
in general, as topological spaces, $\rsp{S}$ and $\rspcl{S}$ are independent of the ambient real algebraic set. 
With this we have that $S_\RR$ is open and dense in $\rspcl{S}$ and the latter is compact Hausdorff.
We will denote $\rsp{\partial}S:= \rsp{S}\setminus S_\RR$ and $\rspcl{\partial}S:= \rspcl{S}\setminus S_\RR$.  \label{n.24}

\begin{defn}
	We say that a point $\alpha\in\rsp S$ is \emph{non-Archimedean} if $\alpha\notin S_\RR$.
\end{defn}	
\begin{remark}\label{rem:point->prime cone} Let $S\subset V\subset \KK^n$ where $V$ is real algebraic and $S$ is semialgebraic. For any real closed extension $\FF$ of $\KK$,  every point $x\in V_\FF$ leads to  
	a prime cone  \label{n.25}
	\bqn
	\alpha_x:=\{f\in\KK[V]:\,f(x)\geq0\}\in \rsp{V}\;
	\eqn 
	and if $x\in S_\FF$ then $\alpha_x\in\rsp S$. 
	While this does not  in general lead to an injection of $V_\FF$, 
	this construction will be useful in representing points in $\rsp{V}$ by points in $V_{\KK^\omega}$, 
	where $\KK^\omega$ is a hyper-$\KK$-field (see Section~\ref{subsec:intro1}). 
	Observe that the image of $V_\FF$ in $\rsp{V}$ consists of the
	points $\alpha\in \rsp{V}$ such that
	there exists a field injection
	$\FF_{\alpha}\hookrightarrow \FF\;.$
	
	There is a useful converse: let $\alpha\in\rsp{S}$, then (Proposition \ref{p.charrsp}) there is $\FF_\alpha$ real closed and $\phi_\alpha:\KK[X_1,\ldots, X_n]\to \FF_\alpha$ such that $\FF_\alpha$ is the real closure of the field of fractions of $\phi_\alpha(\KK[X_1,\ldots, X_n])$ and such that 
	\bqn
	\alpha:=\{f\in\KK[X_1,\ldots,X_n]:\,\phi_\alpha(f)\geq0\}.
	\eqn 
	Let $x_\alpha=(\phi_\alpha(X_1),\ldots,\phi_\alpha(X_n))\in S_{\FF_{\alpha}}$. Then $\alpha=\alpha_{x_\alpha}$.	
\end{remark}


Now that we have introduced the objects, we conclude this section with
a general mechanism  according to which specialization arises through morphisms taking values into order convex subrings (see Section~\ref{subsec:intro1}). 
The proof is an immediate consequence of the definitions.

\begin{lem}\label{lem:special}
	Let $\alpha\in\rsp V$ and assume that the corresponding morphism $\phi_\alpha:\KK[V]\to \FF_\alpha $ 
	takes values  in a order convex subring $\calO\subset \FF_\alpha $.
	Then the composition of $\phi_\alpha $ with the reduction modulo $\calI$, $\mod_\calI:\calO\to\FF_\calO$,
	produces a new point $\alpha\mod\calI\in\rsp V$, which is in the closure of $\alpha$. 
	In particular if $\alpha\in \cons(S)$, where $S$ is a closed semialgebraic subset of $V$,
	then $\alpha\mod\calI$ is in $\cons(S)$ as well.
	If $\calO$ is the convex subring generated by
	$\phi_\alpha(\KK[V])$ then $\alpha\mod\calI=\Ret(\alpha)$.
\end{lem}

\subsection{Density of rational points}\label{s.density}
In this subsection we draw a consequence of a result of Brumfiel \cite[\S 4]{Brum2} adapted to our needs. Let 
$$V\subset \KK^n$$
be real algebraic where $\KK\subset \RR$ is real closed.
\begin{defn}
	A point $\alpha\in\rsp V$ is {\em rational} if  \label{n.26}
	\be
	\item $\alpha$ is a closed point;
	\item $\Frac(\KK[V]/\frak p_\alpha)$ has transcendence degree $1$ over $\KK$.
	\ee
\end{defn}
Our interest in rational points lies in the following:
\begin{lemma}
	If $\alpha\in \rsp V$ is rational, then the valuation on $\Frac(\KK[V]/\frak p_\alpha)$ with respect to any big element in $\phi_{\alpha}(\KK[V])$ is discrete.
\end{lemma}
\begin{proof}
	Let $b\in \phi_\alpha(\KK[V])$ be a big element in $\Frac(\KK[V]/\frak p_\alpha)$. Clearly $b$ has to be transcendental over $\KK$; now there is a unique order on $\KK(b)$ making $b$ a big element and then the corresponding valuation is discrete. Since $\Frac(\KK[V]/\frak p_\alpha)$ is a finite algebraic extension of $\KK(b)$,
	the valuation with respect to $b$ on $\Frac(\KK[V]/\frak p_\alpha)$ is discrete as well.
\end{proof}	

From \cite[Proposition 4.2]{Brum2} we furthermore deduce 
\begin{cor}\label{c.Bru4.2}
	Let $S\subset V$ be a closed semialgebraic subset. Then the set of rational points 
	is dense in $\rspcl\partial(S)$
\end{cor}	


\section{General properties of real spectrum}\label{s.RspGenProp}
In this section we prove some important  properties of the real spectrum of a real algebraic set 
that will play a crucial role in our study of character varieties.
In Section~\ref{subsec:3} we show how to realize any point in the real
spectrum as a homomorphism in a subfield of a hyper-$\KK$ field, 
in Section~\ref{subsec:accessibility} we show that any closed point in the real spectrum compactification can be obtained as limit of a continuous, locally semialgebraic path, and in  Section~\ref{sec:continuity} we prove a continuity lemma which will be useful in the study of the Weyl chamber length compactification.

\subsection{The Hyper-$\KK$ field realization}\label{subsec:3}
If $\KK$ is a real closed field, $V\subset\KK^n$ a real algebraic set and $\FF\supset\KK$ a real closed field, 
we saw in Remark~\ref{rem:point->prime cone}
that any point $y\in V_\FF$ defines a prime cone
\bqn
\alpha_y:=\{f\in\KK[V]:\,f(y)\geq0\}\in \rsp{V}.
\eqn 
The next result (Proposition \ref{prop:1}) asserts that any point $\alpha\in\rsp{V}$ 
arises in this way, where $\FF=\KK^\omega$ is a hyper-$\KK$ field for some ultrafilter $\omega$ depending on $\alpha$.

We start by recalling that, if $\beta\in\rsp{V}$ is a prime cone and
$\fp_\beta=\{f:\,f(\beta)=0\}$ is the prime ideal of $\KK[V]$ associated to $\beta$, the {\em support variety} $V(\beta)$ is defined as
\bqn
V(\beta):=\{x\in V:\,f(x)=0\text{ for all }f\in\fp_\beta\}\,.
\eqn

\begin{prop} For every prime cone $\beta\in\rsp{V}$, the support variety $V(\beta)$  \label{n.27}
	is not empty.
\end{prop}

\begin{proof}  Let  $f_1,\dots, f_r$ be a finite set of generators of
	the prime ideal $\fp_\beta$.  Thus
	\bqn
	V(\beta)=\{x\in V:\,f_i(x)=0,\,i=1,\dots,r\}\,,
	\eqn
	so that $V(\beta)$ is a semialgebraic set.  But $c(V(\beta))\ni\beta$,
	so that $V(\beta)\neq\varnothing$ by Theorem~\ref{thm:constr}(1).
\end{proof}

\begin{remark}  If $\beta\notin V$ then, by chasing the definitions, it is easy to see 
	that $V(\beta)$ is necessarily infinite.
\end{remark}

If $\omega$ is an ultrafilter over $\NN$ and $x=(x_k)_{k\in \NN}$ is a sequence of points in $V$, 
we denote by ${x^\omega} \in V_{\KK^\omega}\subset(\KK^\omega)^n$ the associated point
where, if $x_k=(x_k^{(1)},\ldots,x_k^{(n)})\in\KK^n$, 
\bq\label{eq:xomega}
x^\omega=((x^\omega)^{(1)},\ldots, (x^\omega)^{(n)})
\text{ with }
(x^\omega)^{(i)}=[(x_k^{(i)})]_\omega\,.
\eq

\medskip
Observe that $\alpha_{x^\omega}$ could be equivalently described as the prime cone associated to the homomorphism 
\bqn
\phi^\omega:\begin{array}[t]{ccc}
	\KK[V]&\to&\oK\\f&\mapsto&[(f(x_k))]_\omega
\end{array}
\eqn
and that
\bq
\alpha_{x^\omega}
=\{f\in\KK[V]: f(x_k)\geq 0 \text{ for  $\omega$-almost all } k \} \;.
\eq
In particular $\phi^\omega$ induces a field injection
\bq\label{eq:InjectionInHyperreals}
\FF_{\alpha_{x^\omega}}\hookrightarrow \oK\;.
\eq

%
\begin{prop}\label{prop:1}
	Assume $\KK\subset\RR$ is a real closed field.
	Given $\beta\in\rsp V$ there exists an ultrafilter $\omega$ on $\NN$
	and a sequence $x=(x_k)_{k\in \NN}$ in $V(\beta)$ such that
	$\beta=\alpha_{x^\omega}$.
	Moreover, given any semialgebraic set $S$ such that $\cons(S)$
	contains $\beta$,
	the subset $\{k\in\NN:\;x_k\in S\cap V(\beta)\}$ 
	belongs to $\omega$.           
\end{prop}

\begin{remark}Assume that $\beta=\alpha_y$ where $y\in V_\RR\setminus V$; then the ultrafilter $\omega$ and the sequence $(x_k)_k$ have the property that $\lim_{\omega}x_k=y$ in the topology of $V_\RR$.
\end{remark}	

\begin{proof}[Proof of Proposition~\ref{prop:1}]
	We may assume that $\beta\notin V$;  then $V(\beta)$ is infinite and we choose an injection
	$\NN\to V(\beta)$, $k\mapsto x_k$, such that $\calB:=\{x_k\}_{k\in \NN}$ is dense in $V(\beta)$. This is possible since $\KK\subset\RR$.
	In the proof we will use the family 
	\bqn
	F_\beta:=\{S\subset V:\,S\text{ is semialgebraic and }\cons(S)\ni\beta\}
	\eqn
	to construct an ultrafilter $\omega$ on $\NN$ with the desired property. 
	
	We claim that for each $S\in F_\beta$, the intersection $S\cap \calB$ is not empty. 
	Indeed we will show that, for every $S\in  F_\beta$, $S\cap V(\beta)$ contains a non-empty open subset of $V(\beta)$.
	
	If $S\in F_\beta$, then $S$ is a finite union of sets of the form $U=\{f_1=0,\dots,f_\ell=0,g_1>0,\dots,g_r>0\}$.
	Since at least one of these sets is in $F_\beta$, we may assume that $S=U$.  
	Since $\cons(S)\ni\beta$, then 
	\bqn
	f_1(\beta)=\dots=f_\ell(\beta)=0
	\eqn 
	and hence, by definition of support variety, 
	\bqn
	f_1|_{V(\beta)}=\dots=f_\ell|_{V(\beta)}=0\,.
	\eqn 
	Thus 
	\bqn
	S\cap V(\beta)=\{x\in V(\beta):\,g_1(x)>0,\dots,g_r(x)>0\}\,,
	\eqn
	which is open in $V(\beta)$.
	In addition, since $S\in F_\beta$ and $V(\beta)\in F_\beta$, then
	$S\cap V(\beta)\in F_\beta$ and in particular it is not empty. 
	
	This implies that
	$$\omega_{pf}:=\{\{k\in\NN:\, x_k\in S\cap\calB\},\,S\in F_\beta \}$$
	consists of a family of subsets of $\NN$ that contains all finite intersections  of its elements 
	and hence $\omega_{pf}$ is contained in some ultrafilter $\omega$ on $\NN$.
	
	Let ${x^\omega}\in {V^\omega}\subset({\KK^\omega})^n$ be the point defined by the sequence $(x_k)_{k\in\NN}$ 
	as in \eqref{eq:xomega}.
	We claim that $\beta=\alpha_{x^\omega}$. 
	To see that $\alpha_{x^\omega}\subset\beta$, observe that if $f(x^\omega)\geq0$, there exists $P\in\omega$ 
	such that $f(x_k)\geq0$ for all $k\in P$.  From this we want to deduce that $f(\beta)\geq0$.
	By contradiction, if $f(\beta)<0$, then $S=\{x\in V:\,f(x)<0\}$ is a semialgebraic subset
	with $\beta\in \cons(S)$.  Therefore $S\cap V(\beta)$ is open and not empty,
	and $\{k:\,x_k\in S\cap V(\beta)\}\cap P=\varnothing$,
	which is a contradiction since $\{k:\,x_k\in S\cap V(\beta)\}\subset\omega$.  
	In particular $\fp_{\alpha_{x^\omega}}\subset\fp_\beta$.
	
	To show the other inclusion, we show that if $f(\beta)=0$, then $f(x^\omega)=0$. 
	In fact from $f(\beta)=0$ it follows that $f|_{V(\beta)}=0$.  Thus $f(x_k)=0$ for all $x\in \NN$
	and hence $f(x^\omega)=0$.   Thus $\alpha_{x^\omega}\subset \beta$ and $\frak p_{\alpha_{x^\omega}}=\frak p_\beta$
	which by Lemma~\ref{lem:closed_points}
	implies $\alpha_{x^\omega}= \beta$. \end{proof}

We  apply the preceding result to obtain a representation of non-Archimedean points in $\rsp V$, 
that is points in $\rsp V\setminus V_\RR$, involving Robinson fields. 
This uses order convex subrings of $\KK^\omega$ to obtain specializations of points in $\rsp V\setminus V_\RR$ 
(recall Lemma~\ref{lem:special}).

We realize $\beta\in\rsp V$ as $\alpha_{x^\omega}$ as in Proposition~\ref{prop:1}. 
This leads to an order preserving field injection $\FF_\beta\hookrightarrow\oK$
induced by the map $\KK[V]\to\KK^\omega$, $f\mapsto f(x^\omega)$,
which factors through $\fp_\beta$, by passing to a real closure of the corresponding field of fractions.
Let $\plambda\in\oK$ be a positive infinite element such that
$|{(x^\omega)^{(i)}}|<\plambda$ for all $1\leq i\leq n$,
where
\bqn
x^\omega=((x^\omega)^{(1)},\ldots, (x^\omega)^{(n)})
\eqn 
so that $x^\omega\in(\calO_\plambda)^n$.
We let $\oxl\in V_{\oKl}$ be the image of $\ox$ under reduction modulo $\calI_\plambda$ 
and let $\alpha_{\oxl}\in \rsp V$ be the corresponding point in the real spectrum. 
Using again the norm $N(x)=\sqrt{\sum_{j=1}^n x_j^2}$ from \S\ref{s:Intro2} we have
\begin{cor}\label{c.1oKl}
	Let $\beta=\alpha_{x^\omega}\in \rsp V\setminus V_\RR$ and 
	let $\plambda\in\KK^\omega$ be an infinite element bounding $|(x^\omega)^{(i)}|$, $1\leq i\leq n$. Then
	\begin{enumerate}
		\item If $\FF_\beta\subset\calO_\plambda$, then $\beta=\alpha_{\oxl}$;
		\item if
		$\plambda=[(N(x_k)^2)]_\omega\in{\KK^\omega}$
		then
		\bqn
		\Ret(\beta)=\alpha_{x^\omega_\plambda}\,,
		\eqn
		where $\Ret$ is the retraction map defined in \eqref{eq:retraction}.
	\end{enumerate}
\end{cor}
In particular this implies that any closed point $\beta\in \rspcl V\setminus V_\RR$
can be realized as $\alpha_{x^\omega_\plambda}$, so that for
character varieties points in the real spectrum compactification can be represented by asymptotic
cones of sequences of representations.
Notice that if $b\in\FF_\beta$ is a big element and there exists $m\in\NN_{\geq1}$ such that $b<\plambda^m$,
then the condition $\FF_\beta\subset\calO_\plambda$ in Corollary~\ref{c.1oKl}~(1) is satisfied.

\begin{proof}[Proof of Corollary~\ref{c.1oKl}]
	(1) follows from the fact that if $\FF_\beta\subset\calO_\plambda$, then the reduction modulo $\calI_\plambda$ gives a field homomorphism $i:\FF_\beta\to \oKl$ which is then necessarily injective.
	To see (2), observe that since $\beta\notin V_\RR$,  $\plambda=[(N(x_k)^2)_{k\geq1}]_\omega\in{\KK^\omega}$ is an infinite element and $x^\omega\in V_{\KK^\omega}\cap(\calO_\plambda)^n$.
	Since the quotient map $\calO_\plambda\to\oKl$ is order preserving, we conclude that 
	\bqn
	\beta=\alpha_{x^\omega}\subset\alpha_{x_\plambda^\omega}\,.
	\eqn
	It remains to show that $\alpha_{x_\plambda^\omega}$ is a closed point.  Consider
	\bqn
	\ba
	\phi\colon\KK[V]&\longrightarrow\oKl\\
	f\quad&\mapsto f(x_\plambda^\omega)\,,
	\ea
	\eqn
	and let $\FF(x_\plambda^\omega)$ be the real closure of the field of fractions of $\phi(\KK[V])$.
	Let $P\in\KK[V]$, $P(x):=\sum_{i=1}^nx_i^2$.  
	Then for every $f\in\KK[V]$, $\phi(f)=f(x_\plambda^\omega)$ is bounded by a power of $\phi(P)=P(x_\plambda^\omega)=\plambda$.
	Since $\plambda$ is a big element of $\oKl$ and hence $\oKl$ is Archimedean over $\phi(\KK[V])$,
	thus so is $\FF(x_\plambda^\omega)$.  By Proposition~\ref{prop:closed_points} this implies that $\alpha_{x_\plambda^\omega}$ is closed,
	hence $\alpha_{x_\plambda^\omega}=\Ret(\beta)$.
\end{proof}

\subsection{An accessibility result}\label{subsec:accessibility}
In this section we prove an accessibility result 
which will be useful in proving good topological properties of various compactifications of character varieties. 
The proof uses the full strength of semialgebraic geometry,  as it uses \cite[Corollary 9.3.3]{BCR} 
which shows that the ends  of semialgebraic sets are tame in a very strong sense. 

\medskip
Let, as before, $V\subset \KK^n$ be real algebraic. 
We say that a function $\norm:V\to\KK$ is proper if the inverse image of any bounded interval is a bounded subset of $\KK^n$. For the next proposition the assumption that $\KK$ is countable is crucial, as it guarantees that the topology of $\rspcl V$ is separable.

\begin{prop}\label{prop:2}
	Assume that $\KK\subset\RR$ is countable real closed. Let $S\subset V\subset \KK^n$ be semialgebraic and $\norm: V\to[0,\infty)$ be  continuous, proper, semialgebraic. 
	Then for every $\alpha\in \rspcl  V\cap \cons( S)$ non-Archimedean, that is $\alpha\in \rsp V\setminus V_\RR$, there exists $T\geq 0$ and 
	\bqn
	s_\alpha:[T,\infty)\to S\cap V(\alpha)
	\eqn
	continuous, semialgebraic on bounded intervals
	such that
	\begin{enumerate}
		\item $\norm(s_\alpha(t))=t \quad \forall t\geq T$
		\item $\lim_{t\to\infty} s_\alpha(t)=\alpha$ the convergence taking place in $\rspcl  V$.
	\end{enumerate} 
\end{prop}

\begin{proof}
	As $\KK$ is countable, the topology on $\rsp V$ is second countable. 
	As $\rspcl  V$ is furthermore Hausdorff we can find a sequence $U_1\supset U_2\supset\ldots$ of open semialgebraic sets 
	such that $\{\cons( U_k)|\;k\geq 1\}$ forms a basis of neighbourhoods of $\alpha$ in $\rsp V$, in particular $\cap_{k\geq 1} c( U_k)\cap \rspcl  V=\{\alpha\}$.
	
	Let $E_k$ denote the intersection $U_k\cap S\cap V(\alpha)$. Since $\norm:V\to [0,\infty)$ is proper, the prime cone 
	$\alpha$ is non-Archimedean and $\alpha\in \cons(E_k)$, the image $\norm(E_k)\subset (0,\infty)$ is unbounded semialgebraic, hence contains a ray $[a_k,\infty)$ for some $a_k>0$.
	
	Thus by \cite[Corollary 9.3.3]{BCR}  there exists  $b_k\geq a_k$  such that $\norm|_{E_k}$ is semialgebraically trivial over $(b_k,\infty)$ (see \cite[Definition 9.3.1]{BCR});
	we may  assume that $b_k\leq b_{k+1}$. Set $R_k:=\norm|_{E_k}^{-1}(b_k,\infty)$. From $b_k\leq b_{k+1}$ and $E_k\supset E_{k+1}$, it follows $R_k\supset R_{k+1}$ for all $k\geq 1$. Observe that since $\norm:R_k\to(b_k,\infty)$ is semialgebraically trivial over $(b_k,\infty)$, so is $\norm|_C$ for every connected component $C$ of $R_k$. Also, since $C$ is open and closed in $R_k$, the intersection $C\cap R_{k+1}$ is a union of connected components of $R_{k+1}$. We can thus inductively choose for every $k\geq 1$ a connected component $C_k$ of $R_k$ such that $C_k\supset C_{k+1}$ and $\alpha\in\cons(C_k)$.
	
	Finally choose a sequence $x_k\in C_k$ such that $\norm(x_k)<\norm(x_{k+1})$ for every $k\geq 1$. Observe that, since $x_k\in C_k\subset U_k$ and $\bigcap_{k=1}^\infty \cons(U_k)=\{\alpha\}$, we have $\bigcup_{k=1}^\infty[\norm(x_k),\norm(x_{k+1})]=[\norm(x_1),\infty)$. Now choose, for every $k\geq 1$ a continuous semialgebraic path
	$$p_k:[\norm(x_k),\norm(x_{k+1})]\to C_k$$
	with $p_k(\norm(x_k))=x_k$, $p_k(\norm(x_{k+1}))=x_{k+1}$ and such that $\norm\circ p_k(t)=t$, for all $t\in[\norm(x_k),\norm(x_{k+1})]$.
	Let $s_\alpha:[\norm(x_1),\infty)\to C_1$ be the continuous locally semialgebraic path obtained as concatenation of the paths $p_k$. Then $s_\alpha$ takes values in $S\cap V(\alpha)$ and  $\lim_{t\to\infty}s_\alpha(t)=\alpha$ in $\rspcl V$ because $s_\alpha([\norm(x_k),\infty))\subset U_k$.

\end{proof}
Let now 
\bqn
\norm_\RR:V_\RR\to [0,\infty)_\RR
\eqn
denote the continuous semialgebraic extension.

\begin{cor}\label{c.access}
	Under the assumptions of Proposition \ref{prop:2} 
	there exists $s_\alpha^\RR:[T,\infty)\to S_\RR\cap V(\alpha)_\RR$ locally semialgebraic, continuous, such that 
	\begin{enumerate}
		\item $\norm_\RR\circ s_\alpha^\RR(t)=t$
		\item $\lim_{t\to\infty}s_\alpha^\RR(t)=\alpha.$
	\end{enumerate}
\end{cor}

\begin{proof}
	We define $s_\alpha^\RR$ on each interval $[\norm(x_k),\norm(x_{k+1})]_\RR$ as the semialgebraic extension of $s_\alpha$. 
	Then (1) follows from uniqueness of the extension, and for (2) we have that $s_\alpha^\RR([\norm(x_k),\norm(x_{k+1})])\subset (U_k)_\RR\subset c( U_k)$, 
	which implies that $\lim_{t\to\infty}\norm_\RR(t)=\alpha$.
\end{proof}

\begin{lem}\label{l.access}
	Assume  $S\subset V\subset \KK^n$ is semialgebraic with $\KK$ as in Proposition \ref{prop:2} and $\alpha\in\rspcl{V}\cap \cons(S)$ is non-Archimedean. Let $\omega$ be any non-principal ultrafilter, $(x_k)_{k\in\NN}$ a sequence in $S_\RR\cap V(\alpha)_\RR$ with $\lim x_k=\alpha$ and $\plambda=[(N(x_k)^2)]_\omega\in\oK$. Then $\alpha_{x^\omega_\plambda}=\alpha$.
\end{lem}	
\begin{proof}
	Assume that $\alpha_{\oxl}\neq \alpha$. Since both points are closed (Corollary \ref{c.1oKl} (2)), there exist disjoint open semialgebraic sets $U, U'$  such that $\alpha_{\oxl}\in c(U)$, $\alpha\in c(U')$.
	Then $\omega\{k|\; x_k\in U\}=1$, which implies that $\omega\{k|\; x_k\in (U')^c\}=1$. Since the space $\rspcl V$ is Hausdorff, and, by assumption, the sequence $(x_k)_{k\in\NN}$ converges, its limit is unique and agrees with any possible $\omega$-limit. We deduce that $\alpha=\lim x_k=\lim_\omega x_k\in c((U')^c)$, a contradiction.
\end{proof}
\begin{cor}\label{cor:3.9}
	Under the assumptions of Proposition \ref{prop:2} let $\omega$ be a
	non-principal ultrafilter on $\NN$ and
	$\plambda=(\mu_k)_{k\in\NN}$ with $\lim\mu_k=+\infty$. For
	every $\alpha\in\rspcl V\cap \cons(S)$ non-Archimedean, there exists a
	sequence $(x_k)_{k\in\NN}$ in $V(\alpha)_\RR\cap S_\RR$ such that
	$[(N(x_k)^2)]_{\omega}=[(\mu_k)]_{\omega}$ in $\oK$ and
	$\alpha=\alpha_{\oxl}$.
\end{cor}

\begin{proof}
	We apply Corollary \ref{c.access} with $g:\KK^n\to \KK$, $g(x)=N(x)^2$ to obtain $s_\alpha^\RR:[T,\infty)\to S_\RR\cap V(\alpha)_\RR$  satisfying (1) and (2). Setting $x_k:=s_\alpha^\RR(\mu_k)$, we obtain that $N(x_k)^2=\mu_k$, $lim_{k\to \infty}x_k=\alpha$ and Lemma \ref{l.access} gives $\alpha_{x_\plambda^\omega}=\alpha$.
\end{proof}
\subsection{Continuity of real valued maps}\label{sec:continuity}
In this subsection we establish a continuity result (Proposition~\ref{p.cont}) that will be crucial in many applications, 
the first one being the continuity of the map from the real spectrum of a character variety 
to the Weyl chamber length compactification (see Theorem \ref{t.ThPINTRO}).

Recall that if $\KK\subset\RR$ is a real closed field, $V\subset\KK^n$ an algebraic set
and $T\subset V$ a closed semialgebraic set, 
then $\cons(T)\cap\rspcl{V}$ is a Hausdorff compactification of $T_\RR$ (Proposition~\ref{prop:closed}).

\begin{prop}\label{p.cont} Let $g\colon\KK^n\to\KK$ be a Euclidean norm, and  $T\subset V$ a closed semialgebraic set. 
	If $f\colon T\to\KK$ is a continuous semialgebraic function such that $f(x)>0$ for every $x\in T$,
	then the function 
	\bqn
	\begin{array}{ccc}
		T&\longrightarrow&\RR\\
		x&\mapsto&\displaystyle{\frac{\ln(f(x))}{\ln(2+g(x))}}
	\end{array}
	\eqn
	extends continuously to $\cons(T)\cap\rspcl{V}$ with extension given by 
	\bqn
	\begin{array}{cccc}
		F\colon&\cons(T)\cap\rspcl{V}&\longrightarrow&\RR\\
		&\alpha&\mapsto& \displaystyle{\frac{\log_{b_\alpha}(f(\alpha))}{\log_{b_\alpha}(2+g(\alpha))}}\,,
	\end{array}
	\eqn
	where $b_\alpha\in\FF_\alpha$ is a big element and $f(\alpha)$ is defined below.
\end{prop}
\begin{remark}\
	\be
	\item The value of the extension of $F$ on $T_\RR\subset \cons(T)\cap\rspcl V$ is given by $\frac{\log(f(x))}{\log(2+g(x))}$.
	\item The strategy of proof of Proposition \ref{p.cont} is taken from Brumfiel \cite[Proposition 5.3]{Brum2}
	\ee
\end{remark}
If $\alpha\in \cons(T)$, let $\phi_\alpha\colon\KK[V]\to\FF_\alpha$ 
a homomorphism representing $\alpha$, for a suitable real closed extension $\FF_\alpha$ of $\KK$.  
We give more detail  about  the definition of $f(\alpha)\in\FF_\alpha$ introduced in \S\ref{subsec:real_spectrum}
(see also \cite[Proposition~7.3.1]{BCR}).
Let $f_{\FF_\alpha}\colon T_{\FF_\alpha}\to\FF_\alpha$ be the extension of $f$ to $\FF_\alpha$ as defined in Section~\ref{s:Intro2}.
Observing that $(\phi_\alpha(X_1),\dots,\phi_\alpha(X_n))\in T_{\FF_\alpha}$, we define
\bqn
f(\alpha):=f_{\FF_\alpha}(\phi_\alpha(X_1),\dots,\phi_\alpha(X_n))\in\FF_\alpha\,,
\eqn
where $\KK[X_1,\ldots, X_n]$ is the coordinate ring of $\KK^n$.

Observe that $\FF_\alpha$, and hence $f(\alpha)$, depend on a choice of a real closure 
of the fraction field of $\KK[V]/\alpha\cap(-\alpha)$.  
However the property that $f(\alpha)$ is positive is well defined independently of this choice.
\begin{proof}
	Since $\alpha$ is a closed point, $\FF_\alpha$ is Archimedean over $\phi_\alpha(\KK[V])$
	and hence $g(\alpha)$ is a big element (as in the proof of Proposition~\ref{prop:closed} with $Q=g$).
	Thus, using  the definition of the logarithm with respect to a big element (recall Equation \eqref{eq:Aa} and \eqref{eq:Ba} in \S~\ref{s.valuation}),
	one can verify that, for any big element $b_\alpha$,
	\bqn
	\frac{\log_{b_\alpha}(f(\alpha))}{\log_{b_\alpha}(2+g(\alpha))}={\log_{(2+g(\alpha))}(f(\alpha))}\,.
	\eqn
	Then
	${\log_{(2+g(\alpha))}(f(\alpha))}\leq s\,\text{ if and only if }\,( 2+g(\alpha))^m\geq f(\alpha)^l\,\text{ whenever }\frac{m}{l}> s\,.$
	Thus
	\bqn
	F^{-1}((-\infty,s])=\bigcap_{\frac ml> s}\{\alpha\in \cons(T)\cap\rspcl{V}:\,( 2+g(\alpha))^m\geq f(\alpha)^l\}\,.
	\eqn
	Likewise
	\bqn
	F^{-1}([s,\infty))=\bigcap_{\frac ml<s}\{\alpha\in \cons(T)\cap\rspcl{V}:\, (2+g(\alpha))^m\leq f(\alpha)^l\}\,.
	\eqn
	Since $g$ and $f$ are continuous and semialgebraic, the set
	\bqn
	S_{m,l}=\{ x\in T:\, (2+g(x))^m\geq f(x)^l\}\,,
	\eqn
	is closed and semialgebraic.  It follows then from
        Theorem \ref{thm:constr} 	that
	\bqn 
	c(S_{m,l})=\{\alpha\in \cons(T):\,(2+g(\alpha))^m\geq f(\alpha)^l \}\,,
	\eqn
	which is closed by Theorem~\ref{thm:constr}\,(2).
	Its intersection with $\rspcl{V}$ is closed in $\rspcl{V}$ and hence $F^{-1}((-\infty,s])$ is closed.
	The argument for $F^{-1}([s,\infty))$ is completely analogous.
	Thus for any $s_1<s_2$ in $\RR$, $F^{-1}((-\infty,s_2))$ and $F^{-1}((s_1,\infty))$ are open 
	and so is $F^{-1}((s_1,s_2))$.
	
\end{proof}

\section{Semialgebraic groups}\label{sec:3}
In this section we discuss preliminaries on semisimple (semi)algebraic groups defined over real closed fields. 
The main goal is to set up the classical theory so that  relevant concepts such as the theory of parabolic subgroups, 
Jordan and Cartan decomposition, as well as the theory of proximal elements are done in a semialgebraic way 
and can thus be transfered to any real closed field. 
\subsection{Semisimple groups over real closed fields and their parabolic subgroups}\label{subsec:4.1}
Let $\bG<\SL(n)$ be a connected\footnote{The connectedness of $\bG$ will play an important role, 
for example, in Lemma \ref{lem:4.6}.} semisimple algebraic group defined over a field $\KK$. 
We will, as we may, consider $\bG$ as a functor from commutative $\KK$-algebras to groups (see \cite[0.3]{Borel_Tits})
and use some results of \cite{Borel_Tits} (see also \cite{Borel1}) in terms of this language.  \label{n.28}

Let $\bS$ be a maximal $\KK$-split torus  \label{n.29}
and $X(\bS)$ be the set of rational characters of $\bS$,  \label{n.30}
that are the rational group morphisms from $\bS$ to $\GL(1)$.  Since $\bS$ is $\KK$-split, it follows from \cite[Proposition~1.3]{Borel_Tits} that
all elements in $X(\bS)$ are defined over $\KK$. We denote by ${}_\KK\Phi\subset X(\bS)$ the set of $\KK$-roots of $\bG$ with respect to $\bS$ \label{n.31}
(see \cite[21.1]{Borel1} for the terminology). Let $\bP$ be a minimal parabolic $\KK$-subgroup containing the centralizer $\calZ(\bS)$ of $\bS$ in $\bG$. Then ${}_\KK\Phi\subset X(\bS)$ is endowed with the order such that $\bP$ is associated to the set of positive roots ${}_\KK\Phi^+\subset X(\bS)$. Recall that a \emph{standard parabolic $\KK$-subgroup} is one that contains $\bP$. Then:
\begin{enumerate}
	\item Any parabolic $\KK$-subgroup is $\bG(\KK)$ conjugated to a standard one \cite[21.12]{Borel1}.
	\item $\bQ\mapsto\bQ(\KK)$ is a bijection between the set of standard parabolic $\KK$-subgroup and the set of abstract subgroups of $\bG(\KK)$ containing $\bP(\KK)$ \cite[21.16, Theorem 21.15, 14.15, 14.16]{Borel1}.
\end{enumerate}	

For a detailed description of the standard parabolic $\KK$-subgroups see \cite[21.11]{Borel1}. Our convention, which differs from Borel, but agrees with the current literature on Anosov representations is the following:  Let ${}_\KK\Delta\subset {}_\KK\Phi^+$ be the set of simple roots. Then to every subset $I\subset{}_\KK\Delta$ corresponds a standard parabolic $\KK$-subgroup ${}_\KK\bP_{I}$, \label{n.32}
and conversely; this correspondence has the property that if 
$\bS_I:=\left(\bigcap_{\alpha\in {}_\KK\Delta\setminus I}\ker\alpha\right)^\circ$  \label{n.33}
then $\calZ(\bS_I)$ is a Levi $\KK$-subgroup of ${}_\KK\bP_{I}$.  \label{n.34}
In particular the minimal parabolic subgroup $\bP$ corresponds to the set $I={}_\KK\Delta$.

The proof of the following proposition uses classical results on Cartan subalgebras of real semisimple algebras 
as well as the transfer principle.  Details can be found in \cite{Appenzeller}.

\begin{prop}\label{prop:F-split=K-split} 
	Let $\KK\subset\RR$ be a real closed field 
	and $\FF\supset\KK$ any real closed field containing $\KK$.
	
	\begin{enumerate}
		\item	If $\bS<\bG$ is a maximal $\KK$-split torus, then $\bS$ is maximal $\FF$-split as well.
		\item If $\bG$ is invariant under transposition then $\bS$ can be chosen to consist of symmetric matrices.
		\item $\bP$ is $\FF$-minimal.
	\end{enumerate}
\end{prop}
As a result 
\begin{cor}\label{cor:4.2}Under the same assumptions as in Proposition \ref{prop:F-split=K-split}:
	\begin{enumerate}
		\item Any parabolic subgroup of $\bG$  defined over $\FF$ is $\bG(\FF)$-conjugate to a parabolic subgroup
		defined over $\KK$. In particular standard parabolic $\KK$-subgroups coincide with  standard parabolic $\FF$-subgroups.
		\item $\bQ\mapsto \bQ(\FF)$ is a bijection 
		between the set of standard parabolic $\KK$-subgroups 
		and the set of abstract subgroups of $\bG(\FF)$ containing $\bP(\FF)$.
	\end{enumerate}	
\end{cor}	
{ Since a parabolic subgroup is self normalizing, we can understand the algebraic set $\bGmQ(\KK)$ 
as parametrizing the set of  parabolic $\KK$-subgroups that are
$\bG(\KK)$-conjugate to $\bQ$. 
We then have, for every real closed extension $\FF\supset \KK$,  $\bGmQ(\FF)=\bG(\FF)/\bQ(\FF)$.

} 
%
%
%
%

\subsection{Semialgebraic $\KK$-groups and Cartan decomposition}\label{subsec:Cartan}
\begin{defn}\label{defn:4.3}
	A \emph{semisimple, semialgebraic  $\KK$-group} is a group $G$ \label{n.35}
	for which there exists a connected semisimple linear-algebraic group $\bG\subset\SL(n)$ defined over $\KK$ such that 
	\bqn
	\bG(\KK)^\circ< G< \bG(\KK)\,,
	\eqn
	Here $\bG(\KK)^\circ$ refers to the semialgebraic connected component of the identity of the real algebraic set $\bG(\KK)$. 
	In particular $G$ is a semialgebraic set being a finite union of cosets.
\end{defn}

Here and in the sequel we use the procedure of extending a semialgebraic set $S$ to a larger real closed field $\FF$ 
(Definition \ref{d.extension}). 
Observe that $\bG(\RR)$ is the $\RR$-extension of $\bG(\KK)$, as $\bG$ is defined over $\KK$.  It follows from the transfer principle, Property \ref{ppts.2.10} (3), that $\bG(\RR)^\circ$ is the $\RR$-extension of $\bG(\KK)^\circ$ and hence we have the inclusions
\bqn
\bG(\RR)^\circ< G_\RR< \bG(\RR)\,.
\eqn
Let us now assume that $\bG$ is invariant under transposition and define 
\bqn K:= G\cap\SO(n,\RR),\eqn
where, as usual, 
$
\SO(n,\RR):=\{g\in\SL(n,\RR):\,{}gg^t=\Id\}\,.$

From the same transfer principle we deduce that 
\bqn
K_\RR=G_\RR\cap\SO(n,\RR)\,.
\eqn
Since $\bG(\RR)$ is invariant under transposition, $\bG(\RR)\cap\SO(n,\RR)$ is maximal compact in $\bG(\RR)$,
hence meets every connected component of $\bG(\RR)$, that is
\bqn
\bG(\RR)=(\bG(\RR)\cap\SO(n,\RR))\bG(\RR)^\circ\,.
\eqn
Since $\bG(\RR)^\circ\subset G_\RR$, this implies
\bqn
G_\RR=K_\RR\cdot\bG(\RR)^\circ\,.
\eqn
Since $K_\RR$ is a subgroup of $\SO(n,\RR)$, it is obviously invariant under transposition,
which implies that $G_\RR$, hence $G$, is invariant under transposition.

Let now $\bS<\bG$ be a maximal $\KK$-split torus that we can assume consisting of symmetric matrices.
As above, let ${}_\KK\Phi\subset X(\bS)$ be the set of $\KK$-roots of $\bS$ in $\bG$, ${}_\KK\Phi^+\subset{}_\KK\Phi$ a set of positive roots
and $\Delta\subset{}_\KK\Phi^+$ the set of simple roots,
then $\alpha(\bS(\KK))\subset\KK^\times$ for every $\alpha\in X(\bS)$
and hence, if $S:=\bS(\KK)^\circ$ \label{n.36}
is the semialgebraic connected component of $\bS(\KK)$ containing the identity,
$\alpha(S)\subset\KK_{>0}$.  The \emph{multiplicative Weyl chamber} in $S$ relative to the order ${}_\KK\Phi^+$ is the semialgebraic set defined by 
\bqn
C=\{\nu\in S:\,\alpha(\nu)\geq1\text{ for all }\alpha\in{}_\KK\Phi^+\}\,.
\eqn
Using this we can extend the classical Cartan decomposition to any real closed field:
\begin{prop}[Cartan Decomposition]\label{prop:4.4}  Let $G$ be a semialgebraic semisimple $\KK$-group (as in Definition~\ref{defn:4.3}),
	where we assume $\bG$ to be invariant under transposition.  Let $K$ and $C$ be as above and let $\FF\supset\KK$ be real closed fields.
	Then for every $g\in G_\FF$ there is a unique $c(g)\in C_\FF$ such that \label{n.37}
	\bqn
	g\in K_\FF\,c(g)\,K_\FF\,.
	\eqn
	Furthermore the map $c\colon G_\FF\to C_\FF$ is semialgebraic continuous.
\end{prop}

\begin{proof}  By the transfer principle
	\bqn
	C_\RR=\{\nu\in\bS(\RR)^\circ:\,\alpha(\nu)\geq1\text{ for all }\alpha\in{}_\KK\Phi^+\}
	\eqn
	and the proposition for $\FF=\RR$ is the content of the classical Cartan decomposition,
	taking into account Proposition~\ref{prop:F-split=K-split} and the fact, previously established,
	that $G_\RR$ is invariant under transposition.
	Hence if $\calG$ denotes the projection to $G\times C$ of the semialgebraic subset
	\bqn
	\{(g,c,k_1,k_2)\in G\times C\times K\times K:\,g=k_1\,c\,k_2\}\subset G\times C\times K\times K\,,
	\eqn
	then its extension $\calG_\RR$ to $\RR$ is the graph of a continuous semialgebraic map,
	and so is $\calG$ by the transfer principle, Properties \ref{ppts.2.10}.  Applying the transfer principle to $\FF\supset \KK$ 
	implies the proposition.
\end{proof}

\subsection{The Jordan projection}\label{subsec:Jordan}
Let $\bG(\KK)^\circ<G<\bG(\KK)$ be semialgebraic $\KK$-semisimple,
where $\bG<\SL(n)$ is invariant under transposition.  Let $\bS<\bG$ be a maximal $\KK$-split torus
consisting of symmetric matrices and let $C\subset S:=\bS(\KK)^\circ$ be a closed multiplicative Weyl chamber.
Our aim is to generalize the existence and continuity of the classical refined Jordan projection $\Jord:G_\RR\to C_\RR$.
To this end we first deal with the (refined) Jordan decomposition over $\KK$.

Let $K=G\cap\SO(n)$, $S=\bS(\KK)^\circ$ as in Section~\ref{subsec:Cartan}.
The set $\calE\subset G$ of elements $G$-conjugate to $K$, the set $\calH\subset G$ of elements $G$-conjugate to $S$
and the set $\calU\subset G$ of unipotent elements are all semialgebraic subsets.
The group $G_\RR$ is ``almost real algebraic'' in the sense of \cite[\S~2]{BorelClass}
and $\calE_\RR$ is then the set of elliptic elements in $G_\RR$.  Since $\bS$ is maximal $\RR$-split (Proposition \ref{prop:F-split=K-split} (1)), 
$\calH_\RR$ is the set of hyperbolic elements (``split positive'' in Borel's terminology)
and evidently $\calU_\RR$ is the set of unipotent elements in $G_\RR$.
The refined Jordan decomposition in this context says that, given $g\in G_\RR$, 
there exists unique commuting elements $g_e\in\calE_\RR$, $g_h\in\calH_\RR$ and $g_u\in\calU_\RR$ with
$g=g_eg_hg_u$ \cite[Corollary 2.5]{BorelClass}. 

\begin{lem}\label{lem:4.5}  Let $\FF\supset\KK$ be real closed.  
	Given $g\in G_\FF$ there exist unique commuting elements $g_e\in\calE_\FF$, $g_h\in\calH_\FF$ and $g_u\in\calU_\FF$ 
	with  $g=g_eg_hg_u$.  The map $G\to \calH$, $g\mapsto g_h$ is semialgebraic and so is its extension $G_\FF\to \calH_\FF$.
\end{lem}

\begin{proof}  Let 
	\bqn
	\calT=\{(\epsilon, h, u)\in\calE\times\calH\times\calU:\,[\epsilon,h]=[\epsilon,u]=[u,h]=e\}
	\eqn
	and consider the product map
	\bqn
	\ba
	\calT\quad\,&\longrightarrow \,G\\
	(\epsilon,h,u)&\mapsto \epsilon \,h\,u\,.
	\ea
	\eqn
	Then the extension $\calT_\RR\to G_\RR$ is clearly semialgebraic, and 
	the content of the refined Jordan decomposition is that it is a bijection;
	hence so is $\calT\to G$ and its extension to any real closed field $\FF\supset\KK$.
	Observe that the graph of the map $H:G\to\calH$, $g\mapsto g_h$ is the projection to $G\times\calH$ of the semialgebraic set 
	\bqn
	\{(g,h,\epsilon,u)\in G\times\calH\times\calE\times\calU:\,g=\epsilon \,h\,u,\,[\epsilon,h]=[\epsilon,u]=[u,h]=e\}
	\eqn
	and hence $H$ is semialgebraic.
\end{proof}

Observe that the map $\calT_\RR\to G_\RR$ in the proof of Lemma~\ref{lem:4.5} is continuous bijective,
but its inverse is not continuous, since a sequence of elliptic elements can converge to a non-trivial unipotent element.

In the next lemma, the connectedness of $\bG$ is essential:

\begin{lem}\label{lem:4.6} Let $\FF\supset\KK$ be real closed.
	The $G_\FF$-conjugacy class of any element $h\in\calH_\FF$ meets $C_\FF$ in exacty one point.
	The resulting map $\calH_\FF\to C_\FF$ is semialgebraic.
\end{lem}

\begin{proof}  Consider the projection $\calL$ to $\calH\times C$ of the semialgebraic set 
	\bqn
	\{(h,c,g)\in\calH\times C\times G:\,ghg^{-1}=c\}\,.
	\eqn
	We claim that $\calL_\RR$ is a graph.  This implies the same for $\calL$ 
	and hence for $\calL_\FF$ by the transfer principle, showing the lemma.
	
	To show that $\calL_\RR$ is a graph, it amounts to show that
	the $G_\RR$-conjugacy class of an element $h\in C_\RR$ meets $C_\RR$ only in $h$.  This is standard for $\bG(\RR)^\circ$.
	To conclude we use a result of Matsumoto (see \cite[14.4]{Borel_Tits}) stating that, since $\bG$ is connected,
	then $\bG(\RR)=\bG(\RR)^\circ\bS(\RR)$.  This implies that if $h\in C_\RR$ its $\bG(\RR)$-conjugacy class
	coincide with its $\bG(\RR)^\circ$-conjugacy class, which concludes the proof of the lemma.
\end{proof}

Composing the semialgebraic map $G\to\calH$ from Lemma~\ref{lem:4.5} with the one $\calH\to C$ from Lemma~\ref{lem:4.6},
we obtain the \emph{Jordan projection} \label{n.38}
\bqn
\Jord\colon G\to C
\eqn
and its extension
\bqn
\Jord_\FF\colon G_\FF\to C_\FF\,.
\eqn

\begin{prop}\label{prop:4.7}  The Jordan projection $\Jord_\FF\colon G_\FF\to C_\FF$ is semialgebraic continuous.
\end{prop}

\begin{proof}
	By the usual transfer principle it suffices to show continuity for $\FF=\RR$.
	Let $\calX_\RR=G_\RR/K_\RR$ be the associated symmetric space (of non-compact type)\and let $\pi\colon G_\RR\to \Iso(\calX_\RR)$ 
	be the corresponding homomorphism.  
	Let then  $x_0=e K_\RR\in \calX_\RR$, $\fap\subset T_{x_0}\calX_\RR$ the closed Weyl chamber,
	which, under the Riemannian exponential $\mathrm{Exp}_{x_0}$ goes to $\pi(C_\RR)_* x_0$ as
	well as $j\colon \Iso(\calX_\RR)\to\fap$ the Jordan projection  defined in \cite[Proposition~4.1]{Parreau12}.
	Then the diagram
	\bqn
	\xymatrix{
		\Iso(\calX_\RR)\ar[r]^-j &\fap\ar[rd]^{\mathrm{Exp}_{x_0}}&\\
		&& \pi(C_\RR)_* x_0\\
		G_\RR\ar[uu]^\pi\ar[r]_{\Jord_\RR}&C_\RR\ar[ru]&
	}
	\eqn
	commutes.  Since $j$ is continuous, \cite[Lemma~4.3]{Parreau12} implies the continuity of $\Jord_\RR$.
\end{proof}
\subsection{Proximal elements}\label{s.proximal}
We recall the notions of proximal element and attracting fixed points for elements in $\bG(\RR)$, 
and extend it to elements in $\bG(\FF)$ for any real closed $\FF\supset \KK$. 
With the concepts and notations of Section~\ref{subsec:4.1}, 
recall that  ${}_\KK\bP_{I}$ denotes the  standard parabolic $\KK$-subgroup associated to a subset $I\subset{}_\KK\Delta$ of  simple roots. 
Let $\calF_I:= \bGmPI$ the corresponding projective variety defined over $\KK$ and recall that for any $\FF\supset\KK$
\bqn
\calF_I(\FF)=\quotient{\bG(\FF)}{{}_\KK\bP_I(\FF)}.
\eqn
\begin{defi}\label{defi:proximal}
	A point $\xi_g^+\in\calF_I(\RR)$  is an \emph{attracting fixed point} \label{n.39}
	for $g\in\bG(\RR)$ if $g\cdot\xi_g^+=\xi_g^+$ and the derivative of $g$ at $\xi_g^+$ has spectral radius strictly smaller than one. 
	An element $g\in\bG(\RR)$ is called \emph{$I$-proximal over $\RR$} or simply \emph{$I$-proximal } 
	if it admits an attracting fixed point in $\calF_I(\RR)$.
\end{defi}		
We recall the following classical fact:
\begin{prop}[{See e.g. \cite[Proposition 3.3]{GGKW}}]\label{p.proximal}\
	\begin{enumerate}
		\item An $I$-proximal element $g$ has a unique attracting fixed point in $\calF_I(\RR)$, which henceforth will be denoted $\xi^+_g$
		\item The following are equivalent
		\begin{itemize}
			\item  $g\in\bG(\RR)$ is $I$-proximal.
			\item for all $\alpha\in  I$,\; $\alpha(\Jord_\RR(g))>1$.
		\end{itemize}	
	\end{enumerate}	
\end{prop}	
Here $\Jord_\RR:\bG(\RR)\to C_\RR$ is the Jordan projection as in \S~\ref{subsec:Jordan}. 
Assume without loss of generality that $g\in{}_\KK\bP_I(\RR)$ is I-proximal. Then if $g=g_hg_eg_u$ is the Jordan decomposition of $g$, 
\bqn
\Ad(g)=\Ad(g_h)\Ad(g_e)\Ad(g_u)
\eqn
is the Jordan decomposition of $\Ad(g)$ and 
\bqn
\{\alpha(\Jord_\RR(g))|\; \alpha\in{}_\KK\Phi^+\setminus\langle{}_\KK\Delta\setminus I\rangle\}
\eqn
are the moduli of the eigenvalues of the action of $\Ad(g_h)$ on $\Lie(\bU)$ where $\bU$ is the unipotent radical of ${}_\KK\bP_I$. 
In particular the least absolute value of eigenvalue of such action arises in the set 
\bqn
\{\alpha(\Jord_\RR(g))|\; \alpha\in I\}.
\eqn
This leads to the following notion of attracting fixed points and proximal elements, 
which is valid for any real closed field $\FF$ containing $\KK$ and extends the classical notion in Definition \ref{defi:proximal}:

\begin{defi}\label{defi:proximal2}
Let $\bQ<\bG$ be an $\FF$-parabolic subgroup of $\bG$.
	A point $\xi\in\bGmQ(\FF)$ is an \emph{attracting fixed point} of $g\in\bG(\FF)$ 
	if the action of $\Ad(g_h)$ on the Lie algebra of the unipotent radical of the stabilizer of $\xi$ in $\bG$ 
	has all eigenvalues of modulus in $\FF$ strictly larger than one. 
	An element $g\in\bG(\FF)$ is called \emph{$I$-proximal over $\FF$} if it admits an attracting fixed point in $\calF_I(\FF)$. \label{n.40}
\end{defi}	
Let $\Pr(I, \FF)$ denote the subset of $\bG(\FF)$ consisting of $I$-proximal elements. Then 
\begin{equation}\label{e.proximal}
	\Pr(I,\FF)=\{g\in\bG(\FF)|\; \alpha(\Jord_\FF(g))>1\; \forall \alpha\in I\}.
\end{equation}
As a result,  $	\Pr(I,\FF)$ is semialgebraic and coincides with the $\FF$-extension $\Pr(I,\KK)_\FF$ of $\Pr(I,\KK)$. The following follows then from Proposition \ref{p.proximal} and the transfer principle:
\begin{cor}\label{l.fp} 
	Let $\bG$ be a semisimple connected algebraic group defined over a real closed field $\KK$; $\FF\supset\KK$ real closed;
	\begin{enumerate}
		\item An element $g\in {\bG(\FF)}$ that is $I$-proximal over $\FF$  has a unique attracting fixed point in $\calF_I(\FF)$.			
		\item Let $\LL$ be a real closed intermediate extension $\KK\subset\LL\subset\FF$, and assume $g\in \bG(\LL)$ is $I$-proximal over $\FF$. Then the attracting fixed point of $g$  is in $\calF_I(\LL)$.
	\end{enumerate}
\end{cor}
\begin{proof}
	%
	(1) Follows from Proposition \ref{p.proximal} (1) and Property \ref{ppts.2.10} (2) using the fact that the set of pairs $(g,\xi_g^+)\in \bG\times\calF_I(\KK)$  where $\xi_g^+$ is an attracting fixed point for $g$ is a semialgebraic set (recall Equation \eqref{e.proximal}). \\
	(2) Since a semialgebraic equation defined over $\LL$ that has a solution over $\FF$ already has a solution over $\LL$, the results follows.
\end{proof}	
\subsection{Logarithms}\label{s.log} 
We will need the following extension to real closed fields of the logarithm map $\mathrm{Ln}\colon S_\RR\to\fa$ \label{n.41}
from Lie theory where $\fa$ is the Lie algebra of $S_\RR$. Specifically we assume $\FF\supset\KK$ is a real closed, non-Archimedean extension admitting a big element $b\in\FF$, and denote by $-\log _b$ the associated valuation,  $\calO$ the valuation ring, and $\Lambda:=-\log_b(\FF^*)$ the valuation group.
With $\bS$ and ${}_\KK\Phi\subset X(\bS)$ as in Section~\ref{subsec:Cartan}, and $\alpha\in X(\bS)$,
let 
\bqn
d\alpha\colon\fa\to\RR
\eqn
denote the derivative of $\alpha\colon\bS(\RR)\to\RR^\times$.

\begin{lem}\label{lem:4.13} There is a well defined map \label{n.42}
	\bqn
	\mathrm{Log}_b\colon S_\FF\to\fa, \quad \text{ where $\fa=\Lie(S_\RR)$}
	\eqn
	characterized by
	\bqn
	d\beta(\mathrm{Log}_b(s))=\log_b(\beta(s))
	\eqn
	for all $s\in S_\FF$ and all $\beta\in{}_\KK\Phi$.
	It is a Weyl group equivariant homomorphism, 
	which is surjective if and only if $\Lambda=\RR$.
	Moreover if $\fabp$ and $C_\FF$ are the positive Weyl chambers associated to an order ${}_\KK\Phi^+\subset{}_\KK\Phi$,
	then $\mathrm{Log}_b(C_\FF)\subset\fabp$.
\end{lem}

\begin{proof} Fix an ordering ${}_\KK\Phi^+$ and let ${}_\KK\Delta\subset{}_\KK\Phi^+$ be the corresponding set of simple roots, with $|{}_\KK\Delta|=r$.
	Since 
	\bqn
	\prod_{\alpha\in{}_\KK\Delta}d\alpha\colon\fa\to\RR^r
	\eqn
	is a vector space isomorphism, the existence of $\mathrm{Log}_b$ satisfying
	\bq\label{eq:4.14}
	d\alpha(\mathrm{Log}_b(s))=\log_b(\alpha(s))
	\eq
	for all $\alpha\in{}_\KK\Delta$ and $s\in S_\FF$ is immediate.  But then \eqref{eq:4.14} follows for any $\beta\in{}_\KK\Phi$.

	Finally we know that 
	\bqn
	\prod_{\alpha\in{}_\KK\Delta}\alpha_\RR\colon S_\RR\to (\RR_{>0})^r
	\eqn
	is an isomorphism, which implies the same statement for all real closed $\FF\supset\KK$, that is
	\bqn
	\prod_{\alpha\in{}_\KK\Delta}\alpha_\FF\colon S_\FF\to (\FF_{>0})^r
	\eqn
	is an isomorphism.  But this implies that $\mathrm{Log}_b\colon S_\FF\to\fa$ is surjective if and only if
	\bqn
	\ba
	\FF_{>0}^r&\longrightarrow\RR^r\\
	(x_i)&\mapsto (\log_b(x_i))
	\ea
	\eqn
	is surjective, which is equivalent to $\Lambda=\RR$.
\end{proof}

\section{Non-standard symmetric spaces and metric shadows}\label{sec:5} 
In this section we introduce the non-standard symmetric space $\calX_\FF$ associated to $\bG$ over $\FF$ 
where $\FF$ is real closed with $\FF\supset\KK$ and examine its behavior 
under restriction to order convex subrings and passage to quotient residue field. 
Then we introduce the concept of semialgebraic norm, 
which leads to the existence of an $\FF_{\geq 1}$-valued multiplicative distance function on $\calX_\FF$. In case $\FF$ 
admits an order compatible valuation $\nu$, we get an $\RR_{\geq 0}$-valued pseudodistance whose Hausdorff quotient, 
the metric shadow of $\calX_\FF$, we denote by $\bgof$.

We then show that $\bG(\FF)$ acts transitively by isometries on $\bgof$ and the stabilizer of $[\Id]$ in this action is $\bG(\calO)$ 
where $\calO$ is the valuation ring of $\nu$. Then we analyze the case where $\FF=\rom$ is a Robinson field 
and show that any Weyl group invariant norm on $\frak a=\Lie (S_\RR)$ leads to a $\grom$-invariant distance on $\bgrom$: 
in fact this is shown by identifying $\bgrom$ with an asymptotic cone of $\calX_\RR$ 
for the corresponding Finsler metric associated to said invariant norm on $\fa$. 

For a real closed field $\FF$ admitting a big element and a valuation compatible injection into $\rom$ 
we deduce corresponding results for $\bgof$ and compute the translation length of elements in $\bG(\FF)$. 
If the norm  comes from a Weyl group invariant scalar product on $\fa$, 
$\bgrom$ is CAT(0)-complete and we determine the stabilizers in $\grom$ of the points in the visual boundary $\partial_\infty\bgrom$;  
we show that they are all $\grom$ conjugate to subgroups of the form $\bQ(\rom)$ 
where $\bQ$ is a standard parabolic $\RR$-subgroup of $\bG$.
\subsection{Non-standard symmetric spaces and Cartan projections}\label{subsec:4.4}
{We define a model for the non-standard symmetric space, 
define the Cartan projection and discuss its compatibility with the reduction modulo an order convex subring.}

If $\FF$ is a real closed field, the set
\bqn
\calP^1(n,\FF):=\{A\in M_{n\times n}(\FF):\,\det(A)=1,\, A\text{ is symmetric and positive definite}\}
\eqn
is a semialgebraic set on which $\SL(n,\FF)$ acts via 
\bqn
g_\ast A:=g\,A{g^t}
\eqn
for all $g\in\SL(n,\FF)$ and $A\in\calP^1(n,\FF)$.
Let, as in Section~\ref{subsec:Cartan}, $\bG<\SL(n)$ be a connected semisimple algebraic group defined over $\KK$ 
and assume $\bG$ is invariant under transposition.
Then $\calX:=\bG(\KK)_\ast\Id$ is a semialgebraic subset of $\calP^1(n,\KK)$. Since $\bG$ is invariant under transposition, the extension $\calX_\RR$ of $\calX$ to $\RR$ is the symmetric space
\bqn
\calX_\RR=\bG(\RR)_\ast\Id
\eqn
associated to $\bG(\RR)$.
We have that $\calX_\RR$ is a  closed connected subset of $\calP^1(n,\RR)$ and $\calX_\RR=\bG(\RR)^\circ_*\Id$. 
Hence we deduce by the transfer principle that for any real closed $\FF\supset\KK$
\bqn
\calX_\FF=\bG(\FF)_\ast\Id=\bG(\FF)^\circ_\ast\Id
\eqn
and is closed in $\calP^1(n,\FF)$.  In particular if $\bG(\KK)^\circ<G<\bG(\KK)$, 
then we also have $\calX_\FF=(G_\FF)_\ast\Id$.
If $\FF$ is non-Archimedean,  we will refer to $\calX_\FF$ as to the \emph{non-standard symmetric space associated to $G_\FF$.} \label{n.43}

{We use the Cartan decomposition to induce a $C_\FF$-valued multiplicative distance on $\calX_\FF$,  
	that we will call { Cartan projection}. 
}
Recall that  in Section~\ref{subsec:4.1} we chose a maximal split torus $\bS<\bG$ consisting of symmetric matrices, we denote by $S:=\bS(\KK)^\circ$, by ${}_\KK\Phi\subset X(\bS)$ the set of roots of $\bG$ with respect to $S$, by ${}_\KK\Phi^+\subset{}_\KK\Phi$ a choice of positive roots and by 
\bqn
C:=\{x\in S:\,\alpha(x)\geq1 \text{ for all }\alpha\in{}_\KK\Phi^+\},
\eqn
the multiplicative Weyl chamber.  The Cartan decomposition (Proposition~\ref{prop:4.4}) has the following equivalent reformulation:

\begin{cor}\label{cor:4.8} For every $(x,y)\in\calX_\FF\times\calX_\FF$, the $\bG(\FF)$-orbit of $(x,y)$ intersects
	$\{\Id\}\times (C_\FF)_\ast \Id$ in exactly one point $(\Id, \delta_\FF(x,y)_\ast\Id)$ 
	and the resulting map \label{n.44}
	$$\delta_\FF\colon\calX_\FF\times\calX_\FF\to C_\FF$$
	is semialgebraic continuous.
\end{cor}
\noindent We call the map $\delta_\FF$ the {\em Cartan projection}. 

The Cartan projection behaves well with respect to reduction modulo order convex subrings. For this recall the notation from Section~\ref{s:Intro2}: $\KK\subset \calO\subset\FF$ denotes an order convex subring of $\FF$, $\calI\subset \calO$ its maximal ideal,  ${\FF_\calO}:=\calO/\calI$
the residue field. Recall that we set $V_\FF(\calO)=V_\FF\cap \calO^n$  (Definition \ref{d.2.15}) and denote by $\pi:V_\FF(\calO)\to V_{\FF_\calO}$ the reduction mod $\calI$.
\begin{lem}\label{lem:4.19} \
	\be
	\item Let $g\in\bG(\FF)$ be such that $g_\ast\Id\in\calX_\FF(\calO)$.  Then $g\in\bG(\calO)$ and $\bG(\calO)$ acts transitively on $\calX_\FF(\calO)$.
	\item For all $(x,y)\in\calX_\FF(\calO)\times\calX_\FF(\calO)$,
	\bqn
	\bG(\calO)_\ast(x,y)\cap\left(\{\Id\}\times C_\FF(\calO)_\ast\Id\right)=(\Id, \delta_\FF(x,y)_\ast\Id)
	\eqn
	which implies in particular that 
	\bqn
	\delta_\FF(\calX_\FF(\calO)\times\calX_\FF(\calO))\subset C_\FF(\calO)\,.
	\eqn
	\item Denoting $\delta_{\calO}$ the restriction to $\calX_\FF(\calO)\times\calX_\FF(\calO)$ of $\delta_\FF$, 
	the diagram
	\bqn
	\xymatrix{
		\calX_\FF(\calO)\times\calX_\FF(\calO)\ar[r]^-{\delta_{\calO}}\ar[d]_{\pi\times\pi}&C_\FF(\calO)\ar[d]^\pi\\
		\calX_{{\FF_\calO}}\times\calX_{{\FF_\calO}}\ar[r]_-{\delta_{{\FF_\calO}}} &C_{{\FF_\calO}}
	}
	\eqn
	commutes.
	\ee
\end{lem}
\begin{proof} (1) From $g_\ast\Id=gg^t\in\calX_\FF(\calO)$ we get $\tr(gg^t)\in \calO$,
	in particular $g_{ij}\in \calO$ since $\calO$ is order convex.
	
	\medskip
	\noindent
	(2) Let $(x,y)\in\calX_\FF(\calO)\times\calX_\FF(\calO)$ and $g\in\bG(\FF)$ with 
	\bqn
	g_\ast (x,y)=(\Id,\delta_\FF(x,y)_\ast\Id)\,.
	\eqn
	From $x=g_\ast^{-1}\Id$ and (1), we get $g\in\bG(\calO)$ which, together with $g_\ast y=\delta_\FF(x,y)_\ast\Id\in\calX_\FF(\calO)$, implies
	that $\delta_\FF(x,y)\in C_\FF(\calO)$ and proves (2).
	
	\medskip
	\noindent
	(3) follows from (2), from Corollary~\ref{cor:4.8} applied to ${\FF_\calO}$ and from the equivariance of $\pi$.
\end{proof}

\subsection{Semialgebraic norms and the quotient $\bgof$}\label{s.seminorm}
{ We introduce the concept of (multiplicative) semialgebraic norm on $\bS$, show that any two such norms are equivalent and use them to define a metric quotient  $\bgof$ of $\calX_\FF$ whose set of points is independent of the chosen norm.}

Recall that the Weyl group $W=\calN_K(\bS)/\calZ_K(\bS)$ acts on $\bS$ and hence on $S=\bS(\KK)^\circ$.
A {\em semialgebraic norm on $S$} is a map \label{n.45}
\bqn
N\colon S\to\KK_{\geq1}
\eqn
that is semialgebraic continuous, Weyl group invariant and satisfies the following three properties:
\be
\item[(N1)] $N(gh)\leq N(g)N(h)$ for all $g,h\in S$;
\item[(N2)] $N(g^n)=N(g)^{|n|}$ for all $n\in\ZZ$;
\item[(N3)] $N(g)=1$ if and only if $g=e$.
\ee

\begin{ex}[Semialgebraic norm associated to a dominant weight]\label{e.fundExSemialgebraicNorm}
With the notation of Section \ref{sec:3}, let $\bS$ be a maximal $\KK$-split torus in $\bG$.
Fix in addition $\bS\leq\bT\leq\bB\leq\bP$, where $\bT$ is a maximal split torus in $\bG$ 
and $\bB$ is a Borel subgroup containing it.
This determines compatible orderings on the sets 
$\Phi\subset X(\bT)$ and ${}_\KK\Phi\subset X(\bS)$ of roots of $\bG$ in $\bT$ and $\bS$ respectively, 
such that if $j \colon X(\bT)\to X(\bS)$ denotes the restriction, then
\bqn
j(\Phi^+)={}_\KK\Phi^+\cup\{\mathbb 1\}.
\eqn
If
\bqn
\ba
T&=\{t\in\bT:\,\chi(t)\in\RR_{>0}\text{ for all }\chi\in X(\bT)\}\\
T^+&=\{t\in\bT:\,\chi(t)\geq1\text{ for all }\chi\in X(\bT)\}
\ea
\eqn
and $C\subset S=\bS(\KK)^\circ$ is as in 4.2,
then we have that $S\subset T$ and $C\subset T^+$.
Let $\overline{W}:=\calN(\bT)/\calZ(\bT)$ the (absolute) Weyl group, $w_0\in\overline{W}$ the longest element
and $\lambda\in X(\bT)$ the highest weight of an (absolutely) irreducible representation of $\bG$ with finite kernel.
Let $\eta:=\frac{\lambda}{\lambda\circ w_0}$, and define $\overline{N}\colon T\to\RR_{>0}$ by 
\bqn
\overline{N}(t)=\max_{w\in\overline{W}}\eta(w(t))
\eqn
for all $t\in\RR$.  Then
\bq\label{eq:eta>1}
\eta>1\text{ on } T^+\smallsetminus\{e\}\,,
\eq
and
\bq\label{eq:N}
N=\overline{N}|_S\colon S\to K_{\geq1}
\eq
is a semialgebraic norm.  The fact that $\lambda\circ w_0$ is the lowest weight 
and that the representation has finite kernel implies \eqref{eq:eta>1} as well as $\overline{N}|_{T^+}=\eta|_{T^+}$.
Hence $\overline{N}\geq1$ and $\overline{N}(t)=1$ if and only if $t=e$.
The property 
\bqn
\overline{N}(t_1t_2)\leq\overline{N}(t_1)\overline{N}(t_2)
\eqn
for all $t_1,t_2\in T$ is immediate, and so is also 
\bqn
\overline{N}(t^n)=n\overline{N}(t)
\eqn
for all $n\in\NN_{\geq1}$ and all $t\in T$.
To see that $\overline{N}(t^{-1})=\overline{N}(t)$ one uses that $w_0(T^+)=(T^+)^{-1}$ and $w_0^2=e$.
Then properties (N1), (N2) and (N3) for $\overline{N}$ follow; for the invariance of $N$ under the relative Weyl group
$W=\calN(S)/\calS(T)$ one uses \cite[6.10]{Borel_Tits}.  
\end{ex}

\begin{ex}\label{ex.rootnorm} A concrete example of a semialgebraic norm on $S$  is given by
	\bq\label{eq:N}
	N(s):=\prod_{\alpha\in{}_\KK\Phi}\max\{\alpha(s), \alpha(s)^{-1}\}\,.
	\eq
	We will use this explicit expression in Proposition \ref{prop:4.12}. 
\end{ex}

For any semialgebraic norm $N$ on $S$,  we define the \emph{$\FF_{\geq1}$-valued multiplicative distance} \label{n.46}
\bqn
\ba
D_N^\FF\colon\calX_\FF\times\calX_\FF&\longrightarrow\quad\FF_{\geq1}\\
(x,y)\,\,&\mapsto N_\FF(\delta_\FF(x,y))\,.
\ea
\eqn

Indeed:

\begin{prop}\label{prop:4.10} With the above notation we have:
	\be
	\item $D_N^\FF(x,z)\leq D_N^\FF(x,y)D_N^\FF(y,z)$ for all $x,y,z\in\calX_\FF$;
	\item $D_N^\FF(x,y)=1$ if and only if $x=y$;
	\item Given any two semialgebraic norms $N_1, N_2$, there is a constant $c\in\NN_{\geq1}$ such that 
	\bqn
	(D_{N_2}^\FF)^{1/c}\leq D_{N_1}^\FF\leq (D_{N_2}^\FF)^c\,.
	\eqn
	\ee
\end{prop}
For the proof of this proposition as well as for the next subsection we recall some results of P.~Planche \cite{Planche}.
Let as above $\fa$ be the Lie algebra of $S_\RR$ and $\fabp$ the closed Weyl chamber
corresponding to $C_\RR$ via the exponential map
\bqn
\exp\colon\fa\to S_\RR\,,
\eqn
whose inverse we denote by 
\bqn
\mathrm{Ln}:=(\exp)^{-1}\colon S_\RR\to\fa\,.
\eqn

\begin{thm}[\cite{Planche}]\label{thm:Planche} Given a Weyl group invariant norm $\|\,\cdot\,\|$ on $\fa$, 
	the function $d_{\|\,\cdot\,\|}\colon\calX_\RR\times\calX_\RR\to [0,\infty)$ defined by
	\bqn
	d_{\|\,\cdot\,\|}(x,y):=\|\mathrm{Ln}\delta_\RR(x,y)\|
	\eqn
	is a $\bG(\RR)$-invariant distance function on $\calX_\RR$.
\end{thm}

\begin{proof}[Proof of Proposition~\ref{prop:4.10}]
	The extension $N_\RR$ of $N$ to $S_\RR$ satisfies  the conditions (N1) to (N3).  
	This implies easily that, for all $w\in\fa$,
	\bqn
	\|w\|:=\ln N_\RR(\exp(w))
	\eqn
	is a Weyl group invariant norm on $\fa$ and thus
	\bqn
	d_{\|\,\cdot\,\|}(x,y)=\|\mathrm{Ln}\delta_\RR(x,y)\|=\ln N_\RR\delta_\RR(x,y)
	\eqn
	is an invariant distance on $\calX_\RR\times\calX_\RR$, 
	implying that $D_N^\RR$ satisfies the semialgebraic relations (1) and (2).
	By Properties \ref{ppts.2.10}, $D_N^\KK$, and hence $D_N^\FF$ satisfy (1) and (2).
	
	Finally (3) follows from the fact that all norms on $\fa$ are equivalent.
\end{proof}

Assuming now that $\FF$ admits an order compatible valuation $v=-\log_b$ where $b$ is a big element in $\FF$, it follows then from Proposition~\ref{prop:4.10} that for any semialgebraic norm $N$: \label{n.47}
\bqn
d_N^\FF(x,y):=\log_bD_N^\FF(x,y)=-\nu D_N^\FF(x,y)
\eqn
is a (continuous) semidistance on $\calX_\FF$ taking values in $\Lambda:=\nu(\FF^\times)$. 
Moreover it follows from Proposition~\ref{prop:4.10}(3) that the equivalence relation
\bq\label{eq:eq_rel}
x\sim y\text{ in }\calX_\FF\text{ if and only if }d_N^\FF(x,y)=0
\eq
is independent of the semialgebraic norm $N$.

\begin{defn}\label{defn:bgf}  The \emph{metric shadow of $\calX_\FF$} \label{n.48}
	is the quotient $\bgof$ of $\calX_\FF$ by the equivalence relation in \eqref{eq:eq_rel}.
\end{defn}

Since $\bG(\FF)^\circ$ acts transitively on $\calX_\FF$, it does so on $\bgof$ as well, 
and we proceed to determine the stabilizer of the image $[\Id]\in\bgof$ of $\Id\in\calX_\FF$.
Recall that $\calO=\{\mu\in\FF:\,\nu(\mu)\geq0\}$ is the valuation ring corresponding to $\nu$.

\begin{prop}\label{prop:4.12} \
	\begin{enumerate}
		\item$\bG(\FF)$ acts transitively on $\bgof$ and the stabilizer of $[\Id]$ is $\bG(\calO)$.
		\item The kernel of $\mathrm{Log}_b\colon S_\FF\to\fa$ is $S_\FF(\calO):=S_\FF\cap\bG(\calO)$.
	\end{enumerate}		
\end{prop}

\begin{proof} (1) Let $L<\bG(\FF)$ be the stabilizer of $[\Id]$.  Since $K_\FF<L$,
	the Cartan decomposition (Proposition~\ref{prop:4.4}) implies that $L=K_\FF(L\cap C_\FF)K_\FF$.
	We now compute $L\cap S_\FF$.  For this we will use the semidistance $d_N^\FF$ where
	$N$ is as in Example \ref{ex.rootnorm}.  Chasing the definitions we obtain
	\bqn
	L\cap S_\FF=\{s\in S_\FF:\, \alpha(s)\in\calO^\times\text{ for all }\alpha\in{}_\KK\Phi\}\,.
	\eqn
	This implies that $S_\FF(\calO)\subset L\cap S_\FF$ and we proceed to show the reverse inclusion.
	
	To this end, observe that the subgroup of $X(\bS)$ generated by the set ${}_\KK\Phi$ of roots is of finite index, say $m\geq1$,
	and hence $\chi^m(s)\in\calO^\times$ for all $s\in L\cap S_\FF$ and for all $\chi\in X(\bS)$, 
	which implies that also $\chi(s)\in\calO^\times$.
	This applies in particular to the weights of $\bS$ in $\KK^n$, thus in particular  $\tr(s^2)\in\calO$ for all $s\in L\cap S_\FF$.
	Taking into account that $\bS$ consists of symmetric matrices, then $s\in S_\FF(\calO)$ and hence $L\cap S_\FF\subset S_\FF(\calO)$,
	that is $L\cap S_\FF=S_\FF(\calO)$.  Thus $L= K_\FF S_\FF(\calO)K_\FF=\bG(\calO)$,
	where the last equality follows from the Cartan decomposition and the fact that $K_\FF\subset \bG(\calO)$.
	
	(2) We have $s\in\mathrm{Ker}( \mathrm{Log}_b)$ if and only if $\beta(s)\in\calO^\times$ for all $\beta\in{}_\KK\Phi$,
	which by the proof of (1) implies $s\in S_\FF(\calO)$.  The converse is immediate.
\end{proof}
\begin{remark}\label{rem:4.14} It  follows then from Proposition~\ref{prop:4.12} that $\mathrm{Log}_b$ descends to a well defined injective map
	\bqn
	\mathrm{L}_b\colon (S_\FF)_\ast[\Id]\to\fa
	\eqn
	which is bijective if and only if $\Lambda=\RR$.
\end{remark}

\subsection{Non-standard symmetric spaces and asymptotic cones}\label{subsec:asymptotic}
%
The objective of this section is to relate the metric shadow $\bgrol$ for Robinson fields $\rol$ to appropriate asymptotic cones of the symmetric space $\calX_\RR$.

Let $\|\,\cdot\,\|$ be a Weyl group invariant norm on $\fa$, and let $d_{\|\,\cdot\,\|}=d_{\|\,\cdot\,\|}^\RR$ the Finsler distance on $\calX_\RR$ given by Theorem~\ref{thm:Planche}.
For  a non-principal ultrafilter  $\omega$ on $\NN$, and a sequence $(\lambda_k)_{k\in\NN}=:\prlambda$  in $\RR_{>0}$
with $\lim_\omega\lambda_k=\infty$, we denote by \label{n.49}
\bqn
\Cone(\calX_\RR, \omega,\prlambda,d_{\|\,\cdot\,\|})
\eqn
the asymptotic cone of $\calX_\RR$ with respect to the sequence of distance functions $\left(\frac{d_{\|\,\cdot\,\|}}{\lambda_k}\right)_{k\in\NN}$
and base points  $\Id\in\calX_\RR$.  Let $\mu_k:= e^{\lambda_k}$ and $\rom$ the Robinson field associated to $\omega$
and $\pmb\mu:=(\mu_k)_{k\in\NN}$.  We will take $b=\pmb\mu$ as big element in $\rom$ and denote by $\nu=-\log_b$ the corresponding valuation.

\begin{thm}\label{thm:4.15}
	With the above notation the formula $d_{\|\,\cdot\,\|}^\rom(x,y)=\|\mathrm{Log}_b(\delta_\rom(x,y))\|$ is a pseudodistance on $\calX_\rom$ 
	and the induced distance on $\bgrom$ makes it a metric space isometric to the asymptotic cone
	$\Cone(\calX_\RR,\omega,\prlambda,d_{\|\,\cdot\,\|})$.
\end{thm}

\begin{remark}\label{rem:4.16}\
	\be
	\item Theorem~\ref{thm:4.15} was shown by B.~Thornton \cite{Thornton}
	in the case in which the norm comes form a scalar product.
	As essential ingredient in the proof is a lifting property for real algebraic sets,
	due to Thornton, which we generalized in Proposition \ref{prop:4.18} to semialgebraic sets.
	\item We refer the reader to the diagramm at the beginning of the proof of Theorem \ref{thm:4.15} for the definition of the isometry in Theorem \ref{thm:4.15}.
	\ee
\end{remark}

\medskip

We begin recalling the construction of the asymptotic cone  that proceeds in two steps.
We consider the following subset of the ultraproduct ${}^\omega(\calX_\RR)$ of $\calX_\RR$
\bqn
{}^\omega\calX_{\prlambda}=\left\{[(x_k)]_\omega\in {}^\omega(\calX_\RR):\,\frac{d(x_k,\Id)}{\lambda_k}\text{ is }\omega-\text{bounded}\right\}\,,
\eqn
where, for ease of notation, we set $d=d_{\|\,\cdot\,\|}$.  On ${}^\omega\calX_{\prlambda}$ we define the pseudodistance
\bqn
{}^\omega d_{\prlambda}([(x_k)]_\omega,[(y_k)]_\omega)=\lim_\omega\frac{d(x_k,y_k)}{\lambda_k}\,.
\eqn
The asymptotic cone $\Cone(\calX_\RR,\omega,\prlambda, d_{\|\,\cdot\,\|})$ is then the quotient of ${}^\omega\calX_{\prlambda}$
by the ${}^\omega d_{\prlambda}=0$ relation.

Our main objective now is to relate ${}^\omega\calX_{\prlambda}$ and ${}^\omega d_{\prlambda}$ to $\calX_\rol$ and $d_{\|\,\cdot\,\|}^\rol$.
To this end, let 
\bqn
{}^\omega G_{\prlambda}:=
\left\{[(g_k)]_\omega\in{}^\omega(\bG(\RR)):\,\frac{d(g_k\Id, \Id)}{\lambda_k}\text{ is }\omega-\text{bounded}\right\}
\eqn
and for any subset $E\subset\bG(\RR)$ we let 
\bqn
{}^\omega E_{\prlambda}:={}^\omega G_{\prlambda}\cap{}^\omega E\,.
\eqn
Then ${}^\omega G_{\prlambda}$ acts coordinatewise on ${}^\omega\calX_{\prlambda}$
in a pseudodistance preserving way.
Given $m\in{}^\omega(M_{n,n}(\RR))$ represented by a sequence $(m_k)_{k\in\NN}$, 
we define $\ov m\in M_{n,n}({}^\omega\RR)$ by $\ov m_{ij}:=[(m_{k,ij})]_\omega$.
We wish now to identify the images of ${}^\omega\calX_{\prlambda}$ and ${}^\omega G_{\prlambda}$
under the map $m\mapsto\ov m$.  To this end, recall that $\rom$ is the quotient of the order convex subring
\bqn
\calO_{\pmb\mu}=\{x\in\RR^\omega:\, |x|<\pmb\mu^m\text{ for some }m\in\ZZ\}
\eqn
by its maximal ideal
\bqn
\calI_{\pmb\mu}=\{x\in\RR^\omega:\, |x|<\pmb\mu^m\text{ for all }m\in\ZZ\}\,.
\eqn
Moreover the function 
\bqn
\ba
\ov\nu\colon\calO_{\pmb\mu}&\longrightarrow\,\,\,\RR\cup\{\infty\}\\
x\,\,&\mapsto-\lim_\omega\frac{\ln|x_k|}{\lambda_k}
\ea
\eqn
induces the valuation $\nu=-\log_b$, where $b=\pmb\mu\in\rol$,
through the formula  $\nu\circ\pi=\ov\nu$, where $\pi\colon\calO_{\pmb\mu}\to\rom$ is the canonical projection.

\begin{lem}\label{lem:4.20} With the above notations:
	\be
	\item The map $m\mapsto\ov m$ induces a group isomorphism 
	\bqn
	{}^\omega G_{\prlambda}\to\bG(\calO_{\pmb\mu})
	\eqn
	and an equivariant bijection
	\bqn
	{}^\omega\calX_{\prlambda}\to\calX_{{}^\omega\RR}(\calO_{\pmb\mu})\,.
	\eqn
	\item For $x=[(x_k)]_\omega$ and $y=[(y_k)]_\omega$ in ${}^\omega\calX_{\prlambda}$ we have:
	\bqn
	[(\delta_\RR(x_k,y_k))]_\omega\in{}^\omega(C_\RR)_{\prlambda}
	\eqn
	and
	\bqn
	\ov{[(\delta_\RR(x_k,y_k))]_\omega}=\delta_{\calO_{\pmb\mu}}(\ov x, \ov y)
	\eqn
	(see Lemma~\ref{lem:4.19}(3)).
	\ee
\end{lem}
\begin{proof} (1) We may replace the distance $d$ with the restriction to $\calX_\RR$ of the Riemannian distance $d_R$ on $\calP^1(n,\RR)$, normalized so that
	\bq\label{e.dR}d_R(g_*\Id,\Id)=\sqrt{\sum_{i=1}^n(\ln(\mu_i))^2}.\eq
where $\mu_1\geq \ldots\geq \mu_n>0$ are the eigenvalues of $g{g^t}$.
For the latter we have for all $g\in\SL(n,\RR)$
	\bqn
	\frac{\tr({}gg^t)}{n}\leq e^{d_R(g_\ast\Id,\Id)}\leq (\tr({}gg^t))^{2n-2}
	\eqn
	(see Lemma~\ref{lem:6.7} below).
	Thus 
	\bqn
	\left(\frac{d(g_k\Id,\Id)}{\lambda_k}\right)_\omega
	\eqn 
	is $\omega$-bounded if and only if $(\tr(g_k {}^tg_k)^{1/\lambda_k})_\omega$ is $\omega$-bounded,
	or equivalently there is $m\in\ZZ$ such that 
	\bqn
	\tr(g_k{}^tg_k)\leq(e^m)^{\lambda_k}=\mu_k^m
	\eqn
	for $\omega$-almost all $k\in\NN$, which is turn is equivalent to $[(g_{k,ij})]_\omega\in\calO_\mu$.
	The argument for the second assertion in (1) is analogous.
	\medskip
	\noindent
	(2) is a straightforward verification using Lemma~\ref{lem:4.19}(3).
\end{proof}

\begin{lem}\label{lem:6.7} For every $g\in\SL(n,\RR)$ we have
	$$\frac{\tr(g{g^t})}n\leq e^{d_R(g_\ast \Id,\Id)}\leq \tr(g{g^t})^{2(n-1)}, $$
	where the invariant Riemannian distance $d_R$ on $\calP^1(n,\RR)$ is normalized as in \eqref{e.dR}.
\end{lem}
\begin{proof}
	Let $\mu_1\geq \ldots\geq \mu_n>0$ be the eigenvalues of $g{g^t}$, then
	$$d_R(g_\ast \Id,\Id)=\sqrt{\sum_{i=1}^n(\ln\mu_i)^2}$$
	implies
	$$d_R(g_\ast \Id,\Id)\geq \ln \mu_1\geq \ln\left(\frac{\mu_1+\ldots+\mu_n}n\right)=\ln\left(\frac{\tr(g{g^t})}{n}\right)$$
	and gives the first inequality.
	
	For the second, we may assume that $g{g^t}\neq \Id$; then there is  $1\leq k\leq n-1$ with $\mu_1\geq\ldots\geq \mu_k\geq 1>\mu_{k+1}\geq\ldots\geq \mu_n$. 
	Then if $a=\mu_1\ldots\mu_k$, we have $a^{-1}=\mu_{k+1}\ldots\mu_n$ and
	$$\sum_{i=1}^n|\ln\mu_i|=2\ln a.$$
	Using $a\leq (\mu_1)^{n-1}$, we obtain
	$$d_R(g_\ast\Id,\Id)\leq \sum_{i=1}^n|\ln\mu_i|\leq{2(n-1)}\ln (\mu_1)\leq 2(n-1)\ln(\tr(g{g^t}))$$
	which proves the second inequality.
\end{proof}

The next lemma is crucial as it relates the logarithm map $\mathrm{Log}_b\colon S_\rom\to\fa$ given by Lemma~\ref{lem:4.13} 
to the Lie group logarithm $\mathrm{Ln}\colon S_\RR\to\fa$. 

\begin{lem}\label{lem:4.21}
	Let $s=[(s_k)]_\omega\in{}^\omega(S_\RR)_{\prlambda}$, $\ov s\in S_{{}^\omega\RR}(\calO_{\pmb\mu})$
	the corresponding element given by Lemma~\ref{lem:4.20}(1) and $\pi\colon S_{{}^\omega\RR}(\calO_{\pmb\mu})\to S_\rom$
	the canonical projection.  Then $\left(\frac{\mathrm{Ln}(s_k)}{\lambda_k}\right)_{k\in\NN}$  is a sequence in $\fa$
	and 
	\bqn
	\lim_\omega \frac{\mathrm{Ln}(s_k)}{\lambda_k}=\mathrm{Log}_b(\pi(\ov s))\,.
	\eqn
\end{lem}

\begin{remark}For any commutative $\KK$-algebra $\calA$, a character
	$\alpha\in X(\bS)$ gives rise to a group morphism $\alpha_\calA\colon\bS(\calA)\to\calA^\times$ in a functorial way.
	In particular, the diagram
	\bqn
	\xymatrix{
		\bS(\calO_{\pmb\mu})\ar[r]^{\alpha_{\calO_{\pmb\mu}}}\ar[d]_\pi
		&\calO_{\pmb\mu}^\times\ar[d]^\pi\\
		\bS(\rom)\ar[r]_{\alpha_\rom}&(\rom)^\times
	}
	\eqn
	commutes.
\end{remark}
\begin{proof}[Proof of Lemma~\ref{lem:4.21}]
	Since $d\alpha$ is the derivative of $\alpha_\RR$, we have
	\bqn
	d\alpha\left(\frac{\mathrm{Ln}(s_k)}{\lambda_k}\right)
	=\frac{1}{\lambda_k}\ln\alpha_\RR(s_k)\,.
	\eqn
	Thus
	\bq\label{e.5.14.1}
	d\alpha\left(\lim_\omega\frac{\mathrm{Ln}(s_k)}{\lambda_k}\right)
	=\lim_\omega\frac{\ln\alpha_\RR(s_k)}{\lambda_k}\,.
	\eq
	Next we observe that with the notations in Lemma \ref{lem:4.20},
	\bq\label{e.5.14.2}
	\ov{[(\alpha_\RR(s_k))]_{\omega}}=\alpha_{\calO_{\pmb\mu}}(\ov s)
	\eq
	and therefore
	\bqn
	d\alpha\left(\lim_\omega\frac{\mathrm{Ln}(s_k)}{\lambda_k}\right)
	=-\ov \nu(\alpha_{\calO_{\pmb\mu}}(\ov s))
	=-\nu(\pi(\alpha_{\calO_{\pmb\mu}}(\ov s))
	=-\nu(\alpha_\rom(\pi(\ov s)))\,,
	\eqn
	where the first equality follows from \eqref{e.5.14.1} and \eqref{e.5.14.2},
	while the last from functoriality. Since $-v=\log_b$, 
	we deduce the assertion from Lemma~\ref{lem:4.13}.
\end{proof}

Denoting also by $\pi\colon\calX_{{}^\omega\RR}(\calO_{\pmb\mu})\to\calX_\rom$ the projection map,
we obtain the following formula:

\begin{cor}\label{cor:4.22} For all $x,y\in {}^\omega\calX_{\prlambda}$ we have:
	\bqn
	{}^\omega d_{\prlambda}(x,y)=\|\mathrm{Log}_b\delta_\rom(\pi(\ov x),\pi(\ov y))\|\,.
	\eqn
\end{cor}
\begin{proof}  We have with $x=[(x_k)]_\omega$ and $y=[(y_k)]_\omega$
	\bqn
	{}^\omega d_{\prlambda}(x,y)
	=\lim_\omega\frac{d(x_k,y_k)}{\lambda_k}
	=\lim_\omega\frac{\|\mathrm{Ln}(\delta_\RR(x_k,y_k))\|}{\lambda_k}\,.
	\eqn
	By Lemma~\ref{lem:4.20}(2) we have
	\bqn
	\overline{[(\delta_\RR(x_k,y_k))]_{\omega}}
	=\delta_{\calO_{\pmb\mu}}(\ov x,\ov y)\,,
	\eqn
	which, by Lemma~\ref{lem:4.21} implies that 
	\bqn
	{}^\omega d_{\prlambda}(x,y)
	=\|\mathrm{Log}_b(\pi(\delta_{\calO_{\pmb\mu}}(\ov x,\ov y)))\|
	=\|\mathrm{Log}_b(\delta_\rom(\pi(\ov x),\pi(\ov y)))\|\,,
	\eqn
	where the last equality follows from Lemma~\ref{lem:4.19}(3).
\end{proof}

We now turn to the compatibility of the distance 
\bqn
d_{\|\,\cdot\,\|}^\rom(x,y):=\|\mathrm{Log}_b(\delta_\rom(x,y))\|
\eqn
associated to a Weyl group invariant norm $\|\cdot\|$ on $\fa$  
and the distance 
\bqn
d_{N}^\rom(x,y):=-\nu(N_\rom(\delta_\rom(x,y)))
\eqn
associated to a semialgebraic norm $N$ on $\bS$ 
in Section~\ref{subsec:4.4}:
\begin{lem}\label{l.compatdist}
	If $N$ is a semialgebraic norm on $\bS$, and we denote by $\|\cdot\|_N:=\ln N_\RR(\exp(\cdot))$, then
	$ d_{N}^\rom=d_{\|\,\cdot\,\|_N}^\rom.$
\end{lem}	
\begin{proof}
	First we observe that for $s=[(s_k)]_{\omega}\in{}^\omega(S_\RR)_{\pmb\mu}$, $\ov s\in S_{{}^\omega\RR}(\calO_{\pmb\mu})$
	the corresponding element given by Lemma~\ref{lem:4.20}(1) and $\pi\colon S_{{}^\omega\RR}(\calO_{\pmb\mu})\to S_\rom$
	the canonical projection,
	\bq\label{e.n}N_\rom(\pi(\ov s))=\pi\left(\ov{[(N_\RR(s_k))]_{\omega}}\right).\eq
	Indeed both the left hand side and the right hand side define a function on $S_\rom$ whose graph verifies all the polynomial inequalites (with coefficients in $\KK$) defining $N$.
	
	With the notations  of Lemma \ref{lem:4.20} let $x=[(x_k)]_\omega$, $y=[(y_k)]_\omega$ in ${}^\omega\calX_{\prlambda}$, then with $s_k:=\delta_\RR(x_k,y_k)$, we have $\ov{(s_k)}=\delta_{\calO_{\pmb\mu}}(\ov x, \ov y)$ and hence (Lemma \ref{lem:4.19}(3))
	\bq\label{e.5.14xx}
	\delta_\rom(\pi(\ov x),\pi(\ov y))=\pi\left(\ov{[(s_k)]_{\omega}}\right).
	\eq
	Thus 
	\bqn
	d_N^\rom(\pi(\ov x),\pi(\ov y))=\log_\pmu N_\rom(\delta_\rom (\pi(\ov x),\pi(\ov y)))
	\eqn
	equals, using \eqref{e.n}, $\lim_\omega\frac{\ln N_\RR(s_k)}{\lambda_k}$.
	On the other hand
	\bqn
	d_{\|\,\cdot\,\|_N}^\rom(\pi(\ov x),\pi(\ov y))=\ln N_\RR \exp \mathrm{Log}_\pmu(\delta_\rom(\pi(\ov x),\pi(\ov y))).
	\eqn
	By Lemma \ref{lem:4.21} and \eqref{e.5.14xx} we have 
	\bqn
	\mathrm{Log}_\pmu\delta_\rom(\pi(\ov x),\pi(\ov y))=\lim_\omega\frac{\mathrm{Ln}(s_k)}{\lambda_k}
	\eqn
	and hence 
	\bqn
	d_{\|\,\cdot\,\|_N}^\rom(\pi(\ov x),\pi(\ov y))=\lim_\omega\ln N_\RR\exp\left(\frac{\mathrm{Ln}(s_k)}{\lambda_k}\right)=\lim_\omega\frac{\ln N_\RR\exp(\mathrm{Ln}(s_k))}{\lambda_k},
	\eqn
	since $\ln N_\RR\exp$ is a norm. This leads to
	\bqn
	d_{\|\,\cdot\,\|_N}^\rom(\pi(\ov x),\pi(\ov y))=\lim_\omega\frac{\ln N_\RR(s_k)}{\lambda_k}=d_{N}^\rom(\pi(\ov x),\pi(\ov y)).\qedhere
	\eqn
\end{proof}	
We now have all the ingredients needed to prove Theorem~\ref{thm:4.15}
\begin{proof}[Proof of Theorem~\ref{thm:4.15}]
	Consider the following diagram
	\bqn
	\xymatrix{
		{}^\omega\calX_{\prlambda}\ar[r]\ar[d]&\calX_{{}^\omega\RR}(\calO_{\pmb\mu})\ar[d]^\pi\\
		\Cone(\calX_\RR)&\calX_\rom\ar[d]\\
		&{\bgrom}
	}
	\eqn
	where the horizontal arrow is the equivariant bijection from Lemma~\ref{lem:4.20}(1) 
	and $\Cone(\calX_\RR)$ stands for $\Cone(\calX_\RR,\omega,\prlambda,d_{\|\,\cdot\,\|})$.
	With our notation, we can restate Corollary~\ref{cor:4.22} as follows:
	for all $x,y\in{}^\omega\calX_{\prlambda}$
	\bqn
	{}^\omega d_{\prlambda}(x,y)=d^\rom_{\|\,\cdot\,\|}(\pi(\ov x),\pi(\ov y))\,.
	\eqn
	Since $\pi\colon\calX_{{}^\omega\RR}(\calO_{\pmb\mu})\to\calX_\rom$ is surjective (Proposition~\ref{prop:4.18}),
	this equality implies that $d_{\|\,\cdot\,\|}^\rol(\,\cdot\,,\,\cdot\,)$ defines a pseudodistance on $\calX_\rom$.
	Using Lemma \ref{l.compatdist} and the fact that any two norms on $\fa$ are equivalent, we find constants $c_1,c_2$ with
	\bqn
	c_1d_N^\rom(\,\cdot\,,\,\cdot\,)\leq d_{\|\,\cdot\,\|}^\rom(\,\,\cdot\,,\cdot\,)\leq c_2d_N^\rom(\,\cdot\,,\,\cdot\,)\,,
	\eqn
	from which it follows (see Definition~\ref{defn:bgf}) that $d_{\|\,\cdot\,\|}^\rom(\pi(\ov x),\pi(\ov y))=0$
	if and only if $\pi(\ov x)$ and $\pi(\ov y)$ project to the same point in $\bgrom$.
	Thus the map
	\bqn
	\ba
	{}^\omega\calX_{\prlambda}&\longrightarrow\calX_{{}^\omega\RR}(\calO_{\pmb\mu})\\
	x\,\,\,&\longmapsto \quad\,\,\ov x
	\ea
	\eqn
	defined before Lemma~\ref{lem:4.20} induces a bijective isometry
	\bqn
	\Cone(\calX_\RR)\to\bgrom
	\eqn
	that is equivariant with respect to the (surjective) homomorphism
	\bqn
	{}^\omega G_{\prlambda}\twoheadrightarrow\bG(\rom)\,.\qedhere
	\eqn
\end{proof}

\begin{remark}\label{rem:Lb}
	In view of Remark~\ref{rem:4.14}, Lemma~\ref{lem:4.13} and the fact that $\nu(\rom^\times)=\RR$, we deduce that the map from Remark~\ref{rem:4.14} 
	\bqn
	L_b\colon (S_\rom)_\ast[\Id]\subset\bgrom\to\fa
	\eqn
	is a bijective Weyl group equivariant isometry.
\end{remark}	
\subsection{Metrics on $\bgof$ from Weyl group invariant norms on $\fa$}\label{s.Fbuildings}
In this section we show how any arbitrary Weyl group invariant norm on $\fa$ leads to a natural $\bG(\FF)$-invariant metric on $\bgof$ for real closed fields $\FF$ (with order compatible valuation) admitting a valuation compatible embedding into a Robinson field. We also deduce an explicit formula for the associated translation length (Proposition \ref{prop:6.6.6}). 

Let  $\|\,\cdot\,\|$ be a Weyl group invariant norm on $\fa$.  Define, as in Section~\ref{subsec:asymptotic}, for $x,y\in\calX_\FF$ 
\bq\label{eq:4.14.2}
d_{\|\,\cdot\,\|}^\FF(x,y):=\|\mathrm{Log}_b(\delta_\FF(x,y))\|
\eq
where $\delta_\FF\colon\calX_\FF\times\calX_\FF\to C_\FF$ is the Cartan projection (Corollary~\ref{cor:4.8})
and $\mathrm{Log}_b$ is the map defined in Lemma~\ref{lem:4.13}.

The following corollary of Theorem \ref{thm:4.15} is particularly important when the norm comes from a scalar product, 
since then the metric completion $\obgof$ turns out to be a complete geodesic CAT(0) space. 

\begin{cor}\label{cor:4.17} Let $\FF$ be real closed admitting a big element and assume that there is a valuation compatible
	field injection  $i:\FF\to\rom$.  Then:
	\be
	\item Under the inclusion $\calX_{\FF}\hookrightarrow\calX_{\rom}$ induced by $i$ the function $d_{\|\,\cdot\,\|}^\FF$ is the restriction of $d_{\|\,\cdot\,\|}^\rom$. 
	\item  $d_{\|\,\cdot\,\|}^\FF$ is a pseudodistance on $\calX_\FF$ whose distance zero relation gives a distance on $\bgof$ 
	that has the midpoint property.
	\item If $\|\,\cdot\,\|$ comes from a scalar product, then $d_{\|\,\cdot\,\|}^\FF$ verifies the median inequality.
	In particular the completion $\obgof$ is CAT(0) complete.
	\item If $\|\cdot\|=\ln N_\RR(\exp(\cdot))$ for a semialgebraic norm $N$ on $\bS$, then $ d_{N}^\FF=d_{\|\,\cdot\,\|}^\FF.$
	\ee
\end{cor}
\begin{remark}
	Any field $\FF$ of finite transcendency degree over $\KK\subset \RR$, and thus in particular any field of definition $\FF$ of a  point in the real spectrum of some variety, admits a valuation compatible field injection into a Robinson field (compare Corollary \ref{c.1oKl}).
\end{remark}	
\begin{proof}[Proof of Corollary \ref{cor:4.17}]
	The proof rests on the following compatibility properties which are easy verifications; we list them for further reference. Let $\FF\hookrightarrow\rom$ be the injection where we may assume that the big element $b\in\FF$ goes to $\pmu$. Then the diagram
	$$\xymatrix{\calX_\rom\times \calX_\rom\ar[r]^{\delta_{\rom}}&C_{\rom}\\
		\calX_\FF\times \calX_\FF\ar[u]\ar[r]_{\delta_\FF}&C_{\FF}\ar[u]}$$
	commutes, as does
	$$\xymatrix{S_{\rom}\ar[r]^{\mathrm{Log}_\mu} &\fa\\ S_\FF \ar[u]\ar[ur]_{\mathrm{Log}_b}}.$$	
	Thus $d_{\|\,\cdot\,\|}^\FF=d_{\|\,\cdot\,\|}^{\rom}|_{\calX_\FF\times \calX_\FF}$, which shows the first assertion.
	
	The second assertion follows from the fact that $-\log_b(\FF^*)<\RR$ is a divisible group. For the third assertion, if $\|\cdot\|$ comes from a scalar product on $\fa$, then $d_{\|\cdot\|}$ on $\calX_\RR$ is CAT(0), so is any of its asymptotic cones $\Cone(\calX_\RR,\omega,\prlambda,d_{\|\cdot\|})$ and hence by Theorem \ref{thm:4.15} the space $\bgrom$ with the distance $d_{\|\cdot\|}^\rom$. Hence since $d_{\|\,\cdot\,\|}^\FF=d_{\|\,\cdot\,\|}^{\rom}|_{\calX_\FF\times \calX_\FF}$, the isometric injection $\bgof\to\bgrom$ extends to an isometric injection $\ov\bgof\to\bgrom$, using that $\bgrom$ is complete as it is isometric to an asymptotic cone. 
	Since from (1) $\bgof$ has the midpoint property, this implies that $\ov\bgof$ is geodesic, CAT(0)-complete. Finally the fourth assertion follows from Lemma \ref{l.compatdist}.
\end{proof}

We now turn to a formula for the translation length of $g\in\bG(\FF)$ on $\ov\bgof$ 
with respect to any invariant distance $d^\FF_{\|\cdot\|}$. 
To this aim recall that for an isometry $g$ of a metric space $Z$ 
the \emph{translation length} $\ell(g)$ is defined by \label{n.50}
$$\ell(g):=\inf_{x\in Z} d(x,g_*x).$$
We then have:
\begin{prop}\label{prop:6.6.6}
	For any Weyl group invariant norm $\|\cdot\|$ on $\frak a$ and $g\in\bG(\FF)$, 
	the translation length $\ell_{\|\cdot\|}(g)$ of $g$ with respect to $d_{\|\cdot\|}^\FF$ equals
	$$\ell_{\|\cdot\|}(g)=\|\mathrm{Log}_b\Jord_\FF(g)\|$$
	where 	$\Jord_\FF:\bG(\FF)\to C_\FF$ is the Jordan projection. 
\end{prop}

The first ingredient in the proof is  the following:

\begin{lem}\label{lem:4.32}
	Any CAT(0) geodesic  segment $c:[0,1]\to\ov{\calB_{\bG(\FF)}}$ is also a $d_{\|\cdot\|}^\FF$-geodesic segment 
	for any Weyl group invariant norm $\|\cdot\|$ on $\frak a$. Hence any CAT(0) geodesic is also a $d_{\|\cdot\|}^\FF$-geodesic. 
\end{lem}	
\begin{proof}
	First one shows using the $\bG(\FF)$-action on $\bgof$ and Remark \ref{rem:4.14} that  
	if $x,y\in\calB_{\bG(\FF)}$ then a CAT(0)-midpoint $m$ between $x,y$ is in $\bgof$ and is also a $d^\FF_{\|\cdot\|}$-midpoint. 
	That is if  $c:[0,l]\to\ov{\calB_{\bG(\FF)}}$ is a CAT(0)-geodesic with endpoints in $\bgof$, then 
	$$d^\FF_{\|\cdot\|}(c(0),c(l/2))+d^\FF_{\|\cdot\|}(c(l/2),c(l))=d^\FF_{\|\cdot\|}(c(0),c(l)).$$
	Since every limit of such geodesics satisfies the same, this shows the lemma.
	
\end{proof}

Second we use that if a CAT(0) geodesic $c$ is translated by $g$, 
then  the $d^\FF_{\|\cdot\|}$-translation length of $g$ is computed by $d^\FF_{\|\cdot\|}(c(0),g_*c(0))$, 
which is a direct consequence of the following lemma:

\begin{lem}\label{lem:4.6.9}
	Let $(Z,d)$ be a metric space, $g\in{\rm Is}(Z)$,  $\|g\|:=\lim_{k\to\infty}\frac{d(x,g_*^kx)}k$
	(which is independent of $x\in Z$) \label{n.51}
	and 
	$$Y=\{y\in X:\, d(y,g^k_*y)=kd(y,g_*y)\quad \forall k\geq 1\}.$$
	If $Y\neq \emptyset$ then $Y=\Min(g)$ and $\ell(g)=\|g\|=d(y,g_*y)$ for all $y\in Y$. 	
\end{lem}
\begin{proof}
	We have 
	$$d(x,g_*^k x)\leq \sum_{i=1}^k d(g_*^{i-1}x,g_*^ix)=kd(x,g_*x)$$
	which implies $\|g\|\leq \ell(g)$.
	If $Y\neq \emptyset$ then for all $y\in Y$
	$$\ell(g)\leq d(y,g_*y)=\frac 1k d(y,g_*^ky)=\|g\|\leq \ell(g)$$
	which implies that $Y\subset \Min(g)$.
	But for $x\in\Min(g)$
	$$\frac{d(x,g_*^kx)}{k}\leq d(x,g_*x)=\ell(g)=\|g\|$$
	and since $\|g\|=\inf_{k\geq 1}\frac{d(x,g_*^kx)}k$, we obtain
	$\frac{d(x,g_*^kx)}k=d(x,g_*x)$ for all $k\geq 1$, that is $x\in Y$.
\end{proof}		
In order to find a CAT(0) geodesic translated by $g$, we first observe that in $\calB_{\bG(\FF)}$, 
as opposed to symmetric spaces, the fixpoint set of unipotent elements is non-empty:
\begin{lem}\label{lem:4.6.7}
	Let $u\in\bG(\FF)$ be unipotent. Then the set of $u$-fixed points in $\calB_{\bG(\FF)}$ is non-empty.
\end{lem}	
\begin{proof}
	If $\bN$ is the unipotent radical of the standard minimal parabolic $\KK$-subgroup, and $\frak n$ its Lie Algebra. Then
	$$\exp_\RR:\frak n(\RR)\to \bN(\RR)$$
	is a polynomial $\Ad(\bS(\RR))$-equivariant bijection. Hence we obtain for any real closed field $\FF$ containing $\KK$ a semialgebraic homeomorphism
	$$\exp_\FF:\frak n(\FF)\to \bN(\FF)$$
	Now if $\alpha_1,\ldots, \alpha_r$ are the simple roots we have that 
	\bqn
	\begin{array}{ccc}
		S_\FF&\to&\FF^r_{>0}\\
		s&\mapsto &(\alpha_1(s),\ldots, \alpha_r(s))
	\end{array}
	\eqn
	is an isomorphism. Now choose $s_0\in S_\FF$ so that $\|\alpha_i(s_0)\|\leq\lambda<1$ for all $1\leq i\leq r$.
	Then $\Ad(s_0)$ acts strictly contracting on $\frak n(\FF)$,
	hence every $n\in\bN(\FF)$ is $(\Ad(s_0))^m$ conjugated into $\bN(\calO)$ for some $m\geq 1$. 
	Thus, by Proposition~\ref{prop:4.12}, $n$ fixes $s_0^{-m}[\Id]\in\calB_{\bG(\FF)}$.
	
	The result follows since every unipotent $u\in\bG(\FF)$ is $\bG(\FF)$-conjugated to an element in $\bN(\FF)$ (compare Section~\ref{subsec:4.1}).
\end{proof}	
We now have all the ingredients necessary to prove Proposition \ref{prop:6.6.6}: 
\begin{proof}[Proof of Proposition \ref{prop:6.6.6}]
	We now fix a Weyl group invariant Euclidean norm $\|\cdot\|$ on $\frak a$ and consider $\ov {\calB_{\bG(\FF)}}$ as a complete CAT(0) space. If $g=g_eg_hg_u$ is the Jordan decomposition of $g\in\bG(\FF)$ from Section~\ref{subsec:Jordan}, the closed convex subsets $\FP(g_e)$ and $\FP(g_u)\subset\ov\bgof$ of fixed points are non-empty. Since $g_e$ and $g_u$ commutes, $L=\FP(g_e)\cap\FP(g_u)$ is then non empty: indeed $g_u$ preserves $\FP(g_e)$ hence commutes with nearest point projection on $\FP(g_e)$. The latter do not increase distances, hence the projection of a fixed point of $g_{u}$is a fixed point of $g_{u}$. 
Thus  $L$ is a closed convex subset with $g|_{L}=g_h|_L$.
	
	Since $g_h$ preserves a flat, it is a semisimple isometry of $\ov {\calB_{\bG(\FF)}}$ and hence one of $L$. Let $c:\RR\to L$ be a CAT(0)-geodesic along which $g$ and $g_h$ translate the same amount. Then by Lemma  \ref{lem:4.32}, $c$ is a $d_{\|\cdot\|}^\FF$-geodesic, and by Lemma \ref{lem:4.6.9}
	$$\ell_{\|\cdot\|}^\FF(g)=d_{\|\cdot\|}^\FF(c(0),g c(0))=d_{\|\cdot\|}^\FF(c(0),g_h c(0))=\ell_{\|\cdot\|}^\FF(g_h).$$
	Now the translation length $\ell_{\|\cdot\|}^\FF(g_h)$ can be computed on a geodesic in the maximal flat on which $g_h$ acts as translation by $\Log_b\Jord_\FF(g)$.

	
\end{proof}

We conclude the subsection discussing the compatibility of semialgebraic norms 
and associated length functions with respect to inclusions of symmetric spaces. This will be needed in Section \ref{s.crossratio}.

Let $\bG_1<\bG_2<\SL_n$ be connected, semisimple, defined over $\KK$, 
and invariant under transposition. Given $\bS_i<\bG_i$ maximal $\KK$-split tori consisting of symmetric matrices, 
we may modulo conjugating $\bG_1$ by an element of $K_2=\bG_2\cap\SO(n,\KK)$ assume that $\bS_1<\bS_2$. 
Let $\frak g_i=\frak k_i\oplus\frak p_i$ be the Cartan decompositions of $\frak g_i=\Lie(\bG_i(\RR))$, 
then $\fa_i=\Lie(\bS_i(\RR))<\frak p_i$. 
Let $W_i$ be the respective Weyl groups: $W_i=\calN_{K_i}(S_i)/\calZ_{K_i}(S_i)$ where $S_i=\bS_i(\KK)^\circ$.

\begin{lemma}\label{l.NormA}\
	\begin{enumerate}
		\item If $\|\cdot\|_2$ is a $W_2$-invariant norm on $\fa_2$ then $\|\cdot\|_1:=\|\cdot\|_2|_{\fa_1}$ 
		is a $W_1$-invariant norm on $\fa_1$.
		\item If $N_2:S_2\to\KK_{\geq 1}$ is a semialgebraic norm on $S_2$ so is $N_1:=N_2|_{S_1}$.
	\end{enumerate}	
\end{lemma}	
\begin{proof}
	(1) According to  \cite[Proposition 5.3 and Corollary  5.4]{Planche} the formula 
	$$\|v\|_2:=\|\Ad(k)v\|_2, \quad v\in\fp_2, \quad\Ad(k)v\in\fa_2,$$ 
	defines an $\Ad(K_{2,\RR})$-invariant norm on $\fp_2$ extending $\|\cdot\|_2$ on $\fa_2$. 
	Thus $\|\cdot\|_1=\|\cdot\|_2|_{\fp_1}$ defines an $\Ad(K_{1,\RR})$-invariant norm on $\fp_1$ 
	whose restriction to $\fa_1$ is hence $W_1$-invariant.
	
	(2) Follows immediately from (1).
\end{proof}

Let $d_1,d_2$ be the invariant Finsler metrics on the corresponding symmetric spaces 
$\calX_{1,\RR}=\bG_1(\RR)_*\Id$ and 	$\calX_{2,\RR}=\bG_2(\RR)_*\Id$.
Then $d_1$ is the restriction of $d_2$ to $\calX_{1,\RR}$ and $\calX_{1,\RR}$ is totally geodesic in $\calX_{2,\RR}$. 
Given a real closed field $\FF$ admitting a big element and a valuation compatible injection $\FF\to\rom$ then Lemma \ref{lem:4.32}
and the arguments of Corollary \ref{cor:4.17} 
imply
\begin{prop}\label{p.NormB}\
	\begin{enumerate}
		\item For $g\in\bG_1(\RR)$, $\ell_1(g)=\ell_2(g)$ where $\ell_i$ is the translation length of $g$ acting on $(\calX_{i,\RR},d_i)$	
		\item For $g\in\bG_1(\FF)$, $\ell^\FF_1(g)=\ell^\FF_2(g)$ 
		where $\ell_i^\FF$ is the translation length of $g$ acting on $(\calB_{\bG_i(\FF)},d_i)$	
	\end{enumerate}
	
\end{prop}	

\subsection{Visual boundary of $\bgrom$ and stabilizers}
In this subsection we fix a Weyl group invariant scalar product on $\frak a$ whose corresponding norm is denoted $\|\cdot\|$ 
and consider $\calX_\RR$ and $\bgrom$ endowed with the corresponding CAT(0)-metrics.

Recall that if $Y$ is a CAT(0)-space, the visual boundary $\partial_\infty Y$ is the set of equivalence classes of geodesic rays, 
where two rays are equivalent if their images are at finite Hausdorff distance. 
For the symmetric space $\calX_\RR=\bG(\RR)_*\Id$ it is a classical fact that 
the map $\xi\mapsto \Stab_{\bG(\RR)}(\xi)$ 
which to $\xi\in\partial_\infty \calX_\RR$ associates its stabilizer 
is a surjection from $\partial_\infty\calX_\RR$ to the set of proper parabolic $\RR$-subgroups of $\bG$.

The $\bG(\RR)$ action on $\partial_\infty\calX_\RR$ has a fundamental domain that can be described as follows. 
Let $E_\RR:\frak a\to \calX_\RR$, $w\mapsto \exp(w)_*\Id$; 
then $E_\RR$ is an isometric embedding and $E_\RR(\partial_\infty \fabp)$ is a fundamental domain 
for the $\bG(\RR)$-action on $\partial_\infty\calX_\RR$, 
in fact the map $\xi\mapsto\Stab_{\bG(\RR)}(\xi)$ is a surjection from $E_\RR(\partial_\infty \fabp)$ 
to the set of standard parabolic $\RR$-subgroups of $\bG$.

Next, let $E_\rom$ be the inverse of the map $L_b:(S_\rom)_*\Id\to \frak a$ 
which is a Weyl group invariant isometry (compare Remark \ref{rem:Lb}).

\begin{thm}\label{t.bdrystab}
	Every $\grom$-orbit in $\partial_\infty\bgrom$ meets $E_\rom(\partial_\infty\fabp)$ in a point and the stabilizer in $\grom$ of $E_\rom(\xi)$
	for $\xi\in\partial_\infty\fabp$ is $\bQ(\rom)$, where $\bQ$ is the standard parabolic $\RR$-subgroup of $\bG$ associated with $E_\RR(\xi)$.
\end{thm}

\begin{proof}
	Let $E\colon\fa\to\calX_\RR$ be the restriction to $\fa$ of the Riemannian exponential at $\Id$, $F=E(\fa)$ and $\ov{F}^+=E(\fabp)$.
	Define
	\bqn
	{}^\omega E_{\prlambda}\colon\fa\to\Cone(\calX_\RR)
	\eqn
	so that, for $w\in\fa$, ${}^\omega E_{\prlambda}(w)$ is the image in $\Cone(\calX_\RR)$ of $[(E(\lambda_kw))]_{\omega}\in{}^\omega\calX_{\prlambda}$.
	Then:
	\be
	\item[(i)] ${}^\omega E_{\prlambda}\colon\fa\to\Cone(\calX_\RR)$ is an isometric embedding with image $\Cone(F)$.
	\item[(ii)] Every ${}^\omega G_{\prlambda}$-orbit in $\partial_\infty\Cone(\calX_\RR)$ 
	intersects $\partial_\infty\Cone(\ov{F}^+)={}^\omega E_{\prlambda}(\partial_\infty\fabp)$
	in at least one point.
	\item[(iii)] The diagram
	\bqn
	\xymatrix{
		\Cone(\calX_\RR)\ar[rr]&&\bgrom\\
		&\fa\ar[ul]^{{}^\omega E_{\prlambda}}\ar[ur]_{L_b^{-1}}&
	}
	\eqn
	commutes.
	\ee
	The two first properties belong to the elementary theory of asymptotic cones of symmetric spaces, 
	while the third is a verification using Lemma~\ref{lem:4.21}.
	
	Let now $w\in\fabp$,  $\|w\|=1$, $\xi=\RR_{\geq0}w\in\partial_\infty\fabp$  and $\bP$ be the standard parabolic $\RR$-subgroup of $\bG$ corresponding to $\xi$,
	that is
	\bqn
	\bP(\RR)=\mathrm{Stab}_{\bG(\RR)} E(\xi)\,,
	\eqn
	where $E(\xi)\in\partial_\infty\calX_\RR$.
	
	We claim that 
	\bqn
	\bP(\rom)\subset\mathrm{Stab}_{\bG(\rom)}(L_b^{-1}(\xi))\,.
	\eqn
	Let $[(g_k)]_\omega\in{}^\omega(\bP(\RR))_{\prlambda}$.  Then
	the function $t\mapsto d(E(\lambda_ktw),{g_k}_*E(\lambda_ktw))$
	is convex and bounded on $[0,\infty)$.  In particular
	\bqn
	d(E(\lambda_ktw), {g_k}_*E(\lambda_ktw))\leq d(\Id, {g_k}_*\Id)
	\eqn
	for all $t\geq0$, hence
	\bqn
	\lim_\omega\frac{d(E(\lambda_ktw),{g_k}_*E(\lambda_ktw))}{\lambda_k}
	\leq\lim_\omega\frac{d(\Id,{g_k}_*\Id)}{\lambda_k}
	<\infty\,.
	\eqn
	Thus ${}^\omega(\bP(\RR))_{\prlambda}$ fixes the point ${}^\omega E_{\prlambda}(\xi)\in\partial_\infty\Cone(\calX_\RR)$
	and hence $\bP(\calO_{\pmb\mu})$ fixes $L_b^{-1}(\xi)\in\partial_\infty\bgrom$.
	This implies the claim since  $\pi\colon\bP(\calO_{\pmb\mu})\to\bP(\rom)$ is surjective (Proposition~\ref{prop:4.18}).
	
	From the claim it follows that $\mathrm{Stab}_{\bG(\rom)}(L_b^{-1}(\xi))$
	is a standard parabolic $\rom$-subgroup
	and hence by Corollary~\ref{cor:4.2} there is a standard parabolic $\RR$-subgroup $\bQ\subset\bP$ with $\bQ(\rom)=\mathrm{Stab}_{\bG(\rom)}(L_b^{-1}(\xi))$.
	Assume $\bQ\supsetneq\bP$ and pick $g\in\bQ(\RR)$ with $g\notin\bP(\RR)$.  Then $g_*(E(\xi))\neq E(\xi)$
	and hence $t\mapsto d(g_* E(tw),E(tw))$ is convex and unbounded on $[0,\infty)$.
	Hence there is $k>0$ with 
	\bqn
	d(g_*(E(tw),E(tw))\geq kt
	\eqn
	for all $t\geq0$.  Hence
	\bqn
	\frac{d(g_*E(t\lambda_kw),E(t\lambda_k w))}{\lambda_k}\geq kt\,,
	\eqn
	which implies for the element $\hat g\in{}^\omega(\bQ(\RR))_{\prlambda}$ given by the constant sequence $\hat g=[(g)]_{\omega}$
	that $\hat g({}^\omega E_{\prlambda}(\xi))\neq{}^\omega E_{\prlambda}(\xi)$.
	This implies that $\bQ(\rom)$ does not stabilize $L_b^{-1}(\xi)$, a contradiction.  Thus $\bQ=\bP$.
\end{proof}
\begin{remark}
	If the field $\FF$ is not real closed the stabilizers of points in the boundary of $\bgof$ might not be defined over $\FF$, compare \cite{Brumfiel}.
\end{remark}

\section{Representation varieties and actions on buildings}\label{s.symm}
In this section we study the real spectrum compactification of representation varieties and interpret its points as suitable equivalence classes of representations in real closed fields. We apply the results of \S~\ref{s.RspGenProp} to establish hyper-$\KK$ field realisation of such representations, and discuss several properties of representations in the boundary of suitable semialgebraic subsets including conditions guaranteeing that all such representations admit equivariant harmonic maps on the one hand, and well behaved equivariant framings on the other hand.


\subsection{The real spectrum compactification of representation varieties}\label{subsec:rep}
In this section we realize  representation varieties as semialgebraic sets and interpret the points in their real spectrum compactification as equivalence classes of representations in real closed fields. We also give a criterion for closedness in terms of actions on metric shadows of non-standard symmetric spaces, discuss how non-parabolicity can be transfered to points in the real spectrum compactifications, and draw consequences of the hyper-$\KK$-field realization from \S~\ref{s.RspGenProp}.

Let $\G$ be a finitely generated group, $\bG<\SL_n$ a linear algebraic
group defined over a real closed field $\KK\subset\RR$ and
$\bG(\KK)^\circ<G<\bG(\KK)$ a semialgebraic $\KK$-group.
We first explain how to realize the set $\Hom(\Gamma, G)$ of homomorphisms from $\G$ to $G$ as a semialgebraic set.
For every  finite symmetric generating set
$F=\{\gamma_1,\ldots,\gamma_f\}$ of $\Gamma$,
the evaluation map
\bq\label{e.evaluation}
\ev_F:
\begin{array}[t]{ccc}
	\Hom(\Gamma,\bG(\KK))&\to& M_{n,n}(\KK)^f\\
	\rho&\mapsto & (\rho(\gamma_1),\ldots,\rho(\gamma_f))
\end{array}
\eq
is injective and its image $\calR_F(\Gamma,\bG(\KK))$ \label{n.52} 
is a closed real algebraic subset  of $M_{n,n}(\KK)^f$, while the image $\calR_F(\Gamma,G)$ of $\Hom(\Gamma, G)$ is a closed semialgebraic subset of $\calR_F(\Gamma,\bG(\KK))$.
%
The group $\bG(\KK)$ acts by conjugation on $\Hom(\G,\bG(\KK))$ and diagonally on $M_{n,n}(\KK)^f$, and the map $ev_F$ is equivariant with respect to these actions.

We let $\KK[X^{(\gamma)}_{ij}]_{1\leq i,j\leq n,\, \gamma\in F}$ be the coordinate ring of $M_{n,n}(\KK)^f$. We will use  the coordinate functions $X^{(\gamma)}_{ij}$ also as generators of the coordinate ring $\KK[\calR_F(\Gamma,\bG(\KK))]$.

If now $\FF\supset\KK$ is a real closed extension of $\KK$, the preceeding discussion applies to $\bG$ as a group defined over $\FF$, that is  $\calR_F(\Gamma,\bG(\FF)) \subset M_{n,n}(\FF)^f$ is  real algebraic and contains $\calR_F(\Gamma,G_\FF)$ as a closed semialgebraic subset. In addition it is an easy verification that
$\calR_F(\Gamma,\bG(\FF))=\calR_F(\Gamma,\bG(\KK))_\FF$
and $\calR_F(\Gamma,G_\FF)=\calR_F(\Gamma,G)_\FF.$


%
Our first task is to describe points in the real spectrum
$\rsp{\calR_F(\Gamma,G)}$ of the semialgebraic set
$\calR_F(\Gamma,G)$
%
%
and characterize the ones which are closed,  see \S~\ref{subsec:compactification_semialgebraic}
for definitions.

\begin{defn}
	We say that a representation $\rho:\G\to G_\FF$ \emph{represents a
		point} $\alpha\in \rsp{\calR_F(\Gamma,G)}$ if  $\alpha_{(\rho(\gamma_1),\ldots,\rho(\gamma_f))}=\alpha$ (see Remark ~\ref{rem:point->prime cone}).
\end{defn}	

We will see below that every point in $\rsp{\calR_F(\Gamma,G)}$ is of
this form.

\begin{defn}
	For a representation $\rho:\Gamma\to \bG(\FF)$, we will denote by
	$\KK[\rho]$ the ring generated over $\KK$ by the matrix coefficients
	of $\rho$, namely by the elements $\rho(\g)_{i,j}$ for $\g\in\G$ and $1\leq i,j\leq n$,  and by $\KK(\rho)$ its field of fractions.
\end{defn}

\begin{prop}\label{prop:Rspecrep}
	A point $\alpha\in	\rsp{\calR_F(\Gamma,G)}$ is an equivalence class of pairs $(\rho,\FF)$ 
	where $\FF\supset\KK$, $\rho:\Gamma\to {\GFF}$ is a
	representation, $\FF$ is the real closure of  $\KK(\rho)$, and   $(\rho_1,\FF_1)$ is equivalent to $(\rho_2,\FF_2)$ if there is an order preserving isomorphism $\psi:\FF_1\to\FF_2$ 
	with $\rho_2(\gamma)=\psi(\rho_1(\gamma))$ for all $\gamma\in\Gamma$.
\end{prop} 
\begin{proof}
	Let $\alpha\in \rsp{\calR_F(\Gamma,G)}\subset \rsp{\calR_F(\Gamma,\bG(\KK))}$ and 
	$$\phi_\alpha:\KK[\calR_F(\Gamma,\bG(\KK))]\to \KK_\alpha$$ 
	the corresponding morphism, where $\KK_\alpha$ is the ordered field obtained as field of fraction of the quotient of $\KK[\calR_F(\Gamma,\bG(\KK))]$ by $\frak p_\alpha=\alpha\cap(-\alpha)$ (see Proposition ~\ref{p.charrsp}). Let $\FF_\alpha\supset\KK_\alpha$ be a real closure of $\KK_\alpha$. 
	Then
	$((\phi_\alpha(X_{ij}^{(\gamma)}))_{ij})_\gamma\in \calR_F(\Gamma,G_{\FF_\alpha})$ and thus determines a representation $\rho:\G\to G_{\FF_\alpha}<\bG({\FF_\alpha})$. Clearly $\KK(\rho)=\KK_\alpha$ which shows that ${\FF_\alpha}$ is the real closure of $\KK(\rho)$. 
	
	Conversely if $\rho:\Gamma\to G_\FF$ is a representation and $\FF$ is the real closure of $\KK(\rho)$ we obtain a homomorphism $\phi:\KK[\calR_F(\Gamma,\bG(\KK))]\to \KK(\rho)$ by sending $X_{ij}^{(\gamma)}$ to $\rho(\gamma)_{ij}$; the corresponding prime cone $\alpha$ is the subset of elements in $\KK[\calR_F(\Gamma,\bG(\KK))]$ whose image in $\KK(\rho)$ is non-negative. The statement concerning the equivalence is a straightforward verification using Proposition ~\ref{p.charrsp}(3).
	
\end{proof}

From now on we assume that $\bG$ is connected semisimple, and we turn
to an important characterisation of points in
$\rsp{\partial}\calR_F(\Gamma,G)=\rsp{\calR_F(\Gamma,G)}\setminus \calR_F(\Gamma,G_\RR)$ 
that are closed. Let $(\rho,\FF)$ represent a point in
$\rsp{\partial}\calR_F(\Gamma,G)$
where $\FF$ is the real closure of $\KK(\rho)$,
then $\FF$ admits a non-Archimedean, order compatible valuation
$v:\FF\to \RR\cup \{\infty\}$ (see Proposition ~\ref{prop:closed}),
and  thanks to the discussion of
\S~\ref{s.Fbuildings},  $\rho$ gives rise to an isometric action
on a complete CAT(0)-space $\ov\bgof $ with basepoint $[\Id]$
(cfr. also \cite[\S~6]{BP}).
Recall that the stabilizer of $[\Id]$ in $\bG(\FF)$ is $\bG(\calO)$
where $\calO$ is the valuation ring of $v$.

\begin{prop}\label{prop:Rspecrepcl}
	The following are equivalent:
	\begin{enumerate}
		\item $(\rho,\FF)$ represents a closed point of
		$\rsp{\partial}\calR_F(\Gamma,G)$.
		\item $\sum_{\gamma\in F}\tr({\rho(\gamma)}\rho(\gamma)^t)$ is a big
		element  of $\FF$.
		\item $\rho(\Gamma)\not\subset\bG(\calO)$. 
		\item $[\Id]\in\bgof$ is not a global fixed point of $\rho(\G)$.
	\end{enumerate}
\end{prop}	
\begin{proof}
	This uses Proposition ~\ref{prop:closed_points} and Proposition ~\ref{prop:Rspecrep}. Namely $(\rho,\FF)$ represents a closed point if and only if $\FF$ is Archimedean over $\KK[\rho]$, that is $\KK[\rho]$ contains a big element; in turn this is equivalent to $\sum_{\gamma\in F}\tr({\rho(\gamma)}\rho(\gamma)^t)$ being a big element. If $\|s\|:=e^{-v(s)}$ is the norm associated to the valuation $v$, this is equivalent to $\|\sum_{\gamma\in F}\tr({\rho(\gamma)}\rho(\gamma)^t)\|>1$, that is $\sum_{\gamma\in F}\tr({\rho(\gamma)}\rho(\gamma)^t)\notin\calO$, namely $\rho(\G)\not\subset \bG(\calO)$. 
	In turn this is equivalent to $[\Id]\in\bgof$ not being a
	global fixed point of $\rho(\G)$ (see Proposition ~\ref{prop:4.12}).
\end{proof}	
As a first application of the Tarski--Seidenberg Principle, we can deduce non-parabolicity for many $G_\FF$-valued representations.
\begin{defn}
	A representation $\rho:\Gamma\to G_\FF$ is \emph{$\FF$-parabolic} 
	if there exists a parabolic subgroup $\bQ<\bG$ defined over $\FF$ such that $\rho(\Gamma)\subset \bQ(\FF)$.
\end{defn}

\begin{prop}
	Let $\frakT\subset \calR_F(\Gamma,G(\KK))$ be a semialgebraic subset consisting of representations that are not $\KK$-parabolic. 
	Then for all $(\rho,\FF)$ representing some point in $\rsp \frakT$, $\rho$ is not $\FF$-parabolic.
\end{prop}
\begin{proof}
	We can assume without loss of generality that $\frakT$ is  $\bG(\KK)$-invariant since the set of $\KK$-parabolic representations is invariant under $\bG(\KK)$-conjugation and the image of the map $\bG(\KK)\times \frakT\to \calR_F(\Gamma,\bG(\KK))$, $(g,\rho)\mapsto g\rho g^{-1}$ is semialgebraic.
	
	Assume now by contradiction that there exists an element $(\rho,\FF)\in\rsp \frakT$ that is $\FF$-parabolic, say $\rho(\G)\subset \bQ(\FF)$. 
	Since (see \S~\ref{subsec:4.1})	$\bQ(\FF)$ is
	$\bG(\FF)$-conjugated to a parabolic $\KK$-subgroup $\bQ_0<\bG$,
	we may assume that $\rho(\Gamma)<\bQ_0(\FF)$.
	Thus we have, for the $\KK$-algebraic subset $\calR_F(\Gamma,\bQ_0(\KK))\subset \calR_F(\Gamma,\bG(\KK))$, 
	that the intersection $\rsp \frakT\cap \rsp{\calR_F(\Gamma,\bQ_0(\KK))}$ is not empty. This, by Theorem ~\ref{thm:constr}(1), implies that also $\frakT\cap\calR_F(\Gamma,\bQ_0(\KK))$ is not empty, a contradiction.
\end{proof}

Next we exploit the hyper-$\KK$-field realization in \S~\ref{subsec:3}.
Given an ultrafilter $\omega$ on $\NN$ and a sequence $(\rho_k)_{k\in\NN}$ in $\Hom(\G,\bG(\KK))$ we let ${\rho^\omega}:\G\to\bG(\KK^\omega)$ denote the corresponding representation defined by $\rho^\omega(\gamma)_{ij}=[((\rho_k(\g))_{ij})]_\omega\in\KK^\omega$.
When $\omega$ is non-principal and  $\pmu\in\KK^\omega$ is an infinitely
large element such that $\rho^\omega(\G)\subset \bG(\calO_\pmu)$, we
denote by $\rho^\omega_\pmu:\G\to \bG(\KK^\omega_\pmu)$ the representation obtained by composing
$\rho^\omega$ with the canonical projection
$\bG(\calO_\pmu)\to\bG(\KK^\omega_\pmu)$.

\begin{prop}\label{c.reP2}
	Let $\frakT\subset  \calR_F(\Gamma,G)$ be a semialgebraic
	subset. Let $(\rho,\FF)$ represent some point $\alpha$ in $\rsp{\frakT}$ with $\FF=\FF_\alpha$.
	\begin{enumerate}
		\item 
		There exists $(\rho_k)_{k\in\NN}$ in $\frakT$, an ultrafilter  $\omega$ on $\NN$ and a field injection $i:\FF\to {\KK^\omega}$ so that $i\circ \rho=\rho^\omega$.
		\item In addition, if $\alpha$ is a closed point of $\rsp{\partial}\calR_F(\Gamma,G)$, then $i$ induces a valuation
		compatible field injection $j:\FF\to {\KK^\omega_\pmu}$ so that $j\circ \rho=\rho^\omega_\pmu$, where $\pmu:=i\left(\sum_{\gamma\in F}\tr({\rho(\gamma)}\rho(\gamma)^t)\right)$.
	\end{enumerate}
\end{prop}

\begin{proof}The first statement is an immediate application of Proposition ~\ref{prop:1} together with Proposition ~\ref{prop:Rspecrep}.
	Concerning the second one, since $\alpha \in \rsp{\partial}\calR_F(\Gamma,G)$, the ultrafilter $\omega$ is non-principal. Furthermore, since $\alpha \notin \calR_F(\Gamma,G_\RR)$, $\FF$ is non-Archimedean, and  since $\alpha$ is a closed point, $\sum_{\gamma\in F}\tr({\rho(\gamma)}\rho(\gamma)^t)$ is a big
	element  of $\FF$ (Proposition ~\ref{prop:Rspecrepcl}). This implies that $\pmu:=i\left(\sum_{\gamma\in F}\tr({\rho(\gamma)}\rho(\gamma)^t)\right)$ is an infitinely large element  of $\KK^\omega$, allowing to define $\rho^\omega_\pmu$.
\end{proof}

In turn the accessibility result Corollary~\ref{cor:3.9} can be applied to obtain simultaneous realizations of representations with values in a fixed Robinson field:
\begin{cor}\label{c.repacc}
	Assume $\KK$ is a countable real closed field with $\KK\subset\RR$.
	Fix a non-principal ultrafilter $\omega$ on $\NN$ and an infinitely
	large element $\pmu=[(\mu_k)]_\omega\in\oK$.  Let  $(\rho,\FF)$ represent a closed point $\alpha\in\rspcl{\partial} {\calR_F(\Gamma,G)}$ with $\FF=\FF_\alpha$.
	There exists a valuation
	compatible injection $j:\FF\to {\KK^\omega_\pmu}$ and a sequence
	$(\rho_k)_{k\in\NN}$ in $\calR_F(\Gamma,G_\RR)$ such that
	$\mu_k=\sum_{\gamma\in F}\tr(\rho_k(\gamma)\rho_k(\gamma)^t)$ and
	$j\circ\rho=\rhom$.
\end{cor} 

\subsection{Harmonic maps into CAT(0)-spaces}\label{s.harmonic}
In this section we consider a closed semialgebraic subsets $\frakT\subset \calR_F(\Gamma,G)$ that consists of non-parabolic representations. 
Examples of such sets
are given by Hitchin and maximal representation varieties. 
Our goal is to show that every point in 
$\rspcl  \frakT$ corresponds to an action on a complete CAT(0)-space
that is non-evanescent.

We recall the definition.
\begin{defn} (see \cite[\S 1.2]{MonJams2006})
	An isometric action of a finitely generated group $\Gamma$ on a CAT(0) space $\mathcal Y$ is \emph{non-evanescent} if, given a finite generating set $F$ of $\G$, and denoting the displacement at a point $x$ by \label{n.53}
	$$\calD_F^\rho(x):=\max_{\gamma\in F}d(\rho(\gamma)x,x),$$
	the set of points with displacement bounded by $R$ is a bounded subset for every $R$. 
\end{defn}

Observe that this property implies that $\G$ doesn't have a fixed point in the ideal boundary of $\mathcal Y$; if $\mathcal Y$ is a symmetric space, an action is non-evanescent if and only if it is non-parabolic.

\begin{thm}\label{p.properact}
	Let $\frakT\subset \calR_F(\Gamma,G)$ be a $G$-invariant closed
	semialgebraic subset. Assume that no representation $\rho\in \frakT$
	is parabolic.  Then, for every $(\rho,\FF)$
	representing a point
	$\alpha$ in $\rspcl{\partial} \frakT$ with $\FF=\FF_\alpha$,
	the action of $\G$ on the complete
	CAT(0)-space $\obgof$ is non-evanescent.
\end{thm}

\begin{proof}
	Assume by contradiction that there exists  $(\rho,\FF)$ representing some point $\alpha\in\rspcl\partial \frakT$, 
	with $\FF=\FF_\alpha$, and $R>0$ such that 
	$$\text{Min}_{R,F}(\rho):=\{x\in \ov\bgof |\, \calD_F^\rho(x)<R\}$$
	is unbounded. 
	
	We know from Proposition~\ref{c.reP2} that we can find a non-principal ultrafilter $\omega$, 
	a sequence of scales $\pmu=(e^{\lambda_k})_{k\in\NN}$ 
	and a sequence of representations $(\rho_k)_{k\in\NN}$ in $\frakT$ and a valuation compatible field injection $j:\FF\to\rom$  
	such that $j\circ\rho=\rhom$.
	Furthermore we know from Corollary~\ref{cor:4.17} that there is  an equivariant isometric inclusion $$\ov\bgof \to {\calB_{\bG(\rom)}}$$
	and $\calB_{\bG(\rom)}$ is identified with the asymptotic cone $\Cone(\calX_\RR,\omega, \pmb{\lambda}, d)$ of the symmetric spaces with  constant basepoint $x^{(0)}=[\Id]_{n\geq 1}$.
	In particular, the set
	$\text{Min}_{R,F}(\rhom)=\{x\in \calB_{\bG(\rom)}|\, \calD_F^{\rhom}(x)<R\}$
	is unbounded in $\calB_{\bG(\rom)}$. 
	
	We may take $R$ large enough so that $x^{(0)}\in \text{Min}_{R,F}(\rhom)$. 
	Then choose a sequence $(x^{(k)})_{k\geq 1}$ in $\text{Min}_{R,F}(\rhom)$ with $d(x^{(k)},x^{(0)})\geq k$.
	
	With $x^{(k)}=(x^{(k)}_n)_{n\geq 1}$ we have:
	$$\left\{ \begin{array}{l l}
		\lim_\omega\frac1{\lambda_n}{d(x_n^{(k)},x_n^{(0)})}\geq k &\quad \forall k\\
		\lim_\omega\frac1{\lambda_n}{d(\rho_n(\gamma)x_n^{(k)},x_n^{(k)})}\leq R  &\quad \forall k, \forall \g\in F\\
		\lim_\omega\frac1{\lambda_n}{d(\rho_n(\gamma)x_n^{(0)},x_n^{(0)})}\leq R  &\quad\forall \g\in F.
	\end{array}
	\right.$$
	We can inductively choose, for every $k$, an integer $n(k)$ with $n(k)<n(k+1)$ such that 
	$$\left\{ \begin{array}{ll}
		{d(x_{n(k)}^{(k)},x_{n(k)}^{(0)})}\geq (k-1) {\lambda_{n(k)}}\\
		{d(\rho_{n(k)}(\gamma)x_{n(k)}^{(k)},x_{n(k)}^{(k)})}\leq (R+1){\lambda_{n(k)}}&\quad\forall \g\in F\\
		{d(\rho_{n(k)}(\gamma)x_{n(k)}^{(0)},x_{n(k)}^{(0)})}\leq (R+1){\lambda_{n(k)}}&\quad\forall \g\in F.
	\end{array}
	\right.$$
	
	Take any non-principal ultrafilter $\omega'$, the sequence of scales $\pmu'=(e^{\lambda_{n(k)}})_{k\in\NN}$ and  consider the  representation $\rhomp:\G\to \bG(\RR^{\omega'}_{\pmu'})$ corresponding to $(\rho_{n(k)})_{k\geq 1}$. 
	Observe that $\alpha_{\rhomp}$ is a retraction of a point in $\rsp \frakT$ and thus, since $\rsp \frakT$ is closed,  $\alpha_{\rhomp}$ is in $\rsp \frakT$. 
	We will show that the representation $\rhomp$ is parabolic, thus giving the desired contradiction.
	
	Let $\beta_k:=d(x_{n(k)}^{(k)},x_{n(k)}^{(0)})$ and consider the geodesic segment 
	$$r_k:[0,\beta_k/\lambda_{n(k)}]\to \calX_\RR$$
	joining    $x_{n(k)}^{(0)}$ to $x_{n(k)}^{(k)}$. 
	In the $\frac d{\lambda_{n(k)}}$ distance, the segment $r_k$ has length at least $k-1$, and, by convexity of the distance function in a CAT(0)-space, 
	$$	{d(\rho_{n(k)}(\gamma)r_k(t),r_k(t))}\leq (R+1){\lambda_{n(k)}}$$
	for all $t\in[0,\beta_k/\lambda_{n(k)}]$.
	
	We consider the ray $r:[0,\infty)\to\Cone_{\omega'}(\calX_\RR, x_{n(k)}^{(0)}, \frac{d}{\lambda_{n(k)}})$ defined by $r(t)=(\ov{r_j(t)})_{j\in\NN}$,
	where we set $\ov{r_j(t)}= r_j(t)$ unless $\beta_j\leq t\lambda_{n(j)}$ in which case we set  $\ov{r_j(t)}=x_{n(j)}^{(j)}$. 
	
	Then we have 
	$$d(\rhomp(\g)r(t),r(t))\leq R+1 \quad \forall t\geq0 \quad\forall \g\in F$$
	which implies that every element in $F$, and thus every element in $\G$, fixes the corresponding point in $\partial_\infty\bgromp$. Theorem ~\ref{t.bdrystab} implies then that the representation $\rhomp$ is parabolic, a contradiction.
\end{proof}	
A direct consequence of \cite[Theorem~4.1.2]{KSch} is then the following
\begin{cor}
	Let $\Gamma=\pi_1(M)$ be the fundamental group of a compact
	Riemannian manifold. Let $\frakT\subset \calR_F(\Gamma,G)$ be a
	closed $G$-invariant semialgebraic subset containing no parabolic
	representation.  Then for all $(\rho,\FF)$
	representing a point in $\rsp{\frakT}$,
	there  exists a $\rho$-equivariant harmonic map
	$$f:\wt M\to \ov\bgof.$$
\end{cor}	
We work with the space $\ov\bgof $ (as opposed to $\ov\bgrol$) because the former is separable, 	and canonically associated to the representation $\rho$.
\subsection{Framings}\label{s.framing}
Framings and boundary maps have proved being an important tool in the study of representations varieties and subsets thereof. We mention here three concrete examples, which motivate the study pursued in this section. 

{Anosov representations} $\rho:\G\to G_\RR$, where $\G$ is a Gromov hyperbolic group, can be characterized as those representations that admit a continuous, transverse, dynamics preserving equivariant map $\xi:\partial_\infty\G\to \bGmQ(\RR)$ for some parabolic $\RR$-subgroup $\bQ$ and satisfy additional contraction properties which are automatic in case $\rho$ has Zariski dense image \cite{GW}. 

In the setting where $\G=\pi_1(S)$ is the fundamental group of a finite area, complete, non-compact, hyperbolic surface $S$, Fock and Goncharov used framings defined on the cyclically ordered set of cusps in the boundary of the universal cover $\widetilde S$ to define and parametrize the so-called {positive representations}, a subset of the character variety $\Xi(\pi_1(S),G_\RR)$ of $\pi_1(S)$ in split real Lie groups $G_\RR$, characterized by positivity properties of the associated framing  \cite{FG}. An important difference with the previous setup is that, in this case, the framing is not necessarily uniquely determined by the representation.

More recently a lot of work aims at generalizing the context of Anosov representations to other classes of discrete groups, including relatively hyperbolic groups, as well as various classes of Coxeter groups \cite{ZZrelative, DGKLM}.  For these representations boundary maps still play an important role, but require extra flexibility in the choice of the appropriate domain of definition.

The goal of this subsection is to show how properties of boundary maps can be transferred to the real spectrum compactification. 
These lead to different settings  that find applications in the various situations described above.

Assume first that $\G$ is Gromov hyperbolic, and let $\calH_\G$ be the set of fixed points of infinite order elements in the boundary at infinity
$\partial_\infty\Gamma$ of $\G$. Any infinite order element $\g$ in $\G$ is hyperbolic 
and has a unique attracting fixed point in $\partial_{\infty}\G$ denoted $\gamma^+$. 
\begin{defn}\label{d.dynframing}
	With the concepts and notations of ~\ref{s.proximal}, 
	we say that a representation $\rho:\Gamma\to G_\FF$ admits an \emph{$I$-dynamics preserving framing defined over $\FF$} \label{n.54}
	if there exists a non-empty $\Gamma$-invariant subset $X\subset\calH_\Gamma$ 
	anda $\rho$-equivariant map $\xi:X\to \calF_I(\FF)$ such
	that for every $\g\in \G$ hyperbolic such that $\gamma^+\in X$, $\rho(\g)$ is $I$-proximal over $\FF$ 
	and $\xi(\gamma^+)$ is the attracting fixed point of $\rho(\gamma)$.
\end{defn}
An immediate consequence of Corollary~\ref{l.fp} is:
\begin{remark}\label{l.framing}
	If $\rho$ admits an $I$-dynamics preserving framing $\xi$ defined over $\FF$, it is uniquely defined (Corollary~\ref{l.fp} (i)) and if
	$\LL$ is a real closed intermediary field $\KK\subset\LL\subset\FF$ with $\rho(\Gamma)\subset\bG(\LL)$, then $\xi$
	takes values in $\calF_I(\LL)$ (Corollary~\ref{l.fp} (ii)).
\end{remark}
Anosov representations \cite{GW,KLP} provide interesting and well studied examples of representations encompassed in this setting
\begin{prop}
	Any $\bQ$-Anosov representation $\rho:\Gamma\to G_\RR$ 
	admits  a $I$-dynamics preserving framing defined over $\RR$, where $\bQ$ is conjugate to ${}_\KK\bP_I$.	
\end{prop}	
\begin{proof}
	It follows from \cite[Theorem 1.3]{GGKW} that $\bQ$-Anosov representations admit continuous, $\rho$-equivariant, dynamics-preserving framings defined on the whole $\partial_\infty\Gamma$. The result follows by restriction to $X$. 
\end{proof}	

If a semialgebraic subset $\frakT$ of the representation variety only consists of representations admitting a dynamics preserving framing, the same is true for every representation in the real spectrum compactification of $\frakT$:
\begin{prop}\label{p.dpframing}
	Let $\frakT\subset \calR_F(\Gamma,G)$ be a  semialgebraic subset and $X\subset\calH_\G$ a $\G$-invariant non-empty subset. Fix a subset $I\subset{}_\KK\Delta$ (see ~\ref{s.proximal}) and assume every  $\rho\in\frakT$ admits a $I$-dynamics preserving framing $\xi_\rho:X\to\calF_I(\KK)$. Then every $(\rho,\FF)$ representing a point in $\rsp \frakT$ admits a $I$-dynamics preserving framing $\xi_\rho:X\to\calF_I(\FF)$.
\end{prop}
\begin{proof}
	Fix a point $x\in X$. By definition $x=\gamma^+$ for some element $\gamma\in\G$. Since for every $\rho\in \frakT$, $\rho(\g)$ is $I$-proximal over $\FF$, the same is true for any $\rho\in\rsp \frakT$ (recall Equation \eqref{e.proximal}). We can then define the framing for a representation   $\rho\in\rsp \frakT$ using the uniqueness of attracting fixed points (Corollary ~\ref{l.fp}(ii)).
\end{proof}	
A key feature of Anosov representations is that their associated framing has a continuous extension to the whole $\partial_\infty\Gamma$. This is not to be expected for representations in the real spectrum compactification, since, at least if $\KK$ is countable, for any real closed field $\FF$ arising in the real spectrum compactification, the space $\calF_I(\FF)$ is countable. However some classes of Anosov representations can be characterized through additional semialgebraic properties of the associated framing, properties  that transfers nicely to the real spectrum compactification. 
To be more precise we need another definition:
\begin{defn}
	Let $I\subset{}_\KK\Delta$ and $m\geq 1$. A \emph{configuration space}  \label{n.55}
	is a semialgebraic $G$-invariant subset  $\mathfrak C\subset\calF_I(\KK)^m$. From a configuration space $\mathfrak C$, we obtain for every real closed field $\FF\supset\KK$ a configuration space $\mathfrak C_\FF\subset\calF_I(\FF)^m$.
\end{defn}

The following examples  list some configuration spaces that are important for our discussion. One verifies that all the configuration spaces we discuss can be directly defined as subsets of $\calF_I(\FF)^m$ for every real closed field $\FF$, and such configuration spaces are $\FF$-extension of the corresponding $\KK$-configuration spaces.

\begin{exs}\label{e.confspa}
	The following are configuration spaces; see Examples ~\ref{e.confspa2} below for concrete classes of representations for which these are relevant. In the first two examples we consider the case where ${}_\KK\bP_I$ is self opposite, that is, it is conjugate to an opposite parabolic. Then we say that two points $x_1,x_2\in\calF_I(\KK)$ are opposite if their stabilizers in $\bG$ are opposite (or transverse) parabolic subgroups.
	\begin{enumerate}
		\item The open $\bG(\KK)$-orbit in $\calF_I(\KK)^2$; equivalently
		$$\{(x,y)\in\calF_I(\KK)^2|\, x,y \text{ are opposite}\}.$$
		\item For  $m\geq 2$, any (union of) semialgebraic connected components of 
		$$\calF_I(\KK)^{(m)}:=\{(x_1,\ldots, x_m)\in\calF_I(\KK)^{m}|\, x_i \text{ is opposite to } x_j \quad\forall i\neq j\}$$
		For example:
		\begin{enumerate}
			\item For $\bG=\Sp_{2n}$ the subgroup of $\SL_{2n}$ preserving the standard symplectic form on $\KK^{2n}$ and $\calF_I(\KK)=\{L\subset\KK^{2n}|\, L\text{ is maximal Lagrangian}\}$, then a configuration space playing an important role in \cite{BIW, BP} is 
			$$\mathfrak C=\{(L_1,L_2,L_3)\in\calF_I(\KK)^3|\, \text{ the Maslov index of $(L_1,L_2,L_3)$ is $n$}\}.$$  
			\item For $\bG=\SL_n$ and $\calF_I(\KK)$ the variety of complete flags in $\KK^n$, then a configuration space is the  subset $\mathfrak C$ consisting of  quadruples of  flags that are positive in the sense of Fock-Goncharov \cite{FG, Flamm}.
			\item  For $\bG$ and $I\subset{}_\KK\Delta$ such that $\calF_I(\RR)$ admits a $\Theta$-positive structure (see \cite{GWpos} for definition),
			a configuration space is 
			$$\mathfrak C=\{(x_1,\ldots,x_4)\in\calF_I(\KK)^4|\, (x_1,\ldots,x_4) \text{ is $\Theta$-positive}\}.$$
			Pairs $(\bG,I)$ admitting a $\Theta$-positive structure were classified in \cite{GWalg}.
		\end{enumerate}
		\item For $\bG=\SL_d$  letting $\Gr_l(\KK^d)$ denote the Grassmannian of $l$-dimensional subspaces of $\KK^d$, and $$\calF_I(\KK)=\{(a,b,c)|\; a\in \Gr_k(\KK^d), b\in \Gr_{d-k-1}(\KK^d), c\in \Gr_{d-k+1}(\KK^d), a\subset b\subset c \},$$ 
		a configuration space is the set $\mathfrak C$ of triples $(x,y,z)\in\calF_I(\KK)^{(3)}$ satisfying property $H_k$, namely such that the sum
		\bq\label{e.Hk}
		\left(x^k\cap z^{d-k+1}\right)+\left(y^k\cap z^{d-k+1}\right)+z^{d-k-1}
		\eq
		is direct  \cite[\S~8.2]{PSW}.
		\item For $\bG=\SL_{n+1}$,  $\calF_I(\KK)=\KK\PP^n$, a configuration space is
		$$\mathfrak C=\{(x_1,\ldots,x_{n+2})\in\calF_I(\KK)^{n+2}|\, (x_1,\ldots,x_{n+2}) \text{ are projectively independent}\}.$$ 
	\end{enumerate}
	Observe that Examples (1), (3), (4) make sense for any field $\KK$ while all the  examples in (2) require $\KK$ to be at least ordered.
\end{exs}

With configuration spaces at hand, we can then define
\begin{defn}
	Let $Y\subset X^m$ be a $\G$-invariant subset and $\mathfrak C\subset {\calF_I(\KK)}^m$ be a configuration space. An $\FF$-framing $\phi:X\to \calF_I(\FF)$ is \emph{$(Y,\mathfrak C)$-positive} \label{n.56}
	if for every $m$-tuple $(x_1,\ldots, x_m)\in Y$, $(\phi(x_1),\ldots,\phi(x_m))\in \mathfrak C_{\FF}$. 
\end{defn}
In the next examples we discuss interesting and well studied subsets of character varieties, including the so-called higher Teichm\"uller spaces, which admit $(Y,\mathfrak C)$-positive framings for suitable $Y$ and $\mathfrak C$. Developing analogue theories over any real closed field $\FF$ is compelling, but  beyond the scope of this paper. Nevertheless the results in this section give important tools to study properties of representations in the real spectrum compactification of such subsets. 
\begin{exs}\label{e.confspa2}
	The following are examples of classes of representations admitting a $(Y,\mathfrak C)$-positive $\RR$-framing for the configuration spaces in Examples ~\ref{e.confspa}.
	\begin{enumerate}
		\item Let $\G$ be any Gromov hyperbolic group, $X=\calH_\G$, $Y\subset \calH_\G^2$ the set of pairs of distinct points. If ${}_\KK\bP_I$ is a self-opposite parabolic subgroup, every ${}_\KK\bP_I$-Anosov representation admits a $(Y,\mathfrak C)$-positive $\RR$-framing $\phi:X\to \calF_I(\RR)$ for the set $\mathfrak C$  of opposite points as in Example ~\ref{e.confspa} (1) \cite[Theorem 1.3]{GGKW}.
		\item Let $\G$ be the fundamental group of a closed orientable surface and $X=\calH_\G\subset \partial_{\infty}\Gamma=\mathbb S^1$ with the induced cyclic ordering.
		\begin{enumerate}
			\item If $Y\subset X^3$ denotes the set of positively oriented triples, $\calF_I(\RR)$ is the set of Lagrangians in $\RR^{2n}$, and  $\mathfrak C$ is the set of maximal triples as in Example ~\ref{e.confspa} (2a), then any maximal representation admits  a $(Y,\mathfrak C)$-positive $\RR$-framing $\phi:X\to \calF_I(\RR)$ \cite{BIW, BP}.
			\item If $Y\subset X^4$ denotes the set of cyclically ordered 4-tuples, $\calF_I(\RR)$ is the set of complete flags in $\RR^n$, and  $\mathfrak C$ is the set of positive quadruples as in Example ~\ref{e.confspa} (2b), then  any Hitchin representation admits  a $(Y,\mathfrak C)$-positive $\RR$-framing $\phi:X\to \calF_I(\RR)$ \cite{FG, Flamm}.
			\item If $Y\subset X^4$ denotes the set of cyclically ordered 4-tuples,  $\bG$, $I$ are such that $\calF_I(\RR)$ admits a $\Theta$-positive structure, and  $\mathfrak C$ is the set of $\Theta$-positive quadruples as in Example ~\ref{e.confspa} (2c), then, by definition,  any $\Theta$-positive representation admits  a $(Y,\mathfrak C)$-positive $\RR$-framing $\phi:X\to \calF_I(\RR)$. It was proven in \cite{BGLPW} that $\Theta$-positive representations form connected components of the  character and representation variety.
		\end{enumerate}
		{	\item  Let $\G$ be a finitely generated subgroup of $\PSL_2(\RR)$, $X\subset \calH_\G$, $Y\subset \calH_\G^3$. If $\calF_I(\RR)$ and $\mathfrak C$ are as in Example ~\ref{e.confspa} (3), 
			then, by definition, every representation satisfying property $H_k$ admits a $(Y,\mathfrak C)$-positive framing \cite{PSW}. Properties of representations in the real character variety that satisfy property $H_k$, including the fact that they are positively ratioed, have been studied in the case of fundamental groups $\G$ of closed surfaces in \cite{BeyP}.}
		\item Let $\G<\SL_{n+1}(\RR)$ be a group dividing a strictly convex  subset $\Omega_\G\subset \RR\PP^n$, such as the fundamental group of an $n$-dimensional compact hyperbolic manifold. Then $\Gamma$ is Gromov hyperbolic, and the connected component of the representation variety $\calR_F(\G,\SL_{n+1}(\RR))$ of the natural inclusion $i$ only consists of representations $\rho$ dividing a convex divisible set $\Omega_\rho$. In this case the boundary $\partial\Omega_\rho$ identifies with $\partial_\infty\Gamma$; restricting this identification to $\calH_\G$ we obtain a $I$-dynamics preserving framing.  Such framing is $(Y,\mathfrak C)$-positive, where $Y$ denotes the set of pairwise distinct $(n+2)$-tuples, and $\mathfrak C$ is as in Example ~\ref{e.confspa} (4)  \cite{BenoistIII}.
	\end{enumerate}
\end{exs}	
We then have the following refinement of Proposition ~\ref{p.dpframing}:
\begin{prop}\label{p.dpCframing}
	Let $\frakT\subset \calR_F(\Gamma,G)$ be a  semialgebraic subset, $X\subset\calH_\G$ non-empty, $\G$-invariant,	$Y\subset X^m$ be a $\G$-invariant subset and $\mathfrak C\subset {\calF_I(\KK)}^m$ a configuration space. If every $\rho$ in $\frakT$ admits a $I$-dynamics preserving, $(Y,\mathfrak C)$-positive framing, the same is true for any $\rho\in\rsp \frakT$.
\end{prop}
\begin{remark}
	As an application, in the particular case of the Hitchin representation variety, we deduce from Proposition ~\ref{p.dpCframing} in the case of Example ~\ref{e.confspa2} (2b), that the transversality and positivity hypotheses of \cite[Theorem 1.1, 1.2, 1.3]{Martone} are always satisfied.
\end{remark}


We conclude the section discussing a more general setting, which, for example, includes the framings naturally associated to positive representations in \cite{FG}:

\begin{defn}\label{d.genframing}
	Let $\G$ be a finitely generated group, and $X$ be a set endowed with a $\G$-action with the additional property that for every point $x\in X$ there exists an infinite order element $\g\in\G$ with $\g\cdot x=x$. An \emph{$\FF$-framing} for a representation $\rho\in\calR_F(\G, G_\FF)$ is a $\rho$-equivariant map $\phi:X\to \calF_I(\FF)$. 
\end{defn}	
\begin{example}
	If a group $\G$ is relatively hyperbolic with respect to a family $\calP$ of parabolic subgroups each containing elements of infinite order, then the set $\calP$ is a good example for the set $X$ in Definition ~\ref{d.genframing}. In the special case in which $\G<\PSL(2,\RR)$ is the fundamental group of a finite area, complete, non-compact, hyperbolic surface, the set $\calP$ identifies with the cyclically ordered set of cusp points in $\partial\HH$. For more general classes of groups $\G$ we expect that other geometrically defined sets will prove being relevant, such as the set of fixed point in the Morse boundary \cite{Morse}, or some sets of  subgroups conjugated to special subgroups in the case of Coxeter groups \cite{DGKLM}.
\end{example}	
As another consequence of the Artin-Lang principle we obtain: 
\begin{prop}\label{p.framing3}
	Let $\frakT\subset \calR_F(\Gamma,G)$ be a  semialgebraic subset, $X$ a $\G$-set  each of whose elements is fixed by at least one infinite order element of $\G$, $m\in\NN$, $Y\subset X^m$ a $\G$-invariant subset and $\mathfrak C\subset {\calF_I(\KK)}^m$ a configuration space.
	\begin{enumerate}
		\item If $X$ consists of finitely many $\G$-orbits, every $\rho$ in $\frakT$ admits a $\KK$-framing, and every $\KK$-framing is $(Y,\mathfrak C)$-positive, then the same is true for any    $\rho\in\rsp \frakT$.
		\item If $X$ and $Y$ consist of finitely many $\G$-orbits, and every $\rho$ in $\frakT$ admits a $(Y,\mathfrak C)$-positive $\KK$-framing, then the same is true for any    $\rho\in\rsp \frakT$.
	\end{enumerate}
\end{prop}
\begin{proof}
	(1) Choose representatives $\{x_1,\ldots, x_p\}$ of the $\G$-orbits in $X$, and denote by $P_i<\G$ the stabilizer of the point $x_i$. Since for any $i$ and any $\rho\in \frakT$, $\rho(P_i)$ has a fixed point in $\calF_I(\KK)$, and a fixed point equation is a semialgebraic equation, the same holds true for any $\rho\in\rsp \frakT$ (Theorem ~\ref{thm:constr} (1)). We can define then a framing $\phi$ associated to a representation $\rho\in\rsp \frakT$ by choosing a fixed point for any $i$, and extending to the whole $X$ by $\G$-equivariance.
	The framing $\phi$ is $(Y,\mathfrak C)$-positive as otherwise it would contradict that any framing is $(Y,\mathfrak C)$-positive for any $\rho$ in $\frakT$, again by Theorem ~\ref{thm:constr} (1). 
	\\
	(2) Let, as above, $\{x_1,\ldots, x_p\}$ be representatives of the $\G$-orbits in $X$, denote by $P_i<\G$ the stabilizer of the point $x_i$. We furthermore choose representatives $\{y_1,\ldots, y_l\}$ for the $\G$-orbits in $Y$. For any $1\leq j\leq l$, $$y_j=(\gamma_{j,1}x_{i_{j,1}},\ldots, \gamma_{j,m}x_{i_{j,m}}),$$ for some $i_{j,k}\in\{1,\ldots, p\}$ and elements $\gamma_{j,k}\in \G$. For any $p$-tuple $\xi=(\xi_1,\ldots,\xi_p)\in\left(\calF_I(\KK)\right)^p$ and any $\rho:\G\to \bG(\KK)$ we will write $y_j^\rho(\xi)$ for $(\rho(\gamma_{j,1})\xi_{i_{j,1}},\ldots, \rho(\gamma_{j,m})\xi_{i_{j,m}})$.
	We then have that the set 
	$$\left\{\left.\left(\rho, \xi=(\xi_1,\ldots, \xi_p)\right)\in \frakT\times \left(\calF_I(\KK)\right)^p\right|\;\begin{array}{ll}
		\rho(P_i)\xi_i=\xi_i & \forall 1\leq i\leq p, \\
		y^\rho_j(\xi)\in\mathfrak C & \forall 1\leq j\leq l
	\end{array}\right\}$$
	Is semialgebraic and surjects on $\frakT$. It follows from Property ~\ref{ppts.2.10}(2) that the same holds over any field $\FF$, and allows to define, by equivariance, a $(Y,\mathfrak C)$-positive framing for every $\rho\in \rsp \frakT$.
\end{proof}	
\begin{exs}
	Proposition ~\ref{p.framing3} is applicable in a variety of situations, for example:
	\begin{itemize}
		\item In the case of maximal representations of open surfaces we can consider the set $X$ of cusp points, which consists of finitely many $\G$-orbits. If we choose a set of representatives of these $\G$-orbits, corresponding to elements $g_1,\ldots, g_p$, it is possible to verify that for any choice of a fixed point $P_i$ for $g_i$ in $\calF_I(\RR^{2n})$ the induced framing is positive. In particular Proposition ~\ref{p.framing3}(1) applies. This has been used in \cite{AGRW}.
		\item In order to construct Fock-Goncharov coordinates on the set of positive representations one considers again $X$ the set of cusp points, and can choose $Y$ to be the set of 4-tuples of points in the boundary of pairs of adjacent triangles in a chosen ideal triangulation of the surface $S$. In this case also the set $Y\subset X^4$ consists of finitely many $\G$-orbits, and thus Proposition ~\ref{p.framing3}(2) applies.
	\end{itemize}	
\end{exs}	

\section{Character varieties}\label{s.char}
The character variety is, as a first approximation, the topological space obtained as quotient $\Xi(\Gamma,\GRR)=\Hom_\mathrm{red}(\Gamma, \GRR)/\GRR$ of the set of reductive representations. 
The goal of this section is to define an explicit realization of the character variety of a finitely generated group $\Gamma$ in a semisimple semialgebraic group $G$ 
as a $\KK$-semialgebraic set and to give a characterization of closed points in its real spectrum compactification.
The explicit model will be based on Richardson-Slodowy's work on invariant theory \cite{RS}, which is in turn based on the study of minimal vectors.
\subsection{Reductive representations}
Let as usual $\KK\subset \RR$ be  a real closed field, 
$\FF\supset\KK$ a real closed extension and $\bG(\KK)^\circ\subset
G\subset \bG(\KK)$ be $\KK$-semialgebraic semisimple, where as usual $\bG<\SL_n$ is a connected semisimple $\KK$-subgroup. 

A representation $\rho:\G\to\SL_n(\FF)$ is \emph{reductive} if every $\Gamma$-invariant subspace in $\FF^n$ has a $\G$-invariant complement. 
We deduce from Property \ref{ppts.2.10}(1) that, for every real closed extension $\FF$ of $\KK$,  and every $\rho:\G\to \SL_n(\KK)$, then $\rho$ is reductive if and only if the representation $\rho':\G\to \SL_n(\FF)$ obtained by composing $\rho$ with the injection $\SL_n(\KK)\to\SL_n(\FF)$ is reductive.
We say  furthermore that $\rho:\G\to G_\FF$ is reductive if the
representation $\rho':\G\to \SL_n(\FF)$ obtained by composing $\rho$ with the injection $G_\FF\to\SL_n(\FF)$ is reductive. 

Let $F\subset \G$ be a finite generating set, and $\calR_F^{red}(\Gamma,G_\FF)\subset \calR_F(\Gamma,\SL_n(\FF))$ \label{n.60b}
the image under the evaluation map of the set $\Hom_\mathrm{red}(\G, G_\FF)$ of reductive representations into $G_\FF$ 
(recall Equation \eqref{e.evaluation}). 
A standard argument using quantifier elimination \cite[5.2.2]{BCR} 
shows that $\calR_F^{red}(\Gamma,\SL_n(\KK))$ is a semialgebraic subset of $\calR_F(\Gamma,\SL_n(\KK))$ 
and $\calR_F^{red}(\Gamma,\SL_n(\FF))=\calR_F^{red}(\Gamma,\SL_n(\KK))_\FF$.
It follows then that $\calR_F^{red}(\Gamma,G)$ is a semialgebraic subset of $\calR_F(\Gamma,G)$ 
and that $\calR_F^{red}(\Gamma,G_\FF)=\calR_F^{red}(\Gamma,G)_\FF$.
\subsection{Richardson-Slodowy theory over real closed fields and the proof of Theorem \ref{thm:1}}\label{s.RS}

Assume, as in \S\ref{subsec:Cartan}, that $\bG<\SL_n$ is invariant under transposition. 
We summarize the main points of Richardon-Slodowy construction \cite{RS}, which we will review more in detail, 
and adapt to our discussion in the rest of the section. Since $\bG(\RR)^\circ<G_\RR<\bG(\RR)$ and $\bG$ is connected semisimple, 
$G_\RR$ is Zariski dense in $\bG$.  
In order to construct a semialgebraic model for the quotient of $\calR_F^{red}(\G,G_\RR)$ by $G_\RR$, 
we will follow \cite{RS}, 
remembering that $\calR_F^{red}(\Gamma, G_\RR)$ corresponds to the closed $G_\RR$-orbits in $\calR_F(\G,G_\RR)$ (see \cite{Bre}). 

More precisely, the Richardson--Slodowy theory studies $G_\RR$-orbits in $M_{n,n}(\RR)^f$, 
endowed with the standard scalar product. A point  $A\in M_{n,n}(\RR)^f$ has closed $G_\RR$-orbit 
if and only if the norm along $(G_\RR)_*A$ attains a minimum \cite[Theorem 4.4]{RS}; 
let then $\calM_{G_\RR}$ denote the subset of  vectors of minimal norm. 
Then $\calM_{G_\RR}$  is a closed cone admitting an algebraic description, 
and a $G_\RR$-orbit in $M_{n,n}(\RR)^f$ is closed if and only if it intersects $\calM_{G_\RR}$ 
in which case the intersection is a $K_\RR$-orbit \cite[Theorem 4.5]{RS}. 
The model for $\Xi(\G, G_\RR)$ is then the geometric invariant theory quotient $K_\RR\backslash (\calM_{G_\RR}\cap \calR_F(\G,G_\RR))$. 
This description of $\Xi(\G, G_\RR)$ has two advantages, 
on the one hand it defines a canonical semialgebraic structure on $\Xi(\G, G_\RR)$ (\S \ref{s.canonicity}), 
on the other hand the properness of the quotient map $p:\calM_{G_\RR}\to \Xi(\G, G_\RR)$ 
(Proposition \ref{p:Prsp}) plays an important role in the characterization of closed points in $\rsp{\Xi(\G, G_\RR)}$.


To be more precise, for $f\in \NN_{\geq 1}$, we consider on   $M_{n,n}(\RR)^f$ the standard scalar product:
$$\<(A_1,\ldots, A_f),(B_1,\ldots, B_f)\>=\sum_{i=1}^f\tr(A_i^t B_i),$$
and denote by $\|\cdot\|$ the associated norm.

\begin{defn}
	The set of \emph{$\GRR$-minimal vectors} \label{n.57}
	is the subset
	$\calM_{\GRR} \subset M_{n,n}(\RR)^f$ consisting of the $f$-tuples $A$ 
	such that $\|g_*A\|\geq \|A\|$ for all $g\in G_\RR$.
\end{defn}
Of course $\calM_{\GRR}$ is a closed cone. 
A consequence of \cite{RS} is that  $\calM_{\GRR}$ admits a description as $\RR$-points of a $\KK$-algebraic set, to which we now turn.

With the notation of  Section~\ref{subsec:Cartan}, recall that $K_\RR:= G_\RR\cap \SO(n,\RR)$, and  $K:= G\cap \SO(n,\RR)$.
Then, if $\fg=\Lie(\bG(\RR))$, we have the Cartan decomposition $\fg=\fk+\fp$, 
where  $\fk$ and $\fp$ are real vector spaces defined over $\KK$. We  let then $X_1,\ldots, X_m$ be a $\KK$-basis of $\fp(\KK)$.
Define the $\KK$-algebraic set:
$$\calM_{G}=\{A\in M_{n,n}(\KK)^f:\, \sum_{j=1}^f\<[X_i,A_j],A_j\>=0\quad \forall\, 1\leq i\leq m\};$$
then \cite[Theorem 4.3]{RS} implies $\calM_{G}(\RR)=\calM_{\GRR}$. 
Furthermore \cite[Theorem 4.3]{RS} implies that for $A\in \calM_{\GRR}$ we have
\bq\label{e.GKorbit}
(G_\RR)_* A\cap \calM_{G}(\RR)= (K_\RR)_* A.
\eq

We will now define the quotient of the set of minimal vectors by the $K_\RR $-action by conjugation.  
To this aim we will use a basis of $K_\RR $-invariant polynomials:
\begin{lem}\label{l.Kinv}
	There are  $K_\RR $-invariant homogeneous polynomials $p_1,\ldots, p_l$ on $ M_{n,n}(\RR)^f$ 
	with coefficients in $\KK$ generating the algebra of $K_\RR $-invariant polynomials.
\end{lem}
\begin{proof}
	Let $\calB$ be a basis of $\RR$ as a $\KK$-vector space. 
	Then let $P_1,\ldots, P_s\in\RR[X_{ij}^{(\g)}]^{K_\RR}$ be a finite set of generators of the algebra of $K_\RR$-invariant polynomials. 
	We can of course assume that $P_i$ are homogeneous. Every $P_i$ can be expressed as a finite sum
	\bq\label{eq.Kx}
	P_i=\sum_{\beta\in E}p^{(i)}_\beta \beta
	\eq 
	with $p^{(i)}_\beta\in\KK[X_{ij}^{(\g)}]$ and $E\subset \calB$ finite.  
	Since the representation \eqref{eq.Kx} is unique, we deduce that each $p^{(i)}_\beta$ is $K$-invariant 
	and since $K$ is dense in $K_\RR$, $p^{(i)}_\beta$ is $K_\RR$-invariant. 
	The family $\{p_1,\ldots,p_l\}:=\{p^{(i)}_\beta|\; 1\leq i\leq s, \beta\in E\}$ has the desired properties.
\end{proof}	
Let $\{p_1,\ldots, p_l\}$ be a  set of $K_\RR $-invariant homogeneous polynomials with coefficients in $\KK$  
generating the algebra of $K_\RR $-invariant polynomials  on $ M_{n,n}(\RR)^f$. 
The associated quotient map is the polynomial map of $\KK$-algebraic sets \label{n.58} 
\begin{equation}\label{e.pdef}
	\begin{array}{cccc}
		p:&M_{n,n}(\KK)^f&\to&\KK^l\\
		&x&\mapsto&(p_1(x),\ldots, p_l(x)).
	\end{array}
\end{equation}
As usual $p_\RR$ denotes the extension to $\RR$.
%
Let then $m_i$ be the degree of homogeneity of the polynomial $p_i$ in Lemma \ref{l.Kinv} and $m$ be a multiple of $m_1,\ldots, m_l$.
\begin{prop}\label{prop:pptyP}
	The map $p:M_{n,n}(\KK)^f\to \KK^l$  (see \eqref{e.pdef}) satisfies the following.
	\begin{enumerate}
		\item There exist constants $c_1,c_2\in\NN_{\geq 1}$ such that 
		$$\frac 1{c_1}\|x\|^{2m}\leq \sum_{i=1}^l(p_i(x))^{\frac{2m}{m_i}}\leq c_2\|x\|^{2m},\; \forall x\in M_{n,n}(\KK)^f.$$
		\item The image under $p$ of any closed semialgebraic subset of $\calM_G$ is closed semialgebraic in $\KK^l$.
	\end{enumerate}
	The assertions (1) and (2) hold for the extension to any real closed field $\FF\supset \KK$.
\end{prop}	
\begin{proof}
	(1) Since the fibers of the map $p:M_{n,n}(\RR)^f\to \RR^l$ are the $K_\RR$-orbits (see \cite[\S~7.1]{RS})
	and $K_\RR\cdot 0=0$ we have $p^{-1}(0)=0$. This implies that the homogeneous polynomial 
	\bqn
	\sum_{i=1}^l(p_i(x))^{\frac{2m}{m_i}}
	\eqn
	has strictly positive minimum and maximum in the unit sphere in $M_{n,n}(\RR)^f$ from which the inequalities follow.
	
	Let $S\subset M_{n,n}(\KK)^f$ be  closed semialgebraic; then  $S_\RR\subset M_{n,n}(\RR)^f$ is closed semialgebraic 
	(Properties \ref{ppts.2.10} (1)) and since $p_\RR:M_{n,n}(\RR)^f\to\RR^l$ is proper by (1), $p_\RR(S_\RR)\subset \RR^l$ is closed, 
	it is also semialgebraic and  $p_\RR(S_\RR)=p(S)_\RR$ 
	which implies that $p(S)\subset \KK^l$ is closed semialgebraic by \ref{ppts.2.10} (2).
\end{proof}	

A fundamental example of closed semialgebraic subset of $M_{n,n}(\RR)^f$ to which Proposition \ref{prop:pptyP} applies is 
$$\calM_F(\Gamma,G):=\calM_{G}\cap \calR_F(\Gamma,G).$$
Then \label{n.59} 
$$\Xi_{F,p}(\G,G):=p(\calM_F(\G,G)),$$
is a closed semialgebraic subset of $\KK^l$ and so is  
\bqn
\Xi_{F,p}(\Gamma,G)_\FF=p_\FF(\calM_F(\G,G)_\FF)=p_\FF(\calM_{G_\FF}\cap \calR_F(\Gamma,G_\FF))\subset\FF^l
\eqn
for any real closed $\FF\supset\KK$.

Our aim now is to extend the map $p:\calM_{F}(\Gamma, G)\to \KK^l$  to a semialgebraic map  $P:\calR_F^{red}(\Gamma,G)\to \KK^l$ whose fibers are precisely the $G$-orbits. 
To this aim observe that, since $\calR_F^{red}(\G, G_\RR)$ coincides with the set of closed $G_\RR$-orbits in $\calR_F(\G, G_\RR)$ \cite{Bre}, we have
\bq\label{e.red}\calR_F^{red}(\G, G_\RR)=\{A\in \calR_F(\G, G_\RR)|\; (G_\RR)_* A\cap \calM_{G_\RR}\neq \emptyset\}.\eq
Equation \eqref{e.red} holds over $\KK$, as well as over any real closed extension $\FF\supset \KK$.

The continuous semialgebraic map
$$
\begin{array}{rcc}
	\{(g,A)\in G\times \calR_F(\G,G)|\; g_*A\in \calM_G\}&\to&\KK^l\\
	(g,A)&\mapsto& p(g_*A)	
\end{array}	
$$
doesn't depend on the first variable because of \eqref{e.GKorbit}. This defines by projection and by \eqref{e.red} a continuous semialgebraic map
\bq\label{e.P}
P:\calR_F^{red}(\G,G)\to\KK^l.\eq
\begin{prop}\label{p.proper}
	Let $\FF\supset\KK$ be real closed. Then the closed semialgebraic set $\Xi_{F,p}(\Gamma,G)_\FF$ is the $\FF$-extension of $\Xi_{F,p}(\Gamma,G)\subset \KK^l$. Furthermore
	\begin{enumerate}
		\item The fibers of the continuous semialgebraic map 
		$$p_\FF:\calM_F(\Gamma,G)_\FF\to \Xi_{F,p}(\Gamma,G)_\FF$$
		coincide with the $K_\FF$-orbits.
		\item The fibers of the continuous semialgebraic map 
		$$P_\FF:\calR_F^{red}(\Gamma,G)_\FF\to \Xi_{F,p}(\Gamma,G)_\FF$$
		coincide with the ${\GFF}$-orbits.
	\end{enumerate}
\end{prop}
\begin{proof}
	
	In the case $\KK=\RR$ (1) follows from the fact that $K_\RR$-invariant polynomials separate $K_\RR$-orbits \cite[Section 7.1]{RS}, while (2) is a direct consequence of \eqref{e.GKorbit}.
	
	The general case  follows readily from the case $\KK=\RR$, we prove (1) as (2) is completely analogous. Consider the $\KK$-semialgebraic sets
	$$\calD=\{(x,y,g)\in \calM_F(\Gamma,G)\times \calM_F(\Gamma,G)\times K|\; P(x)=P(y), gx=y \}$$
	and 
	$$\calF=\{(x,y)\in \calM_F(\Gamma,G)\times \calM_F(\Gamma,G)|\; P(x)=P(y) \}$$
	as well as the projection $\pi:\calD\to\calF$.	Then the fact that the fibers of $P_\RR:\calM_F(\Gamma,G)_\RR\to \Xi_{F,p}(\Gamma,G)_\RR$ are precisely the $K_\RR$-orbits is equivalent to the surjectivity of $\pi_\RR:\calD_\RR\to\calF_\RR$; by Properties \ref{ppts.2.10}(2) the same statement holds for $\KK\subset \RR$ and for any real closed $\FF\supset\KK$.
\end{proof}	
We furthermore deduce from Proposition \ref{prop:pptyP}(1):
\begin{prop}\label{p:Prsp}
	The polynomial map $p:\calM_F(\Gamma,G)\to\Xi_{F,p}(\Gamma, G)$ extends to continuous maps 
	$$\begin{array}{cccc}
		\rsp p:&\rsp{ \calM_F(\Gamma,G)}&\twoheadrightarrow&\rsp{\Xi_{F,p}(\Gamma, G)}\\
		\rspcl {p}:&\rspcl{\calM_F(\Gamma,G)}&\twoheadrightarrow&\rspcl{\Xi_{F,p}(\Gamma, G)}.
	\end{array}$$
\end{prop}
\noindent The only thing to verify is that the map $\rsp p$ sends
closed points to closed points, which follows from the following
general lemma 
\begin{lem}
	Let $V\subset \KK^{s}$, $W\subset \KK^n$ be closed semialgebraic and $f:\KK^{s}\to\KK^{n}$ a polynomial map with $f(V)=W$. Then $\rsp f(\rsp{V})=\rsp W$ and $\rsp f(\rspcl{V})\supset\rspcl {W}$. If furthermore there exists $c\in\NN_{\geq 1}$ such that 
	\bq\label{e.l7.6}\|x\|^{2m}\leq c\sum f_i(x)^{2d_i} \eq
	where $f=(f_1, \ldots, f_n)$, $f_i$ is homogeneous of degree $\frac{m}{d_i}\in\NN_{\geq 1}$, $m,d_i\in \NN_{\geq 1}$, 
	then $\rsp f(\rspcl{V})=\rspcl {W}$.
\end{lem}	
\begin{proof}
	The first assertion follows immediately from Property \ref{ppts.2.10}(2), while for the second, if $y\in \rspcl {W}$ is a closed point, $f^{-1}(y)\subset \rsp V$ is a closed  set and hence contains closed points.
	
	Let $(\phi,\FF)$ represent a closed point in $\rsp V$ where $\phi:\KK[X_1,\ldots, X_s]\to \FF$ is a homomorphism as in Proposition \ref{p.charrsp}(3). Then $\FF$ is Archimedean over $\KK[\phi(X_1),\ldots, \phi(X_n)]$ (Proposition \ref{prop:closed_points}) and the composition of $f^*:\KK[Y_1,\ldots, Y_n]\to \KK[X_1,\ldots, X_s]$ with $\phi$ represents $\rsp f(\phi,\FF)\in \rsp W$. Since $\FF$ is Archimedean over $\KK[\phi(X_1),\ldots, \phi(X_s)]$, and $\sum_{i=1}^s\phi(X_i)^{2m}$ is a big element in $\FF$, we deduce  from the inequality \eqref{e.l7.6}
	$$\sum_{i=1}^n f_i(\phi(X_1),\ldots,\phi(X_s))^{2d_i}$$
	is a big element in $\FF$ as well, which shows, by Proposition \ref{p.charrsp}(3), that $\rsp f(\phi,\FF)$ is a closed point.
\end{proof}	
\begin{defn}\label{d.reprapoint}
	A reductive representation $\rho:\G\to G_\FF$ \emph{represents a point} $\alpha\in\rsp{\Xi_{F,p}(\Gamma, G)}$ i
	f for  $P_\FF(\rho)\in{\Xi_{F,p}(\Gamma, G)}_\FF$ 
	we have $\alpha=\alpha_{P_\FF(\rho)}$ (see Remark \ref{rem:point->prime cone}). \label{n.60} 
\end{defn}	

A direct corollary of Proposition \ref{p:Prsp} is the analogous statement for closed $G$-invariant semialgebraic subsets $\frakT$ of 
$\calR_F^{red}(\G,G)$. 
For each such set $\frakT$ we denote by $\Xi\frakT\subset \Xi_{F,p}(\G,G)$ the image $P(\frakT)$, 
which coincides with $p(\frakT\cap\calM_F(\G,G))$.
\begin{cor}\label{c:Prsp}
	Let $\frakT\subset \calR_F^{red}(\G,G)$ be closed, $G$-invariant semialgebraic.
	The map $p:\calM_F(\Gamma,G)\to\Xi_F(\Gamma, G)$ induces continuous maps 
	$$\begin{array}{cccc}
		\rsp p:&\rsp{(\frakT\cap \calM_F(\Gamma,G))}&\twoheadrightarrow&\rsp{\Xi\frakT}\\
		\rspcl {p}:&\rspcl{(\frakT\cap\calM_F(\Gamma,G))}&\twoheadrightarrow&\rspcl{\Xi\frakT}.
	\end{array}$$ 
\end{cor}
As a consequence of the explicit description of the semialgebraic model $\Xi_{F,p}(\G,G)$ of the character variety 
we obtain the existence and uniqueness of the minimal field mentioned in the introduction:
\begin{cor}\label{c.minimalfield}
	Let $\LL$ be a real closed field containing $\KK$ and  $\rho:\G\to G_\LL$ be a reductive representation. 
	Let  $\KK\subset\FF\subset\LL$ be the smallest real closed field containing the coordinates of $P(\rho)\in\LL^l$. 
	Then $\FF$ is $\rho$-minimal in the following sense:
	\begin{enumerate}
		\item $\rho$ can be conjugated into $G_\FF$;
		\item $\rho$ cannot be conjugated into $G_\EE$ for $\EE\subset\FF$ proper real closed.
	\end{enumerate}	
	If in addition $(\rho(\g))_{\g\in F}\in\calM_F(\G,G_\LL)$ then $\rho$ can be conjugated into $G_\FF$ by an element of $K_\LL$.
\end{cor}	

\begin{proof}
	Pick $\rho\in\calR_F^{red}(\G,G_\LL)$ and let $\FF\subset\LL$
	be the smallest real closed field containing $\KK$ and  the coordinates of $P(\rho)\in\LL^l$, 
	so we have that $P(\rho)\in\FF^l$ and $P^{-1}(P(\rho))$ is an $\FF$-semialgebraic set 
	whose extension  $P^{-1}(P(\rho))_\LL$ contains $\rho$ and hence is not empty. 
	It follows from Property \ref{ppts.2.10}(1) that  already $P^{-1}(P(\rho))_\FF$ is not empty, 
	which implies that $\rho$ can be $G_\LL$-conjugated into an element of $\calR_F^{red}(\G, G_\FF)$. 
	Clearly $\FF$ is minimal. The last assertion follows from an analogous argument using Proposition \ref{p.proper}(1).
\end{proof}
This justifies the assertions in the introduction before Theorem \ref{thm:1} 
and, together with Proposition \ref{prop:Rspecrep}, concludes the proof of Theorem \ref{thm:1}. 

\subsection{Minimal vectors, scales and length functions}
The main result of the section is the algebraic inequality in Theorem \ref{t.traceanddisp} relating 
the norm of a minimal vector and the translation length of an explicit generating set, expressed in terms of the traces of the corresponding elements.
This will be key in the characterization of closed points in the real spectrum compactification of $\Xi(\Gamma, G)$.
We will deduce such a result from a theorem proven by Procesi in the case $G_\RR=\SL_n(\RR)$ 
(Theorem~\ref{t.Procesi}), through a detour on the geometry of symmetric spaces.

If $w$ is any word in the alphabet $\{1,\ldots,f\}$, say $w=i_1\ldots i_r$, then $l(w)=r $ and, if $M$ is any associative monoid, 
we denote by $w:M^f\to M$ the product  map, $w(m_1,\ldots, m_f)=m_{i_1}\ldots m_{i_r}$.

\begin{thm}\label{t.traceanddisp}
	Let $m=\lcm\{1,\ldots, 2^n-1\}$, then there exists $c_1=c_1(f,n)$ such that for every $g\in\calM_{\GRR}\cap \GRR^f$,
	$$\|g\|^{2m}\leq n^m\left(c_1\sum_{l(w)\leq 2^{n}-1}\tr(w(g))^{\frac {2m}{l(w)}}\right)^{2(n-1)}.$$
\end{thm}
We first prove Theorem \ref{t.traceanddisp} in the special case of $\SL_n(\RR)$ in which we get the better bound:
\begin{prop}\label{c.procesi}
	Let $m=\lcm\{1,\ldots, 2^{n}-1\}$. Then for every $A\in\calM_{\SL_n(\RR)}$ we have
	$$\|A\|^{2m}=\left(\sum_{i=1}^f\tr(A_i^tA_i)\right)^{m}\leq C\sum_{l(w)\leq 2^n-1}(\tr(w(A)) )^{\frac{2m}{l(w)}}.$$ 
\end{prop}
The proof of Proposition \ref{c.procesi} follows from a general bound on the norm of a minimal vector in terms of the values of a chosen set of generators of $G_\RR$-invariant polynomials, together with a theorem of Procesi describing a special set of $GL_n(\RR)$ invariant polynomials.

To be more precise let $\bH<\GL_n$ be a connected linearly reductive group defined over $\RR$ and $q_1,\ldots,q_s$ be a set of homogeneous generators for the algebra of $\bH(\RR)$-invariant polynomials on $M_{n,n}(\RR)^f$. 
Since $\bH(\RR)$ is Zariski dense in $\bH$, the polynomials $q_1,\ldots,q_s$ extended to $ M_{n,n}(\CC)^f$ also generate the algebra of $\bH$-invariant polynomials. 
We consider the map
$$\begin{array}{cccc}
	q:& M_{n,n}(\CC)^f&\to&\CC^s\\
	&A&\mapsto&(q_1(A),\ldots, q_s(A)).
\end{array}$$
The general bound is the following effective version of \cite[Lemma 6.3]{RS}:
\begin{prop}\label{prop:Ginvariantpoly}
	Let $d_i$ denote the degree of $q_i$, and $d$ be any common multiple of $d_1,\ldots, d_s$. Then there is $C>0$ such that for every $A\in\calM_{\bH(\RR)}$
	$$\|A\|^{2d}\leq C\sum_{i=1}^sq_i(A)^{\frac{2d}{d_i}}.$$
\end{prop}
\begin{proof}
	The fibers of $q: M_{n,n}(\CC)^f\to\CC^s$ separate closed $\bH$-orbits and hence $\bH\cdot\{0\}=\{0\}$ is the unique closed $\bH$-orbit in $q^{-1}(0)$ (cfr. \cite[Lemma 6.4]{RS}). If now $x\in\calM_{\bH(\RR)}$ with $q(x)=0$, then $\bH(\RR)\cdot x$ is closed and hence $\bH\cdot x$ is closed, which implies $0\in\bH\cdot x$ and hence $x=0$. Thus $q^{-1}(0)\cap \calM_{\bH(\RR)}=\{0\}$. 
	If now $S$ is the unit sphere in $M_{n,n}(\RR)^f$, then the homogeneous polynomial 
	$$A\mapsto \sum_{i=1}^sq_i(A)^{\frac{2d}{d_i}}$$
	is strictly positive on $S\cap \calM_{\bH(\RR)}$ from which
	the proposition follows. 
	
\end{proof}
Proposition \ref{c.procesi} follows choosing the following set of invariant polynomials:
\begin{thm}[\cite{Procesi}]\label{t.Procesi}
	The algebra of $\GL_n(\RR)$-invariant polynomials on $M_{n,n}(\RR)^f$ is generated by the polynomials 
	$A\mapsto \tr(w(A))$ 
	where $w$ runs over all words of length less than $2^n-1$. Such a polynomial is homogeneous of degree $l(w)$.
\end{thm}
Observe that any element $s$ in $\GL_n(\RR)$ has a representative of the form $s=gkt$ for some $g\in\SL_n(\RR)$, $k\in O_n(\RR)$, and $t>0$. As a result 
we have $\calM_{\GL_n(\RR)}=\calM_{\SL_n(\RR)}$, 
thus concluding the proof of Proposition \ref{c.procesi}. 
\bigskip

We now turn to the generalization to arbitrary groups $G<\SL_n$.
Taking into account that $(G_\RR)^f\subset (\SL_n(\RR))^f$ we have that a vector
$$A\in\calM_{G_\RR}\cap (G_\RR)^f$$
corresponds to a reductive representation $\rho$ of the free group on $f$ generators into $G_\RR$ which, when composed with the injection $G_\RR\to\SL(n,\RR)$, is still reductive, and hence defines an $\SO_n(\RR)$-orbit in $\calM_{\SL_n(\RR)}\cap (\SL_n(\RR))^f$. We denote by 
$$\Phi:K_\RR\backslash( \calM_{G_\RR}\cap (G_\RR)^f)\to \SO_n(\RR)\backslash( \calM_{\SL_n(\RR)}\cap(\SL_n(\RR))^f)$$
the induced map.

\begin{lem}\label{l.2.10}
	Given $g\in \calM_{G_\RR}\cap (\GRR)^f$, we have
	\bqn
	\|g\|^2\leq n\|\Phi(g)\|^{4(n-1)}.
	\eqn
\end{lem}
\begin{proof}
	For every $g\in \calM_{G_\RR}\cap (\GRR)^f$,
	$$\begin{array}{rl}
		\|g\|^2&=\inf_{h\in G_\RR}\|h^{-1}gh\|^2\\
		&=\inf_{h\in G_\RR}\sum_{i=1}^f\tr((h^{-1}gh)^t(h^{-1}gh))\\
		&\leq n\inf_{h\in G}\sum_{i=1}^fe^{d(g_ih_*\Id,h_*\Id)}\\
		&=n\inf_{h\in \SL_n(\RR)}\sum_{i=1}^fe^{d(g_ih_* \Id,h_*\Id)}\\
		&\leq n\inf_{h\in\SL_n(\RR)}\|h^{-1}gh\|^{4(n-1)}\\
		&=n\|\Phi(g)\|^{4(n+1)}.
	\end{array}
	$$
	Here the first inequality follows from the upper bound in Lemma \ref{lem:6.7}, the equality is a consequence of the fact that $G_\RR$ preserves the convex symmetric subspace $\calX_\RR=(G_\RR)_*\Id\subset\calP^1(n,\RR)$, and the inequality follows from the lower bound in Lemma \ref{lem:6.7}.
\end{proof}

\begin{proof}[Proof of Theorem~\ref{t.traceanddisp}]
	With the notation of Theorem~\ref{t.traceanddisp}, Lemma~\ref{l.2.10} implies that for $g\in\calM_{G_\RR}\cap (G_\RR)^f$, 
	$$\|g\|^{2m}\leq n^m\|\Phi(g)\|^{4(n-1)m}.$$
	Taking into account that $g$ and $\Phi(g)$ are $\SL(n,\RR)$-conjugate, we get from Proposition~\ref{c.procesi} that 
	$$\|\Phi(g)\|^{4(n-1)m}\leq \left(C\sum_{l(w)\leq 2^n-1}(\tr(w(g)) )^{\frac{2m}{l(w)}}\right)^{2(n-1)}$$
	which together with the first inequality implies the theorem.
\end{proof}	
\subsection{Proof of Theorem \ref{thm:1.2}}
We now have all the tools to prove Theorem \ref{thm:1.2} from the introduction, which we restate in a slightly more general form.
\begin{thm}\label{t.1.2text}
	Let $F=F^{-1}$ be a finite generating set of $\Gamma$
	and let $E:=F^{2^n-1}\subset\Gamma$. Let $(\rho,\FF)$ represent a point in $\Xi_{F,p}(\Gamma,G)^\mathrm{RSp}\smallsetminus\Xi_{F,p}(\Gamma,G)_\RR$, and assume $\FF$ is minimal. 
	The following assertions are equivalent:
	\be
	\item\label{it:bdry_cl1text} $(\rho,\FF)$ represents a closed point.
	\item\label{it:bdry_cl2text} The $\Gamma$-action on $\bgf$ does not have a global fixed point.
	\item\label{it:bdry_cl3text} There exists $\eta\in E$ such that $\rho(\eta)$ has positive translation length on $\ov\bgf$.
	\ee
\end{thm}
\begin{proof}
	We keep the same notation as in Section~\ref{s.RS}: $p_1,\ldots, p_l$ are $K$-invariant homogeneous polynomials on $M_{n,n}(\KK)^f$ of degree $m_i$ with coefficients in $\KK$, and they generate the algebra of $K$-invariant polynomials.
	
	$(1)\Rightarrow (3)$ We may assume $(\rho(\g))_{\g\in F}\in\calM_F(\G,G_\FF)$ by Corollary \ref{c.minimalfield}. The field $\FF$ is the real closure of $\KK(p_1(\rho(\g)),\ldots, p_l(\rho(\g)))$ and by hypothesis $\FF$ is Archimedean over $\KK[p_1(\rho(\g),\ldots, p_l(\rho(\g)))]$, that is 
	$$\sum_{i=1}^l p_i(\rho(\g))^{\frac{2m}{m_i}}$$
	is a big element in $\FF$. But then   $\|(\rho(\g))_{\g\in F}\|^{2m}$ is a big element and so is 
	$$\sum_{l(w)\leq 2^{n}-1}\tr(w(g))^{\frac {2m}{l(w)}},\quad\text{ where $g=(\rho(\g))_{\g\in F}$}$$
	by Theorem \ref{t.traceanddisp}. This implies that for some $\eta\in E$, $(\tr(\rho(\eta)))^2$ is a big element.
	
	In order to control the translation length we can fix an $\SL_n(\RR)$-invariant Riemannian distance on $\calP^1(n,\RR)$; this induces a $\bG(\RR)$-invariant Riemannian distance on $\calX_\RR$, hence a Weyl group invariant scalar product on $\frak a$, with corresponding norm $\|\cdot\|$. We may now use $d_{\|\cdot\|}^\FF$ on $\bgof$ and  compute the translation length of $\rho(\eta)\in\bG(\FF)$ using the $\SL_n(\FF)$-Jordan projection of $\rho(\eta)$. Since $(\tr(\rho(\eta)))^2$ is a big element in $\FF$, if $\lambda_1,\ldots,\lambda_n$ are the eigenvalues of $\rho(\eta)$ over $\FF(\sqrt{-1})$, then at least one of $|\lambda_1|,\ldots,|\lambda_n|$ in $\FF_{>0}$ must be a big element, which implies (see Proposition \ref{prop:6.6.6}) that $\ell_{\|\cdot\|}(\rho(\eta))>0$.
	
	\noindent$(3)\Rightarrow (2)$ is clear.
	
	\noindent$(2)\Rightarrow (1)$. Assume that $(\rho,\FF)$ does not have a global fixed point. We may assume $(\rho(\g))_{\g\in F}\in \calM_F(\G, G_\FF)$. Then $[\Id]\in\bgof$ is not a fixed point and hence (Proposition \ref{prop:Rspecrepcl}) $(\rho(\g))_{\g\in F}$ is a closed point in $\rsp{\calM_F(\G, G)}$ which implies by Proposition \ref{p:Prsp} that $p((\rho(\gamma))_{\g\in F})$ is a closed point in $\rsp{\Xi_{F,p}(\Gamma, G)}$.
\end{proof}

\subsection{Proof of Theorem \ref{thm:ggt}}
We now prove Theorem \ref{thm:ggt} from the introduction, which we restate in our usual degree of generality.
\begin{thm}\label{thm:7.16}  Let $\omega$ be a non-principal ultrafilter on $\NN$, 
	$((\rho_k,\RR))_{k\geq1}$ a sequence of representations so that, for every $k$, 
	\bq\label{e.minrepb}(\rho_k(\g))_{\g\in F}\in\calM_F(\Gamma,G_\RR),
	\eq and  $(\rhol,\rol)$  its $(\omega,\pmb{\mu})$-limit for an adapted sequence of scales $\pmb{\mu}$.
	Then:
	\begin{itemize}
		\item $\rhol$ is reductive, and
		\item if $\pmb{\mu}$ is well adapted and  infinite, and $\FF_{\rhol}$ is the $\rhol$-minimal field, 
		then $(\rhol,\rol)$ is $K_{\rol}$-conjugate to a representation $(\pi,\FF_{\rhol})$ 
		that represents a closed point in $\rsp\partial\Xi(\Gamma,G)$.
	\end{itemize}
	Conversely, any $(\rho,\FF)$ representing a closed point in $\rsp\partial\Xi(\Gamma,G)$ arises in this way.
	More precisely for any non-principal ultrafilter $\omega$  and any sequence of scales $\pmb\mu$ giving an infinite element,  there exist an order preserving field injection
	$i\colon \FF\hookrightarrow\rol$ and
	a sequence of  homomorphisms $((\rho_k,\RR))_{k\geq1}$  satisfying \eqref{e.minrepb} for which $\pmb\mu$ is well adapted 
	and such that $i\circ\rho$ and $\rhol$ are $G_{\rol}$-conjugate.
\end{thm}
\begin{proof}
	We apply Proposition \ref{prop:4.18} to the closed semialgebraic subset $\calM_F(\G,G_\RR)\subset M_{n,n}(\RR)$ and to $\RR\subset\calO_\plambda\subset \RR^\omega$, to obtain that the reduction modulo  $\calI_\plambda$ maps $\calM_F(\G,G_{\calO_\plambda})$ to $\calM_F(\G,G_\rol)$; since $\plambda$ is adapted we have $(\rho_k(\g)_{k\geq1})_{\g\in F}\in\calM_F(\G,G_{\calO_\plambda})$ and hence $(\rhol(\g))_{\g\in F}\in\calM_F(\G,G_{\rol})$ which implies in particular that $\rhol$ is reductive.
	
	Assume now that $\plambda$ is well adapted; let $\FF_\rhol$ be the smallest real closed field containing the coordinates of $p(\rho)$, then by Corollary \ref{c.minimalfield} $\rhol$ can be conjugated into $G_{\FF_\rhol}$ by an element $k\in K_\rol$. Let $\pi:\G\to G_{\FF_\rhol}$ be the conjugate representation. Since $\plambda$ is well adapted, 
	$$\ov\plambda=\sum_{\g\in\G}\tr(\rhol(\g)^t\rhol(\g))\in\rol$$
	is a big element and since $k\in K_\rhol$ the latter equals
	$$\ov\plambda=\sum_{\g\in\G}\tr(\pi(\g)^t\pi(\g))\in\FF_\rhol$$
	which is hence a big element in $\FF_\rhol$ which shows that $\pi(\G)$ is not contained in $G_M$ where $M$ is the valuation ring corresponding to the valuation determined by $\ov\plambda$ and hence $(\pi,\FF_{\rhol})$ is a closed point in $\rsp\partial\calR_F(\G,G)$ (Proposition \ref{prop:Rspecrepcl}); thus $(\pi(\g))_{\g\in F}\in\calM_F(\G, G_{\FF_{\rhol}})$ gives a closed point in $\rsp\calM_F(\G, G)$ and hence represents a closed point in $\rsp\partial\Xi_{\G,p}(\G, G)$ by Proposition \ref{p:Prsp}.
	
	Since the map $\rsp p:\rspcl{\calM_F(\G,G)}\to\rspcl{\Xi_{F,p}(\G,G)}$ is surjective (Corollary \ref{c:Prsp}), the second statement follows directly from a result analogue to Corollary \ref{c.repacc} replacing $\calR_F(\Gamma, G)$ by the semialgebraic subset $\calM_F(\Gamma,G)$. 
	
\end{proof}

\begin{remark}
	The hypothesis that $(\rho_k(\g))_{\g\in F}$ are minimal vectors is essential. For instance let $\rho_k:\G\to\SL_2(\RR)$ be the sequence of representations of the free groups on two generators $a,b$ defined by 
	$$\rho_k(a)=\bpm1&k\\0&1\epm\quad\rho_k(b)=\bpm1&0\\e^{-k}&1\epm,\quad k\in\NN.$$
	Then the sequence $\pmu=(k)_{k\geq 1}$ is well adapted and one verifies that 
	$$\rhol(a)=\bpm1&\pmu\\0&1\epm\quad\rhol(b)=\bpm1&0\\0&1\epm.$$
	In particular $\rhol$ is not reductive.
\end{remark}	
\subsection{Real semialgebraic models of character varieties}\label{s.canonicity}
The goal of this section is, on the one hand, to discuss how canonical our realization of the character variety is, on the other hand to show that $\Out(\G)$ naturally acts on $\Xi_{F,p}(\Gamma,G)$ by semialgebraic homeomorphism. We work in the following more general setting:
\begin{defn}
	A \emph{real semialgebraic model for the $G$-character variety of $\G$} is a pair $(F,p)$ where $F=(\g_1,\ldots, \g_f)$ is a labelled finite generating set of the group $\G$, $l$ is a positive integer, and $p:\calR^{red}_F(\G, G)\to \KK^l$ is a continuous semialgebraic map such that 
	\begin{enumerate}
		\item The $G$-orbits in $\calR^{red}_F(\Gamma, G)$ coincide exactly with the fibers of $p:\calR^{red}_F(\Gamma, G)\to \Xi_{F,p}(\Gamma,G)$, where $\Xi_{F,p}(\Gamma,G)$ is the image of $p$.
		\item The $\RR$-extension $p_\RR$ induces a homeomorphism $\calR^{red}_F(\Gamma, G_\RR)/G_\RR\cong \Xi_{F,p}(\Gamma,G)_\RR$.
	\end{enumerate} 
	If $(F,p)$ is a real semialgebraic model for the character variety, we denote by $\Xi_{F,p}(\Gamma,G)\subset \KK^l$ the image of $p$.
\end{defn}	 
It follows from the discussion in Section~\ref{s.RS} that Richardson--Slodowy theory gives the concrete example $(F,P)$  of a real semialgebraic model of the character variety where $P: \calR^{red}_F(\Gamma, G)\to \Xi_{F,p}(\Gamma,G)$ is the extension of $p:\calM_F(\Gamma,G)\to \KK^l$ defined in Equation \eqref{e.P}.

In preparation for the next proposition, which describes how different semialgebraic models are related, let $F_1=(\gamma_1,\ldots \gamma_{f_1}), F_2=(\eta_1,\ldots \eta_{f_2})$ be finite generating sets of $\G$; for every $1\leq j\leq f_2$ let $w_j^{F_1}$ be a word in $\gamma_1,\ldots \gamma_{f_1}$ such that $\eta_j=w_j^{F_1}(\gamma_1,\ldots \gamma_{f_1})$ and define the polynomial map
$$\begin{array}{cccc}
	P_{12}:&M_{n,n}(\KK)^{f_1}&\to &M_{n,n}(\KK)^{f_2}\\
	&(A_1,\ldots, A_{f_1})&\mapsto&\left(w^{F_1}_j(A_1,\ldots, A_{f_1})\right)_{1\leq j\leq f_2}
\end{array}$$
and similarly
$$\begin{array}{cccl}
	P_{21}:&M_{n,n}(\KK)^{f_2}&\to &M_{n,n}(\KK)^{f_1}\\
	&(A_1,\ldots, A_{f_2})&\mapsto&\left(w^{F_2}_i(A_1,\ldots, A_{f_2})\right)_{1\leq i\leq f_1}
\end{array}$$
where $w_i^{F_2}$ is a word in $\eta_1,\ldots \eta_{f_2}$ with $\gamma_i=w_i^{F_2}(\eta_1,\ldots \eta_{f_2})$. 

Let $Q_{12}$ be the restriction of $P_{12}$ to $\calR^{red}_{F_1}(\G,G)$ and $Q_{21}$ the restriction of $P_{21}$ to $\calR^{red}_{F_2}(\G,G)$. Then $Q_{12}$ and $Q_{21}$ are inverse of each other and $G$-equivariant, since $P_{12}$ and $P_{21}$ are $\GL_n(\KK)$-equivariant.
\begin{prop}\label{p.canonicity}
	Let  $(F_1, p_1)$  and $(F_2,p_2)$ be semialgebraic models for the character variety. Then the above maps $Q_{12}$ and $Q_{21}$  induce semialgebraic homeomorphisms
	$$q_{12}:\Xi_{F_1,p_1}(\Gamma,G)\to\Xi_{F_2,p_2}(\Gamma,G),\ \ \ \ 
	q_{21}:\Xi_{F_2,p_2}(\Gamma,G)\to\Xi_{F_1,p_1}(\Gamma,G)$$
	which are inverse of each other.
\end{prop}	 
\begin{proof}
	Clearly $\Graph( Q_{12})\subset\calR_{F_1}^{red}(\Gamma, G)\times \calR_{F_2}^{red}(\Gamma, G)$ is a semialgebraic subset and hence so is 
	$p_1\times p_2(\Graph( Q_{12}))\subset\Xi_{F_1,p_1}(\Gamma,G)\times\Xi_{F_2,p_2}(\Gamma,G)$.
	Since the fibers of $p_1,p_2$ are precisely the $G$-orbits and $Q_{12}$ is $G$-equivariant the subset $p_1\times p_2(\Graph( Q_{12}))$ is the graph of a map 
	$q_{12}:\Xi_{F_1,p_1}(\Gamma,G)\to\Xi_{F_2,p_2}(\Gamma,G)$ which is therefore semialgebraic. It remains to show its continuity, since reversing the role of $F_1,F_2$ this will show the continuity of $q_{21}$, and conclude the proof since $q_{12}$ and $q_{21}$ are inverse of each other. To this end we show the continuity of the $\RR$-extension $(q_{12})_\RR$.
	
	In order to show the continuity of $(q_{12})_\RR$, it clearly suffices to show that for any convergent sequence 
	$(x_k)_{k\in\NN}\subset\Xi_{F_1,p_1}(\Gamma,\bG_\RR)$ with limit $x$, there is a subsequence $(x_{k_\ell})_{\ell\in\NN}$ 
	with 
	\bqn
	\lim_{\ell\to\infty}(q_{12})_\RR(x_{k_\ell})=(q_{12})_\RR(x).
	\eqn
	Thus let $(x_k)_{k\in\NN}$ be any sequence with limit $x$.  Pick $\tilde x\in \calR_{F_1}^{\mathrm{red}}(\Gamma,\bG_\RR)$ 
	and $\Omega\ni\tilde x$ open with $\overline\Omega$ compact.  
	By property (2), $(p_1)_\RR$ is a quotient map for a continuous group action, hence open.  
	Thus $(p_1)_\RR(\Omega)\ni x$ is open and thus there exists $N\in\NN$ and $\tilde x_k\in\Omega$ for all $k\geq N$, 
	such that $(p_1)_\RR(\tilde x_k)=x_k$.  
	Since $\overline\Omega$ is compact, let $(\tilde x_{k_\ell})_{\ell\in\NN}$ be a convergent subsequence with limit $y\in\overline\Omega$.
	Then $(p_1)_\RR(y)=x$ and there exists $g\in\bG_\RR$ with $gy=x$.  Thus $\lim_{\ell\to\infty}g\tilde x_{k_\ell}=\tilde x$, and
	\bqn
	\ba
	(q_{12})_\RR(x)
	&=(p_2)_\RR(Q_{12})_\RR(\tilde x)\\
	&=(p_2)_\RR(Q_{12})_\RR(\lim_\ell g\tilde x_{k_\ell})\\
	&=\lim_\ell(p_2)_\RR g(Q_{12})_\RR(\tilde x_{k_\ell})\\
	&=\lim_\ell(p_2)_\RR(Q_{12})_\RR(\tilde x_{k_\ell})\\
	&=\lim_\ell(q_{12})_\RR(x_{k_\ell}). 
	\ea
	\eqn

\end{proof}

Let $F$ be a finite generating set of $\G$ and 
$$\ev_F:\Hom_{red}(\G,G)\to \calR_F^{red}(\G,G)$$
the evaluation map (see Section \ref{subsec:rep}) and $\alpha\in\Aut\G$. Then we plainly have 
\bqn
\ev_F(\rho\circ\alpha)=\ev_{\alpha(F)}(\rho)=Q_{12}(\ev_{F_1}(\rho))
\eqn
where $F_1=F$ and $F_2=\alpha(F)$. This way we obtain from Proposition \ref{p.canonicity} a semialgebraic homeomorphism
$$\Psi_{\alpha}:\Xi_{F,p}(\G, G)\to\Xi_{F,p}(\G, G)$$
where $(F,p)$ is a semialgebraic model of the character variety. One verifies that 
\begin{enumerate}
	\item $\Psi_{\alpha\beta}=\Psi_\beta\Psi_{\alpha}$ for all $\alpha,\beta\in\Aut(\G)$,
	\item $\Psi_\alpha=\Id$ for every inner automorphism $\alpha$.
\end{enumerate}	

\begin{cor}
	The map $\alpha\mapsto \Psi_{\alpha^{-1}}$ defines an action of $\Out(\G)$ 
	by semialgebraic homeomorphisms on $\Xi_{F,p}(\Gamma,G)$ 
	which extends to an action by homeomorphisms on $\rsp{\Xi_{F,p}(\Gamma,G)}$ preserving $\rspcl{\Xi_{F,p}(\Gamma,G)}$. 
\end{cor}	
\begin{proof}
	This follows from \cite[Proposition 7.2.8]{BCR} stating that 
	a semialgebraic homeomorphism of a semialgebraic set $S$ extends canonically to a homeomorphism of $\rsp S$, 
	which necessarily preserves $\rspcl S$.
\end{proof}	
We outline now another  canonical description of the real spectrum compactification of a character variety, the details are left to the reader.

We endow $G$ with the Euclidean topology as a closed semialgebraic subset of $M_{n,n}(\KK)$, 
and consider $\Hom_{red}(\G,G)\subset \Hom(\G,G)\subset G^\G$ where the latter is endowed with the product topology.
\begin{defn}
	A subset $S\subset \Hom_{red}(\G,G)$ is semialgebraic if there is a finite subset $F\subset \G$ with $|F|=f$ 
	and a semialgebraic subset $\calG\subset G^f$ such that 
\bqn
S=\{f:\G\to G|\; f(\g)\in\calG\; \forall \g\in F\}.
\eqn
\end{defn}	
The set $\mathcal{SA}(\Gamma,G)$ of semialgebraic subsets of $ \Hom_{red}(\G,G)$ 
as well as the set  $\mathcal{SA}_{inv}(\Gamma,G)$ of $G$-invariant semialgebraic subsets 
are boolean algebras of subsets of $ \Hom_{red}(\G,G)$.

Recall that given a set $E$ and a boolean algebra $\calB\subset\calP(E)$ of subsets of $E$ an \emph{ultrafilter} on $\calB$ 
is a subfamily $\calF\subset\calB$ such that 
\begin{enumerate}
	\item	 $E\in\calF$, $\emptyset\notin \calF$,
	\item $D\cap D'\in\calF$ if and only if both $D$ and $D'$ are in $\calF$,
	\item for any $D\in \calB$, either $D\in \calF$ or $E\setminus D\in \calF$. 
\end{enumerate}
Let $\widehat E$ be the set of all ultrafilters on $\calB$; this is the Stone space of $\calB$. Given $D\in\calB$, let $\widehat{D}=\{\calF\in \widehat{E}|\; D\in\calF\}$.
\begin{prop}
	The Stone space of $\mathcal{SA}_{inv}(\Gamma,G)$ endowed with the basis of open sets
	$$\{\widehat D|\; D\subset \Hom_{red}(\G, G) \text{ is $G$-invariant, semialgebraic, open}\}$$
	is homeomorphic to  $ \rsp{\Xi_{F,p}(\Gamma,G)}$ for any semialgebraic model $(F,p)$ of the character variety.
\end{prop}
The argument to prove this relies on two ingredients:
\begin{enumerate}
	\item The ultrafilter theorem in real algebra (see \cite[Proposition 7.1.15]{BCR});
	\item Given a semialgebraic model $p:\calR_F^{red}(\G, G)\to \Xi_{F,p}(\G,G)$, 
	the inverse image $p^{-1}(S)$ of the open semialgebraic subsets $S\subset \Xi_{F,p}(\G,G)$ coincide exactly 
	with the semialgebraic, $G$-invariant open subsets in $\calR_F^{red}(\G,G)$.
\end{enumerate}	
%

\subsection{Fixed points for elements in $\Out(\Gamma,\frakT)$}\label{sec.fixp}
It was observed by Brumfiel that the real spectrum compactification is well behaved from the viewpoint of algebraic topology. For a closed $G$-invariant  semialgebraic subset $\frakT$ of $\calR^{red}_F(\G,G)$, this leads to robust theorems guaranteeing the existence of fixed points for elements in   $\Out(\G,\frakT)$ acting on $\rspcl{(\Xi\frakT)}$.

More specifically, 
for any $X\subset \RR^n$ semialgebraic and every continuous semialgebraic map $\phi:X\to X$, the \emph{graded trace} of $\phi$ is defined by 
$$\tr(\phi_*):=\sum (-1)^i{\rm trace}(\phi_*:H_i(X,\QQ)\to H_i(X,\QQ)).$$
Brumfiel proved:
\begin{thm}[{\cite{Brum88B}}]\label{t.Brufix}
	Let $X\subset \RR^n$ semialgebraic, $\phi:X\to X$ continuous semialgebraic map. If $\tr(\phi_*)\neq 0$ then $\phi:c(X)\to c(X)$ has a fixed point. If $X$ is closed, then $\phi$ has also a closed fixed point.  
\end{thm}	
To give a concrete example in which Theorem \ref{t.Brufix} can be applied to give a non-trival result, recall that the character variety of maximal $\PSp(4,\RR)$ representations of a surface group of genus $g$ has $2(2^{2g}-1)+4g-1$ connected components of which $4g-4$ are smooth and consist entirely of representations with Zariski dense image; they are called the Gothen components and classified by a characteristic number $0< d\leq 4g-4$. We denote them $\Xi^{max}_{0,d}(\pi_1(S),\PSp(4,\RR))$.
\begin{thm}[{\cite{Alessandrini_Collier}}]
	The $d$-th Gothen component is mapping class group equivariant diffeomorphic to a holomorphic fiber bundle over $\Teich(S)$, and for every $\rho\in\Teich(S)$, $\pi^{-1}(\rho)$ is a rank $d+3g-3$ vector bundle over the $(4g-4-d)$-th symmetric power of $\Sigma_g$
\end{thm}

The component for $d=4g-5$ has the homotopy type of the surface. Any mapping class for which the trace of the induced map in $H_1(S,\QQ)$ is different from 2 has a fixed point in $\rsp\partial\Xi^{max}_{0,4g-5}(\pi_1(S),\PSp(4,\RR))$. It is probably possible to extend this result to other components by carefully studying the topology of the symmetric powers of the surface (see \cite{Mil} 
for a description of the homology of symmetric powers),  and the action of the mapping class group on this space. 

\section{The Weyl chamber length compactification}\label{s.WL}
The purpose of this section is to discuss the relation of the real
spectrum compactification of closed semialgebraic subsets of the
character variety and their Weyl chamber length compactification, and
draw interesting applications. We first show in Section~\ref{s.par}
that the real spectrum compactification dominates the Weyl chamber
length compactification, in Section~\ref{s.discrete} we discuss
finiteness properties of length functions in the boundary, and in
particular deduce that integral length functions are dense in the Weyl
chamber length compactification. In Section~\ref{s.realnotfinite} we
deduce topological properties of the Weyl chamber valued length
function map, combining  finiteness properties of boundary length
functions established in Section~\ref{s.discrete} with non-finiteness
properties for length functions of real points. In Section~\ref{s.realvaluedlength} we discuss analogue results for $\RR$-valued length functions and in Section~\ref{sec.fixp2} we establish some properties of length functions associated to fixed points, in the real spectrum, of outer automorphisms.

\subsection{A continuous map to the Weyl chamber length compactification}\label{s.par}
Let $\frakT$ be a closed $G$-invariant semialgebraic subset of $\calR^{red}_F(\G,G)$, 
and denote by $\Xi\frakT$ its image in $\Xi_{F,p}(\G,G)$. We say that a representation in $\calR^{red}_F(\G,G)$ is \emph{bounded} if its image is  contained in a subgroup conjugated to $K_\RR$. \label{n.61}
The goal of the section is to relate the real spectrum compactification $\rspcl{(\Xi\frakT)}$ and its Weyl chamber length compactification $\thp{(\Xi\frakT)}$ as in \cite{Parreau12}, under the assumption that no representation in $\frakT_\RR$ has bounded image. Important examples of sets $\frakT$ satisfying this assumption and of interest to us are the Hitchin or maximal components, or more generally semialgebraic subset consisting of representations among the ones discussed in Example \ref{e.confspa2}.

To every homomorphism $\rho:\G\to G_\RR$ we associate the Weyl chamber valued length function
\bqn\begin{array}{cccc}
	L(\rho):&	\G&\to&(\fap)\\
	&\g&\mapsto&\Ln(\Jord_\RR(\rho(\gamma)))
\end{array}\eqn
here $\Jord_\RR:G_\RR\to C_\RR$ is the Jordan projection introduced in
Section \ref{subsec:Jordan} and $ \Ln:C_\RR\to\fap$ is the logarithm
introduced in Section \ref{s.log}.
As the Jordan projection is invariant by $G_\RR$-conjucagy,
the above length function
$L(\rho)\in(\fap)^\G$ only depends on
the conjugacy class $[\rho]\in\Xi\frakT_\RR$ of $\rho$, and we denote
it by $L([\rho])$.
The discussion in Sections \ref{subsec:Jordan} and \ref{s.log} implies
that the map $L$ can  be defined with the same formula on $\Xi\frakT_\FF$ for any real closed field $\FF$.

The next lemma ensures that, under the assumption that no representation in $\frakT_\RR$ has bounded image, the map $L$ descends to a well defined map $\mathbb PL$ with values in $\mathbb P\left((\fap)^\Gamma\right)$:
\begin{lem}\label{lem:unbounded}
	Let $\rho\in\Hom_\mathrm{red}(\Gamma, \GRR)$. Then $L([\rho]):\Gamma\to\fap$ vanishes identically if and only if $\rho$ is bounded. 
\end{lem}
\begin{proof}
	Clearly if $\rho$ is conjugated in the maximal compact subgroup, then the associated Weyl chamber valued length function vanishes. 
	Conversely assume that for every $\g\in \G$  the Jordan decomposition of $\rho(\gamma)$ has no  hyperbolic part. 
	Let $\bH$ be the Zariski closure of $\rho(\Gamma)$ in $\bG$ which is reductive and defined over $\RR$. We intend to show that $\bH(\RR)$ is compact. 
	Up to passing to a finite index subgroup of $\Gamma$, we can assume that $\bH$ is connected. Then the derived group $\calD\bH$ is connected, semisimple, defined over $\RR$, $\bT=\calZ(\bH)^\circ$ is a torus and we have $\bH=\bT\cdot\calD\bH$, with $\bT\cap\calD\bH$ finite. Then $\bH_{ss}:=\bH/\bT$ is semisimple connected defined over $\RR$ and composing $\rho$ with the projection $\pi_1:\bH\to \bH_{ss}$ we obtain that
	$\pi_1(\rho(\G))<\bH_{ss}(\RR)$ is Zariski dense in $\bH_{ss}$. If $\bH_{ss}(\RR)$ were not compact, $\pi_1(\rho(\Gamma))$ would contain an $\RR$-split element \cite{Prasad}, but this would contradict the vanishing of the Weyl chamber translation length function. One deduces from this that $\calD\bH(\RR)$ is compact as well.
	
	Next, let $\bT=\bT_s\bT_a$ be the decomposition of $\bT$ as an almost direct product of an $\RR$-split torus and an $\RR$-anisotropic torus. Then $\bH=\bT_s\cdot \bT_a\cdot\calD\bH$ with $F=\bT_s\cap(\bT_a\cdot\calD\bH)$ finite; in order to conclude it is enough to show that $\bT_s$ is trivial. Assume that this is not the case. Composing $\rho$ with the projection $\pi_2:\bH\to \bH/\bT_a\calD\bH=\bT_s/F$ we have that $\pi_2(\rho(\G))$ is Zariski dense in $\bT_s/F$. This implies that in the decomposition 
	$$\bH(\RR)^\circ=\bT_s(\RR)^\circ \bT_a(\RR)(\calD\bH(\RR))^\circ$$
	the $\bT_s(\RR)^\circ$-components of the elements of $\rho(\G)$ are not all torsion since otherwise they would be contained in a fixed finite subgroup contradicting the Zariski density of $\pi_2(\rho(\G))$. Thus some $\rho(\g)$ has a non-trivial hyperbolic component in its refined Jordan decomposition, contradicting the hypothesis. We conclude that $\bT_s=\{e\}$ and $\bH(\RR)=\bT_a(\RR)\calD\bH(\RR)$ is compact.
\end{proof}
Let  $\frak T$ be a closed $G$-invariant semialgebraic subset of $\calR^{red}_F(\G,G)$ and $\Xi\frak T$ its image in $\Xi_{F,p}(\G,G)$ which is closed as well. Assume that $\frak T_\RR$ doesn't contain bounded representations. Then for every $[\rho]\in\Xi\frakT_\RR$,  $L([\rho])$ does not vanish identically,
furthermore it follows from Theorem \ref{thm:1.2} and Proposition \ref{prop:6.6.6} that if $[\rho]\in\Xi\frakT_\FF$ represents a point in $\rspcl{\partial}\Xi\frakT$ then $L([\rho])$ does not vanish as well. Thus we obtain a well defined  map 
\bq\label{e.PL}
\mathbb P L:\rspcl{(\Xi\frakT)}\to\mathbb P\left((\fap)^\Gamma\right).\\
\eq
We will show in Theorem \ref{t:Rspec-ThP} that $\mathbb P L$ is
continuous, and induces a continuous surjection on the Weyl chamber
length compactification of $\Xi\frakT_\RR$, whose construction we now recall from \cite{Parreau12}. We denote by $\widehat{\Xi\frakT}_\RR$ the Alexandrov one point compactification of $\Xi\frakT_\RR$, which is the topological space $\Xi\frakT_\RR\cup\{\infty\}$ where a fundamental set of neighbourhoods of $\infty$ is given by the complements of the compact sets of  $\Xi\frakT_\RR$. The \emph{Weyl chamber length compactification} \label{n.62}
$\thp{(\Xi\frakT)}$ of $\Xi\frakT_\RR$ is the closure of $\Xi\frakT_\RR$ under the embedding
\bq\begin{array}{ccccc}
	&	\Xi\frakT_\RR&\to&\widehat{\Xi\frakT_\RR}\times\mathbb P\left((\fap)^\Gamma\right)\\
	&[\rho]&\mapsto&([\rho],\mathbb PL([\rho])).
\end{array}\eq
There are two reasons for this topological trick. The first, obvious one, is that the map $[\rho]\to L([\rho])$ is not necessarily injective, the second one is that even when this map is injective, it is far from clear that it gives a homeomorphism onto its image. All these issue will however be tackled in Section \ref{s.discrete} where we show that the map  has locally compact image and is proper onto its image (see Theorem \ref{thm:realnotfinite}).

Recall from Section \ref{sec.fixp} that we  denote by $\Out(\Gamma, \frakT)$ the subset of the outer automorphism group of $\G$ whose action on $\Xi_{F,p}(\G,G)$, described in \S \ref{s.canonicity}, leaves the subset $\frakT$ invariant.  We define
$\widehat{\mathbb P L}:\rspcl{(\Xi\frakT)}\to\thp{(\Xi\frakT)}$
as follows: 
$$\widehat{\mathbb P L}([\rho])=\left\{
\begin{array}{ll}
	\left([\rho],{\mathbb P L}([\rho])\right) & \text{ if } [\rho]\in\Xi\frak T_\RR\\
	\left(\infty,{\mathbb P L}([\rho])\right)	& \text{ if } [\rho]\in\rspcl\partial\Xi\frak T.
\end{array}
\right.$$
\begin{thm}\label{t:Rspec-ThP}
	For any $G$-invariant, closed, semialgebraic subset $\frakT\subset\calR^{red}_F(\G,G)$ avoiding representations with bounded image, the map 
	\bqn
	\widehat{\mathbb P L}:\rspcl{(\Xi\frakT)}\to\thp{(\Xi\frakT)}
	\eqn
	is continuous, $\Out(\Gamma,\frakT)$-equivariant and surjective.
\end{thm}

\begin{proof}
	In order to define a continuous projection to the Weyl chamber length compactification we rescale  the map $L$ and consider the map
	\bq\label{e.Theta}\begin{array}{cccc}
		\Theta:&\rspcl{(\frakT\cap\calM_{G})}&\to&(\fap)^\Gamma\\
		&\rho&\mapsto &\displaystyle 	\frac{\Log_b\Jord_\FF(\rho(\gamma))}{\log_b\left(2+\sum_{\eta\in F}\tr(\rho(\eta)\rho(\eta)^t)\right)},
	\end{array}\eq
	Here $b\in\FF$ is any big element. It follows from the construction that the map $\Theta$ doesn't depend on the choice of the base $b$ of the logarithm. Since $\Theta$ is $K$-invariant and, by Corollary \ref{c:Prsp}, the projection $p:\rspcl{(\frakT\cap\calM_{G})}\to\rspcl{(\Xi\frakT)}$ is surjective, $\Theta$ descends to a well defined map $\Theta:\rspcl{(\Xi\frakT)}\to(\fap)^\Gamma$.
	
	In order to verify its continuity it is enough to check continuity,
	for all $\alpha\in\Delta$ and $\gamma\in\G$ of the maps 
	$$\begin{array}{ccc}
		\rspcl{(\frakT\cap\calM_{G})}&\to&\RR\\
		\rho&\mapsto& \frac{\log_b\alpha(\Jord_\FF(\rho(\gamma)))}
		{\log_b\left(2+\sum_{\eta\in
				F}\tr(\rho(\eta)\rho(\eta)^t)\right)},
	\end{array}$$
	which 
	follows from Proposition~\ref{p.cont} since $\alpha\circ \Jord$ is semialgebraic.
\end{proof}	
The surjective map in Theorem \ref{t:Rspec-ThP} is highly non injective, as the following example shows.
\begin{example}\label{e.twistbend}\
	Let $\frak C_g\subset\Xi(\Gamma_g,\SL_2(\RR))$ be one of the $2^{2g}$ Teichmüller components 
	describing the Teichmüller space of an oriented closed surface $S_g$ of genus $g\geq 2$, here $\G_g=\pi_1(S_g)$.
	Fix a simple, closed, not null homotopic curve $c$ on $S_g$. 
	Then an example of a length function in  $\thp{(\frak C_g)}$ is given by $L(\eta)=i(c,n)$ 
	where $i$ is the intersection between $c$ and the closed curve $n$ represented by  $\eta$.  
	
	Starting from a fixed hyperbolic structure on $S_g$ let $\rho_k^p$, $\rho_k^{tw}$ be sequences in $\calM_F(\G_g,\SL(2,\RR))$
	 where $F$ is a finite generating set of $\G_g$ 
	 so that $\rho_k^p$ corresponds to having pinched the geodesic $c$ to have length $1/k$ 
	 and  $\rho_k^{tw}$ corresponds to  having added $2\pi k$ to the twist parameter at $c$. 
	 Using Theorem \ref{thmINTRO:Procesi}, we see, using the Collar Lemma, 
	 that  in the first case $\pmu=(e^k)_{k\geq 1}$ is a well adapted sequence of scales, 
	 while in the second it is $\plambda=(k)_{k\geq 1}$. 
	 By Theorem \ref{thm:ggt} both $(\rho^p)^\omega_\plambda$ and $(\rho^{tw})^\omega_\pmu$ 
	 represent closed points in $\rsp\partial\frak C_g$ and they both project to the class of the length function $L$. 
	 Let $\g_0\in\G_g$ represent $c$; observe that  
	 $(\tr((\rho^{ p})^\omega_\plambda(\gamma_0))-2)^{-1}$ is comparable to $\plambda^2$ 
	 and is therefore an infinitely large element in the minimal field of $(\rho^{ p})^\omega_\plambda$ 
	 while $(\tr((\rho^{ tw})^\omega_\pmu(\gamma_0))-2)^{-1}$ is a constant sequence of reals. 
	 Therefore $(\rho^{ p})^\omega_\plambda$ and $(\rho^{ tw})^\omega_\pmu$ represent distinct points in $\rsp\partial\frak C_g$.
\end{example}

We now denote by $\thp\partial(\Xi\frakT):=\thp{(\Xi\frakT)}\setminus \Xi\frakT_\RR$ the set of length functions arising in the boundary of the Weyl chamber length compactification of the semialgebraic set $\Xi\frakT$. As a consequence of Proposition~\ref{prop:2} we obtain: 
\begin{cor}\label{c.length}
	For every $[L]\in\thp\partial(\Xi\frakT)$ there exists a continuous locally semialgebraic path $\rho_t:[0,\infty)\to \frakT\cap\calM_{G}$ such that, for every $\gamma\in\Gamma$ we have
	$$\lim_{t\to\infty}\frac{\Ln\Jord_\RR(\rho_t(\gamma))}{t}=L(\gamma).$$ 
\end{cor}
\begin{proof}
	Let $L:\G\to\fap$ be such that $[L]\in\thp\partial\frakT$. It follows from the proof of Theorem \ref{t:Rspec-ThP} that there exists $(\rho,\FF)\in\rspcl{(\frakT\cap\calM_{G})}$ with $L(\rho)=cL$ for some constant $c\in\RR_{>0}$.
	
	Proposition~\ref{prop:2} applied to the continuous, proper, semialgebraic function 
	$$\begin{array}{cccc}
		g:&\frakT\cap\calM_{G}&\to& \KK\\
		&\rho&\mapsto& 2 + \sum_{\eta\in F}\tr(\rho(\eta)\rho(\eta)^t)
	\end{array}$$ 
	ensures that we can choose a path $\pi_\rho:[0,\infty)\to (\frakT\cap\calM_{G})_\RR$ such that
	\begin{enumerate}
		\item $\lim_{s\to\infty}\pi_\rho(s)=\rho$
		\item $g(\pi_\rho(s))=s.$
	\end{enumerate}
	We define $\rho_t:=\pi_\rho(e^{ct})$;
	the claim follows from the continuity of the map $\Theta$ from Equation \eqref{e.Theta}.
\end{proof}

\begin{remark}
	The hypothesis that $\frakT$ avoids representations with bounded
	image is not essential here. It can be removed using a more general 
	construction for the Weyl chamber length compactification we now detail.
	%
	Let $\frakT$ be a 
	closed $G$-invariant semialgebraic subset of $\calR^{red}_F(\G,G)$ and
	$\Xi\frak T$ its image in $\Xi_{F,p}(\G,G)$ which is closed as well.
	We first use the continuous map
	\bqn
	\mathbb P L:\Xi\frakT_\RR-\Xi(\G, K_\RR)\to \mathbb P\left((\fap)^\Gamma\right),\\
	\eqn
	which is defined on the complement of the  compact subset $\Xi(\G,
	K_\RR)$  of bounded representations,
	to glue the space $\mathbb P\left((\fap)^\Gamma\right)$ at the
	infinity of $\Xi\frakT_\RR$:
	More precisely, we  endow the disjoint union $\Xi\frakT_\RR \sqcup \mathbb P\left((\fap)^\Gamma\right)$ with the topology
	extending the topology of $\Xi\frakT_\RR$ such that a fundamental set
	of neighbourhoods of $[L]\in \mathbb P\left((\fap)^\Gamma\right)$ is given by the subsets
	$$U(O,\frakK):= \left({\mathbb PL}^{-1}(O)-\frakK \right) \sqcup O$$
	where $O$ runs through neighbourhoods of $[L]$ in
	$\mathbb P\left((\fap)^\Gamma\right)$ and $\frakK$ runs through compact subsets of $\Xi\frakT_\RR$.
	Then we define the \emph{Weyl chamber length compactification}  
	$\thp{(\Xi\frakT)}$ of $\Xi\frakT_\RR$  as the closure of
	$\Xi\frakT_\RR$ in $\Xi\frakT_\RR \sqcup \mathbb P\left((\fap)^\Gamma\right)$.
	Note that this space is metrizable and that
	a sequence $([\rho_k])_{k\geq1})$ in
	$\frakT_\RR$ converges to some $[L]\in \mathbb
	P\left((\fap)^\Gamma\right)$ if and anly if $[\rho_k]$ eventually gets
	out any compact subset of $\frakT_\RR$ and $\mathbb P L ([\rho_k])$
	converges to $[L]$ in  $\mathbb P\left((\fap)^\Gamma\right)$.
	In particular, it is isomorphic to the compactification defined in
	\cite{Parreau12}, and, if $\frak T_\RR$ does not contain any bounded
	representation, to the compactification $\thp{(\Xi\frakT)}$ defined in
	page \pageref{n.62}, by identifying boundary point $[L]$ with $(\infty,[L])$.
	
	The conclusions  of Theorem \ref{t:Rspec-ThP} then holds
	up to replacing the map $\widehat{\mathbb P L}$ by the map
	$$\widetilde{\mathbb P L}:\rspcl{(\Xi\frakT)}\to \thp{\Xi\frakT}$$
	extending the identity on $\Xi\frakT$, and defined by
	$\widetilde{\mathbb P L}([\rho])={\mathbb P L}([\rho]) \in \mathbb P\left((\fap)^\Gamma\right)$ for  $[\rho]\in\rspcl\partial\Xi(\G,G)$.
	Indeed
	$\mathbb P L$ extends to  a well-defined
	$\Out(\Gamma,\frakT)$-equivariant map
	\bqn
	\mathbb P L:\rspcl{(\Xi\frakT_\RR)}-\Xi(\G, K_\RR)\to \mathbb P\left((\fap)^\Gamma\right),\\
	\eqn
	which is continuous by the arguments in the proof of Theorem
	\ref{t:Rspec-ThP}, and the continuity of the map
	$$\widetilde{\mathbb P L}:\rspcl{(\Xi\frakT)}\to \Xi(\G,G_\RR) \sqcup \mathbb P\left((\fap)^\Gamma\right)$$
	follows.
	Moreover, since $\Xi\frakT_\RR$ is dense in $\rspcl{(\Xi\frakT)}$, the image of $\widetilde{\mathbb P L}$ is the closure of $\Xi\frakT_\RR$ in $\Xi\frakT_\RR \sqcup \mathbb P\left((\fap)^\Gamma\right)$, namely $\thp{(\Xi\frakT)}$.
	
	Corollary \ref{c.length} follows with no modifications.
	In the sequel, the following results are still valid in that context:
	Corollary  \ref{c.integrallengths},
	Theorem \ref{t.lengthA},   modulo the analoguous
	modifications (namely, replacing $\fap$ by $\RR_{\geq 0}$ and $L$ by
	$\ell_N$ in the above), and Corollary \ref{c.denseZNlenght}.
	%
	%
	%
\end{remark}

\subsection{Finiteness properties of length functions}\label{s.discrete}
In this section we will establish a certain finiteness property of length functions associated to points in the boundary $\rsp\partial(\Xi\frakT):=\rsp{(\Xi\frakT)}\setminus(\Xi\frakT)_\RR$ where $\frakT\subset\calR_F(\G,G)$ is a closed, $G$-invariant, semialgebraic set. To this end, for a subgroup $\Lambda\subset \RR$ of the additive group $\RR$ we define  the \emph{$\Lambda$-valued weight lattice}
$$\faL:=\{v\in\fa|\; d\alpha(v)\in\Lambda,\; \forall \alpha\in\Phi\}.$$ 


\begin{prop}\label{p.finlen}
	Let $\FF$ be real closed admitting a big element $b$, $\rho:\G\to\bG(\FF)$,  $m=\dim\bG$. Denoting by $\LL_\rho<\FF$ the field generated over $\KK$ by all the matrix coefficients of $\rho$ and $\valu_\rho:=\log_b(\LL_\rho)$, we have 
	$$\Log_b(\Jord_\FF(\rho(\g)))\in 
	\mathfrak a_{\frac1{2m!}\valu_\rho}.$$
	In particular if $\FF=\FF_\rho$ is $\rho$-minimal, the Weyl chamber valued length function associated to $\rho$ takes values in the finite dimensional $\QQ$-vector space $\mathfrak a_{\Lambda_\rho}$.
\end{prop}
\begin{cor}\label{c.integrallengths}
	The set of length functions in $\thp\partial(\Xi\frakT)$ admitting a representative taking values in the weight lattice $\mathfrak a_\ZZ$ 
	is dense.
\end{cor}
\begin{proof}[Proof of Corollary \ref{c.integrallengths}]
	This follows immediately from the density in $\rspcl\partial(\Xi\frakT)$ of rational points, see Corollary \ref{c.Bru4.2}.
\end{proof}
Proposition \ref{p.finlen} follows from the following lemma, together with Proposition \ref{p.valuautom}
\begin{lem}\label{l.discJord}
	Let $\KK\subset\LL\subset\FF$ with $\FF$ real closed admitting a big element $b$, and $\LL$ any field. Let $\valu:=\log_b(\LL)$. Then for every $g\in\bG(\LL)$
	$$\Log_b(\Jord_\FF(g))\in \mathfrak a_{\frac1{2m!}\valu}.$$
\end{lem}	
\begin{proof}
	Let $g\in\bG(\LL)$; we know from Section~\ref{subsec:Jordan} that the refined Jordan decomposition of $g$ in $\bG(\FF)$ is obtained from the classical multiplicative Jordan decomposition
	$$g=g_s\cdot g_u$$
	where $g_s,g_u\in\bG(\LL)$ are the semisimple and the unipotent part \cite[Theorem 4.4]{Borel1}, and  $g_s=g_h\cdot g_e$ is the decomposition in $\bG(\FF)$ into hyperbolic resp. elliptic part. Let $\ov \FF=\FF[\sqrt{-1}]$. Then $\Ad(g_s)$ is diagonalizable over $\ov \FF$, and since $\Ad(g_s)=\Ad(g_h)\Ad(g_e)$,  the eigenvalues of $\Ad(g_s)$  are of the form  $\rho\cdot\lambda$, $\ov \rho\cdot\lambda$, where $\rho\in\ov\FF$, $|\rho|=1$  is an eigenvalue of $\Ad(g_e)$ and $\lambda\in\FF_{>0}$ is an eigenvalue of $\Ad(g_h)$.
	
	Since $g_h$ is conjugate to $\Jord_\FF(g)\in C_\FF$, the set  
	$$\{\alpha(\Jord_\FF(g))^2|\; \alpha\in \Phi\}$$
	coincides with the products $(\rho\cdot\lambda)(\ov\rho\cdot\lambda)$ of the eigenvalues of $\Ad(g_s)$ and  hence it is contained in an extension $\EE(g)\subset\FF$ of $\LL$ of degree at most $m!$.
	
	Hence $$m!\log_b(\alpha(\Jord_\FF(g))^2)\in\valu$$
	for every $\alpha$ in $\Phi$ \cite[XII \S4 Proposition 12]{Lang} which implies 
	$$\Log_b(\Jord_\FF(g))\in\frac1{2m!}\mathfrak a_{\valu}.$$
\end{proof}
\begin{remark} One can construct examples where $\LL_\rho=\KK(T,S)$ with the ordering in Example \ref{ex:weirdval} so that $\log_b(\LL_\rho)=\QQ$. In particular $\mathfrak a_{\frac1{2m!}\valu_\rho}$ is not a finitely generated group. We don't know if the set $\{\Log_b(\Jord_\FF(\rho(\g)))|\; \g\in\G\}$ is always contained in a finitely generated group.
\end{remark}	
\subsection{Infinite generation for length functions of real points}\label{s.realnotfinite}
Let, as in Section~\ref{s.par}, $\frakT$ be a closed $G$-invariant semialgebraic subset of $\calR_F^{red}(\G, G)$, $\Xi\frakT$ its image in the character variety $\Xi_{F,p}(\G,G)$; assuming that all representations in $\frakT$ are unbounded, we have a well defined map (see \eqref{e.PL})
\bqn
\mathbb P L:\rspcl{(\Xi\frakT)}\to\mathbb P\left((\fap)^\Gamma\right)
\eqn
which we restrict to the Archimedean points $(\Xi\frakT)_\RR$. Observe that the latter is the image 
in the character variety $\Xi_{F,p}(\G,G_\RR)$ of the $G_\RR$-invariant semialgebraic subset $\frakT_\RR$   of $\calR_F^{red}(\G, G_\RR)$.

When $\G=\pi_1(S)$, for a compact surface $S$ of genus $g\geq 2$,  $\bG=\SL_2$, and $\Xi\frakT_\RR$ 
the Teichm\"uller space, it is a non-trivial fact that ${\mathbb P L(\Xi\frakT_\RR)}$ is open in $\ov{\mathbb P L(\Xi\frakT_\RR)}$ 
and the injective map 
\bqn
\mathbb P L:\Xi\frakT_\RR\to\mathbb P L(\Xi\frakT_\RR)
\eqn
is proper. We will show that this is a general phenomenon for non-elementary representations.

\begin{defn}\label{def:elementary_rep} We say that a reductive representation $\rho:\G\to G$ is \emph{elementary} \label{n.63} 
if the semisimple part of the connected component of the Zariski closure of $\rho(\G)$ is $\KK$-anisotropic.
\end{defn}

Taking into account that $\bG$ is defined over $\KK$, this is equivalent to either of the following:
\begin{enumerate}
	\item The Zariski closure in $\bG(\RR)$ of $\rho(\G)$ is a compact extension of a virtually Abelian group.
	\item The Zariski closure in $\bG(\RR)$ of $\rho(\G)$ is amenable.
\end{enumerate}	
\begin{thm}\label{thm:realnotfinite}
	Assume $\frakT$ (or equivalently $\frakT_\RR$) consists of non-elementary representations. Then $\mathbb P L({(\Xi\frakT_\RR)})$ is open in $\ov{\mathbb P L({(\Xi\frakT_\RR)})}$ and the map
	\bqn
	\mathbb P L:{\Xi\frakT_\RR}\to\mathbb PL(\Xi\frakT_\RR)
	\eqn
	is proper.	
\end{thm}
The proof uses three ingredients: the continuity of the map  $\mathbb P L:\rspcl{(\Xi\frakT)}\to\mathbb P\left((\fap)^\Gamma\right)$
from Theorem \ref{t:Rspec-ThP}, the finiteness property in Proposition \ref{p.finlen}, 
and Corollary 1 and Proposition 1 in Prasad-Rapinchuk \cite{PraRap}. 
We use the latter to show the next Proposition.
Recall that $\bG\subset \SL_n$ is defined over $\KK$. 
We keep the notation of Section~\ref{subsec:Jordan}:  $\bS$ is a maximal $\KK$-split torus, $C$ is a closed multiplicative Weyl chamber, 
$\fa=\Lie(\bS(\RR))$, $\Jord_\RR:\bG(\RR)\to C_\RR$ is the Jordan projection and $\Phi$ is the set of $\KK$-roots of $\bS$ in $\bG$.

\begin{prop}\label{p.infdim}
	Let $\rho:\G\to\bG(\RR)$ be non-elementary and $\G$ finitely generated. Then 
	$$\{\Ln(\Jord_\RR(\rho(\g)))|\; \gamma\in\Gamma\}$$
	generates an infinite dimensional $\QQ$-vector subspace of $\fa$. 
	Furthermore if $\chi\in X(\bS)$ is larger than $1$ on the interior of $C$, then the same conclusion holds for  
	$$\{\ln\chi(\Jord_\RR(\rho(\g)))|\; \gamma\in\Gamma\}<\RR.$$
\end{prop}		
\begin{remark}  The above result extends to the following more general situation.
Let $R_1,\dots,R_r$ be virtually solvable subgroups of $\Gamma$ and let $\calC\subset\Gamma$ be a subset such that 
\be
\item no elements of $\calC$ is conjugate to one in $\cup_{j-1}^rR_j$;
\item for any element in $\gamma\in\Gamma$ not conjugate to one in $\cup_{j-1}^rR_j$
there is $n\in\ZZ\smallsetminus\{0\})$ with $\gamma^n\in\calC$.
\ee
Then if $\rho\colon\Gamma\to\bG(\RR)$ is non-elementary,
$\{Ln(J_\RR(\rho(\gamma))):\,\gamma\in\calC\}$ generates an infinite dimensional $\QQ$-vector space of $\fa$.
A typical example of this situation is when $\Gamma=\SO(n,1)$ is, say, a non-uniform lattice with $r$ non-equivalent cusps,
$R_1,\dots, R_r$ the corresponding stabilizers and $\calC$ the set of primitive hyperbolic elements.
\end{remark}
\begin{proof}[Proof of Proposition \ref{p.infdim}]
	Let $\MM$ be the field generated by the matrix coefficients of $\rho(\G)$; 
	then $\MM$ is finitely generated and we may assume, by exchanging $\MM$ with a larger finitely generated field, 
	that $\bG$ is defined over $\MM$. The Zariski closure $\bH$ of $\rho(\G)$ is defined over $\MM$ 
	and so are $\bH^\circ$ and its derived subgroup $\calD(\bH^\circ)$. 
	Define $\Gamma_1:=\rho^{-1}(\bH^\circ)$ and $\Gamma_2=[\Gamma_1, \Gamma_1]$. 
	Then $\rho(\Gamma_2)\subset \calD(\bH^\circ)$ is Zariski dense in $\calD(\bH^\circ)$ 
	and let $\Lambda<\rho(\Gamma_2)$ be Zariski dense and finitely generated. Let $\bL:=\calD(\bH^\circ)$.
	
	Then the results of \cite{PraRap} apply to $\Lambda<\bL(\MM)$. 
	
	Assume by contradiction that the $\bQ$-vector space spanned by 
	\bqn
	\{\Ln(\Jord_\RR(\rho(\g)))|\; \gamma\in\Gamma\}
	\eqn
	is finite dimensional. Then so is the subspace spanned by 
	\bqn
	\{\Ln(\Jord_\RR(\gamma))|\; \gamma\in\Lambda\}.
	\eqn
	Let $m-1$ be its dimension. 
	Then given any $n$-tuple $\gamma_1,\ldots,\gamma_m$ in $\Lambda$ there are integers $(a_1,\ldots,a_m)\neq (0,\ldots, 0)$ with
	\bqn
	\sum_{i=1}^m a_i\Ln(\Jord_\RR(\gamma_i))=0,
	\eqn
	hence 
	\bqn
	\prod_{i=1}^m \Jord_\RR(\gamma_i)^{a_i}=\Id.
	\eqn
	Now let $\gamma_1,\ldots, \gamma_m\in\Lambda$ be $\RR$-regular elements given by \cite[Corollary 1]{PraRap} 
	and to which \cite[Proposition 1]{PraRap} applies: a particular case of it is that if $\bT_i$ is the maximal torus 
	$\bT_i:=\calZ_\bG(\gamma_i)^\circ$, and $\chi_i\in X(\bT_i)$ is a non-trivial character, 
	then any relation of the form $\prod_{i=1}^m \chi_i(\gamma_i)^{s_i}=1$ with $s_i\in\ZZ$ implies $s_1=\ldots=s_m=0$.
	
	Any torus $\bT$ as above has the following property: $\bT$ is maximal in $\bL$ defined over $\RR$ 
	and its split part $\bT_s$ is maximal $\RR$-split in $\bL$; all such tori are $\bL(\RR)$-conjugate. Let 
	\bqn
	\bT=\bT_s\cdot \bT_a
	\eqn
	be the decomposition into almost direct product where $\bT_a$ is the largest anisotropic subtorus in $\bT$. 
	Let $e:=|\bT_s\cap \bT_a|$. Then for every $\chi\in X(\bT_s)$, $\chi^{e}$ extends uniquely to a character of $\bT$ trivial on $\bT_a$ 
	and on any torsion element in $\bT_s(\RR)$.
	
	Now the Jordan decomposition of $\gamma_i$ is obtained as follows (see \cite{BorelClass}): 
	write 
	\bq\label{e.JorT}\bT_i(\RR)=\bT_{i,s}(\RR)^\circ\cdot K_i\eq
	where $K_i$ is the largest compact subgroup of $\bT_i(\RR)$. Observe that $K_i$ is the product of the 2-torsion in $\bT_{i,s}(\RR)$ with $\bT_{i,a}(\RR)$. The decomposition in Equation \eqref{e.JorT} is a direct product, and 
	$$\gamma_i=(\gamma_i)_h\cdot(\gamma_i)_e$$
	is the decomposition into hyperbolic and elliptic part. Next let $g_i\in\bG(\RR)$ such that $g_i\bT_{i,s}g_i^{-1}\subset \bS$ and $\Jord(\gamma_i)=g_i(\gamma_i)_hg_i^{-1}\in C_\RR$.
	
	Now consider $p:=\prod_{\alpha\in\Phi^+}\alpha$ and let $\chi_i$ be the character obtained by extending to $\bT_i$ the character defined in $\bT_{i,s}$ by 
	$$t\mapsto p(g_itg_i^{-1})^{2e}.$$
	Then, since the character is trivial on $K_i$ we get:
	$$\begin{array}{rl}
		\chi_i(\gamma_i)&=p(g_i(\gamma_i)_hg_i^{-1})^{2e}\\
		&=p(\Jord(\gamma_i)^{2e})>1
	\end{array}$$
	and we get applying $p$ to
	$$\prod_{i=1}^m\Jord(\gamma_i)^{a_i}=\Id$$
	that
	$$\prod_{i=1}^m\chi_i(\gamma_i)^{a_i}=\Id$$
	which implies by \cite[Proposition 1]{PraRap} that $a_1=\ldots=a_m=0$.
	
	The second claim follows along the same lines.
\end{proof}	
\begin{proof}[Proof of Theorem \ref{thm:realnotfinite}]
	We have 
	\bqn
	\mathbb P L (\rspcl{(\Xi\frakT)})=\ov{\mathbb P L(\Xi\frakT_\RR)}.
	\eqn
	Indeed $\mathbb P L$ is continuous, $\rspcl{(\Xi\frakT)}$ is compact, 
	and the image $(\Xi\frakT)_\RR$ of $\frakT_\RR$ is dense in $\rspcl{(\Xi\frakT)}$. 
	Let $\partial\rspcl{(\Xi\frakT)}:=\rspcl{(\Xi\frakT)}\setminus\Xi\frakT_\RR$, which is compact. Clearly
	\bqn
	\mathbb P L(\partial\rspcl{(\Xi\frakT)})\supset \ov{\mathbb P L(\Xi\frakT_\RR)}\setminus\mathbb P L(\Xi\frakT_\RR).
	\eqn
	If there were not equality, we would have a pair $(\rho,\FF)$ representing a point in $\partial\rspcl{(\Xi\frakT)}$, 
	a pair $(\pi,\RR)$ in $\frakT_\RR$ and a constant $c>0$ 
	such that $\Log_b(\Jord_\FF(\rho(\g)))=c\Ln(\Jord_\RR(\pi(\gamma)))$, 
	which by Proposition \ref{p.infdim} would lead to a contradiction to Proposition \ref{p.finlen}.
	
	Hence $\mathbb P L(\partial\rspcl{(\Xi\frakT)})= \ov{\mathbb P L(\Xi\frakT_\RR)}\setminus\mathbb P L(\Xi\frakT_\RR)$ 
	and hence $\mathbb P L(\Xi\frakT_\RR)$ is open in $ \ov{\mathbb P L(\Xi\frakT_\RR)}$.
	
	Concerning properness, let $M\subset  {\mathbb P L(\Xi\frakT_\RR)}$ be compact.
	Then $F=\mathbb P L^{-1}(M)$ is a closed subset of $\Xi\frakT_\RR$.  If it were not compact,
	there would exist a sequence $(x_n)_{n\in\NN}\in F$ leaving any compact subset of $\Xi\frakT_\RR$.
	Let $y\in\partial\Xi^\mathrm{RSp}_{\mathrm{cl}}$ be a limit point of a subsequence $(x_{n_\ell})_{\ell\in\NN}$.
	Then $\mathbb P L(y)=\lim_{\ell\to\infty}\mathbb P L(x_{n_\ell})\in M$,
	which, by Proposition~\ref{p.finlen} and Proposition~\ref{p.infdim} leads to a contradiction.
\end{proof}
\subsection{Real valued length functions}\label{s.realvaluedlength}
In this section we discuss real valued length functions associated to reductive representations and establish results analogous to the ones in \S\ref{s.par} and \ref{s.discrete} for the corresponding length function compactifications.

With the concepts introduced in Example~\ref{e.fundExSemialgebraicNorm}, 
let $N\colon S\to \KK_{\geq1}$ be the semialgebraic norm 
associated to the highest weight of an absolutely irreducible representation of $\bG$ with finite kernel.

Let $\FF$ be a real closed field admitting a big element $b\in \FF_{> 1}$ and $\rho:\Gamma\to G_\FF $ a reductive representation. Define
\bqn
\ell_N(\rho)(\g):=\log_bN_\FF(\Jord_\FF(\rho(\g))), \quad \g\in\G
\eqn
where as usual if $\FF\subset\RR$, $\log_b$ denotes the natural logarithm.

It follows then from Lemma \ref{lem:unbounded} and the fact that for $w\in\fa$, $\|w\|:=\ln N_\RR(\exp(w))$ defines a norm on $\fa$ 
that, if $\rho:\G\to G_\RR$ is reductive and unbounded, $\ell_N(\rho):\G\to\RR_{\geq 0}$ is not identically zero. 
Since obviously $\ell_N(\rho)$ only depends on the $G$-conjugacy class of $\rho$, 
we will use from now on the notation $\ell_N([\rho])$, where $[\rho]\in\Xi_{F,p}(\G,G_\RR)$.

Let $\frak T\subset \calR_F^{red}(\G,G)$ be a closed, $G$-invariant, semialgebraic subset and $\Xi\frak T$ its image in $\Xi_{F,p}(\G,G)$. 
Then if $\frak T_\RR$ does not contain any bounded representation $\ell_N([\rho])\in \RR^\G_{\geq 0}$ 
defines an element $\mathbb P \ell_N([\rho])\in\mathbb P(\RR^{\G}_{\geq0})$.

This leads, as in \S\ref{s.par}, to a well defined map 
\bqn
\mathbb P\ell_N:\rspcl{(\Xi\frak T)}\to\mathbb P(\RR^\G_{\geq 0}). 
\eqn
The length function compactification $\nl{(\Xi\frak T)}$ of $\Xi\frak T_\RR$ is obtained 
as the closure of the image of $\Xi\frak T_\RR$ under the embedding

\bqn\begin{array}{ccccc}
	&	\Xi\frakT_\RR&\to&\widehat{\Xi\frakT_\RR}\times\mathbb
	P\left((\RR_{\geq 0}^\Gamma)\right)\\
	&[\rho]&\mapsto&([\rho],\mathbb P\ell_N([\rho])).
\end{array}\eqn
We have then a map
\bqn
\widehat{\mathbb P \ell_N}:\rspcl{(\Xi\frakT)}\to\nl{(\Xi\frakT)}\eqn
constructed as in \S\ref{s.par}, and
\begin{thm}\label{t.lengthA}
	For any $G$-invariant, closed, semialgebraic subset $\frak T\subset\calR^{red}_F(\G,G)$ such that $\frak T$ avoids bounded representations, the map 
	\bqn
	\widehat{\mathbb P \ell_N}:\rspcl{(\Xi\frakT)}\to\nl{(\Xi\frakT)}
	\eqn
	is continuous, $\Out(\G,\frak T)$-equivariant, surjective.
\end{thm}	

Now we turn to finiteness properties of the length functions in $\nl\partial\Xi\frak T:=\nl{(\Xi\frak T)}\setminus(\Xi\frak T)_\RR$.
With the notations of \S\ref{s.discrete},  Proposition \ref{p.finlen} gives:
\begin{prop}\label{prop.lengthB}
	Let $l$ be the index in $X(\bS)$ of the subgroup generated by $\Phi$. Then 
	$$\log_b\chi_\FF(\Jord_\FF(\rho(\g)))\in\frac1{2lm!}\valu_\rho,\quad \forall \chi\in X(\bS)$$
	where $\valu_\rho=\log_b(\LL_\rho)$ is as in Proposition \ref{p.finlen}. \end{prop}	
\begin{proof}
	Raising to the $l^{th}$-power we see that for every $\chi\in X(S)$ and every $s\in S$
	$$d\chi(\Log_b(s))=\log_b(\chi_\FF(s)).$$
	Thus for every $\g\in\G$,
	$$\log_b(\chi_\FF(\Jord_\FF(\rho(\g))))=d\chi_\RR(\Log_b(\Jord_\FF(\rho(\g)))).$$
	With Proposition \ref{p.finlen} this implies that 
	$$\log_b(\chi_\FF(\Jord_\FF(\rho(\g))))\in d\chi_\RR(\fa_{\frac 1{2m!}\valu_\rho}).$$
	and applying again that $\chi^l$ is a product of elements in $\Phi$ we get
	$$d\chi_\RR(\fa_{\frac 1{2m!}\lambda_\rho})\subset\frac 1{2lm!}\valu_\rho$$
	which shows the proposition.
\end{proof}	
With the same arguments as in Corollary \ref{c.integrallengths} we obtain
\begin{cor}\label{c.denseZNlenght}
	The set of length functions in $\nl\partial(\Xi\frakT)$ admitting a representative taking values in $\ZZ$ 
	is dense.
\end{cor}	
Using Theorem \ref{t.lengthA}, Proposition \ref{p.infdim} and Proposition \ref{prop.lengthB} we deduce the analogue of Theorem \ref{thm:realnotfinite}:
\begin{thm}\label{thm:PPLNproper}
	Assume $\frakT\subset\calR_F^{red}(\G,G)$ is a closed, $G$-invariant semialgebraic set consisting of non-elementary representations. Then $\mathbb P \ell_N({(\Xi\frakT)_\RR})$ is open in $\ov{\mathbb P \ell_N({(\Xi\frakT_\RR)})}$ and the map
	\bqn
	\mathbb P \ell_N:{(\Xi\frakT)_\RR}\to\mathbb P\ell_N((\Xi\frakT)_\RR)
	\eqn
	is proper.	
\end{thm}
\subsection{Length functions at fixed points for  $\Out(\Gamma,\frakT)$}\label{sec.fixp2}
We now turn to general properties of length functions of points fixed by elements of $\Out(\Gamma,\frakT)$. 
To this aim, let $(\rho,\FF_\rho)$ represent a point in $\rsp\partial\Xi(\G,G)=\rsp{\Xi(\G,G)}\setminus\Xi(\G,G_\RR)$, 
then $\FF_\rho$ is real closed, non-Archimedean and $\rho$-minimal. 
We let as usual $\log_b:\FF_\rho\to\RR\cup\{\infty\}$ denote the valuation given by the choice of a big element $b\in\FF_\rho$.

Let $\psi\in\Aut(\G)$.  By Proposition \ref{p.charrsp}(3) the point in  $\rsp\partial\Xi(\G,G)$ defined by $(\rho,\FF_\rho)$ is $\psi$-fixed if and only if there exists $g\in G_{\FF}$, $\FF=\FF_\rho$ and $\alpha\in\Aut(\FF)$ with 
$$\rho\circ\psi(\gamma)=g\alpha(\rho(\gamma))g^{-1}\quad \forall \g\in\Gamma.$$

Recall from Equation \eqref{e.PL} in Section~\ref{s.par} that the length function associated to $\rho$ is given by 
$$L(\rho)(\g):=\Log_b(\Jord_\FF(\rho(\g)))\in\fap.$$
Then we have
\begin{prop}
	If $[(\rho,\FF)]\in\rsp\partial\Xi(\G,G)$ is $\psi$-fixed, and $\alpha\in\Aut(\FF)$ is the corresponding automorphism, then for all $\g\in\G$
	$$L(\rho\circ\psi)(\g)=\log_b(\alpha(b))L(\rho)(\g)$$
	where $\log_b(\alpha(b))>0$ is the algebraic number given in Proposition \ref{p.valuautom}(2). 
\end{prop}	
\begin{proof}
	Observe that since $\KK\subset\RR$, whe have $\alpha|_\KK=id_\KK$; hence since $K$ and $\bS$ are defined over $\KK$ we have that the sets $\calE_\FF$ of elliptic elements and  $\calH_\FF$ of hyperbolic elements are $\alpha$-stable, as is the set of unipotent elements. In addition since all characters $\beta\in X(\bS)$ are defined over $\KK$, we have $\alpha(\beta(s))=\beta(\alpha(s))$ for all $s\in\bS_\FF$ and in particular $\beta(C_\FF)>1$ for all $\beta\in\Phi^+$. As a result
	$$\alpha(\Jord_\FF(g))=\Jord_\FF(\alpha(g)).$$
	Next observe that for all $s\in\bS_\FF$, $\beta\in X(\bS)$,
	$$\begin{array}{rl}d\beta(\Log_b(\alpha(s)))&=\log_b(\beta(\alpha(s)))\\&=\log_b(\alpha(\beta(s)))\\&=\log_b(\alpha(b))\log_b(\beta(s))
	\end{array}$$
	which implies that 
	$$\Log_b(\alpha(s))=\log_b(\alpha(b))\Log_b(s)$$
	and hence 
	$$\begin{array}{rl}
		L(\rho\circ\psi)(\g)&=\Log_b(\Jord_\FF(\alpha(\rho(\g))))\\
		&=\Log_b(\alpha(\Jord_\FF(\rho(\g))))\\
		&=\log_b(\alpha(b))\Log_b(\Jord_\FF(\rho(\g))).
	\end{array}$$
\end{proof}	

\section{An application for fundamental groups of surfaces: cross-ratios and positively ratioed representations}\label{s.crossratio}
In this section we specialize to the case where $\G<\PSL(2,\RR)$ is a torsion free cocompact lattice, $\bG$ is a semisimple semialgebraic $\KK$-group with $\bG<\SL_d$, and turn our attention to closed semialgebraic subsets of $\Xi(\G,G)$ consisting of $k$-positively ratioed representations where $2k\leq d$. This property of a representation $\rho:\G\to\SL_d(\RR)$ allows, following Martone-Zhang \cite{Martone_Zhang}, to associate a geodesic current to $\rho$ whose intersection function will be a length function of $\rho$ coming from a semialgebraic norm associated to a highest weight of $\SL_d$.

The notion of $k$-positively ratioed representations can be defined over any real closed field and is in terms of a  cross-ratio on a certain flag manifold. {We will show that the cross-ratio map extends continuously to the real spectrum compactification and deduce that, for $k$-positively ratioed representations, integral length functions  correspond to intersection with  weighted multicurves. }

More precisely, let $\Delta=\{\alpha_1,\ldots,\alpha_{d-1}\}$ be the standard set of simple roots of $\SL_d$,  $k\in\NN$ with $2k\leq d$, and $I=\{\alpha_k,\alpha_{d-k}\}$. Then we have the identification  (see notations of \S\ref{s.proximal})
$$\calF_I(\FF):=\{(a, A)\in\Gr_k(\FF^d)\times\Gr_{d-k}(\FF^d), \; a\subset A\}.$$
where $\Gr_l(\FF^d)$ denotes the Grassmannian of $l$-dimensional  subspaces of $\FF^d$.  The corresponding standard parabolic $\bP_I$ is the stabilizer of $(\langle e_1,\ldots,e_k\rangle, \langle e_1,\ldots, e_{d-k}\rangle)$.

The $k$-th multiplicative cross-ratio is the function defined on the set $\calF_I^{(4)}$ of pairwise transverse  4-tuples of flags by
%
$$\begin{array}{clcl}
	\CR_k:&\quad\calF_I^{(4)}(\FF)&\to&\FF\\
	&((a,A),(b,B),(c,C),(d,D))&\mapsto& \displaystyle\frac{\ov a\wedge \ov C}{\ov a\wedge\ov B}\frac{\ov d\wedge\ov B}{\ov d\wedge\ov C}.
\end{array}
$$
where $\ov a$, $\ov d$ are lifts of $a,d$ to $\wedge^k\FF^d$,  $\ov B$, $\ov C$ are lifts of $B,C$ to $\wedge^{d-k}\FF^d$ and we identify $\wedge^d\FF^d$ with $ \FF$ by sending $e_1\wedge\ldots\wedge e_d$ to 1. It is easy to verify that $\CR_k$ doesn't depend on any such choice, and therefore defines an algebraic map $\calF_I^{[4]}(\FF)\to\FF$.

Now we identify the boundary of $\Gamma$ with the boundary $\partial\mathbb H$ of the upper half plane and let $\calH_\Gamma\subset\partial\mathbb H$ be the set of fixed points of elements in $\G\setminus \{e\}$. 

\begin{defn}
	A representation $\rho:\Gamma\to\SL_d(\FF)$ is \emph{$k$-positively ratioed} if \label{n.64} 
	\be
	\item it admits an $I$-dynamics preserving framing 
	$\phi:\calH_\G\to \calF_I(\FF)$ (see Definition \ref{d.dynframing}),
	\item $\phi^*\CR_k$ is positive: for every positively oriented 4-tuple $(\xi_1,\xi_2,\xi_3,\xi_4)\in\calH_\G^4$
	$$\CR_k(\phi(\xi_1),\phi(\xi_2),\phi(\xi_3), \phi(\xi_4))>1.$$
	\ee
\end{defn}	

\begin{example} Some of the classes discussed in Examples \ref{e.confspa2} give examples of positively ratioed representations; we keep a numbering consistent with Example \ref{e.confspa2}.
	\be
	\item[(2a)] Maximal representations $\rho:\Gamma\to\Sp_{2n}(\FF)<\SL_{2n}(\FF)$  are $n$-positively ratioed for every real closed field $\FF$: indeed the restriction of $\CR_n$ to $\calL(\FF^{2n})$ is the cross-ratio  considered in \cite{BP}. 
	\item[(2b)] Hitchin representations $\rho:\Gamma\to\SL_d(\RR)$ 
	are $k$-positively ratioed for all $k$ \cite{Martone_Zhang}.
	\item[(2c)] $\Theta$-positive representations $\rho:\Gamma\to\SO(p,q)<\SL_{p+q}(\RR)$ with $1\leq p\leq q$ 
	are $k$-positively ratioed for all $1\leq k\leq p-1$ \cite{BeyPSO}.
	\item[(3)] Representations $\rho:\Gamma\to\SL_d(\RR)$  satisfying property $H_k$ are $k$-positively ratioed \cite{BeyP}.
	\ee
\end{example}	

Let now as in \cite{BIPP-PCBT} $\calH_\G^{[4]}$ be  the set of positively oriented quadruples in $\calH_\G$. Assume that $\FF$ admits a big element  $b$ and define
$[\cdot,\cdot,\cdot,\cdot]:\calH_\G^{[4]}\to[0,\infty)$ by the formula
$$[\xi_1,\xi_2,\xi_3,\xi_4]:=\log_{b^2}\big(\CR_k(\phi(\xi_1),\phi(\xi_2),\phi(\xi_3),\phi(\xi_4))\CR_k(\phi(\xi_3),\phi(\xi_4),\phi(\xi_1),\phi(\xi_2))\big)$$
then we obviously have:
\bq\label{e.cr1} [\xi_1,\xi_2,\xi_3,\xi_4]=[\xi_3,\xi_4,\xi_1,\xi_2]\quad\text{ for all } (\xi_1,\xi_2,\xi_3,\xi_4)\in\calH_\G^{[4]}
\eq
and a direct computation gives 
\bq\label{e.cr2} [\xi_1,\xi_2,\xi_4,\xi_5]=[\xi_1,\xi_2,\xi_3,\xi_5]+[\xi_1,\xi_3,\xi_4,\xi_5]
\eq
Thus $[\cdot,\cdot,\cdot,\cdot]$ defines a positive cross-ratio on $\calH_\G$ and we proceed to state a formula for the period of any $\g\in\G\setminus\{\Id\}$. To this end, let $\bS_d(\FF)$ be the diagonal subgroup of $\SL_d(\FF)$,  $\chi_k$ the character given by
$$\chi_k(\diag(\lambda_1,\ldots,\lambda_d))=\frac{\lambda_1\ldots\lambda_k}{\lambda_d\ldots\lambda_{d-k+1}},$$
and $\Jord_\FF:\SL_d(\FF)\to C_\FF$ the Jordan projection where 
$$C_\FF=\left\{\left.\diag(\lambda_1,\ldots,\lambda_d)
\right|\; \lambda_1\geq\lambda_2\geq\ldots\geq\lambda_d>0\right\}.$$
Then it follows from the fact that $\rho(\g)$ is $I$-proximal and \cite[Lemma 3.9]{BeyP} that for $\g\in\G\setminus\{\Id\}$ and $x\in\calH_\G\setminus\{\g_-,\g_+\}$,
$$\begin{array}{rl}\per(\g)&=[\g_-,x,\g x,\g_+]\\
	&=2\log_{b^2}\chi_k(\Jord_\FF(\rho(\g)))\\
	&=\log_{b}\chi_k(\Jord_\FF(\rho(\g))).\end{array}$$

Now let 
$$S_d=\left\{\left.\diag(\lambda_1,\ldots,\lambda_d)
\right|\; \lambda_i>0, \lambda_i\in\KK\right\}$$
then $\chi_k:S_d\to\KK_{\geq 1}$ is a highest weight satisfying the symmetry condition (recall Example \ref{e.fundExSemialgebraicNorm}) and hence 
$$\begin{array}{cccc}N_k:&S_d&\to&\KK_{\geq 1}\\&\diag(\lambda_1,\ldots,\lambda_d)&\mapsto&\displaystyle\max_{\sigma\in\calS_d}\chi_k(\diag(\lambda_{\sigma(1)}\ldots,\lambda_{\sigma(d)}))\end{array}$$
is a semialgebraic norm on $S_d$. Modulo conjugating $\bG$ by an element of $\SO(d,\KK)$ we may assume that there is a maximal $\KK$-split torus $\bS\subset\bG$ with $\bS\subset\bS_d$. Then $N_k|_{S}=\nu_k$ gives a semialgebraic norm on $S$ (Lemma \ref{l.NormA}) and Proposition \ref{p.NormB} implies that if $\rho$ takes values in $G$ for the associated length function $\ell_k$ we have
$$\ell_k([\rho])(\g)=\ln\chi_k(\Jord(\rho(\g))).$$

If $\FF$ is real closed with big element $b$ and $\rho:\G\to G_\FF$ is a reductive representation then Proposition \ref{p.NormB} also implies that 
$$\ell_k([\rho])(\g)=\log_b\chi_k(\Jord_\FF(\rho(\g))).$$
Thus it follows from \cite[Theorem 1.2]{BIPP-PCBT}:
\begin{cor}\label{c.posrat}
	Let $G<\SL_d(\KK)$ be a semisimple semialgebraic $\KK$-group and assume $S\subset S_d$. Let $\nu_k$ be the semialgebraic norm on $S$ obtained by restricting 
	$$(\lambda_1,\ldots,\lambda_d)\mapsto\max_{\sigma\in\calS_d}\frac{\lambda_{\sigma(1)}\ldots\lambda_{\sigma(k)}}{\lambda_{\sigma(d)}\ldots\lambda_{\sigma(d-k+1)}}.$$
	Let $\rho:\G\to G_\FF\subset\SL_d(\FF)$ be a positively $k$-ratioed representation where $\FF$ is a real closed field admitting a big element $b$. Then there is a geodesic current $\mu_\rho$ on $\Sigma=\G\backslash\mathbb H$ such that
	$$i(\mu_\rho,\g)=\log_b\nu_k(\Jord_\FF^G(\rho(\g)))$$
	where $\Jord^G$ is the Jordan projection of $G$.
\end{cor}	
Now  we turn our attention to semialgebraic subsets of $\calR_F^{red}(\G,G)$ only consisting of $k$-positively ratioed representations.

\begin{prop}\label{p.posrat}
	Let $\frakT\subset\calR_{F}^{red}(\G,G)$ be a semialgebraic set consisting of $k$-positively ratioed representations. 
	Then for every real closed field $\FF\supset\KK$ the same holds for $\frakT_\FF$. 
	In particular if $\rho\in\calR_F^{red}(\G,G_\FF)$ represents a point in $\rsp\frakT$ then $\rho$ is $k$-positively ratioed.
\end{prop}	

Let $\Prox_I(\SL_d(\FF))$ be the subset of $I$-proximal elements in $\SL_d(\FF)$. 
The following lemma is straightforward using Proposition \ref{p.proximal} and Definition \ref{defi:proximal2}.
\begin{lemma}\label{l.proxF}\
	\begin{enumerate}
		\item$\Prox_I(\SL_d(\FF))$ is semialgebraic and coincides with the $\FF$-extension of $\Prox_I(\SL_d(\KK))$.
		\item The map
		\bqn
		\begin{array}{ccc}\Prox_I(\SL_d(\FF))&\to&\calF_I(\FF)\\g&\mapsto&g_+\end{array}
		\eqn
		is continuous semialgebraic and coincides with the $\FF$-extension of the corresponding map for $\KK$.
		\item The function $\CR_k:\calF_I(\FF)^{(4)}\to\FF$ is semialgebraic continuous 
		and coincides with the $\FF$-extension of the corresponding map for $\KK$.
	\end{enumerate}	
\end{lemma}	
\begin{proof}[Proof of Proposition \ref{p.posrat}]
	Define, for every $\g\in\G\setminus\{\Id\}$
	$$\Prox_\g(\G,G_\FF):=\{\rho\in\calR_F^{red}(\G,G_\FF)|\,\rho(\g)\text{ is $I$-proximal}\}.$$
	In view of Lemma \ref{l.proxF} it is a semialgebraic subset, with $\Prox_\g(\G,G)_\FF=\Prox_\g(\G,G_\FF)$. 
	Since $\frakT\subset\Prox_\g(\G,G)$ for all $\g\in\G\setminus\{\Id\}$, 
	we also have $\frakT_\FF\subset\Prox_\g(\G,G_\FF)$ for all $\g\in\G\setminus\{\Id\}$. 
	Then given $\rho\in\frakT_\FF$, $\phi(\g_+):=\rho(\g)_+$ defines an $I$-dynamics preserving framing $\phi:\calH_\G\to\calF_I(\FF)$. 
	
	Next let $\g_1,\g_2,\g_3,\g_4\in\G$ such that $(\g_{1,+},\g_{2,+},\g_{3,+},\g_{4,+})$ is positively oriented.
	Then 
	\bqn
	\frakT\subset\left\{\rho\in\left.\bigcap_{i=1}^4\Prox_{\g_i}(\G,G)\right|\;\CR_k(\rho(\g_1)_+,\rho(\g_2)_+,\rho(\g_3)_+,\rho(\g_4)_+)>1\right\}		\eqn
	and hence by Lemma \ref{l.proxF}
	\bqn
	\frakT_\FF\subset
	\left\{\rho\in\left.\bigcap_{i=1}^4\Prox_{\g_i}(\G,G_\FF)\right|\;\CR_k(\rho(\g_1)_+,\rho(\g_2)_+,\rho(\g_3)_+,\rho(\g_4)_+)>1\right\}.
	\eqn
	This shows the first assertion of Proposition \ref{p.posrat}; the second is an immediate consequence of the first.
\end{proof}	

Now let $\frakT\subset\calR_F^{red}(\G,G)$ be a closed $G$-invariant semialgebraic subset 
and $\Xi\frakT$ its image in $\Xi_{F,p}(\G,G)$. If $\rho\in\frakT_\FF$ represents a closed point of $\rsp{\Xi\frakT}$ 
then it follows from Theorem \ref{thm:1.2} (3) and Corollary \ref{c.posrat} that the associated geodesic current $\mu_\rho$ does not vanish. 
Observe that $\mu_\rho$ depends on the choice of a big element in $\FF$, 
any other choice of big element would lead to a positive multiple of $\mu_\rho$. Thus we obtain a well defined map
$$\mathbb P\calC:\rspcl{(\Xi\frakT)}\to\mathbb P\mathcal C(\Sigma)$$
into the projectivized space of geodesic currents $\calC(\Sigma)$ on $\Sigma$. \label{n.65} 

For the length function
$$\ell_k([\rho])(\g)=\log_b\nu_k(\Jord_\FF^G(\rho(\g)))$$
associated to the semialgebraic norm $\nu_k$ on $S$, we have

\begin{thm}\label{t.crmain}
	We have a commutative diagram of continuous maps 
	$$\xymatrix{\rspcl{(\Xi\frakT)}\ar[r]^{\mathbb P\calC}\ar[dr]_{\mathbb P\ell_k}&\mathbb P(\calC(\Sigma))\ar[d]^{i(\cdot,\cdot)}\\&\mathbb P(\RR^\G_{\geq 0})}$$
	where the vertical map is given by the intersection.
\end{thm}	
\begin{proof}
	It remains to show the continuity of $\mathbb P\calC$. As in \cite[Section 4.3]{BIPP-PCBT}, 
	let  $\mathcal{CR}(\calH_\G)$ be the topological vector space of cross-ratios on $\calH_\G$ with the topology of pointwise convergence, 
	and $\mathcal{CR}^+(\calH_\G)$ the closed convex cone of positive ones.
	
	To every $\rho\in\calR_\G^{red}(\G,G_\FF)$ representing a closed point in $\rspcl{(\Xi\frakT)}$ we can associate a positive cross-ratio
	$$[\,\cdot\,,\,\cdot\,,\,\cdot\,,\,\cdot\,]_{\rho,b}\in\mathcal{CR}^+(\calH_\G)$$
	where we indicate the dependence on $\rho$ and the big element $b$; this cross-ratio doesn't vanish identically 
	since its periods give the length function $\ell_k([\rho])$. Thus we obtain a well defined map:
	$$\rspcl{(\Xi\frakT)}\to\PP(\mathcal{CR}^+(\calH_\G)):=\fracmod{\RR_{>0}}{\mathcal{CR}^+(\calH_\G)\setminus\{0\}}.$$
	The map $\mathbb P\calC$ is the composition of this map with the map 
	$$\PP(\mathcal{CR}^+(\calH_\G))\to\PP(\calC(\Sigma))$$
	which to a positive cross-ratio associates its geodesic current. 
	The latter map being continuous \cite[Proposition 4.10]{BIPP-PCBT}, it remains to show continuity of the former. 
	To this end we proceed as in the proof Theorem \ref{t:Rspec-ThP} and consider the lift
	\bq\label{e.crfinal}
	\begin{array}{ccc}
		\rspcl{(\frakT\cap\calM_{G})}&\to&\mathcal{CR}^+(\calH_\G)\\
		\rho&\mapsto &\displaystyle 	\frac{[\,\cdot\,,\,\cdot\,,\,\cdot\,,\,\cdot\,]_{\rho,b}}{\log_b\left(2+\sum_{\eta\in F}\tr(\rho(\eta)\rho(\eta)^t)\right)},
	\end{array}
	\eq
	which is independent of the choice of the big element $b$. 
	Now 
	$$ \begin{array}{l}
		[(\g_{1,+},\g_{2,+},\g_{3,+},\g_{4,+})]_{\rho,b}=\\
		\quad=\log_b\big(\CR_k(\rho(\g_1)_+,\rho(\g_2)_+,\rho(\g_3)_+,\rho(\g_4)_+)\CR_k(\rho(\g_3)_+,\rho(\g_4)_+,\rho(\g_1)_+,\rho(\g_2)_+)\big).\end{array}$$
	By Lemma \ref{l.proxF}	the function $f:\frakT\cap\calM_G\to\KK$ given by 
	$$f(\rho)=\CR_k(\rho(\g_1)_+,\rho(\g_2)_+,\rho(\g_3)_+,\rho(\g_4)_+)\CR_k(\rho(\g_3)_+,\rho(\g_4)_+,\rho(\g_1)_+,\rho(\g_2)_+)$$
	is semalgebraic continuous and $f(\rho)>0$. Also $g(\rho)=2+\sum_{\eta\in F}\tr(\rho(\eta)^t\rho(\eta))$ is semialgebraic continuous with $g(\rho)>2$. The continuity of \eqref{e.crfinal} follows then from Lemma \ref{l.proxF} (3) and Proposition  \ref{p.cont}.
\end{proof}

Combining Corollary \ref{c.denseZNlenght}, Theorem \ref{t.crmain} and \cite[Corollary 1.7]{BIPP-PCBT}, stating that if a cross-ratio is integral valued then the associated current is a not necessarily simple multicurve,  we conclude
\begin{cor}\label{c.multicurves}
	Let  $\frakT\subset \Hom_\mathrm{red}(\Gamma,G)$ be a $G$-invariant, {closed}, semialgebraic subset only consisting of positively ratioed representations.  	Then the set of length functions 
	arising as intersection with integral weighted multicurves that are not necessarily simple
	is dense in $\partial^{\nu_kL}{\Xi\frakT}$.
\end{cor}	

\begin{proof}
	Let $[(\rho,\FF)]\in \rspcl{(\Xi\frakT)}$; denote by $\LL_\rho$ the field generated by the matrix coefficients of $\rho$, and assume that $\valu:= v(\LL_\rho)$ is a discrete subgroup of $\RR$. The same argument as in \cite[Theorem 7.3]{BIPP-PCBT} ensure that for all $(\xi_1,\xi_2,\xi_3,\xi_4)\in \calH_\G^{(4)}$,
	$[\xi_1,\xi_2,\xi_3,\xi_4]_{\rho,b}\in \frac 1{(8d)! }\valu$. We immediately deduce from \cite[Corollary 1.7]{BIPP-PCBT} that a multiple of the current associated to $[\,\cdot\,,\,\cdot\,,\,\cdot\,,\,\cdot\,]_{\rho,b}$ is a non-necessarily simple integral weighted multicurve.
\end{proof}	
\section{List of notations and concepts}

\begin{tabular}{p{.10\textwidth} p{.90\textwidth}}
	$\ov\LL^r$& The real closure of the ordered field $\LL$, page \pageref{n.rc}\\
	$\calO$& order convex subring of an ordered field $\FF$, page \pageref{n.2}\\
	$\FF_\calO$& residue field $\calO/\calI$ of an ordered convex subring, page \pageref{n.3}\\
	$\oK$& hyper-$\KK$ field, page \pageref{n.4} \\
	$\calO_{\plambda}$ & order convex subring of $\oK$ associated to $\plambda$, page \pageref{n.5}
	\\$\oKl$& Robinson field, page \pageref{n.6} 
	\\$v$& (non-Archimedean) valuation, page \pageref{n.7}
	\\$v_b$& valuation associated to the big element $b$, page \pageref{n.8}
	\\$S$& semialgebraic set, page \pageref{n.9}
	\\& semialgebraically connected, page \pageref{n.10}
	\\$S_\FF$& $\FF$-extension of the semialgebraic set $S$, page \pageref{n.11}
	\\ $V_\FF(\calO)$& $\calO$ points of the set $V_\FF$, page \pageref{n.12}
	\\$\Rspec(\calA)$& {real spectrum}  of a ring $\calA$, page \pageref{n.13}
	\\$\fp_\alpha$&ideal associated to a prime cone $\alpha$, page \pageref{n.14}
	
\end{tabular}

\begin{tabular}{p{.25\textwidth} p{.80\textwidth}}
	\\$\phi_\alpha$&  $\phi_\alpha:\calA\to \calA/\fp_{\alpha}< \FF_\alpha$, page \pageref{n.15}
\\& constructible set, page \pageref{n.16}
\\& spectral topology, page \pageref{n.17}
\\& specialization, page \pageref{n.18}
	\\$\Ret$& retraction $\Ret:\Rspec(\calA)\to \Rspecc(\calA)$, page \pageref{n.19}
\\& a subring $R_2$ is  Archimedean over $R_1$, page \pageref{n.20}
\\$P\subset \calA$& (proper) cone, page \pageref{n.21}

\\$\rsp V$& the real spectrum compactification of the real algebraic set $V$, page \pageref{n.22}
\\$\rspcl V$& the subset of closed points in $\rsp V$, page \pageref{n.22}
\\$c(S)$& constructible set associated to a semialgebraic subset $S$, page \pageref{n.23}
	\\$\rsp S$& the real spectrum compactification of the semialgebraic set $S$, page \pageref{n.23b}
\\$\rspcl S$& the subset of closed points in $\rsp S$, page \pageref{n.23b}
\\$\rsp{\partial}S$ & non-Archimedean points, page \pageref{n.24}
\\	$\alpha_x$& prime cone associated to $x\in V_\FF$, page \pageref{n.25}
	\\& {rational point in $\rsp V$}, page \pageref{n.26}
	\\$V(\beta)$& support variety of $\beta\in\rsp V$, page \pageref{n.27}\\
	
	$\bG<\SL(n)$& connected, semisimple algebraic group defined over $\KK$, page \pageref{n.28}
	\\$\bS$& maximal $\KK$-split torus, page \pageref{n.29}
	\\$X(\bS)$& rational characters of $\bS$, page \pageref{n.30}
	\\${}_\KK\Phi\subset X(\bS)$& $\KK$-roots of $\bG$ with respect to $\bS$, page \pageref{n.31}
	\\${}_\KK\bP_{I}$& standard parabolic subgroup, page \pageref{n.32}
	\\$\bS_I$& $\left(\bigcap_{\alpha\in {}_\KK\Delta\setminus I}\ker\alpha\right)^\circ$, page \pageref{n.33}
	\\$\calZ(\bS_I)$& Levi $\KK$-subgroup of ${}_\KK\bP_{I}$, page \pageref{n.34}
	\\$G$& semisimple, semialgebraic  $\KK$-group, page \pageref{n.35}
	\\$S$& $\bS(\KK)^\circ$, page \pageref{n.36}
	\\$g\in K_\FF\,c(g)\,K_\FF$& Cartan decomposition, page \pageref{n.37}
	\\$\Jord_\FF$& Jordan projection $\Jord_\FF:G_\FF\to C_\FF$, page \pageref{n.38}
	\\$\xi_g^+$& attracting fixed point of a $I$-proximal element, page \pageref{n.39}
	\\& $I$-proximal over $\FF$, page \pageref{n.40}
	\\$\Ln$& logarithm map $\mathrm{Ln}\colon S_\RR\to\fa$, page \pageref{n.41}
	\\$\Log_b$& Logarithm map $ S_\FF\to\fa$, page \pageref{n.42}
	\\$\calX_\FF$& non-standard symmetric space, page \pageref{n.43}
	\\$\delta_\FF$& Cartan projection $\delta_\FF\colon\calX_\FF\times\calX_\FF\to C_\FF$, page \pageref{n.44}
	\\$N$& semialgebraic norm on $S$, page \pageref{n.45}
	\\$D_N^\FF$& $\FF_{\geq1}$-valued multiplicative distance, page \pageref{n.46}
	\\$d_N^\FF$& semidistance on $\calX_\FF$, page \pageref{n.47}

\end{tabular}

\begin{tabular}{p{.20\textwidth} p{.80\textwidth}}
	\\$\bgof$& metric shadow of $\calX_\FF$, page \pageref{n.48}
\\$\Cone(\calX_\RR, \omega,\prlambda,d_{\|\,\cdot\,\|})$& asymptotic cone, page \pageref{n.49}
\\$\ell(g)$& {translation length}, page \pageref{n.50}
\\$\|g\|$& $\lim_{k\to\infty}\frac{d(x,g_*^kx)}k$, page \pageref{n.51}
\\$\calR_F(\Gamma,\bG(\KK))$& representation variety, page \pageref{n.52} 
\\$\calD_F^\rho(x)$& displacement of $\rho$ at $x$, $\max_{\gamma\in F}d(\rho(\gamma)x,x)$, page \pageref{n.53}
	\\$\xi$& $I$-dynamics preserving framing defined over $\FF$, page \pageref{n.54}
	\\$\mathfrak C$& Configuration space, page \pageref{n.55}
	\\& $(Y,\mathfrak C)$-positive framing, page \pageref{n.56}
	\\$\calR_F^{red}(\Gamma,G)$& set of reductive representations, page \pageref{n.60b}
	\\$\calM_G$& $G$-minimal vectors, page \pageref{n.57}
	\\$\Xi_{F,p}(\G,G)$& character variety, page \pageref{n.59} 
	\\& $(\rho,\FF)$ represents a point $\alpha\in\rsp{\Xi_{F,p}(\Gamma, G)}$, page \pageref{n.60}
	\\& {bounded representation}, page \pageref{n.61}
	\\$\thp{(\Xi\frakT)}$& Weyl chamber length compactification, page \pageref{n.62}
	\\& elementary representation, page \pageref{n.63} 
	\\& $k$-positively ratioed rapresentation, page \pageref{n.64} 
	\\$\calC(\Sigma)$& geodesic currents on $\Sigma$, page \pageref{n.65} 
\end{tabular}

\end{document}